\documentclass{amsart}
\usepackage{cO2}
\theoremstyle{remark}
\newtheorem{exa}[thm]{Example}
\usepackage[all]{xy}
\let\Prad=U
\newcommand\vPrad{\vec\Prad}
\newcommand\bPrad{\mo\Prad}
\begin{document}
\title[Asymptotic expansions II]{Explicit asymptotic expansions in $p$-adic harmonic analysis II}
\subjclass[2000]{Primary 22E50, 22E35}
\keywords
	{$p$-adic group,
	supercuspidal representation,
	orbital integral,
	Harish-Chandra character}

\author{Loren Spice}
\address{Department of Mathematics \\
Texas Christian University \\
Fort Worth, TX 76129}
\email{l.spice@tcu.edu}
\thanks{The author was partially supported by Simons Foundation grant 636151,
and also thanks the American Institute of Mathematics, which supported this research as part of a SQuaRE program, for their hospitality.}

\begin{abstract}
We unwind the induction implicit in the ``one-step'' asymptotic expansions of \cite{spice:asymptotic} to describe how to turn asymptotic expansions, such as (but not limited to) the Harish-Chandra--Howe local character expansion, for depth-\(0\) supercuspidal characters into Kim--Murnaghan-type asymptotic expansions for arbitrary positive-depth, tame, supercuspidal characters.  Doing so requires analogous results for Fourier transforms of orbital integrals on the Lie algebra of a reductive group, which are of independent interest.
\end{abstract}

\maketitle

\tableofcontents

\section{Introduction}
\label{sec:intro}

\subsection{Motivation}
\label{subsec:motivation}

Recent work in the representation theory of \(p\)-adic groups has lent increasing support to the idea, articulated by Paul Sally but due essentially to Harish-Chandra, that ``characters tell all''.  See, for example, the work of Tasho Kaletha \cites{kaletha:simple-wild,kaletha:epipelagic,kaletha:regular-sc} on the supercuspidal local Langlands correspondence, in which character formulas play a crucial role.  However, the difficulty of obtaining character formulas has made it difficult to realise their full potential.  At least for supercuspidal characters (and later for discrete-series representations \cite{rader-silberger:submersion}), Harish-Chandra \cite{hc:harmonic} proved an integral formula, but it can be difficult explicitly to evaluate.  This approach has been used successfully in a few cases \cites{sally-shalika:characters,corwin-moy-sally:gll,adler-corwin-sally:division-formulas,debacker:thesis,spice:thesis,adler-spice:explicit-chars}, but it seems impractical to carry out in a general setting.  For example, this author tried for some time to generalise the work in \cite{spice:thesis} to \(\SL_4\), but was not able to do so.  In retrospect, it seems that direct evaluation of the Harish-Chandra integral is most suited to cases where we have a single-orbit theory of the type that seems first to have been discussed in \cites{murnaghan:chars-u3,murnaghan:chars-sln,murnaghan:chars-gln}.  When this is not the case, it seems most useful to study the behaviour of characters, not \textit{via} the local character expansion in terms of Fourier transforms of nilpotent orbital integrals, but rather \textit{via} a modified expansion in terms of Fourier transforms of orbital integrals having a fixed, or a collection of fixed, semisimple parts.  See \cites{jkim-murnaghan:charexp,jkim-murnaghan:gamma-asymptotic}.

These ``asymptotic expansions'' were initially understood only near the identity.  In \cite{spice:asymptotic}, we showed that such expansions could be constructed and centred at essentially \emph{any} semisimple element, by analogy with the Harish-Chandra local character expansion \cite{adler-korman:loc-char-exp}.  The program developed in most of \cite{spice:asymptotic} built on the homogeneity work of DeBacker \cite{debacker:homogeneity}, especially on its later development by Kim and Murnaghan \cites{jkim-murnaghan:charexp,jkim-murnaghan:gamma-asymptotic}, to give a general framework that could, in principle, be used to study arbitrary characters, and even arbitrary invariant distributions on the group.  See \xcite{spice:asymptotic}*{\S\xref{sec:dist}} and, specifically, \xcite{spice:asymptotic}*{Theorem \xref{thm:dist-G-to-G'}}.  However, implementing this study requires a way of matching distributions on groups and their tame, twisted Levi subgroups.  A general perspective on this matching, perhaps informed by Hecke-algebra isomorphisms, is not yet available, but, in the case of tame, supercuspidal representations constructed by Yu \cite{yu:supercuspidal}, a substitute is built in:  the data Yu use include enough information to construct a representation, not just of a single group, but also of a tower of tame, twisted Levi subgroups.  This provides the needed matching, and so allows us to apply the general framework of \cite{spice:asymptotic} to the study of supercuspidal characters.  The main result is \xcite{spice:asymptotic}*{Theorem \xref{thm:asymptotic-pi-to-pi'}}.

Despite this, there are two questions that are left unanswered by \xcite{spice:asymptotic}*{Theorem \xref{thm:asymptotic-pi-to-pi'}}.  The first is an unexpected technical one.  Namely, there is an apparently trivial mismatch between \cite{gerardin:weil}*{(2.7)}, where a representation of a parabolic subgroup \(P\) of a symplectic group is twisted by a quadratic character \(\chi^{E_+}\) of \(P\), and the statement of \cite{gerardin:weil}*{Theorem 2.4(b)}, where that twist is omitted.  The proof of \cite{yu:supercuspidal}*{Lemma 14.6} uses the version without the twist, and is incorrect, so the question arises of how to fix it while preserving the essential idea of Yu's construction.  While writing \cite{spice:asymptotic}, the author expected the resolution of this to be a minor fix, but it turned out to involve 5 years of work and the diverse perspectives and expertise of several collaborators.  A key aspect of the eventual fix was that, while the author found this problem while investigating character computations, Kaletha discovered an apparently different issue while investigating the Langlands correspondence.  Both of us encountered a formula in which a factor that we expected to see did not appear, but neither of us could figure out an appropriate modification to make our factor appear.  The unexpected punchline was that the obvious factors that each of us knew must appear were different!  It turned out that, very nearly, the actual correction factor that was needed was not either one, but rather their \emph{product}---another indication of the tight interplay between character computations and the Langlands correspondence.  For more details, see \xcite{fintzen-kaletha-spice:twist}*{introduction and Theorem \xref{thm:main}}.  See also \cites{hakim:supercuspidal,fintzen:supercuspidal}.

The second question that arises comes from the inductive nature of Yu's construction and our approach to it.  Namely, as mentioned, Yu's construction produces not just a supercuspidal representation \(\pi\) of a group \(G\), but also a usually smaller-depth supercuspidal representation \(\pi'\) of a tame, twisted Levi subgroup \(G'\), and it is this structure that we use in \xcite{spice:asymptotic}*{Theorem \xref{thm:asymptotic-pi-to-pi'}}; but Yu's construction \emph{also} produces ever smaller-depth supercuspidal representations \(\pi_i\) of ever smaller, twisted Levi subgroups \(G^i\), eventually bottoming out in \(\pi_0\), which is a twist of a depth-\(0\) supercuspidal representation.  It is natural to hope that one could iterate \xcite{spice:asymptotic}*{Theorem \xref{thm:asymptotic-pi-to-pi'}} to relate the character of \(\pi\) to that of \(\pi_0\), but doing so is unexpectedly subtle.  Namely, starting at the bottom for notational clarity (although our proofs proceed the other way), we could use an asymptotic expansion of the character of \(\pi_0\) to obtain an asymptotic expansion of the character of \(\pi_1\), and an asymptotic expansion of the character of \(\pi_1\) to obtain an asymptotic expansion of the character of \(\pi_2\), but \emph{the two kinds of asymptotic expansion of \(\pi_1\) do not match}:  they involve different orbits.  It turns out that the right perspective on addressing this problem is to look, not just at the characters themselves, but at the Fourier transforms of orbital integrals that appear in their asymptotic expansions.  This requires an analogue for invariant distributions on Lie algebras of the results on invariant distributions on the group from \cite{spice:asymptotic}.  Fortunately, the proofs involved are similar where they are not easier, so establishing this generalisation is not difficult once it is clear it is necessary.

This paper thus allows us directly and explicitly to reflect information about depth-\(0\), supercuspidal character formulas in positive-depth, supercuspidal character formulas.  See Theorem \ref{thm:char-unwind}.  In the optimal case where the depth-\(0\) supercuspidal character has a single-orbit theory---for example, for Deligne--Lusztig representations, as shown in \cite{debacker-reeder:depth-zero-sc}---we obtain a fully explicit character formula for regular supercuspidal representations that carries a pleasing combination of number-theoretic and harmonic-analysis information.  See \xcite{fintzen-kaletha-spice:twist}*{\S\xref{sub:char}}.

\subsection{Structure of the paper}

In \S\ref{sec:notation}, we lay out the standard notation that we will use.  Most of \S\ref{sec:cuspidal} is also about notation, but we do point out the important discussion of the `twisted' Weil representation (our remedy to the error created by the mis-stated \cite{gerardin:weil}*{Theorem 2.4(b)}, using \xcite{fintzen-kaletha-spice:twist}*{Theorem \xref{thm:main}}), which we use in \S\ref{subsec:Weil-rep} to define the `twisted' Gauss sum (Notation \ref{notn:twisted-Gauss}) needed in \xcite{spice:asymptotic}*{Definition \xref{defn:rho-pi}}.

\S\S\ref{sec:dual-Jordan} and \ref{sec:funny-centraliser} involve general structure theory, with little or nothing that is particular to groups over \(p\)-adic fields.  The main content of \S\ref{sec:dual-Jordan} is Definition \ref{defn:dual-disc}, which defines a discriminant function on the dual Lie algebra in a uniform fashion, avoiding the \textit{ad hoc} approach taken in \xcite{spice:asymptotic}*{Definition \xref{defn:normal-harm}}, and Lemma \ref{lem:disc-as-det}, which computes the discriminant in a way that parallels \xcite{debacker-spice:stability}*{Remark \xref{rem:disc:roots}}.
This section is based crucially on the ideas of Kac and Weisfeiler \cite{kac-weisfeiler:Zg}.
In an earlier version of this paper, we said that,
with our Proposition \ref{prop:dual-Borel} in place of \cite{kac-weisfeiler:Zg}*{Lemma 3.3},
the proofs of \cite{kac-weisfeiler:Zg}*{\S3, Theorem 4(ii--v)} go through unchanged without the requirement that \bG be almost simple, even when the characteristic of \field is \(2\) and some almost-simple component of \bG is an odd orthogonal group.  We thank Cheng-Chiang Tsai for pointing out to us that, actually, \cite{kac-weisfeiler:Zg}*{\S3.8 and Theorem 4iv)} are not correct even under the assumptions of \cite{kac-weisfeiler:Zg}.  In particular, uniqueness of Jordan decompositions on the dual Lie algebra can fail in small characteristic; see \cite{spice-tsai:jordan}*{\S7}.  This complicates the proof of Lemma \ref{lem:ss-of-dual-blob}, but the main results of \cite{spice-tsai:jordan} still allow us to prove that result.

In \S\ref{sec:funny-centraliser}, we define the generalised centraliser subgroups \(\CC\bG i(\gamma)\) attached to a semisimple element \(\gamma\) of \(G\) or \(\Lie(G)\) whose existence we had to hypothesise in \xcite{spice:asymptotic}*{Hypothesis \xref{hyp:funny-centraliser}}.  Such groups had already been defined for \(\gamma\) a semisimple element of \(G\conn\) (rather than of some possibly non-identity component) in \xcite{adler-spice:good-expansions}*{Definition \xref{defn:funny-centralizer}}, but \S\ref{sec:funny-centraliser} has new content even if \bG is connected, since it allows the subgroups \(\CC\bG i(\gamma)\) to be disconnected.  This change means that they are much better behaved under descent, obeying the property \(\CC\bG i(\gamma) \cap \bG' = \CC{\bG'}i(\gamma)\) (Remark \initref{rem:gp-dfc-facts}\subpref{sub}) that their identity components in general do not.

Because of \S\ref{sec:funny-centraliser}, some of what was a hypothesis in \cite{spice:asymptotic} becomes a theorem in this paper.  In addition, there were a few errors in \cite{spice:asymptotic}, which we describe in Appendix \ref{app:errata}.  In light of this, rather than specifying the changes necessary to \cite{spice:asymptotic}, it seemed to make sense simply to re-state the affected hypotheses, with the necessary changes included, in this paper.  We do so in \S\ref{sec:hyps}.  These changed hypotheses also affect \xcite{spice:asymptotic}*{\S\xref{sec:depth-matrix}}, requiring a different way of specifying depths, which we introduce in Appendix \ref{app:depth}, and a different way of constructing compact, open subgroups of \(G\) when \(\gamma\) is not compact, which we discuss in \S\ref{sec:subgps}.

As mentioned in \S\ref{subsec:motivation}, although our goal in this paper is to understand characters, we do so, as usual, by using an exponential-type map to reduce crucial questions to the Lie algebra, where we must study Fourier transforms of orbital integrals.  We discuss this situation in \S\ref{sec:orbits}.  The machinery here is in almost all respects the same as, or easier than, that on the group, so \S\ref{sec:orbits} mostly contains the statements of results on the Lie algebra to indicate the few changes that are necessary from the group case.  The exception is Theorem \ref{thm:sample-orb-to-orb'}, which is the analogue of \xcite{spice:asymptotic}*{Theorem \xref{thm:isotypic-pi-to-pi'} and Corollary \xref{cor:isotypic-pi-to-pi'}} but requires a geometric rather than spectral approach.  Although much of this material is intended for use in character computations, it may also be interesting in its own right; for example, Theorem \ref{thm:orb-asymptotic-exists} provides a quantitative version of the Shalika-germ expansion as a special case, and moreover allows us to move it (so that it is centred at an element other than the origin) and to refine it, as in the passage from the local character expansion to the asymptotic expansions of Kim and Murnaghan.

Similarly, in \S\ref{sec:characters}, we address the mild issues created by a subtle change in perspective from \cite{spice:asymptotic}*{\S\xref{sec:cuspidal}} when we want to compute the character value at an element not of the form \(\sbjhd\gamma r\dotm\mexp(\sbjtl Y r)\), but rather one of the form \(\sbjhd\gamma{R_{-1}}\dotm\mexp(\sbjhd Y r + \sbjtl Y r)\).  These difficulties come mostly from the fact that we do not require our mock-exponential map to be multiplicative even on commuting elements, so most of \S\ref{sec:characters} is not needed if, for example, we work in characteristic \(0\) and \mexp is the usual exponential map.  An exception is Lemma \ref{lem:gp-Gauss-combine}, on the interaction of Gauss sums attached to \(G\) and to \(\Lie(G)\), which has content even when \mexp is the exponential map.

Finally, \S\ref{sec:unwind} ties all these strands together by unwinding the induction implicit in \S\S\ref{sec:orbits}, \ref{sec:characters}.  We found the process of carrying out the induction subtle, and enjoy the unexpected way that the pieces fit together in Theorem \ref{thm:orb-unwind} and its near-, but not completely, identical sibling Theorem \ref{thm:char-unwind}.

\S\ref{subsec:orb-unwind} allows us, in most cases, to reduce the problem of computing an asymptotic expansion for the Fourier transform of an orbital integral to the analogous problem for Fourier transforms of \emph{nilpotent} orbital integrals.  When dealing with Fourier transforms of semisimple orbital integrals, we recover and generalise a version of a formula of Waldspurger \cite{waldspurger:loc-trace-form}*{Proposition VIII.1}; in particular, we explicitly compute its range of validity.  See Theorem \ref{thm:orb-unwind} and Corollary \ref{cor:orb-lc}.  \S\ref{subsec:char-unwind} is the analogous discussion for supercuspidal characters.  It completely reduces the study of positive-depth, tame, supercuspidal representations to that of depth-\(0\), supercuspidal representations, and is effective in the sense that, given an explicit asymptotic expansion for the depth-\(0\) representation, we can immediately turn it into an explicit asymptotic expansion for the positive-depth representation.

\subsection{Acknowledgements}

This paper benefited greatly from my work with several other people.  Chief among these were my conversations with Tasho Kaletha that led to our work with Jessica Fintzen; in particular, the realisation that we held two pieces of the solution to the same puzzle.  Kaletha also provided crucial encouragement to overcome the technical hurdles in unwinding the induction of \cite{spice:asymptotic}.  Yeansu Kim and Sandeep Varma were very patient with my long delay in completing my portion of \cite{ykim-spice-varma:bernstein}.  Varma particularly encouraged me to try to streamline the hypotheses of \xcite{spice:asymptotic}*{\S\xref{sec:depth-matrix}}, and focussed my attention on the importance of the strong representability hypothesis imposed on the characters \(\phi_i\) in \S\ref{subsec:char-unwind}.  My work on this paper stalled for a long time because of technical difficulties in \S\ref{sec:funny-centraliser}, and the work with Jeff Adler and Josh Lansky that became \cite{adler-lansky-spice:actions} turned out to give me exactly the tools to overcome that remaining hurdle.
Cheng-Chiang Tsai pointed out to me the relevance of \cite{waldspurger:endoscopie} in providing conditions for the validity of the hypotheses in \S\ref{sec:hyps}, and identified some subtleties involving the results of \cite{kac-weisfeiler:Zg}, particularly the possible failure of uniqueness of Jordan decompositions on the dual Lie algebra, that affect the proof of Lemma \ref{lem:ss-of-dual-blob}.
I thank all of these people, as well as Stephen DeBacker, with whom I had many useful discussions.
I also thank again the American Institute of Mathematics and the Simons Foundation, whose collaboration-focussed funding encouraged this work.

While I was writing my thesis, my advisor, Paul Sally, told me that ``the character theory of \(p\)-adic groups is still in its infancy'' (a diagnosis that appears in slightly different form in the introduction to \cite{spice:thesis}).  My experience with this paper indicates that the character theory has passed through its terrible twos, in which the only things it seemed to know how to say were ``why?\@'' and ``no!''  I hope that, with this paper, the character theory may be said to have moved at least into its awkward teenage years.

As will be obvious, reading this paper requires careful attention to many subtle details, although some familiarity with, say, the early works on the subject, such as \cite{corwin-moy-sally:gll} (where I first encountered these ideas), may help to see through the bookkeeping to the essential ideas.  I thank the anonymous referee who took the time thoroughly to understand both the big-picture and small-picture views of this paper, and who made many valuable suggestions that removed errors and improved the exposition.

The most recent part of this work was completed while sheltering at home during the COVID-19 pandemic.  This meant that I relied crucially, even more so than I usually would, on the support and encouragement of my wife, Anna Spice.  To make a long story short, I thank her for her endless supplies of both.

\numberwithin{thm}{subsection}
\section{Notation}
\label{sec:notation}

Let \field be a field, and \(\matnotn{kalg}\algfield/\field\) an algebraic closure.  Let \(\matnotn{ksep}\sepfield/\field\) be the maximal separable extension inside \(\algfield\).  All schemes we consider are of finite type, but not necessarily smooth.

We write \matnotn R\tR for \(\R \sqcup \set{\Rp r}{r \in \R} \sqcup \sset\infty\), and impose on \(\tR \sqcup \sset{-\infty}\) the order and addition discussed in \xcite{spice:asymptotic}*{\S\xref{sec:p-adic-group-basics}}.

We continue the convention, due to Cheng-Chiang Tsai and used in \cite{spice:asymptotic}*{\S\xref{sec:p-adic-group-basics}}, by which, if \(S\) is a set and \map x S{\tR \cup \sset{-\infty}} is a function, then we write \matnotn[\sbtl{\mc G}x r]{Gxr}{\sbtl S x r} for the subset \(\set{g \in S}{x(g) \ge r}\), and \matnotn[\sbtlp{\mc G}x r]{Gxrplus}{\sbtlp S x r} for \(\sbtl S x{\Rp r}\), for every \(r \in \tR \cup \sset{-\infty}\).  If \(S = \mc G\) is a group and \(x\) is a group valuation, in the sense that we have \(x(g) = x(g\inv)\) and \(x(g h\inv) \ge \min \sset{x(g), x(h)}\) for all \(g, h \in \mc G\), then each \(\sbtl{\mc G}x r\) is a group, and we put \(\matnotn{Gxr}{\sbat{\mc G}x r} = \sbtl{\mc G}x r/\sbtlp{\mc G}x r\) for every \(r \in \tR \cup \sset{-\infty}\).  If there is a canonical choice of \(x\), as on a valued field, then we often suppress the name of it, writing just \matnotn[\sbjtl{\mc G}r]{Gr}{\sbjtl S r}, \matnotn[\sbjtlp{\mc G}r]{Grplus}{\sbjtlp S r}, and \matnotn{Gr}{\sbjat{\mc G}r}.  Thus, for example, \(\sbjtl\tR 0\) means \(\set{r \in \tR}{r \ge 0}\); and, if \field is a valued field, then \(\sbjtl\field 0\) is its ring of integers, \(\sbjtlp\field 0\) is the maximal ideal in the ring of integers, and \(\sbjat\field 0\) is the residue field of \field.

\subsection{Affine algebraic groups}

We will always use ``scheme over \field'' to mean ``affine scheme of finite type over \field'', and ``algebraic group over \field'' to mean ``affine group scheme of finite type over \field.''  We do not assume that schemes or algebraic groups are smooth, unless explicitly stated.

Let \bG be an algebraic group over \field.  Write \matnotn{LieG}{\Lie(\bG)} for the Lie algebra of \bG, and \matnotn{LieGdual}{\Lie^*(\bG)} for its dual.
We put \(\matnotn{XG}{\clat(\bG)} = \Hom(\bG, \GL_1)\) and \(\matnotn{XG}{\cclat(\bG)} = \Hom(\GL_1, \bG)\).  We have that \(\clat(\bG)\) is a group and, if \bG is commutative, so is \(\cclat(\bG)\), with the operation in both cases being pointwise multiplication.  Despite this, we will usually write the group operations additively.

We write \matnotn{DerG}{\Der\bG} for the derived subgroup of \bG.
If \vbG is a sequence \((\bG^0, \bG^1, \dotsc, \bG^\ell)\) of groups, then we write \matnotn{DerG}{\Der\vbG} for the sequence of derived subgroups \((\Der\bG^0, \Der\bG^1, \dotsc, \Der\bG^\ell)\).

Let \bX be a scheme over \field, and \anonmap{\bG \times \bX}\bX an action.
If \(\bX'\) is a subscheme of \bX, then we write \matnotn{Norm}{\gNorm_\bG(\bX')}
(respectively, \matnotn{Cent}{\Cent_\bG(\bX')}) for the stabiliser (respectively, the fixer) of \(\bX'\) in \(\bG\).
Thus, \(\Cent_\bG(\bX')\) is normal in \(\gNorm_\bG(\bX')\).  We write \matnotn{Weyl}{\Weyl(\bG, \bX')} for \(\gNorm_\bG(\bX')/{\Cent_\bG(\bX')}\).  (This notation is most familiar when
\bG is a connected, reductive group acting on \(\bX = \bG\) by conjugation, and
\(\bX'\) is a maximal split torus in \bG, but we do not require that.)  If \bX is a group, \bG is a subgroup of \bX (acting on \bX by conjugation), and \(\bX'\) is a \bG-stable subgroup of \bX such that \anonmap{\bG \ltimes \bX'}\bX is an isomorphism, then we have not just the notation \(\Cent_\bG(\bX')\), for the subgroup of \bG that acts trivially on \(\bX'\), but also \(\Cent_{\bX'}(\bG)\), for the subgroup of \bG-fixed points in \(\bX'\).

If \bX is a pointed scheme over \field, such as an algebraic group, then we write \matnotn{Xconn}{\bX\conn} for the component of \bX containing the distinguished point.  If we denote a variety over \field by a boldface letter, like \bX, then we sometimes denote its set of rational points by the corresponding non-boldface letter, like \(\bX(\field) = X\).  For example,
	\(\Cent_G(\Gamma)\) means \(\Cent_\bG(\Gamma)(\field)\),
	\(\Lie(G)\) and \(\Lie^*(G)\) mean \(\Lie(\bG)(\field)\) and \(\Lie^*(\bG)(\field)\),
	\matnotn{DerG}{\Der G} means \((\Der\bG)(\field)\),
and	\matnotn{Gconn}{G\conn} means \(\bG\conn(\field)\).
In particular, \(\Der G\) need not be the derived subgroup of \(G\), and \(G\conn\) need not be the identity component of \(G\), even if \(G\) carries a topology.
However, if \(g\) belongs to \(\bG(\field)\), then we always mean by \matnotn g{\sgen g} the smallest algebraic subgroup of \(\bG\) containing \(g\), not the abstract group generated by \(g\).

We say that a subgroup \bP of \bG is a \term[subgroup!parabolic]{parabolic subgroup} of \bG if \(\bP_\algfield\) contains a Borel subgroup of \(\bG\conn_\algfield\).  Thus, \bP is a parabolic subgroup of \bG if and only if \(\bP \cap \bG\conn\) is a parabolic subgroup of \(\bG\conn\).  We say that a subgroup \bM of \bG is a \term[subgroup!parabolic!Levi component of]{Levi component} of \bG if and only if the restriction to \(\bM_\algfield\) of the canonical quotient map \anonmap{\bG_\algfield\conn}{\bG_\algfield\conn/{\operatorname R\textsub u(\bG_\algfield\conn)}} is an isomorphism, where \(\operatorname R\textsub u(\bG_\algfield\conn)\) is the unipotent radical of \(\bG_\algfield\conn\).  Thus, as for parabolic subgroups, \bM is a Levi component of \bG if and only if \(\bM \cap \bG\conn\) is a Levi component of \(\bG\conn\).  Thus, if \bM is a Levi component of \bG, then \(\bG_\algfield\conn\) is the semi-direct product of \(\bM_\algfield \cap \bG_\algfield\conn\) and \(\operatorname R\textsub u(\bG_\algfield\conn)\); but we need not have that \(\bM_\algfield\) is a semi-direct factor of \(\bG_\algfield\).
We say that \bG is \term[group!reductive]{reductive} if \bG is smooth and \(\operatorname R\textsub u(\bG_\algfield\conn)\)
is trivial.  In particular, we do not require that a reductive group be connected; and Levi components are reductive.
A subgroup \bM of a reductive group \bG is called a \term[subgroup!Levi]{Levi subgroup} of \bG if \bM is a Levi component of a parabolic subgroup of \bG, and a \term[subgroup!Levi!twisted]{twisted Levi subgroup} of \bG if \(\bM_\algfield\) is a Levi subgroup of \(\bG_\algfield\).  Note that a Levi subgroup of \bG is usually \emph{not} a Levi component of \bG.

If \bG is reductive, then we write \matnotn{rss}{\bG\rss} and \matnotn{rss}{\Lie(\bG)\rss} for the varieties of regular, semisimple elements of \bG and \(\Lie(\bG)\).

If \(\Gamma\) is a diagonalisable group and \bV is a \(\Gamma\)-module, then, for each \(\alpha \in \clat(\Gamma)\), we write \(\bV_\alpha\) for the \(\alpha\)-weight space for \(\Gamma\) on \bV; and write \matnotn{Weight}{\Weight(\bV, \Gamma)} and \matnotn{Root}{\Root(\bV, \Gamma)} for the sets of weights, and of non-\(0\) weights, of \(\Gamma\) on \bV.  If \bV equals \(\Lie(\bG)\), with the linear action coming from a group action of \(\Gamma\) on \bG, then we write \matnotn{Weight}{\Weight(\bG, \Gamma)} and \matnotn{Root}{\Root(\bG, \Gamma)} in place of \(\Weight(\Lie(\bG), \Gamma)\) and \(\Root(\Lie(\bG), \Gamma)\).  (This notation is most familiar when \bG is a connected, reductive group and \(\Gamma\) is a maximal split torus in \bG, acting on \bG by conjugation, but we do not require that.)

\subsection{Reductive groups over local fields}
\label{subsec:p-adic}

Beginning with \S\ref{sec:hyps}, we will assume that \field is a field that is complete and locally compact with respect to a non-trivial discrete valuation.  We say in brief that \field is a \term[field!non-Archimedean local]{non-Archimedean local field}.  Note that we do not require a non-Archimedean local field to have characteristic \(0\).

Let \field be a non-Archimedean local field, and \bG a reductive group over \field.

Write \(\matnotn{kun}\unfield/\field\) for the maximal unramified extension inside \(\algfield/\field\), and \(\matnotn{ktame}\tamefield/\field\) for the maximal tame extension inside \(\algfield/\field\).  We say that a subgroup \(\bG'\) of \bG is a \term[subgroup!Levi!twisted!tame]{tame, twisted Levi subgroup} if \(\bG'_\tamefield\) is a Levi subgroup of \(\bG_\tamefield\).

Fix a complex, additive character \matnotn{Lambda}\AddChar of \field that is trivial on \(\sbjtlp\field 0\) but not on \(\sbjtl\field 0\).  We will denote the resulting character of \(\sbjat\field 0\) also by \AddChar.  If \(V\) is a vector space and \(X^*\) is an element of the dual space \(V^*\), then we write \matnotn{LambdaX}{\AddChar_{X^*}} for the character \(\AddChar \circ X^*\) of \(V\), and \matnotn{LambdaXbar}{\AddChar_{X^*}^\vee} for its contragredient (i.e., complex conjugate).

If \(U\) is an open subset of \(G\), then we write \matnotn{Hecke}{\Hecke(U)} for the space of locally constant, compactly supported functions on \(G\) that vanish outside \(U\).  If \(K\) is a compact, open subgroup of \(G\) such that \(K\dotm U\dotm K\) equals \(U\), and \(\rho\) is a representation of \(K\) on a complex vector space \(V\), then \(K \times K\) acts (by simultaneous left and right multiplication) on \(\Hecke(U) \otimes_\C \End_\C(V)\), and we write \(\Hecke(U\sslash K, \rho)\) for the (\(\rho \times \rho\))-isotypic subspace of \(\Hecke(U) \otimes_\C \End_\C(V)\).

There is a canonical choice of Haar measure on \(G\) and \(\Lie(G)\) described in \cite{debacker-reeder:depth-zero-sc}*{\S5.1}.  For every point \(x \in \BB(G)\), these measures assign mass \(\card{\sbat G x 0}\cdot\card{\sbat{\Lie(G)}x 0}^{-1/2}\) to \(\sbtl G x 0\) and mass \(\card{\sbat{\Lie(G)}x 0}^{1/2}\) to \(\sbtl{\Lie(G)}x 0\).  We will always use these Haar measures, denoting them by \matnotn{dg}{\upd g} and \matnotn{dY}{\upd Y}.

If \(U\) is a non-empty, compact, open subset of \(G\), then, for every \(f \in \Hecke(U)\), we write \matnotn{char}{\chrc{U, f}} for the function that equals \(\meas(U)\inv f\) on \(U\), and is \(0\) elsewhere; and put \(\matnotn{int}{\uint_U f(g)\upd g} = \int_G \chrc{U, f}(g)\upd g\).  We use similar notation on the Lie algebra.

If \(V\) is a finite-dimensional \field-vector space and \(f\) belongs to \(\Hecke(V)\), then the Fourier transform \matnotn{fhat}{\hat f} of \(f\) is the element of \(\Hecke(V^*)\) given by
\[
\hat f(v^*) = \int_V f(v)\AddChar_{v^*}(v)\upd v
\qforall{\(v^* \in V^*\).}
\]
Similarly, for every \(f^* \in \Hecke(V^*)\), we define
\[
\matnotn{fcheck}{\check f^*}(v)
= \int_{V^*} f^*(v^*)\ol{\AddChar_v(v^*)}\upd{v^*}
\qforall{\(v \in V\).}
\]
Note that these depend on the choice of Haar measures \(\upd v\) and \(\upd{v^*}\) on \(V\) and \(V^*\).  Given a choice of \(\upd v\), there is a unique choice of \(\upd{v^*}\), called the \term[Haar measure!dual]{dual Haar measure} to \(\upd v\), so that \(\Check{\Hat f}\) equals \(f\) for all \(f \in \Hecke(V)\).  Technically speaking, since we have not specified a choice of \(\upd v\), the functions \(\hat f\) and \(\check f^*\) are not well defined, although the measures \(\hat f(v^*)\upd{v^*}\) and \(\check f^*(v)\upd v\) are.  In practice, \(V\) will be the Lie algebra of a reductive \(p\)-adic group, equipped with its canonical measure, so that this ambiguity will not cause a problem; and then we will equip the Lie-algebra dual \(V^*\) with the dual measure.

As in \xcite{debacker-spice:stability}*{Definition \xref{defn:normal-harm}}, if \(\pi\) is an admissible representation of \(G\), then we write \matnotn{Theta}{\Theta_\pi} for the scalar character of \(\pi\), which is the function on \(G\rss\) that represents the distribution character in the sense that
\[
\tr \pi(f) \qeqq \int_G \Theta_\pi(g)f(g)\upd g
\]
for all \(f \in \Hecke(G)\) \cite{adler-korman:loc-char-exp}*{Proposition 13.1}; and \matnotn{Phi}{\Phi_\pi} for the function \anonmapto g{\abs{\redD_G(g)}^{1/2}\Theta_\pi(g)} on \(G\rss\).  Similarly, if \OO is a coadjoint orbit in \(\Lie^*(G)\), then we write \matnotn{mu}{\mu^G_\OO} for the orbital-integral distribution on \(\Lie^*(G)\) given by integration against some invariant measure on \OO, if it converges, and \matnotn{muhat}{\muhat^G_\OO} for the function on \(\Lie(G)\rss\) that represents its Fourier transform, in the sense that
\[
\mu^G_\OO(\hat f) \qeqq \int_{\Lie(H)} \muhat^G_\OO(Y)f(Y)\upd Y
\]
for all \(f \in \Hecke(\Lie(G)\rss)\); and \(\matnotn{Ohat}{\Ohat^G_\OO}\) for the function \anonmapto Y{\abs{\redD_G(Y)}^{1/2}\abs{\redD_G(\OO)}^{1/2}\muhat^G_\OO(Y)} on \(\Lie(G)\rss\), where \(\redD_G\) is defined as a function on \(\Lie^*(G)\) in Definition \ref{defn:dual-disc} below.
(Although convergence of orbital integrals as distributions on all of \(\Lie^*(G)\) is not known in general, it is known regardless of characteristic that, \emph{when they converge}, their Fourier transforms are represented on \(\Lie(G)\rss\) by a locally constant function.  See \cite{adler-debacker:mk-theory}*{Theorem A.1.2}.)
We sometimes write \(\mu^G_{\mc S}\) in place of \(\mu^G_\OO\), and similarly for \(\muhat\) and \(\Ohat\), if \mc S is a non-empty subset of \OO.

\label{OO-defn}
If \mc S is a subset of \(\Lie^*(G)\), then we write \(\OO^G(\mc S)\) for the set of all co-adjoint orbits of \(G\) in \(\Lie^*(G)\) that intersect every neighbourhood of \mc S in the analytic topology on \(\Lie^*(G)\).  In general, if \mc S and \(\Lie^*(G)\) belong to some common superset, then we write \(\OO^G(\mc S)\) for \(\OO^G(\mc S \cap \Lie^*(G))\).
(See, for example, Theorem \ref{thm:orb-asymptotic-exists}.)

\subsection{Moy--Prasad filtration of reductive groups over local fields}
\label{subsec:MP}

We continue with the non-Archimedean local field \field, and reductive group \bG over \field, of \S\ref{subsec:p-adic}.

We have the (enlarged) Bruhat--Tits building \matnotn{BG}{\BB(\bG, \field)} of \bG, which we usually denote by just \matnotn{BG}{\BB(G)}.  For every point \(x \in \BB(G)\), we have associated Moy--Prasad depth functions on \(G\), \(\Lie(G)\), and \(\Lie^*(G)\), all of which we denote by \matnotn{depth}{\depth_x}.

If a maximally split maximal torus in \(\bG_\unfield\) is induced, or becomes so after a tame extension---for example, if \bG is adjoint or simply connected, or if \(\bG_\tamefield\) is split---then Moy and Prasad defined, for each \(x \in \BB(G)\), associated canonical filtrations of a compact open subgroup \(\sbtl G x 0\) of \(G\), and of \(\Lie(G)\).  However, as discussed in \cite{kaletha-prasad:bt-theory}*{\S13.1}, defining such filtrations for general reductive groups requires first defining them for tori.  This definition should be admissible, schematic, connected, and congruent, in the sense of \cite{yu:models}*{\S\S4.3, 4.4}.  While the Moy--Prasad filtration is admissible and schematic \loccit*{Proposition 4.5}, it is neither connected nor congruent
\citelist{
	\cite{yu:models}*{\S4.6}
	\cite{kaletha-prasad:bt-theory}*{Example B.10.6}
}.  We thank Gopal Prasad for discussing these issues with us.

Because of problems that can occur when dealing with non-tame tori, even induced ones (see Example \ref{exa:not-most-singular}), we require in \S\S\ref{sec:orbits}--\ref{sec:unwind} that \(\bG_\tamefield\) be split.
(This should also be regarded as being in force for \xcite{spice:asymptotic}*{\S\S\xref{sec:qualitative}, \xref{sec:quantitative}}.  See Remark \ref{rem:explicit-tame}.)
In this setting, since the Moy--Prasad filtration is the unique admissible filtration on a tame torus \cite{kaletha-prasad:bt-theory}*{Proposition B.10.5}, and since Moy--Prasad filtrations admit good descent properties for split groups, or along tame extensions (see Lemma \initref{lem:dual-MP-ascent}\subpref{tame-or-split}), there is no ambiguity in the definition.
It is probably possible to find a way around this explicit tameness restriction, but, since our most natural sufficient conditions require tameness anyway (see \S\ref{subsec:sufficient}), we do not do so here.

The depth function on \(G\) is often taken to be defined only on the parahoric subgroup of \(G\) at \(x\), but, for convenience, we extend it to take the value \(-\infty\) outside the parahoric subgroup \(\sbtl G x 0\) (in particular, outside \(G\conn\)).

These depth functions are group valuations, and, as usual, we denote the associated filtrations by, for example, \matnotn{Gxr}{\sbtl G x r} rather than by \(\sbtl G{\depth_x}r\).  Thus, for example, for every \(x \in \BB(G)\), we have that \(\sbtl G x 0\) is a parahoric subgroup, \(\sbtlp G x 0\) is the pro-unipotent radical of \(\sbtl G x 0\), and \(\sbat G x 0\) is the group of \(\sbjat\field 0\)-rational points of the associated reductive quotient.
Note that \(\sbtl G x r\) equals \(\sbtl{G\conn}x r\) for all \(x \in \BB(G)\) and \(r \in \sbjtl\tR 0\).

Recall that, for every \(x \in \BB(G)\) and \(r \in \tR\), the corresponding Moy--Prasad sublattice of \(\Lie^*(G)\) is defined by
\[
\sbtl{\Lie^*(G)}x r = \sett{X^* \in \Lie^*(G)}{\(\pair{X^*}Y \in \sbjtlp{\field\unram}0\) for all \(Y \in \sbtlpp{\Lie(\bG)(\unfield)}x{-r}\)},
\]
where \(\sbtlpp{\Lie(\bG)(\unfield)}x{-r}\) means \(\Lie(\bG)(\unfield)\) if \(r\) is \(\infty\), and \(\sbtl{\Lie(\bG)(\unfield)}x{-s}\) if \(r\) equals \(\Rp s\) with \(s \in \R\).

\begin{rem}
\label{rem:dual-MP-by-MP}
Fix real numbers \(r\) and \(r^*\), and elements \(X^* \in \sbtl{\Lie^*(G)}x{r^*}\) and \(Y \in \sbtl{\Lie(G)}x r\).  We claim that \(\pair{X^*}Y\) belongs to \(\sbjtl\field{r + r^*}\).  If not, then \(\pair{X^*}Y\inv Y\) belongs to \(\sbtlpp{\Lie(G)}x{-r^*}\), so \(1 = \pair[\big]{X^*}{\pair{X^*}Y\inv Y}\) belongs to \(\sbjtlp\field 0\), which is a contradiction.

The analogous claim where \(r\) and \(r^*\) are allowed to range over all of \tR follows easily.
\end{rem}

For every algebraic extension \(E/\field\) of finite ramification degree, every \(x \in \BB(G)\), and every \(r \in \tR\), we have that \(\sbtl{\Lie(G)}x r\) is contained in \(\sbtl{\Lie(\bG)(E)}x r\), and that \(\sbtl G x r\) is contained in \(\sbtl{\bG(E)}x r\) if \(r\) is non-negative, but, in full generality, the analogous containment can fail on the dual Lie algebra.  We spend some time discussing this issue.

Lemma \ref{lem:dual-MP-by-root} shows how to compute depths on a root space in the dual Lie algebra.  Lemma \ref{lem:dual-MP-by-cclat} shows how to compute the depth on the dual Lie algebra of a \emph{split} torus.

\begin{lem}
\label{lem:dual-MP-by-root}
Suppose that \bG is quasi-split.

Let \bS be a maximal split torus in \bG, and put \(\bT = \Cent_\bG(\bS)\).  Let \(a\) be an element of \(\Root(\bG, \bS)\), and \(\alpha\) a weight of \(\bT_\sepfield\) on \((\Lie(G)_a)_\sepfield\).  For \(X^* \in \Lie^*(G)_{-a}\) and \(Y \in \Lie(G)_a\), put \(\pair{X^*}Y_\alpha = \pair{X^*_{-\alpha}}{Y_\alpha}\), where the subscripts denote projection on appropriate weight spaces for \bT.

Write \(\field_\alpha\) for the fixed field in \(\sepfield\) of the stabiliser in \(\Gal(\sepfield/\field)\) of \(\alpha\).  If \(X^* \in \Lie^*(G)_{-a}\) and \(Y \in \Lie(G)_a\) satisfy \(\ord \pair{X^*_{-\alpha}}{Y_\alpha} = \sup \ord(\tr_{\field_\alpha/\field}\inv(1))\), where the subscripts denote \(\bT_\sepfield\)-equivariant projections on the appropriate weight spaces, then \(\depth_x(X^*)\) equals \(-\depth_x(Y)\).
\end{lem}

\begin{note}
Recall that \(\Lie(\bG)_a\) denotes the \(a\)-weight space of \bS in the Lie algebra of \bG, \emph{not} necessarily the Lie algebra of the associated root subgroup of \bG.  The latter also contains \(\Lie(\bG)_{2a}\), which may be non-trivial.
\end{note}

\begin{proof}
Put \(r = \depth_x(Y)\).  Note that \(\pair{X^*}Y\) equals \(\tr_{\field_\alpha/\field} \pair{X^*}Y_\alpha\).

Write \(\Lie(\bG_{\field_\alpha})_\alpha\) for the descent to \(\field_\alpha\) of \(\Lie(\bG_\sepfield)_\alpha\).  The projection map \anonmap{\Lie(G)_a}{\Lie(\bG_{\field_\alpha})_\alpha(\field_\alpha)} is an isomorphism, and so equips \(\Lie(G)_a\) canonically with the structure of a \(1\)-dimensional \(\field_\alpha\)-vector space.  Moreover, we have that \(\Lie(G)_a \cap \sbtl{\Lie(G)}x{r + d}\) equals \(\sbjtl{(\field_\alpha)}d\dotm Y\) for all \(d \in \R\) \cite{bruhat-tits:reductive-groups-2}*{\S4.2.7}.

By assumption, there is some \(t_0 \in \sbjtl{(\field_\alpha)}0\) such that \(\pair{X^*}{t_0 Y} = \tr_{\field_\alpha/\field}(t_0\pair{X^*}Y_\alpha)\) equals \(1\).  Since \(t_0 Y\) belongs to \(\sbtl{\Lie(G)}x r\), we have that \(X^*\) does not belong to \(\sbtlpp{\Lie^*(G)}x{-r}\).

If \(Y'\) belongs to \(\sbtlp{\Lie(G)}x r\), then \(\pair{X^*}{Y'}\) equals \(\pair{X^*}{Y'_a}\), where \(Y'_a\) is the \bS-equivariant projection of \(Y'\) on \(\Lie(\bG)_a\).  Since \(Y'_a\) still belongs to \(\sbtlp{\Lie(G)}x r\), it equals \(t_+ Y\) for some \(t_+ \in \sbjtlp{(\field_\alpha)}0\).  Therefore, \(\pair{X^*}{Y'}\) equals \(\tr_{\field_\alpha/\field}(t_+\pair{X^*}Y_\alpha)\), which belongs to \(\sbjtlp\field 0\).

Since \(Y' \in \sbtlpp{\Lie(G)}x r\) was arbitrary, we have that \(X^*\) belongs to \(\sbtl{\Lie^*(G)}x{-r}\).
\end{proof}

\begin{lem}
\label{lem:dual-MP-by-cclat}
Let \bT be a split torus.
Fix \(r \in \tR\) and \(X^* \in \Lie^*(T)\).
Then \(X^*\) belongs to \(\sbjtl{\Lie^*(T)}r\) if and only if \(\pair{X^*}{\upd\lambda(1)}\) lies in \(\sbjtl\field r\) for all \(\lambda \in \cclat(\bT)\).
\end{lem}

\begin{proof}
Suppose first that \(X^*\) belongs to \(\sbjtl{\Lie^*(T)}r\).  For every \(\lambda \in \cclat(\bT)\), since \(\pair{\upd\chi}{\upd\lambda(1)} = \pair\chi\lambda\) belongs to \(\Z \subseteq \sbjtl\field 0\) for every \(\chi \in \clat(\bT)\), we have that \(\upd\lambda(1)\) belongs to \(\sbjtl{\Lie(T)}0\), and hence by Remark \ref{rem:dual-MP-by-MP} that \(\pair{X^*}{\upd\lambda(1)}\) belongs to \(\sbjtl\field r\).

Conversely, suppose that \(X^* \in \Lie^*(T)\) satisfies \(\pair{X^*}{\upd\lambda(1)} \in \sbjtl\field r\) for all \(\lambda \in \cclat(\bT)\).  Fix an element \(Y\) of \(\sbjtlpp{\Lie(T)}{-r}\), so that \(\pair{\upd\chi}Y\) belongs to \(\sbjtlpp\field{-r}\) for all \(\chi \in \clat(\bT)\).  If we choose dual bases \((\chi_i)_i\) and \((\lambda_i)_i\) for \(\clat(\bT)\) and \(\cclat(\bT)\), then we have that \(Y\) equals \(\sum_i \pair{\upd\chi_i}Y\cdot\upd\lambda_i(1)\), so that \(\pair{X^*}Y = \sum_i \pair{\upd\chi_i}Y\pair{X^*}{\upd\lambda_i(1)}\) is a sum of elements of \(\sbjtlpp\field{-r}\dotm\sbjtl\field r \subseteq \sbjtlp\field 0\), and hence belongs to \(\sbjtlp\field 0\).  Since \(Y \in \sbjtlpp{\Lie(T)}{-r}\) was arbitrary, we have that \(X^*\) belongs to \(\sbjtl{\Lie^*(T)}r\) as desired.
\end{proof}

Example \ref{exa:not-most-singular} shows that dual Moy--Prasad sublattices might not grow upon passing to field extensions.  Actually, examples of this behaviour are not hard to manufacture, but ours shows one surprising consequence of this failure, namely, that our analogue Lemma \ref{lem:unique-approx*} of \xcite{adler-spice:good-expansions}*{Proposition \xref{prop:unique-approx}} can fail for non-split groups.

\begin{exa}
\label{exa:not-most-singular}
Let \field be a local, non-Archimedean field of residual characteristic \(3\) whose residue field has at least \(9\) elements, and \(E/\field\) a totally wildly ramified, cubic, Galois extension.  Choose elements \(a_+, a_- \in \sbjtl\field 0\) whose images in \(\sbjat\field 0\) are non-\(0\), and neither equal nor opposite.  Let \bG be the quasi-split, adjoint group over \field of type \(\prescript3{}{\mathsf D_4}\) that splits over \(E\).  Let \(\bB^+\) be a Borel subgroup of \bG, \bT a maximal torus in \(\bB^+\), and \(\Simple(\bB^+_\sepfield, \bT_\sepfield)\) the set of roots in \(\Root(\bG_\sepfield, \bT_\sepfield)\) that are simple with respect to the order coming from \(\bB^+_\sepfield\).  Write \(\alpha_+\) for the element of \(\Simple(\bB^+_\sepfield, \bT_\sepfield)\) that is fixed by \(\Gal(\sepfield/\field)\), and \(\alpha_-\) for any other element of \(\Simple(\bB^+_\sepfield, \bT_\sepfield)\).  There are unique elements \(\Gamma\) and \(X^*\) of \(\Lie^*(T)\) such that
\(\pair\Gamma{\upd\alpha_+^\vee(1)}\) equals \(a_+\),
\(\pair{X^*}{\upd\alpha_+^\vee(1)}\) equals \(0\), and
\(\pair\Gamma{\upd\alpha_-^\vee(1)}\) and \(\pair{X^*}{\upd\alpha_-^\vee(1)}\) both equal \(a_-\).

Write \(\tr_{E/\field}\) for the natural map \anonmap{\Lie(\bT)(E)}{\Lie(T)}.
Since \bT is induced, there is only one choice of admissible filtration \cite{kaletha-prasad:bt-theory}*{Proposition B.10.5}.  Namely, for all \(d \in \tR\), we have that \(\sbjtl{\Lie(T)}d\) is spanned, as a \(\sbjtl\field 0\)-module, by \(\upd\alpha_+(\sbjtl\field d)\) and \(\tr_{E/\field}(\upd\alpha_-(\sbjtl E d))\).  For all \(t \in \sbjtlp\field 0\), we have that \(\pair\Gamma{\upd\alpha_+(t)}\) equals \(a_+ t\), and so belongs to \(\sbjtlp\field 0\); for all \(t \in \sbjtl\field 0\), we have that \(\pair{X^*}{\upd\alpha_+(t)}\) equals \(0\), and so belongs to \(\sbjtlp\field 0\); and, for all \(t \in \sbjtl E 0\), we have that \(\pair\Gamma{\tr_{E/\field}(\upd\alpha_-(t))}\) and \(\pair{X^*}{\tr_{E/\field}(\upd\alpha_-(t))}\) equal \(\tr_{E/\field}(a_- t) \in \tr_{E/\field}(\sbjtl E 0)\), and so belong to \(\sbjtlp\field 0\).  That is, \(\Gamma\) belongs to \(\sbjtl{\Lie^*(T)}0\) and \(X^*\) belongs to \(\sbjtlp{\Lie^*(T)}0\).  Since \(\pair\Gamma{\upd\alpha_+(1)}\) equals \(a_+\), which does not belong to \(\sbjtlp\field 0\), we have that \(\Gamma\) does not belong to \(\sbjtlp{\Lie^*(T)}0\).

By our choice of \(a_+\) and \(a_-\) (and an inspection of the root system of type \(\mathsf D_4\)), we have that \(\ord \pair\Gamma{\upd\alpha^\vee(1)}\) equals \(0\) for all \(\alpha \in \Root(\bG_\sepfield, \bT_\sepfield)\).  This is \cite{yu:supercuspidal}*{\S8, \textbf{GE1}}, with \(\bG' = \bT\), and then \cite{yu:supercuspidal}*{Lemma 8.1} shows that \cite{yu:supercuspidal}*{\S8, \textbf{GE2}} is also satisfied (since \(2\) is the only torsion prime for the root system of type \(\mathsf D_4\)).

Finally, we have that \(\Gamma - X^*\) belongs to \(\Gamma + \sbjtlp{\Lie^*(T)}0\), but \(\pair{\Gamma - X^*}{\upd\alpha_-^\vee(1)}\) equals \(0\).  Thus \(\Cent_\bG(\Gamma - X^*)\) is strictly larger than \(\Cent_\bG(\Gamma)\).

We have that \(\Cent_\bG(\Gamma - X^*)\) is a Levi subgroup of \bG.  In particular, we may regard \(\BB(\Cent_G(\Gamma - X^*))\) as a subset of \(\BB(G)\).  The dual-Lie-algebra analogue of \cite{adler-debacker:bt-lie}*{Lemma 3.5.3} gives that the depths in \(\Lie^*(G)\) of \(\Gamma\) and \(\Gamma - X^*\) are the same as their depths in \(\Lie^*(T)\), hence equal \(0\).  For every \(x \in \BB(\Cent_G(\Gamma - X^*))\), we have that \(\depth_x(\Gamma - X^*)\) equals \(\depth(\Gamma - X^*) = 0\).  If \(x\) is sufficiently close to \(\BB(T)\) that \(\depth_x(X^*)\) remains positive, then \(\depth_x(\Gamma)\) also equals \(0 = \depth(\Gamma)\).  That is, \cite{jkim-murnaghan:charexp}*{Theorem 2.3.1} fails in this case (which is not a counterexample to the theorem, since the hypothesis \cite{jkim-murnaghan:charexp}*{\S1.4, (HGT)} also fails).
\end{exa}

Example \ref{exa:not-most-singular} involves a wildly ramified extension.  Lemma \initref{lem:dual-MP-ascent}\subpref{tame-or-split} shows that, as long as we can restrict to tame ramification---a common restriction when dealing with Moy--Prasad filtrations---things are much better behaved.

\begin{lem}
\initlabel{lem:dual-MP-ascent}
Fix \(x \in \BB(G)\) and \(r \in \tR\), and let \(E/\field\) be an algebraic extension of finite ramification degree.
\begin{enumerate}
\item\sublabel{arbitrary}
\(\sbtl{\Lie^*(G)}x r\) contains \(\sbtl{\Lie^*(\bG)(E)}x r \cap \Lie^*(G)\).
\item\sublabel{unram}
\(\sbtl{\Lie^*(G)}x r\) equals \(\sett{X^* \in \Lie^*(G)}{\(\pair{X^*}Y \in \sbjtlp\field 0\) for all \(Y \in \sbtlpp{\Lie(G)}x{-r}\)}\).
\item\sublabel{tame-or-split}
If \(E/\field\) is tamely ramified, or if \bG is split, then \(\sbtl{\Lie^*(G)}x r\) equals \(\sbtl{\Lie^*(\bG)(E)}x r \cap \Lie^*(G)\).
\end{enumerate}
\end{lem}

\begin{note}
The difference between \locpref{unram} and the definition of \(\sbtl{\Lie^*(G)}x r\) is that, in \locpref{unram}, we impose the stated condition only for \(Y \in \sbtlpp{\Lie(G)}x{-r}\), not for all \(Y \in \sbtlpp{\Lie(\bG)(\unfield)}x{-r}\).
\end{note}

\begin{proof}
%
%
Fix an element \(X^*\) of \(\sbtl{\Lie(\bG)(E)}x r \cap \Lie^*(G)\).  Since \(\sbtlpp{\Lie(\bG)(\unfield)}x{-r}\) is contained in \(\sbtlpp{\Lie(\bG)(E\unram)}x{-r}\), we have for every \(Y \in \sbtlpp{\Lie(\bG)(\unfield)}x{-r}\) that \(\pair{X^*}Y\) belongs to \(\unfield \cap \sbjtlp{E\unram}0 = \sbjtlp{\field\unram}0\), and hence that \(X^*\) belongs to \(\sbtl{\Lie^*(G)}x{-r}\).  This shows \locpref{arbitrary}, and one containment in \locpref{tame-or-split}.

Now \locpref{unram}, and the case of \locpref{tame-or-split} where \(E/\field\) is tame, will follow by ``approximation by finite, Galois subextensions'' if we prove the following claim.
Suppose that
\(X^* \in \Lie^*(G)\) satisfies \(\pair{X^*}Y \in \sbjtlpp\field 0\) for all \(Y \in \sbtlpp{\Lie^*(G)}x{-r}\).
Let \(E/\field\) be a
	finite,
	Galois,
	tamely ramified extension,
and let \(Y\) be an element of \(\sbtlpp{\Lie^*(\bG)(E)}x{-r}\).
Then \(\pair{X^*}Y\) belongs to \(\sbjtlp E 0\).
To prove this, it suffices to observe that, for every \(t \in \sbjtl E 0\),
the value of \(X^*\) at \(\sum_{\sigma \in \Gal(E/\field)} \sigma(t Y)\) equals \(\tr_{E/\field} t\pair{X^*}Y\), and belongs to \(\sbjtlp\field 0\).
Since \(E/\field\) is tamely ramified, this implies that \(\pair{X^*}Y\) belongs to \(\sbjtlp E 0\), as desired.

It remains only to prove \locpref{tame-or-split} when \bG is split.  Let \bT be a split maximal torus in \bG such that \(x\) belongs to the apartment of \bT.  We have the direct-sum decomposition
\[
\sbtlpp{\Lie(G)}x{-r} =
\sbjtlpp{\Lie(T)}{-r} \oplus
\bigoplus_{\alpha \in \Root(\bG, \bT)}
	(\Lie(G)_{-\alpha} \cap \sbtlpp{\Lie^*(G)}x{-r}),
\]
which dualises to give
\[
\sbtl{\Lie^*(G)}x r =
\sbjtl{\Lie^*(T)} r \oplus
\bigoplus_{\alpha \in \Root(\bG, \bT)}
	(\Lie^*(G)_\alpha \cap \sbtl{\Lie^*(G)}x r),
\]
and analogous decompositions hold over \(E\).  Lemma \ref{lem:dual-MP-by-cclat} gives that \(\sbjtl{\Lie^*(T)}r\) is contained in \(\sbjtl{\Lie^*(\bT)(E)}r\), and hence in \(\sbtl{\Lie^*(\bG)(E)}x r\).  Now fix \(\alpha \in \Root(\bG, \bT)\), and suppose that \(X^*\) belongs to \(\Lie^*(G)_\alpha \cap \sbtl{\Lie^*(G)}x r\).  Let \(Y\) be the element of \(\Lie(G)_{-\alpha}\) such that \(\pair{X^*}Y\) equals \(1\).  Then Lemma \ref{lem:dual-MP-by-root} gives that \(Y\) does not belong to \(\sbtlpp{\Lie(G)}x{-r}\); since \(\Lie(G)_{-\alpha} \cap \sbtlpp{\Lie(\bG)(E)}x{-r}\) is contained in \(\sbtlpp{\Lie(G)}x{-r}\), it follows that \(Y\) does not belong to \(\sbtlpp{\Lie(\bG)(E)}x{-r}\); and another application of Lemma \ref{lem:dual-MP-by-root} gives that \(X^*\) belongs to \(\sbtl{\Lie^*(\bG)(E)}x r\).  Thus \(\sbtl{\Lie^*(G)}x r\) is contained in \(\Lie^*(G) \cap \sbtl{\Lie^*(\bG)(E)}x r\).  The reverse containment, and hence equality, follows from \locpref{arbitrary}.
\end{proof}

\numberwithin{thm}{section}
\section{Jordan decompositions on the dual Lie algebra}
\label{sec:dual-Jordan}

Throughout \S\ref{sec:dual-Jordan}, let
	\field be an algebraically closed field
and	\bG a smooth algebraic group over \field.
We will eventually assume (before Definition \ref{defn:dual-Jordan}) that \bG is reductive, but we do not do so yet.  Let \(\bB^\pm\) be a pair of opposite Borel subgroup of \bG with unipotent radicals \(\bPrad^\pm\), and \bT a maximal torus in \(\bB^+ \cap \bB^-\).

In \xcite{spice:asymptotic}*{Definition \xref{defn:normal-harm}}, an \textit{ad hoc} definition was given for a discriminant function on the dual Lie algebra of a reductive group.  That definition involved some auxiliary choices that made it not completely satisfactory.  We devote \S\ref{sec:dual-Jordan} to remedying this deficiency.  The discriminant function itself is defined in Definition \ref{defn:dual-disc}.

\begin{defn}
\label{defn:embed-dual}
If \bH is a smooth algebraic subgroup of \bG that is normalised by \bT, and for which \(\Lie(\bH)\) is a sum of weight spaces of \bT on \(\Lie(\bG)\), then we write \matnotn{Lie}{\Lie(\bH)^\perp} for the \bT-stable complement to \(\Lie(\bH)\) in \(\Lie(\bG)\), and embed \(\Lie^*(\bH)\) in \(\Lie^*(\bG)\) by extending trivially across \(\Lie(\bH)^\perp\).
\end{defn}

\begin{rem}
\label{rem:embed-dual}
Preserve the notation of Definition \ref{defn:embed-dual}.

Our embeddings of \(\Lie^*(\bB^\pm)\) and \(\Lie^*(\bPrad^\pm)\) in \(\Lie^*(\bG)\) differs from the one in \cite{kac-weisfeiler:Zg}*{\S3, p.~140}.  Namely, our \(\Lie^*(\bB^+)\) is their \(\Lie^*(\bB^-)\), and, similarly, our \(\Lie^*(\bPrad^+)\) is their \(\Lie^*(\bPrad^-)\).  Our choice has the advantage that it makes sense for more subgroups \bH (not just parabolic subgroups of reductive subgroups of reductive groups \bG, and their unipotent radicals), but their choice has the advantage that the extension maps \anonmap{\Lie^*(\bB^+)}{\Lie^*(\bG)} and \anonmap{\Lie^*(\bPrad^+)}{\Lie^*(\bG)} are \(\bB^+\)-equivariant.  Fortunately, since it is the space \(\Lie^*(\bB^+)\) itself, not its labelling (by \(\bB^+\) or \(\bB^-\)), that matters to us, this will cause no difficulty.

Note that \(\Lie^*(\bH)\) is the sum of the \(\alpha\)-weight spaces of \bT on \(\Lie^*(\bG)\), as \(\alpha\) ranges over \(-\Weight(\bH, \bT)\).

The embedding \(\Lie^*(\bH) \to \Lie^*(\bG)\) can depend on the choice of \bT, even if the characteristic of \field is large (or \(0\)), \bG is reductive, and we insist that \bT is actually contained in \bH.  For example, if \bG is \(\GL_2\) and \bH is the Borel subgroup consisting of upper-triangular matrices, then choosing the maximal torus \(\bT_1\) in \bH consisting of diagonal matrices gives \(\Lie(\bH)^\perp_1\) spanned by \(\begin{pmatrix} 0 & 0 \\ 1 & 0 \end{pmatrix}\), but choosing \(\bT_2 = \Int\begin{pmatrix} 1 & 1 \\ 0 & 1 \end{pmatrix}\bT_1\) gives \(\Lie(\bH)^\perp_2\) spanned by \(\begin{pmatrix} 1 & -1 \\ 1 & -1 \end{pmatrix}\).

Contrary to the claim in \cite{kac-weisfeiler:Zg}*{\S3, p.~140}, this dependence can also occur in small residual characteristic even if both \bG and \bH are reductive.  For example, when the characteristic of \field is \(2\), the orthogonal group of a non-degenerate quadratic form \(q\) in \(4\) variables is contained in the symplectic group of the polar \anonmapto{(x, y)}{q(x + y) - q(x) - q(y)} of \(q\).
Thus we obtain an embedding of \(\bH = \operatorname{SO}_4\) in \(\bG = \operatorname{Sp}_4\).  Let \(\bT_1\) be a 
torus that is maximal in \bH, hence also in \bG.  Then \(\Root(\bH, \bT_1)\) consists of the short roots in \(\Root(\bG, \bT_1)\).  If we let \(\alpha\) be a short root and \(\beta\) a long root in \(\Root(\bG, \bT_1)\) so that \(\alpha + \beta\) also belongs to \(\Root(\bG, \bT_1)\), hence is a short root, then the conjugate of a non-\(0\) element of the root space \(\Lie(\bG)_\beta \subseteq \Lie(\bH)^\perp_1\) by a non-\(1\) element \(u\) in the \(\alpha\)-root group \(\bU_\alpha\) has a non-\(0\) component in \(\Lie(\bG)_{\alpha + \beta} \subseteq \Lie(\bH)\), hence does not lie in \(\Lie(\bH)^\perp_1\).  Thus, choosing the maximal torus \(\bT_2 = \Int(u)\bT_1\) gives \(\Lie(\bH)^\perp_2 = \Ad(u)\Lie(\bH)^\perp_1\), which does not equal \(\Lie(\bH)^\perp_1\).

Suppose, though, that \bG and \bH are reductive and \bH contains \bT.  If further \(\Root(\bH, \bT)\) is closed in \(\Root(\bG, \bT)\) (as always happens when we have \(6 \ne 0\) in \field, or, regardless of characteristic, if \bH is the fixer in \bG of a subscheme of \bT or \(\Lie(\bT)\)), then \(\Lie(\bH)^\perp\) is \bH-stable, even if \bH is not connected.  Indeed, \(\Lie(\bH)^\perp\) is clearly \(\gNorm_\bH(\bT)\)-stable.  Suppose that \(\alpha, \beta \in \Root(\bG, \bT)\) satisfy \(\alpha \in \Root(\bH, \bT)\) and \(\beta \notin \Root(\bH, \bT)\).  Then, for all \(i \in \sbjtl\Z 0\), closure implies that \(i\alpha + \beta\) does not belong to \(\Root(\bH, \bT)\).  Since the conjugate of an element of the root space \(\Lie(\bG)_\beta\) by an element of the root subgroup \(\bU_\alpha\) lies in \(\sum_{i \in \sbjtl\Z 0} \Lie(\bG)_{i\alpha + \beta} \subseteq \Lie(\bH)^\perp\), we have that \(\Lie(\bH)^\perp\) is stabilised by \(\bU_\alpha\).  Since \bH is generated by its root subgroups and \(\gNorm_\bH(\bT)\), we have shown the desired stability.

It follows that, under these conditions on \bG and \bH, the subspace \(\Lie(\bH)^\perp\), and hence the embedding of \(\Lie^*(\bH)\) in \(\Lie^*(\bG)\), is independent of the choice of maximal torus \bT in \bG that is contained in \bH.
\end{rem}

Our Proposition \ref{prop:dual-Borel} is \cite{kac-weisfeiler:Zg}*{Lemma 3.3}.  Unfortunately, the proof of that result involves restrictions that we do not want to impose, and without which it is possible for the set \(\Omega\) of regular semisimple elements \cite{kac-weisfeiler:Zg}*{\S3, p.~14} to be empty (for example, if \(p = 2\) and \(\bG = \operatorname{PGL}_2\)%
).
Instead, we follow \cite{borel:linear}*{Proposition 14.25} in considering the orbit of a regular nilpotent element.

\begin{prop}
\label{prop:dual-Borel}
Every orbit of \(G\) in \(\Lie^*(G)\) contains an element of \(\Lie^*(B^+)\).
\end{prop}

\begin{proof}
We may, and do, assume,
upon replacing \bG by its quotient by its unipotent radical (which is contained in \(\bB^+\)), that \bG is reductive.
Note that \(\bB^-\) normalises the Lie algebra of its unipotent radical \(\bPrad^-\), and hence normalises the annihilator in \(\Lie^*(\bG)\) of \(\Lie(\bPrad^-)\), which is just (the image of) \(\Lie^*(\bB^+)\).

Write \(\Simple(\bB^+, \bT)\) for the set of roots in \(\Root(\bG, \bT)\) that are simple with respect to the order coming from \(\bB^+\).
For each \(\alpha \in \Simple(\bB^+, \bT)\), choose a non-\(0\) vector \(E^*_\alpha \in \Lie^*(G)\) that lies in (the image of) the dual of the root space \(\Lie(G)_\alpha\) on \(\Lie(\bG)\), i.e., in the (\(-\alpha\))-root space on \(\Lie^*(\bG)\); and then write \(E^*\) for the sum of the \(E^*_\alpha\).  We claim that the transporter scheme \(\bS \ldef \operatorname{Trans}_\bG(E^*, \Lie^*(\bB^+)) = \set{g \in \bG}{\Ad^*(g)E^* \in \Lie^*(\bB^+)}\) equals \(\bB^-\).  Indeed, since \bS is clearly stable under left translation by \(\bB^-\), it suffices by the Bruhat decomposition \cite{milne:algebraic-groups}*{Proposition 21.70} to show that, if \(n \in \gNorm_G(\bT)\) satisfies \(n\bPrad^- \cap \bS \ne \emptyset\), then \(n\) lies in \(T\).

Choose \(u^- \in \Prad^-\) such that \(\Ad^*(n u^-)E^*\) belongs to \(\Lie^*(B^+)\).
For every \(\alpha \in \Simple(\bB^+, \bT)\), we have that
	\begin{itemize}
	\item the \(\alpha\)-root space is mapped by \(1 - \Ad(u^-)\inv\) into \(\Lie(B^-)\), on which \(E^*\) is trivial;
	\item so \((1 - \Ad^*(u^-))E^*\) is trivial on the \(\alpha\)-root space;
	\item so \(\Ad^*(u^-)E^*\) equals \(E^*\), which equals \(E^*_\alpha\), on the \(\alpha\)-root space, hence is \emph{non}-trivial there;
	\item so \(\Ad^*(n u^-)E^*\) is non-trivial on the \(n\dota\alpha\)-root space;
	\item so \(n\dota\alpha\) is \emph{not} a root of \bT in \(\Lie(\bPrad^-)\).
	\end{itemize}
Since \(\alpha \in \Simple(\bB^+, \bT)\) was arbitrary, this implies that \(n\) lies in \(T\), as desired.

Now consider the maps \anonmap[\overset b\mapdefaultarrow]{\anonmap[\overset a\mapdefaultarrow]{\bG \times \Lie^*(\bG)}{\bG \times \Lie^*(\bG)}}{\bG/\bB^- \times \Lie^*(\bG)}, where \(a\) is given by \anonmapto{(g, X^*)}{(g\bB^-, \Ad^*(g)X^*)} and \(b\) is the natural quotient map.  Write \(\widetilde\bV\) for the image of \(\bG \times \Lie^*(\bB^+)\) under \(b \circ a\).  Then one computes easily that \(b\inv(\widetilde\bV)\) equals \(a(\bG \times \Lie^*(\bB^+))\), which is closed since \(a\) is an isomorphism of varieties; so \(\widetilde\bV\) is closed, since \(b\) is a quotient map, hence an open map.  Since the projection \(\anonmap{\widetilde\bV}{\bG/\bB^-}\) is surjective, and the fibres are conjugates of \(\Lie^*(\bB^+)\), which has dimension \(\dim(\bB^+) = \dim(\bB^-)\), we have that \(\widetilde\bV\) is equi-dimensional of dimension \(\dim \bG = \dim \Lie^*(\bG)\); and that the fibre of the projection \(\anonmap{\widetilde\bV}{\Lie^*(\bG)}\) over \(E^*\) is a singleton.  It follows that, if \bV is the irreducible component of \(\widetilde\bV\) containing \((\bB^-, E^*)\), then the fibres over some Zariski-dense, Zariski-open subset of the image of \bV in \(\Lie^*(\bG)\) are finite \cite{borel:linear}*{Theorem AG.10.1}.  Since \bV is equi-dimensional of dimension \(\dim \Lie^*(\bG)\), it follows that the image of \bV in \(\Lie^*(\bG)\), hence also the image of \(\widetilde\bV\), contains a Zariski-dense and Zariski-open subset of \(\Lie^*(\bG)\).  Finally, since \(\bG/\bB^-\) is complete, we have that the projection \(\anonmap{\bG/\bB^- \times \Lie^*(\bG)}{\Lie^*(\bG)}\) is closed, and so the image of \(\widetilde\bV\) is closed.  It follows that the image of \(\widetilde\bV\) is all of \(\Lie^*(\bG)\).
\end{proof}

For the rest of this section, we assume (in addition to the standing assumption that \field is algebraically closed) that \bG is reductive.

\begin{defn}[\cite{kac-weisfeiler:Zg}*{\S3, p.~140}]
\label{defn:dual-Jordan}
An element of \(\Lie^*(G)\) is semisimple if it belongs to the dual Lie algebra of some maximal torus of \bG, and nilpotent if it belongs to the dual Lie algebra of the unipotent radical of some Borel subgroup of \bG.

For every \(X^* \in \Lie^*(B^+)\), \(S^* \in \Lie^*(T)\), and \(N^* \in \Lie^*(\Prad^+)\), we say that \((S^*, N^*)\) is a \((\bB^+, \bT)\)-standard Jordan decomposition of \(X^*\) if and only if
\(X^*\) equals \(S^* + N^*\), and
\(N^*\) annihilates the root space for \bT in \(\Lie(\bG)\) corresponding to \(\alpha\)
for every \(\alpha \in \Root(\bG, \bT)\) such that \(\pair{S^*}{\upd\alpha^\vee(1)}\) is non-\(0\).

In general, for every triple \((X^*, S^*, N^*)\) of elements of \(\Lie^*(G)\), we say that \((S^*, N^*)\) is a \term[Jordan decomposition!of dual-Lie-algebra element]{Jordan decomposition} of \(X^*\) if there is some \(g \in G\conn\) such that \(\Ad^*(g)(X^*, S^*, N^*)\) is a \((\bB^+, \bT)\)-standard Jordan decomposition of \(\Ad^*(g)X^*\).
\end{defn}

By
\cite{spice-tsai:jordan}*{Theorem 1(ii)},
every element \(X^* \in \Lie^*(G)\) admits a Jordan decomposition \((\matnotn{Xsemi}{X^*\semi}, \matnotn{Xnil}{X^*\nilp})\).  However, contrary to the assertion in \cite{kac-weisfeiler:Zg}*{Theorem 4(iv)}, this decomposition can fail to be unique; see
\cite{spice-tsai:jordan}*{\S7}.

\begin{rem}
\label{rem:which-Borus}
For every \(X^* \in \Lie^*(G)\), Jordan decomposition \((X^*\semi, X^*\nilp)\) of \(X^*\), and \(g \in G\), we have that \((\Ad^*(g)X^*\semi, \Ad^*(g)X^*\nilp)\) is a Jordan decomposition of \(\Ad^*(g)X^*\).  In particular, the definition of a Jordan decomposition of \(X^*\) does not depend on the choice of Borel--torus pair \((\bB^+, \bT)\) in \bG.
\end{rem}

\project[Q:  Is \(\Cent_\bG(X^*)\) always equal to \(\Cent_\bG(X^*\semi) \cap \Cent_\bG(X^*\nilp)\) (equivalently, contained in \(\Cent_\bG(X^*\semi)\))?  This is obvious for the maximal smooth subgroups.
\endgraf
A:  No.  Cheng-Chiang's example that gave rise to \cite{spice-tsai:jordan}*{\S\ref{sec:nonuni}} shows that this can fail, even for maximal smooth subgroups.]

\begin{lem}
\label{lem:dual-contract}
For every \(X^* \in \Lie^*(G)\) and Jordan decomposition \((X^*\semi, X^*\nilp)\) of \(X^*\), there is a cocharacter \(\lambda\) of \(\Cent_\bG(X^*\semi)\) such that \(\lim_{t \to 0} \Ad^*(\lambda(t))X^*\nilp\) equals \(0\).
\end{lem}

\begin{proof}
We may, and do, assume, upon replacing \((X^*, X^*\semi, X^*\nilp)\) by a suitable \(G\conn\)-conjugate, that \((X^*\semi, X^*\nilp)\) is a \((\bB^+, \bT)\)-standard Jordan decomposition of \(X^*\).  Remember that, for our choice of embedding, \(\Lie^*(\bPrad^+)\) is the sum of the root spaces for \bT in \(\Lie^*(\bG)\) corresponding to roots in \(\Root(\bB^-, \bT)\).

In particular, since \(X^*\nilp\) belongs to \(\Lie^*(\Cent_\bG(X^*\semi))\), it lies in the sum of the root spaces corresponding to roots \(-\alpha\) with \(\alpha \in \Root(\Cent_{\bB^+}(X^*\semi), \bT)\).
We may take \(\lambda\) to be a cocharacter of \bT, hence of \(\Cent_\bG(X^*\semi)\), such that \(\bB^-\) is the parabolic subgroup of \bG associated to \(\lambda\) (for example, the sum of the coroots for \bT in \bG associated to elements of \(\Root(\bB^-, \bT)\)).
\end{proof}

\begin{lem}
\label{lem:dual-centre}
For every \(X^* \in \Lie^*(G)\), Jordan decomposition \((X^*\semi, X^*\nilp)\) of \(X^*\), and \(Z^* \in \Lie^*(G)^{G\conn}\), we have that \((X^*\semi + Z^*, X^*\nilp)\) is a Jordan decomposition of \(X^* + Z^*\).
\end{lem}

\begin{proof}
We may, and do, assume, upon replacing \((X^*, X^*\semi, X^*\nilp)\) by a suitable \(G\conn\)-conjugate, that \((X^*\semi, X^*\nilp)\) is a \((\bB^+, \bT)\)-standard Jordan decomposition of \(X^*\).  Since \(Z^*\) is \(G\conn\)-fixed, it already has a \((\bB^+, \bT)\)-standard Jordan decomposition \((Z^*\semi, Z^*\nilp)\).

Since \bG is smooth, we have that \(Z^*\) is \(\bG\conn\)-fixed.
Lemma \ref{lem:dual-contract} gives that there is some cocharacter \(\lambda\) of \bG such that \(\lim_{t \to 0} \Ad^*(\lambda(t))Z^*\nilp\) equals \(0\), so that \(Z^* = \lim_{t \to 0} \Ad^*(\lambda(t))Z^*\) equals \(Z^*\semi\).  In particular, \(Z^*\) belongs to \(\Lie^*(T)\), so \(X^*\semi + Z^*\) belongs to \(\Lie^*(T)\).

Since \(Z^*\) is \(\bG\conn\)-fixed, we have that \(\Cent_{\bG\conn}(X^*\semi + Z^*)\) equals \(\Cent_{\bG\conn}(X^*\semi)\), so \(\Lie^*(\Cent_\bG(X^*\semi + Z^*))\) equals \(\Lie^*(\Cent_\bG(X^*\semi))\), and hence contains \(X^*\nilp\).  Since also \(X^*\nilp\) is nilpotent,
the result follows.
\end{proof}

\begin{cor}
\label{cor:dual-centre}
Every element of \(\Lie^*(G)^{G\conn}\) belongs to \(\Lie^*(T)\).
\end{cor}

\begin{proof}
Applying Lemma \ref{lem:dual-centre} to the Jordan decomposition \((0, 0)\) of \(X^* = 0 \in \Lie^*(G)\) gives that \((Z^*, 0)\) is a Jordan decomposition of \(Z^*\).  That is, \(Z^*\) is semisimple.  \textit{A priori} this means only that there is a maximal torus \(\bT_1\) in \bG such that \(Z^*\) lies in \(\Lie^*(T_1)\); but then there exists \(g \in G\conn\) such that \(\Int(g)\bT_1\) equals \bT, and then \(Z^* = \Ad^*(g)Z^*\) belongs to \(\Ad^*(g)\Lie^*(T_1) = \Lie^*(T)\).
\end{proof}

\begin{lem}
\label{lem:class-disc}
Let \mc S be a finite subset of \(\Lie(T)\) stable under \(\Weyl(\bG, \bT)(\field)\).  The functions \(\map{d_{\mc S}}{\Lie^*(B^+)}\field\) and \(\map{d_{\mc S}\red}{\Lie^*(B^+)}\field\), defined by
\[
d_{\mc S}(X^*) = \prod_{Y \in \mc S} \pair{X^*}Y
\qandq
d_{\mc S}\red(X^*) = \prod_{\substack
	{Y \in \mc S \\
	\pair{X^*}Y \ne 0}
} \pair{X^*}Y
\]
for all \(X^* \in \Lie^*(B^+)\), are constant on the intersection with \(\Lie^*(B^+)\) of every orbit of \(G\) in \(\Lie^*(G)\).
\end{lem}

\begin{proof}
Recall that \(\Lie^*(\bB^+)\) is identified with the annihilator of \(\Lie(\bPrad^-)\) in \(\Lie^*(\bG)\).

Suppose that \(X^* \in \Lie^*(B^+)\) and \(g \in G\) are such that \(\Ad^*(g)X^*\) lies in \(\Lie^*(B^+)\).  Write \(g\) as \(u^- n v^-\), with \(u^-\) and \(v^-\) in \(\Prad^-\) and \(n\) in \(\gNorm_G(\bT)\).  By hypothesis, \(X^*\) and \(\Ad^*(g)X^*\) are trivial on \(\Lie(\Prad^-)\), hence constant on \(\Prad^-\)-orbits through elements of \(\Lie(T)\), so that \(\pair{\Ad^*(g)X^*}Y\) equals \(\pair{\Ad^*(g)X^*}{\Ad(u^-)Y} = \pair{\Ad^*(v^-)X^*}{\Ad(n)\inv Y}\), and \(\pair{X^*}Y\) equals \(\pair{X^*}{\Ad(v^-)\inv Y} = \pair{\Ad^*(v^-)X^*}Y\), for all \(Y \in \mc S\).  Since \(\mc S\) is stabilised by \(n^{-1}\), it follows that
\begin{align*}
d_{\mc S}\red(\Ad^*(g)X^*) \qeqq
& \prod_{\substack
	{Y \in \mc S \\
	\pair{\Ad^*(v^-)X^*}{\Ad(n)\inv Y} \ne 0}
} \pair{\Ad^*(v^-)X^*}{\Ad(n)\inv Y} \\
={} & \prod_{\substack
	{Y' \in \mc S \\
	\pair{\Ad^*(v^-)X^*}{Y'} \ne 0}
} \pair{\Ad^*(v^-)X^*}{Y'} = d_{\mc S}\red(X^*).
\end{align*}
The argument for \(d_{\mc S}\) is almost identical.
\end{proof}

Definition \ref{defn:dual-disc} produces an extension to \(\Lie^*(G)\) of the function \(\redD_G\) defined only on \(\Lie^*(T)\) in \xcite{debacker-spice:stability}*{Definition \xref{defn:disc}}.

\begin{defn}
\label{defn:dual-disc}
With the notation of Lemma \ref{lem:class-disc}, there are unique extensions of \(d_{\mc S}\) and \(d_{\mc S}\red\) from \(\Lie^*(B^+)\) to \(G\)-invariant functions on \(\Lie^*(G)\).  We denote these extensions again by the same symbols.

We write \matnotn{DG}{\redD_\bG} for \(\map{d_{\mc S_\bG}\red}{\Lie^*(G)}{\field\mult}\), where \(\mc S_\bG\) is \(\set{\upd\coroot\alpha(1)}{\alpha \in \Root(\bG, \bT)}\).  If \bH is a reductive subgroup of \bG containing \bT, then we write \matnotn{DGH}{\Disc_{\bG/\bH}} for \(\map{d_{\mc S_{\bG/\bH}}}{\Lie^*(\bH)(\field)}\field\), where \(\mc S_{\bG/\bH}\) is \(\set{\upd\coroot\alpha(1)}{\alpha \in \Root(\bG/\bH, \bT)}\).
\end{defn}

\begin{rem}
\label{rem:dual-disc}
Whenever the dual Chevalley restriction theorem holds, in the sense that the restriction map \(\anonmap{\field[\Lie^*(\bG)]^\bG}{\field[\Lie^*(\bT)]^{\Weyl(\bG, \bT)}}\) is an isomorphism, the function \(d_{\mc S}\) of Definition \ref{defn:dual-disc} extends to an element of \(\field[\Lie^*(\bG)]^\bG\) (although \(d_{\mc S}\red\) is never polynomial unless it is identically \(1\)).  For example, this happens when \bG is almost simple and the characteristic of \field is not \(2\) \cite{kac-weisfeiler:Zg}*{Theorem 3.4(i)};
or, more generally, when the characteristic of \field is not \(2\), and \bG has no factor of odd orthogonal type \cite{spice-tsai:jordan}*{Theorem 1(iv, v)}.
However, \(d_{\mc S}\) can fail to be polynomial when the characteristic of \field equals \(2\).  For example, suppose that \(\bG = \operatorname{PGL}_2\); \bT, respectively \(\bB^+\), is the image in \bG of the group of diagonal, respectively upper-triangular, matrices in \(\GL_2\); and \mc S is the singleton containing the image in \(\Lie(\bT)\) of \(\operatorname{diag}(1, 0)\) (which is \(\Weyl(\bG, \bT)(\field)\)-fixed by our assumption on the characteristic).
Then the (non-degenerate, \(\GL_2\)-invariant) trace pairing \anonmap{\gl_2 \otimes \gl_2}{\gl_1} descends to a non-degenerate, \(\GL_2\)-invariant pairing \anonmap{\mf{pgl}_2 \otimes \mf{sl}_2}{\gl_1}, thus furnishing a \(\GL_2\)-equivariant isomorphism \(\Lie^*(\bG) = \mf{pgl}_2^* \cong \mf{sl}_2\) that identifies \(\Lie^*(\bB^+)\) with the subalgebra of \(\mf{sl}_2\) consisting of \emph{lower}-triangular matrices; and, with respect to this identification, the function \(d_{\mc S}\) is given on lower-triangular matrices by \(\anonmapto{\begin{pmatrix} a & 0 \\ c & -a \end{pmatrix}}a\).  Since the \(\operatorname{PGL}_2(\field)\)-orbit (which is the same as the \(\GL_2(\field)\)- and the \(\operatorname{SL}_2(\field)\)-orbits) of an arbitrary element \(\begin{pmatrix} a & b \\ c & -a \end{pmatrix} \in \mf{sl}_2(\field)\) contains an element of the form \(\begin{pmatrix} \sqrt{a^2 - b c} & 0 \\ * & -\sqrt{a^2 - b c} \end{pmatrix}\),
we find that the function \(d_{\mc S}\) is given on \(\mf{pgl}_2^*(\field) \cong \mf{sl}_2(\field)\) by the non-polynomial function \anonmapto{\begin{pmatrix} a & b \\ c & -a \end{pmatrix}}{\sqrt{a^2 - b c} = a + \sqrt{b c}}.

Nonetheless, even when the characteristic of \field equals \(2\), we have that \(d_{\mc S}\) extends to a polynomial function on \(\Lie^*(\bG)\) whenever \(\mc S \subseteq \Lie(T)\) admits a lift to the Lie algebra of the simply connected cover of \bG; for then we may, and do, assume that \bG itself is simply connected.  In this case, we have that \(\Lie^*(\bT)\) is spanned by the images of the fundamental weights, so that the proof of \cite{springer-steinberg:conj}*{\S II.3.17\textquotesingle} goes through unchanged.  This applies, in particular, to \(\Disc_{\bG/\bH}\).
\end{rem}

\begin{lem}
\label{lem:disc-as-det}
Let \(\Gamma\) be an element of \(\Lie^*(T)\), viewed as a subspace of \(\Lie^*(G)\).  Put \(\bG' = \Cent_\bG(\Gamma)\).
Then the map
\[
\map{i_\Gamma \ldef \ad^*(\anondot)\Gamma}{\Lie(G)/{\Lie(G')}}{\Lie^*(G)/{\Lie^*(G')}}
\]
is invertible; and, for every \(X^* \in \Lie^*(G')\), the determinant of the endomorphism \(i_\Gamma\inv \circ \bigl(\ad^*(\anondot)X^*\bigr)\) of \(\Lie(G)/{\Lie(G')}\) is \(\Disc_{G/G'}(X^*)\Disc_{G/G'}(\Gamma)\inv\).
\end{lem}

\begin{proof}
By \cite{kac-weisfeiler:Zg}*{Lemma 3.1(ii, iii)}, we have that \(\bG'\) is reductive, and \(\Root(\bG', \bT)\) equals \(\set{\alpha \in \Root(\bG, \bT)}{\pair\Gamma{\upd\alpha^\vee(1)} = 0}\).

For each \(\alpha \in \Root(\bG, \bT)\), let \(E_\alpha\) be a non-\(0\) vector in the corresponding root subspace \(\Lie(G)_\alpha\) of \(\Lie(G)\).  Then the image of \(\set{E_\alpha}{\alpha \in \Root(\bG/\bG', \bT)}\) in \(\Lie(G)/\Lie(G')\) is a basis of \(\Lie(G)/\Lie(G')\).  If \(\alpha\) and \(\beta\) belong to \(\Root(\bG/\bG', \bT)\), then \(\pair{i_\Gamma(E_\alpha)}{E_\beta} = \pair{\ad^*(E_\alpha)\Gamma}{E_\beta} = \pair\Gamma{[E_\beta, E_\alpha}]\) is \(0\) unless \(\alpha + \beta\) equals \(0\), in which case it is a non-\(0\) multiple of \(\pair\Gamma{\upd\alpha^\vee(1)}\), hence is itself non-\(0\).  That is, the matrix of \(i_\Gamma\) with respect to our chosen basis is anti-diagonal, with non-\(0\) entries on the anti-diagonal.  In particular, \(i_\Gamma\) is invertible.

Fix \(X^* \in \Lie^*(G')\).  By Proposition \ref{prop:dual-Borel}, there is a Borel subgroup \(\bB\primethen+\) of \(\bG'\) such that \(X^*\) belongs to \(\Lie^*(B\primethen+)\).  There is some \(g' \in G'\) such that \(\Int(g')\bT\) is contained in \(\bB\primethen+\).  Upon replacing \bT by \(\Int(g')\bT\), we may, and do, assume that \bT is contained in \(\bB\primethen+\).

Choose a total ordering on \(\Root(\bG/\bG', \bT)\) extending the partial order coming from \(\bB\primethen+\).  For every \(\alpha \in \Root(\bG/\bG', \bT)\), we have that \(i_\Gamma\inv(\ad^*(E_\alpha)X^*)\) lies in \(\Lie(G)_\alpha + \sum_{\beta > \alpha} \Lie(G)_\beta\), and satisfies
\[
\pair\Gamma{\comm{E_{-\alpha}}{i_\Gamma\inv(\ad^*(E_\alpha)X^*)}}
= \pair{\ad^*(E_\alpha)X^*}{E_{-\alpha}}
= \pair{X^*}{\comm{E_{-\alpha}}{E_\alpha}},
\]
so that its \(\Lie(G)_\alpha\)-component is the product of the scalar \(\pair{X^*}{\comm{E_{-\alpha}}{E_\alpha}}\pair\Gamma{\comm{E_{-\alpha}}{E_\alpha}}\inv\) and the vector \(E_\alpha\).  Since \(\comm{E_{-\alpha}}{E_\alpha}\) is a non-\(0\) multiple of \(\upd\coroot\alpha(1)\), the scalar equals \(\pair{X^*}{\upd\coroot\alpha(1)}\pair\Gamma{\upd\coroot\alpha(1)}\inv\).

The matrix of the indicated endomorphism with respect to our chosen basis is thus upper triangular with explicitly computed diagonal entries.  The result follows from Definition \ref{defn:dual-disc}.
\end{proof}

\numberwithin{thm}{subsection}
\section{The subgroups $\CC\bG i(\gamma)$}
\label{sec:funny-centraliser}

Throughout \S\ref{sec:funny-centraliser}, let \field be a field, and \bG an algebraic group over \field.  Unless otherwise stated, we do not assume that \bG is smooth or connected.  We will eventually assume (in \S\ref{subsec:reductive-gp-funny}) that \bG is reductive, but we do not do so yet.

In \xcite{spice:asymptotic}*{Hypothesis \xref{hyp:funny-centraliser}}, the existence of groups \(\CC\bG i(\gamma\pinv)\) (and, in \xcite{spice:asymptotic}*{Hypothesis \xref{hyp:gamma}}, also \(\CC\bG i(\gamma\pinv[2])\)) depending on \(i \in \tR\) and satisfying certain conditions is postulated.  In \S\ref{sec:funny-centraliser}, we define these groups for all \(i \in \sbjtl\tR 0\), in which case the superscript `\(\pinv\)' is irrelevant, as well as analogues associated to an element of \(\Lie(G)\).  See Definition \ref{defn:vGvr} for how we handle the absence of these groups when \(i\) is negative.

\subsection{Group actions and Lie algebras}
\label{subsec:act-Lie}

Throughout \S\ref{subsec:act-Lie}, we fix
	an algebraic group \(\Gamma\) over \field
and	an action \anonmap{\Gamma \times \bG}\bG of \(\Gamma\) on \bG by automorphisms.
Unless otherwise indicated, we do not assume that \(\Gamma\) is smooth or connected.

Lemmas \ref{lem:Lie-cent} and \ref{lem:fixed-surjective} are structural results that we use constantly throughout \S\ref{sec:funny-centraliser}, usually without explicit mention.

\begin{lem}
\label{lem:Lie-cent}
We have that \(\Lie(\Cent_\bG(\Gamma))\) equals \(\Cent_{\Lie(\bG)}(\Gamma)\).
\end{lem}

\begin{proof}
This is almost \cite{milne:algebraic-groups}*{Proposition 10.34}, except that that result requires that \(\Gamma\) be a subgroup of \bG acting by conjugation.  The proof goes through \textit{literatim} without that requirement.  Alternatively, we could embed \bG and \(\Gamma\) in some common overgroup \(\widetilde\bG\), such as \(\Gamma \ltimes \bG\), so that \(\Gamma\) is contained in \(\gNorm_{\widetilde\bG}(\bG)\); and then apply \cite{milne:algebraic-groups}*{Proposition 10.34} directly to deduce the equalities
\[
\Lie(\Cent_\bG(\Gamma))
= \Lie(\Cent_{\widetilde\bG}(\Gamma) \cap \bG)
= \Cent_{\Lie(\widetilde\bG)}(\Gamma) \cap \Lie(\bG)
= \Cent_{\Lie(\bG)}(\Gamma).\qedhere
\]
\end{proof}

\begin{lem}[\cite{conrad-gabber-prasad:prg}*{Example A.1.12 and Proposition A.8.14(1)}]
\label{lem:fixed-surjective}
If
	\bG is smooth,
	\(\Gamma\) is of multiplicative type,
and	\anonmap\bG{\bG'} is a quotient by a \(\Gamma\)-stable, normal subgroup,
then \anonmap{\Cent_\bG(\Gamma)\conn}{\Cent_{\bG'}(\Gamma)\conn} is a quotient map.
\end{lem}

\begin{lem}
\label{lem:cent-by-Lie}
Suppose that \bJ is a 
\(\Gamma\)-stable, connected subgroup of \bG such that \(\Cent_\bJ(\Gamma)\) is smooth.
If \(\Lie(\bJ)\) contains (respectively, is contained in) \(\Cent_{\Lie(\bG)}(\Gamma)\), then \bJ contains (respectively, is contained in) \(\Cent_\bG(\Gamma)\conn\).
\end{lem}

\begin{proof}
We have that \(\Lie(\Cent_\bG(\Gamma))\) equals \(\Cent_{\Lie(\bG)}(\Gamma)\) and \(\Lie(\Cent_\bJ(\Gamma))\) equals \(\Cent_{\Lie(\bJ)}(\Gamma)\).

We now use \cite{milne:algebraic-groups}*{Proposition 10.15}, which says that, if a smooth subgroup of a connected group has the same Lie algebra as the full group, then the smooth subgroup and the connected group are equal.

If \(\Lie(\bJ)\) is contained in \(\Cent_{\Lie(\bG)}(\Gamma)\), then \(\Lie(\Cent_\bJ(\Gamma))\) equals \(\Lie(\bJ)\), so \(\Cent_\bJ(\Gamma)\conn\) equals \(\bJ\conn = \bJ\), so that \bJ is contained in \(\Cent_\bG(\Gamma)\conn\).

If \(\Lie(\bJ)\) contains \(\Cent_{\Lie(\bG)}(\Gamma)\), then \(\Lie(\Cent_\bJ(\Gamma))\) equals \(\Lie(\Cent_\bG(\Gamma))\), so \(\Cent_\bJ(\Gamma)\conn\) equals \(\Cent_\bG(\Gamma)\conn\), so that \bJ contains \(\Cent_\bG(\Gamma)\conn\).
\end{proof}

\subsection{$\gamma$ an automorphism of \bG}
\label{subsec:gp-funny}

{
\let\usualCC=\CC
\let\usualCCp=\CCp
\let\usualZZ=\ZZ
\renewcommand\CC[2]{\Cent_{#1}^{(#2)}}
\renewcommand\ZZ[2]{\Zent_{#1}^{(#2)}}
\newcommand\wCC[2]{\vphantom{\widetilde\Cent}\smash{\widetilde\Cent}_{#1}^{(#2)}}

Definition \ref{defn:GL-fc} begins with an elementary linear-algebra definition on which we pin the entirety of our construction.

The groups \(\ZZ{\GL(\bV)}{\mc L}(\gamma)\) and \(\CC{\GL(\bV)}{\mc L}(\gamma)\) in Definition \ref{defn:GL-fc} are not obviously intrinsic; that is, they appear to depend not just on the group \(\GL(\bV)\), but on its specific representation on \bV.  Proposition \initref{prop:gp-dfc-rep}\subpref{regular} shows that in fact \(\CC{\GL(\bV)}{\mc L}(\gamma)\), hence also its centre \(\ZZ{\GL(\bV)}{\mc L}(\gamma)\), \emph{can} be defined purely in terms of the algebraic-group structure on \(\GL(\bV)\).

\begin{defn}
\label{defn:GL-fc}
Suppose that
	\begin{itemize}
	\item \bV is a (possibly infinite dimensional) vector space over \field,
	\item \(\gamma\) is a locally diagonalisable automorphism of \bV,
and	\item \mc L is a subgroup of \(\field\mult\).
	\end{itemize}
For every \(\lambda \in \field\mult\), we define the \mc L-close \(\lambda\)-eigenspace for \(\gamma\) in \bV to be the sum of the eigenspaces for \(\gamma\) in \bV corresponding to eigenvalues in \(\lambda\mc L\), and write \(\CC\bV{\mc L}(\gamma)\) for the \mc L-close \(1\)-eigenspace for \(\gamma\) in \bV.

If \bV is finite dimensional, then we write \(\ZZ{\GL(\bV)}{\mc L}(\gamma)\) for the split torus in \(\GL(\bV)\) corresponding to the direct-sum decomposition of \bV as the sum of its \mc L-close eigenspaces for \(\gamma\), and \(\CC{\GL(\bV)}{\mc L}(\gamma)\) for \(\Cent_{\GL(\bV)}(\ZZ{\GL(\bV)}{\mc L}(\gamma))\).

If \field is equipped with a discrete valuation, then, for each \(i \in \sbjtl\tR 0\), put
\(\usualCC\bV i(\gamma) = \CC\bV{\sbjtl\field i}(\gamma)\), \(\usualZZ{\GL(\bV)}i(\gamma) = \ZZ{\GL(\bV)}{\sbjtl\field i}(\gamma)\), and \(\usualCC{\GL(\bV)}i(\gamma) = \CC{\GL(\bV)}{\sbjtl\field i}(\gamma)\).
\end{defn}

In the notation of Definition \ref{defn:GL-fc}, we have that \(\CC{\GL(\bV)}{\mc L}(\gamma)\) is the subgroup of \(\GL(\bV)\) preserving every \mc L-close eigenspace of \(\gamma\) in \(\bV\).

Let \anonmap\bG{\GL(\matnotn V\bV)} be a faithful, finite-dimensional representation of \bG, and fix a semisimple element \(\matnotn{gamma}\gamma \in \gNorm_{\GL(V)}(\bG)\).  Thus, the automorphism of \bG induced by \(\gamma\) is semisimple, in the sense of \cite{steinberg:endomorphisms}*{p.~51}, so the automorphism of \(\bG_\algfield\) induced by \(\gamma\) is quasi-semisimple, in the sense of \cite{steinberg:endomorphisms}*{p.~59}, by \cite{steinberg:endomorphisms}*{Theorem 7.5}.
We do not require that \(\gamma\) belong to \(G\).

\begin{prop}
\initlabel{prop:gp-dfc-rep}
Let \mc L be a subgroup of \(\field\mult\).
Suppose that \(\gamma\) is diagonalisable.
\begin{enumerate}
\item\sublabel{regular}
The group \(\bG \cap \CC{\GL(\bV)}{\mc L}(\gamma)\) is the subgroup of \bG that stabilises every \mc L-close eigenspace for \(\gamma\) in the regular representation \(\field[\bG]\) of \bG.
\item\sublabel{all}
For every (\(\sgen\gamma \ltimes \bG\))-module \bW (not necessarily faithful or finite dimensional), we have that \(\bG \cap \CC{\GL(\bV)}{\mc L}(\gamma)\) stabilises every \mc L-close eigenspace for \(\gamma\) in \bW.
\end{enumerate}
\end{prop}

\begin{proof}
Put \(\Gamma = \sgen\gamma\).  (Remember that this means the smallest subgroup scheme of \bG containing \(\gamma\), not the abstract group generated by \(\gamma\).)

Our central observation is that, if \(\wtilde\bG'\) is a subgroup functor of \bG, then the set of (\(\Gamma \ltimes \bG\))-modules \bW for which \(\wtilde\bG'\) stabilises every \mc L-close eigenspace for \(\gamma\) in \bW
is closed under
	\begin{itemize}
	\item subspaces, since close eigenspaces in a submodule \(\bW'\) of \bW are intersections with \(\bW'\) of close eigenspaces in \bW;
	\item quotients, since close eigenspaces in a quotient module \(\bW'\) of \bW are the images in \(\bW'\) of close eigenspaces in \bW;
	\item duals, since close eigenspaces in \(\contra\bW\) are dual to close eigenspaces in \bW;
	\item directed unions, since an eigenvector in \(\bigcup_{i \in I} \bW_i\) lies in some \(\bW_i\);
	\item direct sums, since close eigenspaces in \(\bW_1 \oplus \bW_2\) are sums of close eigenspaces in \(\bW_1\) and \(\bW_2\);
and	\item tensor products, since close eigenspaces in \(\bW_1 \otimes \bW_2\) are sums of tensor products of eigenspaces in \(\bW_1\) and \(\bW_2\).
	\end{itemize}
(There is `coalescing' in the tensor-product case, where different close eigenspaces in \(\bW_1\) and \(\bW_2\) contribute to the same close eigenspace in \(\bW_1 \otimes \bW_2\), but this does not cause trouble.)

Write \(\wCC\bG{\mc L}(\gamma)\) for the subgroup functor of \bG that preserves every \mc L-close eigenspace for \(\gamma\) in \(\field[\bG]\), and analogously for \(\wCC{\Gamma \ltimes \bG}{\mc L}(\gamma)\).

For every \field-vector space \bM, we write \(\bM\textsub{\bG-triv}\) for the trivial \bG-module structure on \bM, and similarly for `(\(\Gamma \ltimes \bG\))-triv'.  Then
\anonmap{\field[\Gamma]\textsub{\bG-triv} \otimes \field[\bG]}{\field[\Gamma \ltimes \bG]} is a \bG-module isomorphism, so the map \anonmap{\wCC\bG{\mc L}(\gamma)}{\Gamma \ltimes \bG} factors through \anonmap{\wCC{\Gamma \ltimes \bG}{\mc L}(\gamma)}{\Gamma \ltimes \bG}.

For \locpref{all}, since \bW is the union of its finite-dimensional (\(\Gamma \ltimes \bG\))-submodules, we may, and do, assume that \bW is finite dimensional.  Then the matrix-coefficient map \anonmap{\bW \otimes \contra\bW\textsub{(\(\Gamma \ltimes \bG\))-triv}}{\field[\Gamma \ltimes \bG]} given by \anonmapto{w \otimes \contra w}{\anonmapto{(\delta, g)}{\pair{\contra w}{(\delta, g)\dota w}}} dualises to an embedding \anonmap\bW{\field[\Gamma \ltimes \bG] \otimes \bW\textsub{(\(\Gamma \ltimes \bG\))-triv}}.  In particular, since \(\wCC{\Gamma \ltimes \bG}{\mc L}(\gamma)\) stabilises every \mc L-close eigenspace for \(\gamma\) in \(\field[\Gamma \ltimes \bG] \otimes \bW\textsub{(\(\Gamma \ltimes \bG\))-triv}\), it also stabilises every \mc L-close eigenspace for \(\gamma\) in \bW.  Applying this to \(\bW = \bV\) gives that \(\wCC\bG{\mc L}(\gamma) \subseteq \wCC{\Gamma \ltimes \bG}{\mc L}(\gamma)\) is contained in \(\bG \cap \CC{\GL(\bV)}{\mc L}(\gamma)\).  This establishes \locpref{all}.

It remains only to show that \(\bG \cap \CC{\GL(\bV)}{\mc L}(\gamma)\) is contained in \(\wCC\bG{\mc L}(\gamma)\).  As in \cite{deligne:hodge-cycles}*{proof of Proposition 3.1(a)}, the closed embedding \anonmap{\anonmap\bG{\GL(\bV)}}{\gl(\bV) \times \gl(\contra\bV)} realises \(\field[\bG]\) as a quotient of \(\field[\gl(\bV)] \otimes \field[\gl(\contra\bV)]\).  Since \(\bG \cap \CC{\GL(\bV)}{\mc L}(\gamma)\) preserves every \mc L-close eigenspace for \(\gamma\) in the symmetric algebras \(\field[\gl(\bV)]\) and \(\field[\gl(\contra\bV)]\), it also preserves every such close eigenspace in \(\field[\bG]\), as desired.  This establishes \locpref{regular}.
\end{proof}

Now let \matnotn L{\mc L} be a \(\Gal(\sepfield/\field)\)-stable subgroup of \(\field\twosup\sepsup\multsup\).

\begin{defn}
\label{defn:gp-dfc}
Write \(\CC\bV{\mc L}(\gamma)\) and \(\CC\bG{\mc L}(\gamma)\) for the descents to \field of \(\CC{\bV_\sepfield}{\mc L}(\gamma)\) and \(\CC{\GL(\bV_\sepfield)}{\mc L}(\gamma) \cap \bG_\sepfield\).
If \sepfield is equipped with a \(\Gal(\sepfield/\field)\)-fixed discrete valuation, then, for each \(i \in \sbjtl\tR 0\), put \(\usualCC\bV i(\gamma) = \CC\bV{\sbjtl{\field\twosup\sepsup\multsup}i}(\gamma)\) and \(\usualCC\bG i(\gamma) = \CC\bG{\sbjtl{\field\twosup\sepsup\multsup}i}(\gamma)\).
\end{defn}


\begin{rem}
\initlabel{rem:gp-dfc-facts}
\hfill\begin{enumerate}
\item\sublabel{central}
Proposition \initref{prop:gp-dfc-rep}\subpref{regular} gives that the group \(\CC\bG{\mc L}(\gamma)\) depends only on (\bG, \mc L, and) the automorphism of \bG induced by \(\gamma\), not on the specific element \(\gamma \in \GL(V)\) inducing that automorphism or the faithful representation \anonmap\bG{\GL(\bV)}.
\item\sublabel{sub}
If \(\bG'\) is a \(\gamma\)-stable subgroup of \bG, then \(\CC{\bG'}{\mc L}(\gamma)\) equals \(\CC\bG{\mc L}(\gamma) \cap \bG'\).
\item\sublabel{unfunny}
If \(\mc L = 1\) is the trivial subgroup, then \(\Cent_\bV(\gamma)\) equals \(\CC\bV{\mc L}(\gamma)\) and \(\Cent_\bG(\gamma)\) equals \(\CC\bG{\mc L}(\gamma)\).
\item\sublabel{monotone}
If \(\mc L'\) is a \(\Gal(\sepfield/\field)\)-stable subgroup of \(\field\twosup\sepsup\multsup\), then \(\CC\bV{\mc L}(\gamma) \cap \CC\bV{\mc L'}(\gamma)\) equals \(\CC\bV{\mc L \cap \mc L'}(\gamma)\) and, since every \mc L- and every \(\mc L'\)-close eigenspace is a sum of (\(\mc L \cap \mc L'\))-close eigenspaces, also \(\CC\bG{\mc L}(\gamma) \cap \CC\bG{\mc L'}(\gamma)\) equals \(\CC\bG{\mc L \cap \mc L'}(\gamma)\).
\item\sublabel{section}
If \(\gamma\) is diagonalisable, \mc G is the subgroup \(\set{\chi(\gamma)}{\chi \in \clat(\sgen\gamma)}\) of \(\field\mult\) generated by the eigenvalues of \(\gamma\) on \bV, and there is an embedding \map f{\mc G/(\mc G \cap \mc L)}{\field\mult}, then the element \(\delta \ldef f \circ \gamma\) of \(\Hom(\clat(\sgen\gamma), \field\mult) = \sgen\gamma(\field)\) lies in \(\ZZ{\GL(V)}{\mc L}(\gamma)\) and satisfies \(\CC\bG{\mc L}(\gamma) = \Cent_\bG(\delta)\).
\end{enumerate}
\end{rem}

Note that Lemma \ref{lem:gp-dfc-im} does not claim that the image of \(\CC{\wtilde\bG}{\mc L}(\gamma)\) \emph{equals} \(\CC\bG{\mc L}(\gamma)\), and indeed this can fail in general (as is already seen by considering \(\wtilde\bG = \GL_2\), \(\bG = \GL_2\), \(\mc L = 1\), \(\tilde\gamma = \operatorname{diag}(1, -1)\), and the natural map \anonmap{\wtilde\bG}\bG, if \(-1\) does not equal \(1\) in \field); but see Proposition \initref{prop:gp-cfc-facts}\subpref{quotient}.

Recall that \(\sgen\gamma\) denotes the smallest subgroup scheme of \(\GL(\bV)\) containing \(\gamma\), not the abstract group generated by \(\gamma\); and similarly for \(\sgen{\tilde\gamma}\), regarded as a subgroup scheme of the implicit overgroup of \(\wtilde\bG\) in which \(\tilde\gamma\) lies.

\begin{lem}
\label{lem:gp-dfc-im}
Suppose that
	\(\wtilde\bG\) is an algebraic group over \field,
	\(\tilde\gamma\) is a semisimple automorphism of \(\wtilde\bG\),
and	\map f{\sgen{\tilde\gamma} \ltimes \wtilde\bG}{\sgen\gamma \ltimes \bG} is a homomorphism sending \(\tilde\gamma\) to \(\gamma\).
Then the image of \(\CC{\wtilde\bG}{\mc L}(\tilde\gamma)\) under \(f\) is contained in \(\CC\bG{\mc L}(\gamma)\).
\end{lem}

\begin{proof}
We may, and do, assume, upon replacing \field by \sepfield, that \field is separably closed.

Composing \(f\) with the faithful representation \anonmap{\sgen\gamma \ltimes \bG}{\GL(\bV)} yields a representation of \(\sgen{\tilde\gamma} \ltimes \wtilde\bG\) on \bV.

Proposition \initref{prop:gp-dfc-rep}\subpref{all} gives that \(\CC{\wtilde\bG}{\mc L}(\gamma)\) preserves every \mc L-close eigenspace for \(\tilde\gamma\) in \bV, so that its image in \(\GL(\bV)\) lies in \(\CC{\GL(\bV)}{\mc L}(\tilde\gamma) = \CC{\GL(\bV)}{\mc L}(\gamma)\); but this means that its image in \bG lies in \(\bG \cap \CC{\GL(\bV)}{\mc L}(\gamma) = \CC\bG{\mc L}(\gamma)\).
\end{proof}

Lemma \initref{lem:gp-dfc-power}\subpref{exact} refines Remark \initref{rem:gp-dfc-facts}\subpref{section}.  Together with Remark \initref{rem:gp-dfc-facts}\subpref{monotone}, it establishes part of \xcite{spice:asymptotic}*{Hypothesis \initxref{hyp:gamma}(\subxref{bi-Lie-Lie})} (the part that does not involve Moy--Prasad filtrations) when \field is a valued field in which the valuation of \(2\) is \(0\).

\begin{lem}
\initlabel{lem:gp-dfc-power}
\hfill\begin{enumerate}
\item\sublabel{contain}
For every \(n \in \Z\), we have the containments
\(\CC\bV{\mc L}(\gamma) \subseteq \CC\bV{\mc L^n}(\gamma^n)\) and \(\CC\bG{\mc L}(\gamma) \subseteq \CC\bG{\mc L^n}(\gamma^n)\).
If \(\mu_n(\sepfield)\) is contained in \mc L, then both containments are equalities.
\item\sublabel{exact}
If
	the automorphism of \bG induced by \(\gamma\) has finite order \(N\),
and	if \(n \in \Z\) is such that \(\mu_n(\sepfield)\) contains \(\mu_N(\sepfield) \cap \mc L\) and is contained in \mc L,
then \(\Cent_\bG(\gamma^n)\) equals \(\CC\bG{\mc L}(\gamma)\).
\end{enumerate}
\end{lem}

\begin{proof}
We may, and do, assume, upon replacing \field by \sepfield, that \(\gamma\) is diagonalisable.

\locpref{contain} follows from the fact that, for every \(\lambda \in \field\mult\), the \(\mc L^n\)-close \(\lambda^n\)-eigenspace for \(\gamma^n\) is the sum of the \mc L-close \(\zeta\lambda\)-eigenspaces for \(\gamma\), as \(\zeta\) ranges over \(\mu_n(\field)\).  In particular, if \(\mu_n(\field)\) is contained in \mc L, then the \(\mc L^n\)-close \(\lambda^n\)-eigenspace for \(\gamma^n\) equals the \mc L-close \(\lambda\)-eigenspace for \(\gamma\).


For \locpref{exact}, use Remark \initref{rem:gp-dfc-facts}\subpref{central} to replace \(\gamma\) by its image in \(\Weyl(\Gamma, \bG)(\field)\), and \anonmap\bG{\GL(\bV)} by a faithful representation \anonmap{\Weyl(\Gamma, \bG) \ltimes \bG}{\GL(\bV)}.  Now the element \(\gamma \in \GL(V)\) itself, not just its action on \bG, has order \(N\).  We have that the \mc L-close eigenspaces for \(\gamma\) in \bV are precisely the eigenspaces for \(\gamma^n\) in \bV, so that \(\CC{\GL(\bV)}{\mc L}(\gamma)\) equals \(\Cent_{\GL(\bV)}(\gamma^n)\), and hence that \(\CC\bG{\mc L}(\gamma)\) equals \(\Cent_\bG(\gamma^n)\).
\end{proof}

\begin{lem}
\label{lem:reg-rep-roots}
If \bG is a split, connected, reductive group and \bT is a split, maximal torus in \bG, then every weight of \bT in \(\field[\bG]\) belongs to \(\Z\Root(\bG, \bT)\).
\end{lem}

\begin{proof}
The action of \bT on \bG, and hence on \(\field[\bG]\), factors through \(\bT/{\Zent(\bG)}\); so all weights of \bT on \(\field[\bG]\) lie in \(\clat(\bT/{\Zent(\bG)}) = \Z\Root(\bG, \bT)\).
\end{proof}

\begin{prop}
\label{prop:gp-dfc-weights}
If \(\gamma\) is diagonalisable, then \(\field[\CC\bG{\mc L}(\gamma)]\) is the maximal quotient algebra of \(\field[\bG]\) on which all eigenvalues of \(\gamma\) lie in \mc L.
\end{prop}

\begin{proof}
Since \(\CC\bG{\mc L}(\gamma)\) equals \(\CC{\GL(\bV)}{\mc L}(\gamma) \cap \bG\), so that \(\field[\CC\bG{\mc L}(\gamma)]\) 
equals \(\field[\CC{\GL(\bV)}{\mc L}(\gamma)] \otimes_{\field[\GL(\bV)]} \field[\bG]\), it suffices to handle the case \(\bG = \GL(\bV)\).


In this case, by Definition \ref{defn:GL-fc}, we have that \(\CC\bG{\mc L}(\gamma)\) equals \(\Cent_\bG(\ZZ\bG{\mc L}(\gamma))\).  Upon using the closed embedding \anonmap\bG{\Lie(\bG) \oplus \gl_1} given by \anonmapto g{(g, \det(g)\inv)} to identify \bG with a closed subset of \(\Lie(\bG) \oplus \gl_1\), we find that \(\CC\bG{\mc L}(\gamma) = \Cent_\bG(\ZZ\bG{\mc L}(\gamma))\) equals \(\Cent_{\Lie(\bG) \oplus \gl_1}(\ZZ\bG{\mc L}(\gamma)) \cap \bG\).   Put \(\bW = \Lie(\bG) \oplus \gl_1\), \(\bW' = \Cent_\bW(\ZZ\bG{\mc L}(\gamma))\), \(A = \field[\bW]\), and \(A' = \field[\bW']\).  As before, it suffices to show that \(A'\) is the maximal quotient algebra of \(A\) on which all eigenvalues of \(\gamma\) lie in \mc L.

Since \(\ZZ\bG{\mc L}(\gamma)\) is a split torus (by Definition \ref{defn:GL-fc}), we have that \(\bW\) is diagonalisable as a \(\ZZ\bG{\mc L}(\gamma)\)-module \cite{milne:algebraic-groups}*{Theorem 12.12}.  In particular, the sum \(\bW^\perp\) of the non-trivial-weight spaces for \(\ZZ\bG{\mc L}(\gamma)\) in \bW is a \(\ZZ\bG{\mc L}(\gamma)\)-stable complement to \(\bW' \ldef \Cent_\bW(\ZZ\bG{\mc L}(\gamma))\) in \bW.  We use this complement to furnish a \(\ZZ\bG{\mc L}(\gamma)\)-equivariant embedding of \(\bW\thendual\perp\) in \(\bW^*\), with image the sum of the non-trivial-weight spaces for \(\ZZ\bG{\mc L}(\gamma)\) in \(\bW^*\).  Thus, the ideal of \(A\) generated by the non-trivial-weight vectors for \(\ZZ\bG{\mc L}(\gamma)\) in \(A\) contains \(\field[\bW'] \otimes_\field \bW\thendual\perp\), and hence the kernel of \anonmap{A \cong \field[\bW'] \otimes_\field \field[\bW^\perp]}{\field[\bW'] = A'}.  That is, \(A'\) is the maximal quotient algebra of \(A\) on which \(\ZZ\bG{\mc L}(\gamma)\) acts with only trivial weights.

We are finally trying to show that the maximal quotient algebra of \(A\) on which \(\ZZ\bG{\mc L}(\gamma)\) acts with only trivial weights is the same as the maximal quotient algebra of \(A\) on which \(\gamma\) acts with all eigenvalues in \mc L, i.e., that the ideals of \(A\) generated by the non-trivial-weight spaces for \(\ZZ\bG{\mc L}(\gamma)\), and by the eigenvectors for \(\gamma\) corresponding to eigenvalues outside \mc L, are equal.

Let \mc B be a basis of \bV consisting of eigenvectors for \(\gamma\).  For each \(v \in \mc B\), write \(v^*\) for the corresponding element of the dual basis of \(\bV^*\).  We have that \(\ZZ\bG{\mc L}(\gamma)\) and \(\gamma\) both act trivially on \(\field[\gl_1]\).  It remains to understand their action on \(\field[\Lie(\bG)]\), which is the polynomial algebra on \(\set{v^* \otimes w}{v, w \in \mc B}\), where we abuse notation by writing \(v^* \otimes w\) for the matrix coefficient \anonmapto g{\pair{v^*}{g\dota w}}.  Each \(v^* \otimes w\) is an eigenvector for \(\gamma\) and a weight vector for \(\ZZ\bG{\mc L}(\gamma)\), with trivial weight for \(\ZZ\bG{\mc L}(\gamma)\) if and only if the eigenvalue for \(\gamma\) lies in \mc L.  The result follows.
\end{proof}

Proposition \initref{prop:gp-dfc-smooth}\subpref{Lie} shows that an appropriate restatement of \xcite{spice:asymptotic}*{Hypothesis \initxref{hyp:funny-centraliser}(\subxref{Lie})} holds.

\begin{prop}
\initlabel{prop:gp-dfc-smooth}
\hfill\begin{enumerate}
\item\sublabel{Lie}
\(\Lie(\CC\bG{\mc L}(\gamma))\) equals \(\CC{\Lie(\bG)}{\mc L}(\gamma)\).
\item\sublabel{smooth}
If \bG is smooth, then so is \(\CC\bG{\mc L}(\gamma)\).
\end{enumerate}
\end{prop}

\begin{proof}
We may, and do, assume, upon replacing \field by \sepfield, that \(\gamma\) is diagonalisable.

Write \((\sbtl{\field(\bG)}{e_\bG}n)_{n \in \Z}\) and \((\sbtl{\field(\CC\bG{\mc L}(\gamma))}{e_\bG}n)_{n \in \Z}\) for the filtrations on \(\field(\bG)\) and \(\field(\CC\bG{\mc L}(\gamma))\) induced by the integer-valued valuation corresponding to the identity element \(e_\bG\).  Thus, for example, \(\sbtl{\field(\bG)}{e_\bG}0\) is the local ring \(\mc O_{\bG, e_\bG}\) of \bG at \(e_\bG\), \(\sbtlp{\field(\bG)}{e_\bG}0 = \sbtl{\field(\bG)}{e_\bG}1\) is the maximal ideal \(\mf m_{\bG, e_\bG}\) of \(\mc O_{\bG, e_\bG}\), and \(\sbat{\field(\bG)}{e_\bG}1\) is \(\mf m_{\bG, e_\bG}/\mf m_{\bG, e_\bG}^2\).

Since \(\gamma\) is diagonalisable and \anonmap{\field[\bG] \cap \sbtl{\field(\bG)}{e_\bG}1}{\sbat{\field(\bG)}{e_\bG}1} is surjective, we have that every eigenvector for \(\gamma\) on \(\sbat{\field(\bG)}{e_\bG}1\) admits a lift to an eigenvector for \(\gamma\) on \(\field[\bG] \cap \sbtl{\field(\bG)}{e_\bG}1\) with the same eigenvalue.  That is, if we write \(I\) for the ideal in \(\field[\bG]\) generated by the eigenvectors for \(\gamma\) corresponding to eigenvalues not in \mc L, then the image of \(I \cap \sbtl{\field(\bG)}{e_\bG}1\) in \(\sbat{\field(\bG)}{e_\bG}1\) is the kernel of its maximal quotient \bX on which \(\gamma\) acts with all eigenvalues in \mc L.  Proposition \ref{prop:gp-dfc-weights} gives that \(I\) is the ideal of functions vanishing identically on \(\CC\bG{\mc L}(\gamma)\).  Therefore, the composition \anonmap{\field[\bG] \cap \sbtl{\field(\bG)}{e_\bG}1}\bX, which is surjective, factors through the quotient \anonmap{\field[\bG] \cap \sbtl{\field(\bG)}{e_\bG}1}{\field[\CC\bG{\mc L}(\gamma)] \cap \sbtl{\field(\CC\bG{\mc L}(\gamma))}{e_\bG}1}.  It follows that \bX is a quotient of \(\sbat{\field(\CC\bG{\mc L}(\gamma))}{e_\bG}1\).  Since \bX is the largest quotient of \(\sbat{\field(\bG)}{e_\bG}1\) on which \(\gamma\) acts with all eigenvalues in \mc L, and since all eigenvalues of \(\gamma\) on \(\field[\bG]/I \cong \field[\CC\bG{\mc L}(\gamma)]\), hence on \(\field(\CC\bG{\mc L}(\gamma))\), lie in \mc L, we have that the quotient \anonmap{\sbat{\field(\CC\bG{\mc L}(\gamma))}{e_\bG}1}\bX is an isomorphism.  That is, \(\sbat{\field(\CC\bG{\mc L}(\gamma))}{e_\bG}1\) is the maximal quotient of \(\sbat{\field(\bG)}{e_\bG}1\) on which \(\gamma\) acts with all eigenvalues in \mc L.

Dualising shows that \(\Lie(\CC\bG{\mc L}(\gamma)) = \Hom(\sbat{\field(\CC\bG{\mc L}(\gamma))}{e_\bG}1, \field)\) is the maximal subspace of \(\Hom(\sbat{\field(\bG)}{e_\bG}1, \field) = \Lie(\bG)\) on which \(\gamma\) acts with all eigenvalues in \mc L.  This establishes \locpref{Lie}.

If \bG is smooth, then we show that \(\sbtl{\field(\CC\bG{\mc L}(\gamma))}{e_\bG}0\) is regular exactly as in \cite{milne:algebraic-groups}*{Lemma 13.6}, except that we choose our lift of a basis for \(\sbat{\field(\bG)}{e_\bG}1\) not as a subset of \(\Cent_{\sbtlp{\field(\bG)}{e_\bG}0}(\gamma) \cup \set{f - \gamma\dota f}{f \in \Cent_{\sbtlp{\field(\bG)}{e_\bG}0}(\gamma)^\perp}\) but among \(\CC{\sbtlp{\field(\bG)}{e_\bG}0}{\mc L}(\gamma)\) and those functions of the form \(\lambda f - \gamma\dota f\), where \(\lambda\) lies in \mc L and \(f\) lies in the sum of the eigenspaces for \(\gamma\) in \(\CC{\sbtlp{\field(\bG)}{e_\bG}0}{\mc L}(\gamma)\) corresponding to eigenvalues outside of \mc L.  This establishes \locpref{smooth}.
\end{proof}

\begin{rem}
\label{rem:concrete-dfc}
Suppose that \(g \in G\) is such that \(\Int(g)\gamma\) commutes with \(\gamma\).

If \(g\) belongs to \(\CC G{\mc L}(\gamma)\), hence to \(\CC{\GL(V)}{\mc L}(\gamma)\), then we deduce successively the equalities \(\CC\bV{\mc L}([g, \gamma]) = \bV\), then \(\CC{\GL(\bV)}{\mc L}([g, \gamma]) = \GL(\bV)\), and finally \(\CC\bG{\mc L}([g, \gamma]) = \bG\).

The converse does not hold in general (as is already seen by considering \(\bG = \GL_2\), \(\mc L = 1\), \(\gamma = \operatorname{diag}(1, -1)\), and \(g = \operatorname{antidiag}(1, 1)\), if \(-1\) does not equal \(1\) in \field), but does when the adjoint representation is faithful.  If \(\CC\bG{\mc L}([g, \gamma])\) equals \bG, then Proposition \initref{prop:gp-dfc-smooth}\subpref{Lie} gives the equality \(\CC{\Lie(\bG)}{\mc L}([g, \gamma]) = \Lie(\bG)\).  That is,
\(\Ad(g)\) belongs to \(\CC{\GL(\Lie(G))}{\mc L}(\gamma)\).  If the adjoint representation of \bG is faithful, then this implies that \(g\) belongs to \(\CC G{\mc L}(\gamma)\).  If \bG is connected and reductive, then this can be leveraged to give a concrete description of \(\CC\bG{\mc L}(\gamma)\), along the lines of Definition \ref{defn:G'}.
\end{rem}

As in Definition \ref{defn:embed-dual}, we need to specify a way of regarding \(\Lie^*(\CC\bG{\mc L}(\gamma))\) as a subspace of \(\Lie^*(\bG)\).
We do so as in the discussion preceding \xcite{spice:asymptotic}*{Definition \xref{defn:vGvr}}.

\begin{defn}
\label{defn:embed-fc-dual}
There is a unique \(\gamma\)-stable complement \matnotn[\Lie(\bH)^\perp]{Lie}{\Lie(\CC\bG{\mc L}(\gamma))^\perp} to \(\Lie(\CC\bG{\mc L}(\gamma)) = \CC{\Lie(\bG)}{\mc L}(\gamma)\) in \(\Lie(\bG)\), namely, the sum of the eigenvalues for \(\gamma\) in \(\Lie(\bG)\) corresponding to eigenvalues outside \mc L.  We identify \(\Lie^*(\CC\bG{\mc L}(\gamma))\) with the subset of \(\Lie^*(\bG)\) that annihilates \(\Lie(\CC\bG{\mc L}(\gamma))^\perp\).
\end{defn}

\begin{lem}
\initlabel{lem:gp-dfc-nearby}
Suppose that \(t \in \gNorm_{\GL(V)}(\bG)\) is semisimple and commutes with \(\gamma\).
\begin{enumerate}
\item\sublabel{rep}
If \bW is a (\(\sgen{\gamma, t} \ltimes \bG\))-module and \(\lambda\) belongs to \(\field\twosup\sepsup\multsup\), then the intersections with \(\CC{\bW_\sepfield}{\mc L}(t)\) of the \mc L-close \(\lambda\)-eigenspaces for \(\gamma\) and \(\gamma t\inv\) in \(\bW_\sepfield\) are equal.  In particular, the intersections with \(\CC\bW{\mc L}(t)\) of \(\CC\bW{\mc L}(\gamma)\) and \(\CC\bW{\mc L}(\gamma t\inv)\) are equal.
\item\sublabel{equal}
The intersections with \(\CC\bG{\mc L}(t)\) of \(\CC\bG{\mc L}(\gamma)\) and \(\CC\bG{\mc L}(\gamma t\inv)\) are equal.
\item\sublabel{sub}
Let \(\mc L'\) be a \(\Gal(\sepfield/\field)\)-stable subgroup of \mc L.
If
	\(\CC\bG{\mc L}(t)\) contains both \(\CC\bG{\mc L}(\gamma)\) and \(\CC\bG{\mc L}(\gamma t\inv)\), and
	\(\CC\bG{\mc L}(\gamma)\) equals \(\Cent_\bG(\gamma)\),
then \(\CC\bG{\mc L'}(\gamma t\inv)\) equals \(\Cent_\bG(\gamma) \cap \CC\bG{\mc L'}(t)\).
\end{enumerate}
\end{lem}

\begin{proof}
We may, and do, assume, upon replacing \field by \sepfield, that \(\gamma\) and \(t\) are diagonalisable.

Since we may replace \(\gamma\) by \(\gamma t\inv\) and then \(t\) by \(t\inv\), a containment in one direction immediately implies equality.

We prove a slightly more general version of \locpref{rep}, namely, that, for every \(\lambda, \mu \in \field\mult\), the intersections with the \mc L-close \(\mu\)-eigenspace for \(t\) in \bW of the close \(\lambda\)-eigenspace for \(\gamma\) and the close \(\lambda\mu\inv\)-eigenspace for \(\gamma t\inv\) are equal.  Indeed, the intersection of the close \(\mu\)-eigenspace for \(t\) with the close \(\lambda\)-eigenspace for \(\lambda\) is a sum of simultaneous eigenspaces for \(\gamma\) and \(t\) in \bW.  On each such eigenspace, the eigenvalue of \(\gamma\) lies in \(\lambda\mc L\) and that of \(t\) lies in \(\mu\mc L\), so \(\gamma t\inv\) acts by multiplication by an element of \(\lambda\mu\inv\mc L\).  This shows that the intersection is contained in the close \(\lambda\mu\inv\)-eigenspace for \(\gamma t\inv\).  As mentioned, this implies the desired equality.

\locpref{rep} follows by taking \(\mu = 1\).

Each close eigenspace for \(\gamma\) in \bV is a sum of simultaneous close eigenspaces for \(\gamma\) and \(t\) in \bV.  We have shown that each such simultaneous eigenspace is also a simultaneous close eigenspace for \(\gamma t\inv\) and \(t\), hence is preserved by \(\CC\bG{\mc L}(t) \cap \CC\bG{\mc L}(\gamma t\inv)\).  This shows that \(\CC\bG{\mc L}(t) \cap \CC\bG{\mc L}(\gamma t\inv)\) is contained in \(\CC\bG{\mc L}(\gamma)\), hence in \(\CC\bG{\mc L}(t) \cap \CC\bG{\mc L}(\gamma t\inv)\).  As mentioned, this implies \locpref{equal}.

In the situation of \locpref{sub}, we have by \locpref{equal} that \(\CC\bG{\mc L}(\gamma t\inv)\) equals \(\CC\bG{\mc L}(\gamma) = \Cent_\bG(\gamma)\).  Then Remark \initref{rem:gp-dfc-facts}(\subref{central}, \subref{sub}, \subref{monotone}) gives that \(\CC\bG{\mc L'}(\gamma t\inv)\) equals \(\CC\bG{\mc L}(\gamma t\inv) \cap \CC\bG{\mc L'}(\gamma t\inv) = \Cent_\bG(\gamma) \cap \CC\bG{\mc L'}(t\inv) = \CC{\Cent_\bG(\gamma)}{\mc L'}(\gamma t\inv) = \CC{\Cent_\bG(\gamma)}{\mc L'}(t\inv) = \Cent_\bG(\gamma) \cap \CC\bG{\mc L'}(t\inv)\).  Since \(\CC\bG{\mc L'}(t\inv)\) equals \(\CC\bG{\mc L'}(t)\), this establishes \locpref{sub}.
\end{proof}

Lemma \ref{lem:shallow-dfc} is another refinement of Remark \initref{rem:gp-dfc-facts}\subpref{section}.  Together with Remark \initref{rem:gp-dfc-facts}\subpref{monotone}, it shows that an appropriate restatement of \xcite{spice:asymptotic}*{Hypothesis \initxref{hyp:funny-centraliser}(\subxref{basic})} holds.  First, we need Notation \ref{notn:ord-gamma}, whose notation is essentially as in \cite{deligne:support}*{\S1, p.~155}.

\begin{notn}
\label{notn:ord-gamma}
If \sepfield is equipped with a discrete valuation \(\ord\), then we write \matnotn{mgamma}{m_\gamma} for the element of \(\cclat(\sgen\gamma_\sepfield) \otimes_\Z \Q\) satisfying \(\pair\chi{m_\gamma} = \ord \chi(\gamma)\) for all \(\chi \in \clat(\sgen\gamma_\sepfield)\).
\end{notn}

\begin{rem}
\label{rem:ord-gamma}
Using Notation \ref{notn:ord-gamma}, if \(\ord\) is \(\Gal(\sepfield/\field)\)-fixed, then so is \(m_\gamma\); so we may, and do, regard \(m_\gamma\) as an element of \(\cclat(\sgen\gamma) \otimes_\Z \Q\).
\end{rem}

We abuse notation slightly in Lemma \initref{lem:shallow-dfc}\subpref0 by writing \(\Cent_\bG(m_\gamma)\) in place of ``\(\Cent_\bG(\operatorname{im}(N m_\gamma))\), where \(N\) is a positive integer so that \(N m_\gamma\) belongs to \(\cclat(\sgen\gamma)\)''.

\begin{lem}
\initlabel{lem:shallow-dfc}
Suppose that \sepfield is equipped with a \(\Gal(\sepfield/\field)\)-fixed valuation.
\hfill\begin{enumerate}
\item\sublabel0
\(\usualCC\bG 0(\gamma)\) equals \(\Cent_\bG(m_\gamma)\).
\item\sublabel{0+}
Write \(\sbjtl{\ol\gamma}0\) for the automorphism of \(\usualCC\bG 0(\gamma)\) induced by \(\gamma\).  Suppose that the characteristic \(p\) of \(\sbjat\field 0\) is positive, and that there is a \(p\)-power \(q\) such that \(\sbjat{\ol\gamma}0 \ldef \lim_{n \to \infty} \sbjtl{\ol\gamma}0^{q^n}\) exists (in the analytic topology on \(\Weyl(\GL(\bV), \usualCC\bG 0(\gamma))(\field)\)).  Then \(\usualCCp\bG 0(\gamma)\) equals \(\Cent_{\usualCC\bG 0(\gamma)}(\sbjat{\ol\gamma}0)\).
\end{enumerate}
\end{lem}

\begin{proof}
We may, and do, assume, upon replacing \field by \sepfield, that \(\gamma\) is diagonalisable.

For every \(\lambda \in \field\mult\), the \(\sbjtl{\field\mult}0\)-close \(\lambda\)-eigenspace for \(\gamma\) on \bV is precisely the sum of eigenspaces whose eigenvalues have the same valuation as \(\lambda\).  Thus \(\usualCC{\GL(\bV)}0(\gamma) = \Cent_{\GL(\bV)}(\usualZZ{\GL(\bV)}0(\gamma))\) equals \(\Cent_{\GL(\bV)}(m_\gamma)\).  \locpref0 follows upon intersecting with \bG.

Remark \initref{rem:gp-dfc-facts}(\subref{sub}, \subref{monotone}) give that \(\usualCCp\bG 0(\gamma)\) equals \(\usualCCp\bG 0(\gamma) \cap \usualCC\bG 0(\gamma) = \usualCCp{\usualCC\bG 0(\gamma)}0(\gamma)\).  Thus, for \locpref{0+}, we may, and do, assume that \bG equals \(\usualCC\bG 0(\gamma)\).  As in the proof of Lemma \initref{lem:gp-dfc-power}\subpref{exact}, but now referring to \cite{landvogt:compactification}*{\S0.6, (iii)} for the compatibility properties of the analytic topology, we may, and do, assume, upon replacing \(\gamma\) by its image in \(\sgen\gamma/{\Cent_{\sgen\gamma}(\bG)}\) and then changing \bV, that \(\sbjat\gamma 0 \ldef \lim_{n \to \infty} \gamma^{p^n}\) exists in the analytic topology on \(\GL(V)\).

Now \(\sbjtlp\gamma 0 \ldef \sbjat\gamma 0\inv\gamma\) is pro-\(p\)-torsion in the analytic topology on \(\GL(V)\), so that all the eigenvalues of \(\sbjtlp\gamma 0\) on \bV lie in \(\sbjtlp{\field\mult}0\); that is, that \(\usualCCp\bV 0(\sbjtlp\gamma 0)\) equals \bV.  It follows from Lemma \ref{lem:gp-dfc-nearby} that \(\usualCCp\bG 0(\sbjat\gamma 0)\) equals \(\usualCCp\bG 0(\gamma)\).  Since \(\sbjat\gamma 0\) has order \(q - 1\), and \(\mu_{q - 1}(\field) \cap \sbjtlp{\field\mult}0\) is trivial, we have by Lemma \initref{lem:gp-dfc-power}\subpref{exact} that \(\usualCCp\bG 0(\sbjat\gamma 0)\) equals \(\Cent_\bG(\sbjat\gamma 0)\).
\end{proof}

\begin{rem}
\label{rem:shallow-dfc}
We discuss a variant of Lemma \initref{lem:shallow-dfc}\subpref{0+}, in which both the hypotheses and the conclusion are strengthened slightly.
Continue to suppose that \sepfield is equipped with a \(\Gal(\sepfield/\field)\)-fixed valuation.

Suppose that \(\gamma\) belongs to \bG.  Let \bA be the maximal split, normal torus in \(\usualCC\bG 0(\gamma)\), and write \(\ol\gamma\) for the image of \(\gamma\) in \((\usualCC\bG 0(\gamma)/\bA)(\field)\).  Suppose that the characteristic \(p\) of \(\sbjat\field 0\) is positive, and that there is a \(p\)-power \(q\) such that \(\sbjat{\ol\gamma}0 \ldef \lim_{n \to \infty} \sbjtl{\ol\gamma}0^{q^n}\) exists.  (This notation conflicts with that in Lemma \initref{lem:shallow-dfc}\subpref{0+}, but only slightly; the automorphisms of \(\usualCC\bG 0(\gamma)\) induced by the elements denoted here and there by \(\sbjat{\ol\gamma}0\) are the same.)

Since the \(\Gal(\sepfield/\field)\)-cohomology of \(\bA(\sepfield)\) vanishes, there is a lift \(\sbjat\gamma 0\) of \(\sbjat{\ol\gamma}0\) to \(\usualCC G 0(\gamma)\).  Lemma \initref{lem:shallow-dfc}\subpref{0+} gives that \(\usualCCp\bG 0(\gamma)\) equals \(\Cent_{\usualCC\bG 0(\gamma)}(\sbjat\gamma 0)\).  If \field contains an element \(t\) of non-\(0\) valuation, then we can replace \(\sbjat\gamma 0\) by its translate by \((N m_\gamma)(t) \in A\) for a suitable positive integer \(N\), and then Lemma \initref{lem:shallow-dfc}\subpref0 implies that \(\usualCCp\bG 0(\gamma)\) actually equals \(\Cent_\bG(\sbjat\gamma 0)\).
\end{rem}

Lemma \initref{lem:gp-dfc-unique}\subpref{nearby-sub} is similar to Lemma \initref{lem:gp-dfc-nearby}\subpref{sub}, but both the hypotheses and the conclusions are weaker.

\begin{lem}
\initlabel{lem:gp-dfc-unique}
\hfill\begin{enumerate}
\item\sublabel{around}
If \bJ is a \(\gamma\)-stable, connected, smooth subgroup of \bG such that \(\Lie(\bJ)\) is contained in (respectively, contains) \(\CC{\Lie(\bG)}{\mc L}(\gamma)\), then \bJ is contained in (respectively, contains) \(\CC\bG{\mc L}(\gamma)\conn\).
\setcounter{tempenumi}{\value{enumi}}
\end{enumerate}
Suppose that \bG is smooth.
\begin{enumerate}
\setcounter{enumi}{\value{tempenumi}}
\item\sublabel{unique}
\(\CC\bG{\mc L}(\gamma)\conn\) is the unique \(\gamma\)-stable, connected, smooth subgroup of \bG whose Lie algebra is \(\CC{\Lie(\bG)}{\mc L}(\gamma)\).
\item\sublabel{nearby-equal}
Suppose that \(t \in \gNorm_{\GL(V)}(\bG)\) is semisimple and commutes with \(\gamma\).  If \(\CC{\Lie(\bG)}{\mc L}(t)\) contains \(\CC{\Lie(\bG)}{\mc L}(\gamma) + \CC{\Lie(\bG)}{\mc L}(\gamma t\inv)\), then \(\CC\bG{\mc L}(\gamma t\inv)\conn\) equals \(\CC\bG{\mc L}(\gamma)\conn\).
\item\sublabel{nearby-sub}
Let \(\mc L'\) be a \(\Gal(\sepfield/\field)\)-stable subgroup of \mc L.
If, in addition to the hypotheses of \locpref{nearby-equal}, we have that \(\CC{\Lie(\bG)}{\mc L}(\gamma)\) equals \(\Cent_{\Lie(\bG)}(\gamma)\), and if \(\mc L'\) is a \(\Gal(\sepfield/\field)\)-stable subgroup of \mc L, then \(\CC\bG{\mc L'}(\gamma t\inv)\conn\) is the identity component of \(\Cent_\bG(\gamma) \cap \CC\bG{\mc L'}(t)\).
\end{enumerate}
\end{lem}

\begin{proof}
The proof of \locpref{around} is almost exactly the same as that of Lemma \ref{lem:cent-by-Lie}.  Namely, if \(\Lie(\bJ)\) is contained in \(\CC{\Lie(\bG)}{\mc L}(\gamma)\), then \(\CC{\Lie(\bJ)}{\mc L}(\gamma) = \CC{\Lie(\bG)}{\mc L}(\gamma) \cap \Lie(\bJ)\) equals \(\Lie(\bJ)\), so Proposition \ref{prop:gp-dfc-smooth} gives that \(\CC\bJ{\mc L}(\gamma)\conn\) is a smooth subgroup of \bJ whose Lie algebra equals \(\CC{\Lie(\bJ)}{\mc L}(\gamma) = \Lie(\bJ)\), and hence that \(\CC\bJ{\mc L}(\gamma)\conn\) equals \bJ.  In particular, by Remark \initref{rem:gp-dfc-facts}\subpref{sub}, we have that \(\bJ = \CC\bJ{\mc L}(\gamma)\conn\) is contained in \(\CC\bG{\mc L}(\gamma)\conn\).

On the other hand, if \(\Lie(\bJ)\) contains \(\CC{\Lie(\bG)}{\mc L}(\gamma)\), then \(\CC{\Lie(\bJ)}{\mc L}(\gamma)\) equals \(\CC{\Lie(\bG)}{\mc L}(\gamma) \cap \Lie(\bJ) = \CC{\Lie(\bG)}{\mc L}(\gamma)\), so again Proposition \ref{prop:gp-dfc-smooth} gives that \(\CC\bJ{\mc L}(\gamma)\conn\) is a smooth subgroup of \(\CC\bG{\mc L}(\gamma)\conn\) whose Lie algebra equals \(\CC{\Lie(\bJ)}{\mc L}(\gamma) = \CC{\Lie(\bG)}{\mc L}(\gamma) = \Lie(\CC\bG{\mc L}(\gamma))\), and hence that \(\CC\bJ{\mc L}(\gamma)\conn\) equals \(\CC\bG{\mc L}(\gamma)\conn\).  In particular, we have that \(\CC\bG{\mc L}(\gamma)\conn = \CC\bJ{\mc L}(\gamma)\conn\) is contained in \bJ.  This establishes \locpref{around}.

\locpref{unique} follows from Proposition \initref{prop:gp-dfc-smooth}\subpref{Lie} and \locpref{around}.

In the situation of \locpref{nearby-equal}, Lemma \initref{lem:gp-dfc-nearby}\subpref{rep} gives that \(\CC{\Lie(\bG)}{\mc L}(\gamma t\inv)\) equals \(\CC{\Lie(\bG)}{\mc L}(\gamma)\).  Then \locpref{nearby-equal} follows from yet another application of Proposition \ref{prop:gp-dfc-smooth}, and \locpref{unique}.

Finally, as in the proof of Lemma \initref{lem:gp-dfc-nearby}\subpref{sub}, if \(\CC{\Lie(\bG)}{\mc L}(\gamma)\) equals \(\Cent_{\Lie(\bG)}(\gamma)\), then Lemma \ref{lem:cent-by-Lie} and Proposition \initref{prop:gp-dfc-smooth}\subpref{Lie} give that \(\CC\bG{\mc L}(\gamma)\conn\), which we have just shown equals \(\CC\bG{\mc L}(\gamma t\inv)\conn\), equals \(\Cent_\bG(\gamma)\conn\); so Remark \initref{rem:gp-dfc-facts}(\subref{sub}, \subref{monotone}, \subref{central}) gives that \(\CC\bG{\mc L'}(\gamma t\inv)\conn\) equals \(\CC{\CC\bG{\mc L}(\gamma t\inv)\conn}{\mc L'}(\gamma t\inv)\conn = \CC{\Cent_\bG(\gamma)\conn}{\mc L'}(\gamma t\inv)\conn\), which is the identity component of \(\Cent_\bG(\gamma) \cap \CC\bG{\mc L'}(t\inv)\).  Since
\(\CC\bG{\mc L'}(t\inv)\) equals \(\CC\bG{\mc L'}(t)\), we have shown \locpref{nearby-sub}.
\end{proof}

\subsection{$\gamma$ an automorphism of \bG and \bG reductive}
\label{subsec:reductive-gp-funny}

Throughout \S\ref{subsec:reductive-gp-funny}, suppose that \bG is reductive, but not necessarily connected, and that \(\gamma\) is a semisimple automorphism of \bG.

\begin{defn}
\label{defn:gamma-large}
A torus \bT in \bG is called \term[torus!gamma-large@{$\gamma$}-large]{\(\gamma\)-large} if \(\Cent_\bG(\gamma)\conn \cap \bT\) is a maximal torus in \(\Cent_\bG(\gamma)\conn\).
\end{defn}

Lemma \initref{lem:gamma-large}\subpref{quass} is \emph{nearly} contained in \cite{digne-michel:non-connexe}*{Th\'eor\`eme 1.8(iv)}, but that result works over an algebraically, not a separably, closed field.

\begin{lem}
\initlabel{lem:gamma-large}
\hfill\begin{enumerate}
\item\sublabel{reductive}
\(\Cent_\bG(\gamma)\conn\) is reductive.
\item\sublabel{bijection}
The map \anonmapto\bT{\Cent_\bG(\gamma)\conn \cap \bT} from \(\gamma\)-large maximal tori in \bG to maximal tori in \(\Cent_\bG(\gamma)\conn\) is a bijection, with inverse \anonmapto\bS{\Cent_\bG(\bS)}.
\item\sublabel{Levi}
If \bP is a \(\gamma\)-stable, parabolic subgroup of \bG, and \bM is a \(\gamma\)-stable Levi component of \bP, then every \(\gamma\)-large torus in \bM is also \(\gamma\)-large in \bG.
\item\sublabel{stays-large}
Every \(\gamma\)-large maximal torus is \(\gamma^n\)-large for all integers \(n\), and \(\gamma\)-stable.
\item\sublabel{quass}
Every split, \(\gamma\)-large maximal torus in \bG is contained in a \(\gamma\)-stable Borel subgroup of \bG.
\end{enumerate}
\end{lem}

\begin{proof}
Remember that the automorphism of \(\bG_\algfield\) induced by \(\gamma\) is quasi-semisimple \cite{steinberg:endomorphisms}*{Theorem 7.5}.

\locpref{reductive} follows from \cite{nagata:cr}*{p.\ 95, Theorem 2}
and \cite{conrad-gabber-prasad:prg}*{Proposition A.8.12}.  That the map \anonmapto\bS{\Cent_\bG(\bS)} of \locpref{bijection} sends maximal tori in \(\Cent_\bG(\gamma)\conn\) to \(\gamma\)-large maximal tori in \bG is contained in \cite{digne-michel:non-connexe}*{Th\'eor\`eme 1.8(iv)}.  Since \locpref{reductive} gives that the centraliser in \(\Cent_\bG(\gamma)\conn\) of a maximal torus \bS in \(\Cent_\bG(\gamma)\conn\) is \bS, we have that the map \anonmapto\bT{\Cent_\bG(\gamma)\conn \cap \bT} of \locpref{bijection} is a left inverse to \anonmapto\bS{\Cent_\bG(\bS)}.  To show that it is also a right inverse, let \bT be a \(\gamma\)-large maximal torus in \bG, and put \(\bS = \Cent_\bG(\gamma)\conn \cap \bT\).  Since \(\Cent_\bG(\bS)\) clearly contains \bT, and we have already shown that it is a torus, we have that \(\Cent_\bG(\bS)\) equals \bT.  This establishes \locpref{bijection}.

\locpref{stays-large} follows from \locpref{bijection}.

In the situation of \locpref{Levi}, we have by
\cite{digne-michel:non-connexe}*{Proposition 1.11(ii)} that \(\Cent_\bM(\gamma)\conn\) is a Levi subgroup of \(\Cent_\bG(\gamma)\conn\).  Thus, if \bT is a \(\gamma\)-large torus in \bM, so that \(\Cent_\bT(\gamma)\conn\) is a maximal torus in \(\Cent_\bM(\gamma)\conn\), then \(\Cent_\bT(\gamma)\conn\) is also a maximal torus in \(\Cent_\bG(\gamma)\conn\); that is, \bT is \(\gamma\)-large in \bG.

For \locpref{quass}, let \bT be a split, \(\gamma\)-large maximal torus in \bG, put \(\bS = \Cent_\bG(\gamma)\conn \cap \bT\), and let \(\lambda\) be a cocharacter of \bS such that \(\Cent_\bG(\bS)\) equals \(\Cent_\bG(\lambda)\).  Then the parabolic subgroup \(\bB \ldef P_\bG(\lambda)\) of \bG associated to \(\lambda\) \cite{springer:lag}*{Proposition 8.4.5} admits \(\Cent_\bG(\lambda) = \Cent_\bG(\bS)\) as a Levi component by \cite{borel:linear}*{Proposition 20.5}.  Since \(\Cent_\bG(\bS)\) equals \bT by \locpref{bijection}, it follows that \(\bB\) is a Borel subgroup of \bG.  It is \(\gamma\)-stable because \(\lambda\) is fixed by \(\gamma\).
\end{proof}

Proposition \initref{prop:gp-cfc-facts}(\subref{Levi}, \subref{quotient}), together with Proposition \initref{prop:gp-dfc-smooth}\subpref{smooth}, show that \xcite{spice:asymptotic}*{Hypothesis \initxref{hyp:funny-centraliser}(\subxref{Levi})} holds.

\begin{prop}
\initlabel{prop:gp-cfc-facts}
\hfill\begin{enumerate}
\item\sublabel{reductive}
\(\CC\bG{\mc L}(\gamma)\conn\) is reductive.
\item\sublabel{torus}
If \bT is a \(\gamma\)-large maximal torus in \bG, then \(\CC\bG{\mc L}(\gamma)\conn \cap \bT\) is a maximal torus in \(\CC\bG{\mc L}(\gamma)\), hence equals \(\CC\bT{\mc L}(\gamma)\conn\); and \(\Cent_\bG(\CC\bG{\mc L}(\gamma)\conn \cap \bT)\) equals \bT.
\item\sublabel{regular}
If \bS is a maximal torus in \(\CC\bG{\mc L}(\gamma)\), then there is a unique maximal torus \bT in \bG containing \bS, and \bS equals \(\CC\bG{\mc L}(\gamma)\conn \cap \bT\).
\item\sublabel{Levi}
If \(\bP^+\) is a \(\gamma\)-stable, parabolic subgroup of \bG and \bM is a \(\gamma\)-stable Levi component of \(\bP^+\), then \(\CC{\bP^+}{\mc L}(\gamma)\) is a parabolic subgroup of \(\CC\bG{\mc L}(\gamma)\), \(\CC\bM{\mc L}(\gamma)\) is a Levi component of \(\CC{\bP^+}{\mc L}(\gamma)\), and the unipotent radical of \(\CC{\bP^+}{\mc L}(\gamma)\conn\) is \(\CC{\bPrad^+}{\mc L}(\gamma)\), where \(\bPrad^+\) is the unipotent radical of \(\bP\thenconn+\).
\item\sublabel{big-cell}
If, with the notation of \locpref{Levi},
	\begin{itemize}
	\item \bG is connected,
	\item \(\bP^-\) is the parabolic subgroup of \bG opposite to \(\bP^+\) relative to \bM,
and	\item \(\bPrad^-\) is the unipotent radical of \(\bP^-\),
	\end{itemize}
then \(\CC\bG{\mc L}(\gamma) \cap (\bPrad^-\dotm\bM\dotm\bPrad^+)\) equals \(\CC{\bPrad^-}{\mc L}(\gamma)\dotm\CC\bM{\mc L}(\gamma)\dotm\CC{\bPrad^+}{\mc L}(\gamma)\), and is an open, dense subscheme of \(\CC\bG{\mc L}(\gamma)\).
\item\sublabel{quotient}
The image of \(\CC\bG{\mc L}(\gamma)\conn\) under a quotient \anonmap\bG{\bG'} by a \(\gamma\)-stable, central subgroup of \bG is \(\CC{\bG'}{\mc L}(\gamma)\conn\).
	\project[Can we change this to:  Quotient by a subgroup of multiplicative type or by a smooth subgroup?]
\end{enumerate}
\end{prop}

\begin{proof}
We may, and do, assume, upon replacing \field by \sepfield and \bG by \(\bG\conn\), that \field is separably closed and \bG is connected.

Note that \(\CC\bG{\mc L}(\gamma)\conn\) contains \(\Cent_\bG(\gamma)\conn\).

In the situation of \locpref{torus}, put \(\bS' = \Cent_\bG(\gamma)\conn \cap \bT\) and \(\bS = \CC\bG{\mc L}(\gamma)\conn \cap \bT\).  Lemma \initref{lem:gamma-large}\subpref{bijection} gives that \(\bS'\) is a torus, and \(\Cent_\bG(\bS')\) equals \bT.  In particular, \(\bS = \CC\bG{\mc L}(\gamma)\conn \cap \bT\) equals \(\CC\bG{\mc L}(\gamma)\conn \cap \Cent_\bG(\bS') = \Cent_{\CC\bG{\mc L}(\gamma)\conn}(\bS')\), which is smooth \cite{conrad-gabber-prasad:prg}*{Proposition A.8.10(2)} and connected \cite{borel:linear}*{Corollary 11.12}, and therefore a torus.  Further, \(\Cent_\bG(\bS)\), which clearly contains \bT, is also contained in \(\Cent_\bG(\bS') = \bT\), so that we have equality; and \(\Cent_{\CC\bG{\mc L}(\gamma)\conn}(\bS) = \CC\bG{\mc L}(\gamma)\conn \cap \Cent_\bG(\bS)\) equals \(\CC\bG{\mc L}(\gamma)\conn \cap \bT = \bS\), so that \bS is a maximal torus in \(\CC\bG{\mc L}(\gamma)\conn\).  This gives \locpref{torus}.

For (\locref{Levi}, \locref{big-cell}), use Lemma \initref{lem:gamma-large}\subpref{Levi} to obtain a \(\gamma\)-large maximal torus \bT in \bM.  Put \(\delta_{\bP^+} = \sum_{\alpha \in \Root(\bP^+, \bT)} \coroot\alpha\), so that \(\delta\) is a \(\gamma\)-fixed cocharacter of \bT, hence a cocharacter of \(\Cent_\bT(\gamma)\conn \subseteq \CC\bG{\mc L}(\gamma)\conn\).  Then \(\bP^+\) is the associated parabolic subgroup \(\operatorname P_{\bG\conn}(\delta)\) of \(\bG\conn\), and \bM is its associated Levi component \(\Cent_{\bG\conn}(\delta)\), so \cite{conrad-gabber-prasad:prg}*{p.~49 and Lemma 2.1.5} give that \(\CC\bG{\mc L}(\gamma)\conn \cap \bP^+ = \CC\bG{\mc L}(\gamma)\conn \cap \CC{\bP^+}{\mc L}(\gamma)\) and \(\CC\bG{\mc L}(\gamma)\conn \cap \bM = \CC\bG{\mc L}(\gamma)\conn \cap \CC\bM{\mc L}(\gamma)\) are the analogous subgroups \(\operatorname P_{\CC\bG{\mc L}(\gamma)\conn}(\delta)\) and \(\Cent_{\CC\bG{\mc L}(\gamma)\conn}(\delta)\) of \(\CC\bG{\mc L}(\gamma)\conn\).  For the statements about \(\bPrad^+\), we must take a bit more care.  Namely, we observe that \(\bPrad^+\) equals \(\operatorname U_\bG(\delta)\), so that \(\CC{\bPrad^+}{\mc L}(\gamma) = \CC\bG{\mc L}(\gamma) \cap \bPrad^+\) equals \(\operatorname U_{\CC\bG{\mc L}(\gamma)}(\delta)\).  Since \(\operatorname U_{\CC\bG{\mc L}(\gamma)}(\delta)\) is connected \cite{conrad-gabber-prasad:prg}*{Proposition 2.1.8(4)}, it equals \(\operatorname U_{\CC\bG{\mc L}(\gamma)\conn}(\delta)\).
(\locref{Levi}, \locref{big-cell}) will follow once we have proven \locpref{reductive} (using 
\cite{conrad-gabber-prasad:prg}*{Proposition 2.1.8(3)} for \locpref{big-cell}).

For (\locref{reductive}, \locref{regular}, \locref{quotient}), use Lemma \initref{lem:gamma-large}\subpref{bijection} to choose a \(\gamma\)-large maximal torus \bT in \bG.  Write \(\bT'\) for the image of \bT in \(\bG'\).
By \cite{demazure-grothendieck:SGA-3.3}*{Expos\'e XXIV, Corollaire 1.7}, there is a positive integer \(N = \card{\pi_0(\sgen\gamma)}\) such that the automorphism of \bG induced by \(\gamma^N\) belongs to \(\Int_\bG(\bG)(\field) \cap \Cent_{\Aut(\bG)}(\bT)(\field) = \Int_\bG(\bT)(\field)\).  There is thus a natural extension map \anonmap{\Z\Root(\bG, \bT) = \clat(\Int_\bG(\bT))}{\clat(\sgen{\gamma^N, \Int_\bG(\bT)})}, which we use to make sense of the notation \(\chi(\gamma^N)\) for \(\chi \in \Z\Root(\bG, \bT)\).

\locpref{regular} follows from \locpref{torus} for the particular maximal torus \(\bS = \CC\bG{\mc L}(\gamma)\conn \cap \bT\); but, since the maximal tori in \(\CC\bG{\mc L}(\gamma)\conn\) are \(\CC G{\mc L}(\gamma)\conn\)-conjugate, \locpref{regular} follows in general.

To prove \locpref{quotient}, we will temporarily assume that we have proven \locpref{reductive}.
Write \(q\) for the quotient map \anonmap\bG{\bG'}.
It follows from Lemma \ref{lem:fixed-surjective} that \anonmap{\Cent_\bT(\gamma)\conn}{\Cent_{\bT'}(\gamma)\conn} is surjective,
and from Lemma \initref{lem:gamma-large}\subpref{bijection} that \(\Cent_{\bT'}(\gamma)\conn\) equals \(\Cent_{\bG'}(\gamma)\conn \cap \bT\) and that there is a cocharacter \(\lambda'\) of \(\Cent_{\bT'}(\gamma)\conn\) such that \(\Cent_{\bG'}(\lambda)\conn\) equals \(\bT'\).  Upon replacing \(\lambda'\) by a positive-integer multiple if necessary, we may, and do, arrange that \(\lambda'\) is of the form \(q \circ \lambda\) for some cocharacter \(\lambda\) of \(\Cent_\bT(\gamma)\conn\).
We have that \(\bP^\pm \ldef \operatorname P_\bG(\pm\lambda)\) are opposite \(\gamma\)-stable, parabolic subgroups of \bG with common Levi component \(\Cent_\bG(\lambda) = \bT\) and unipotent radicals \(\bPrad^\pm \ldef \operatorname U_\bG(\pm\lambda)\); and \(\bP\primethen\pm \ldef \operatorname P_{\bG'}(\lambda')\) are opposite \(\gamma\)-stable, parabolic subgroups of \(\bG'\) with common Levi component \(\bPrad\primethen\pm \ldef \operatorname U_\bG(\pm\lambda')\).
It follows from Lemmas \ref{lem:fixed-surjective} and \initref{lem:gp-dfc-power}\subpref{exact} 
that \(q(\CC\bT{\mc L}(\gamma)\conn)\) equals \(\CC{\bT'}{\mc L}(\gamma)\conn\).

Since \(\bPrad^\pm\) are unipotent, they intersect the multiplicative-type subgroup \(\ker(q)\) trivially, so that \(q\) restricts to isomorphisms \(q^\pm\) of \(\bPrad^\pm\) onto \(q(\bPrad^\pm) = \bPrad\primethen\pm\).  Let us temporarily write \bH for the subgroup of \bG generated by \(\CC\bT{\mc L}(\gamma)\conn\) and \(q^{\pm\,{-1}}(\CC{\bPrad\primethen\pm}{\mc L}(\gamma))\).  Then \(q(\bH)\) is generated by \(q(\CC\bT{\mc L}(\gamma)\conn) = \CC{\bT'}{\mc L}(\gamma)\conn\) and \(\CC{\bPrad\primethen\pm}{\mc L}(\gamma)\), hence contains a dense open subgroup of \(\CC{\bG'}{\mc L}(\gamma)\) by \locpref{big-cell}, hence equals \(\CC{\bG'}{\mc L}(\gamma)\).  Further, \(\Lie(\bH)\) is generated by \(\Lie(\CC\bT{\mc L}(\gamma))\) and \(\upd q^{\pm\,{-1}}(\Lie(\CC{\bPrad\primethen\pm}{\mc L}(\gamma)))\).  Proposition \initref{prop:gp-dfc-smooth}\subpref{Lie} gives that \(\Lie(\CC{\bPrad^\pm}{\mc L}(\gamma))\) equals \(\CC{\Lie(\bPrad^\pm)}{\mc L}(\gamma)\), and analogously with \(\bPrad\primethen\pm\) in place of \bT.  Thus \(\upd q^{\pm\,{-1}}(\Lie(\CC{\bPrad\primethen\pm}{\mc L}(\gamma))) = \upd q^{\pm\,{-1}}(\CC{\Lie(\bPrad\primethen\pm)}{\mc L}(\gamma)\) equals \(\CC{\Lie(\bPrad^\pm)}{\mc L}(\gamma) = \Lie(\CC{\bPrad^\pm}{\mc L}(\gamma))\).  It follows that \(\Lie(\bH)\) is generated by \(\Lie(\CC\bT{\mc L}(\gamma))\) and \(\Lie(\CC{\bPrad^\pm}{\mc L}(\gamma))\), hence, by \locpref{big-cell} again, equals \(\Lie(\CC\bG{\mc L}(\gamma))\).  Since \bH is smooth and connected, we have by Lemma \initref{lem:gp-dfc-unique}\subpref{unique} that \bH equals \(\CC\bG{\mc L}(\gamma)\conn\).  That is, \(q(\CC\bG{\mc L}(\gamma)\conn)\) equals \(\CC{\bG'}{\mc L}(\gamma)\conn\), as desired.

It remains only to prove \locpref{reductive}.  By Lemma \initref{lem:gp-dfc-unique}\subpref{unique}, it suffices to show that there is a \(\gamma\)-stable, connected, reductive subgroup of \bG whose Lie algebra is \(\CC{\Lie(\bG)}{\mc L}(\gamma)\).  We prove this by induction on \(N\), splitting our argument into three cases.



The first, base, case is when \(N = 1\).
Since evaluation at \(\gamma = \gamma^N\) furnishes a homomorphism \anonmap{\Z\Root(\bG, \bT)}{\field\mult}, we have that \(\Root(\CC{\Lie(\bG)}{\mc L}(\gamma), \bT) = \set{\alpha \in \Root(\bG, \bT)}{\alpha(\gamma) \in \mc L}\) is closed in \(\Root(\bG, \bT)\).
Put \(\bH = \Cent_\bG\bigl(\bigcap_{\alpha \in \Root(\CC{\Lie(\bG)}{\mc L}(\gamma), \bT)} \ker(\alpha)\bigr)\conn\).
(We emphasise that we are taking the identity component of the centraliser, not the centraliser of the identity component, so that \(\bigcap \ker(\alpha)\), not just its identity component, is central in \(\CC\bG i(\gamma)\conn\).)
We have by \cite{borel:linear}*{Proposition 13.20} that \(\Lie(\bH)\) is the sum of \(\Lie(\bT)\) and the root subspaces for \bT in \(\Lie(\bG)\) corresponding to roots
in \(\Z\Root(\CC{\Lie(\bG)}{\mc L}(\gamma), \bT) \cap \Root(\bG, \bT) = \Root(\CC{\Lie(\bG)}{\mc L}(\gamma), \bT)\);
that is, \(\Lie(\bH)\) equals
\(\CC{\Lie(\bG)}{\mc L}(\gamma)\).




In the other two cases, we identify an appropriate non-trivial divisor \(n\) of \(N\).  By induction, we then may, and do, assume, upon replacing \bG by \(\CC\bG{\mc L^n}(\gamma^n)\conn\) and \bT by \(\CC\bG{\mc L^n}(\gamma^n) \cap \bT\), that \(\Lie(\bG)\) equals \(\CC{\Lie(\bG)}{\mc L^n}(\gamma^n)\).


The second of three cases is when \(\mu_N(\field) \cap \mc L\) is trivial, in which case we put \(n = N\).


Remember that we may evaluate elements \(\chi\) of \(\Z\Root(\bG, \bT)\) at \(\gamma^N = \gamma^n\), and that, by the assumption that \(\CC{\Lie(\bG)}{\mc L^n}(\gamma^n)\) equals \(\Lie(\bG)\), all \(\chi(\gamma^n)\) lie in \(\mc L^n\).  Thus there is a unique element \(t \in (\bT/{\Zent(\bG)})(\field)\) such that, for all \(\chi \in \clat(\bT/{\Zent(\bG)})\), we have that \(\chi(t)\) lies in \(\mc L\) and \(\chi(t)^n\) equals \(\chi(\gamma^n)\).  Note that \(t\) belongs to \(\Cent_{\bT/{\Zent(\bG)}}(\gamma)(\field)\).

Put \(\bH = \Cent_\bG(\gamma t\inv)\conn\).
By construction, the only eigenvalue for \(\gamma t\inv\) on \(\Lie(\bG)\) that lies in \mc L is \(1\), so \(\Lie(\bH) = \Cent_{\Lie(\bG)}(\gamma t\inv)\) equals \(\CC{\Lie(\bG)}{\mc L}(\gamma t\inv)\).  Lemma \initref{lem:gp-dfc-unique}\subpref{nearby-equal} gives that \(\CC{\Lie(\bG)}{\mc L}(\gamma t\inv)\) equals \(\CC{\Lie(\bG)}{\mc L}(\gamma)\).
The third case is when \(\mu_N(\field) \cap \mc L\) is non-trivial.  Let \(n\) be a divisor of \(N\) such that \(\mu_N(\field) \cap \mc L\) equals \(\mu_n(\field)\).  Then \(\Lie(\bG) = \CC{\Lie(\bG)}{\mc L^n}(\gamma^n)\) equals \(\CC{\Lie(\bG)}{\mc L}(\gamma)\) by Lemma \initref{lem:gp-dfc-power}\subpref{contain}.
%
%
%
%
\end{proof}

\subsection{$\gamma \in \Lie(G)$}
\label{subsec:Lie-funny}

The results of \S\ref{subsec:Lie-funny} almost exactly parallel those of \S\S\ref{subsec:gp-funny}, \ref{subsec:reductive-gp-funny}, except that now we regard the element \(\gamma\) as living in \(\Lie(G)\) rather than being an automorphism of \bG.  Note that, though \(\gamma\) now lives on the Lie algebra, we are still defining subgroups of \bG, not just sub-Lie algebras of \(\Lie(\bG)\).  Since the proofs are identical or easier, we omit them, instead indicating the analogous results in \S\S\ref{subsec:gp-funny}, \ref{subsec:reductive-gp-funny}.

There is a slight notational conflict, in which our notation now requires us to determine whether to view an automorphism \(\gamma\) of a finite-dimensional vector space \bV as lying in \(\GL(\bV)\) or \(\mathfrak{gl}(\bV)\) before understanding how to make sense of notation like \(\CC{\GL(\bV)}{\mc L}(\gamma)\).  We will rely on context to make this clear.

We could perhaps generalise the results of \S\ref{subsec:Lie-funny} to handle more general derivations of \(\Lie(\bG)\) rather than elements of \(\Lie(G)\), but we have no need of this generalisation, and it is not clear to us whether `semisimple derivations' have the same nice properties as semisimple automorphisms; so we do not pursue this generalisation.

\begin{defn}[Definition \ref{defn:GL-fc}]
\label{defn:gl-fc}
Suppose that
	\begin{itemize}
	\item \bV is a (possibly infinite dimensional) vector space over \field,
	\item \(\gamma\) is a locally diagonalisable endomorphism of \bV,
and	\item \mc L is a subgroup of \field.
	\end{itemize}
For every \(\lambda \in \field\), we define the \mc L-close \(\lambda\)-eigenspace for \(\gamma\) in \bV to be the sum of the eigenspaces for \(\gamma\) in \bV corresponding to eigenvalues in \(\lambda + \mc L\), and write \(\CC\bV{\mc L}(\gamma)\) for the \mc L-close \(0\)-eigenspace for \(\gamma\) in \bV.
\end{defn}

Throughout \S\ref{subsec:Lie-funny}, fix a semisimple element \(\matnotn{gamma}\gamma \in \Lie(G)\), and a \(\Gal(\sepfield/\field)\)-stable subgroup \matnotn L{\mc L} of \sepfield.

\begin{defn}[Definition \ref{defn:gp-dfc} and Proposition \ref{prop:gp-dfc-weights}]
\label{defn:Lie-dfc}
Let \(\sepfield[\bG]_{\gamma \in \mc L}\) be the maximal quotient algebra of \(\sepfield[\bG]\) on which all eigenvalues of \(\gamma\) lie in \mc L.  Write \(\CC\bG{\mc L}(\gamma)\) for the descent to \field of \(\operatorname{Spec} \sepfield[\bG]_{\gamma \in \mc L}\).
If \sepfield is equipped with a \(\Gal(\sepfield/\field)\)-fixed valuation, then, for every \(i \in \tR\), put \(\usualCC\bG i(\gamma) = \CC\bG{\sbjtl{\field\sep}i}(\gamma)\).
\end{defn}

\begin{prop}[Propositions \ref{prop:gp-dfc-rep}, \ref{prop:gp-dfc-weights}, and \ref{prop:gp-dfc-smooth}]
\initlabel{prop:Lie-dfc-facts}
\hfill\begin{enumerate}
\item
\(\CC\bG{\mc L}(\gamma)\) is a subgroup of \bG.
\item
\(\Lie(\CC\bG{\mc L}(\gamma))\) equals \(\CC{\Lie(\bG)}{\mc L}(\gamma)\).
\item
If \bG is smooth, then so is \(\CC\bG{\mc L}(\gamma)\).
\item
If \anonmap{\wtilde\bG}\bG is a morphism and \(\tilde\gamma \in \Lie(\wtilde G)\) is a semisimple element mapping to \(\gamma\), then the image of \(\CC{\wtilde\bG}{\mc L}(\tilde\gamma)\) is contained in \(\CC\bG{\mc L}(\gamma)\).
\item\sublabel{sub}
If \(\bG'\) is a subgroup of \bG such that \(\gamma\) belongs to \(\Lie(G')\), then \(\CC\bG{\mc L}(\gamma) \cap \bG'\) equals \(\CC{\bG'}{\mc L}(\gamma)\).
\item
For every \bG-module \bW, we have that \(\CC\bG{\mc L}(\gamma)\) stabilises every \mc L-close eigenspace for \(\gamma\) in \bW.
\setcounter{tempenumi}{\value{enumi}}
\end{enumerate}
Suppose that \(t \in \Lie(G)\) is semisimple and commutes with \(\gamma\).
\begin{enumerate}
\setcounter{enumi}{\value{tempenumi}}
\item\sublabel{nearby-equal}
The intersections with \(\CC\bG{\mc L}(t)\) of \(\CC\bG{\mc L}(\gamma)\) and \(\CC\bG{\mc L}(\gamma - t)\) are equal.
\item\sublabel{nearby-sub}
Let \(\mc L'\) be a \(\Gal(\sepfield/\field)\)-stable subgroup of \mc L.
If
	\(\CC\bG{\mc L}(t)\) contains both \(\CC\bG{\mc L}(\gamma)\) and \(\CC\bG{\mc L}(\gamma - t)\), and
	\(\CC\bG{\mc L}(\gamma)\) equals \(\Cent_\bG(\gamma)\), then
\(\CC\bG{\mc L'}(\gamma - t)\) equals \(\CC\bG{\mc L}(\gamma) \cap \CC\bG{\mc L'}(t)\).
\end{enumerate}
\end{prop}

\begin{prop}[Proposition \ref{prop:gp-dfc-smooth} and Lemma \ref{lem:gp-dfc-unique}]
\initlabel{prop:Lie-smooth-dfc-facts}
\hfill\begin{enumerate}
\item
If \bJ is a \(\gamma\)-stable, connected, smooth subgroup of \bG such that \(\Lie(\bJ)\) is contained in (respectively, contains) \(\CC{\Lie(\bG)}{\mc L}(\gamma)\), then \bJ is contained in (respectively, contains) \(\CC\bG{\mc L}(\gamma)\conn\).
\setcounter{tempenumi}{\value{enumi}}
\end{enumerate}
Suppose that \bG is smooth.
\begin{enumerate}
\setcounter{enumi}{\value{tempenumi}}
\item
\(\CC\bG{\mc L}(\gamma)\) is smooth.
\item\sublabel{unique}
\(\CC\bG{\mc L}(\gamma)\conn\) is the unique \(\gamma\)-stable, smooth, connected subgroup of \bG whose Lie algebra is \(\CC{\Lie(\bG)}{\mc L}(\gamma)\).
\item\sublabel{nearby-equal}
Suppose that \(t \in \Lie(G)\) is semisimple and commutes with \(\gamma\).  If \(\CC{\Lie(\bG)}{\mc L}(t)\) contains \(\CC{\Lie(\bG)}{\mc L}(\gamma) + \CC{\Lie(\bG)}{\mc L}(\gamma - t)\), then \(\CC\bG{\mc L}(\gamma - t)\conn\) equals \(\CC\bG{\mc L}(\gamma)\conn\).
\item\sublabel{nearby-sub}
If, in addition to the hypotheses of \locpref{nearby-equal}, we have that \(\CC{\Lie(\bG)}{\mc L}(\gamma)\) equals \(\Cent_{\Lie(\bG)}(\gamma)\), and if \(\mc L'\) is a \(\Gal(\sepfield/\field)\)-stable subgroup of \mc L, then \(\CC\bG{\mc L'}(\gamma - t)\conn\) is the identity component of \(\CC\bG{\mc L}(\gamma) \cap \CC\bG{\mc L'}(t)\).
\end{enumerate}
\end{prop}

\begin{prop}[Proposition \ref{prop:gp-cfc-facts}]
\initlabel{prop:Lie-cfc-facts}
Suppose that \bG is reductive.
\begin{enumerate}
\item
\(\CC\bG{\mc L}(\gamma)\conn\) is reductive.
\item
If \bT is a maximal torus in \bG such that \(\gamma\) belongs to \(\Lie(T)\), then \bT is also a maximal torus in \(\CC\bG{\mc L}(\gamma)\).
\item\sublabel{Levi}
If \(\bP^+\) is a parabolic subgroup of \bG and \bM is a Levi component of \(\bP^+\) such that \(\gamma\) belongs to \(\Lie(M)\), then \(\CC{\bP^+}{\mc L}(\gamma)\) is a parabolic subgroup of \(\CC\bG{\mc L}(\gamma)\) and \(\CC\bM{\mc L}(\gamma)\) is a Levi component of \(\CC{\bP^+}{\mc L}(\gamma)\).
\item\sublabel{big-cell}
If, with the notation of \locpref{Levi},
	\begin{itemize}
	\item \bG is connected,
	\item \(\bP^-\) is a parabolic subgroup of \bG opposite to \(\bP^+\) relative to \bM,
and	\item \(\bPrad^-\) is the unipotent radical of \(\bP^-\),
	\end{itemize}
then \(\CC\bG{\mc L}(\gamma) \cap (\bPrad^-\dotm\bM\dotm\bPrad^+)\) equals \(\CC{\bPrad^-}{\mc L}(\gamma)\dotm\CC\bM{\mc L}(\gamma)\dotm\CC{\bPrad^+}{\mc L}(\gamma)\).
\item\sublabel{quotient}
If \anonmap\bG{\bG'} is a quotient by a central subgroup of \bG and \(\gamma'\) is the image of \(\gamma\) in \(\Lie(\bG')\), then \anonmap{\CC\bG{\mc L}(\gamma)\conn \subseteq \bG}{\bG'} factors through a quotient map \anonmap{\CC\bG{\mc L}(\gamma)\conn}{\CC{\bG'}{\mc L}(\gamma')\conn}.
\end{enumerate}
\end{prop}
}

\numberwithin{thm}{subsection}
\section{Hypotheses}
\label{sec:hyps}

Throughout \S\ref{sec:hyps} and the remainder of the paper, we assume that
	\matnotn k\field is a non-Archimedean local field
with valuation \matnotn{ord}\ord,
and	\matnotn G\bG is a (not necessarily connected) reductive group over \field.
Our main results require that the residual characteristic of \field is not \(2\), but we do not assume this yet.

For convenience, we will say
``{\let\noterm=\relax\term[torus!tame]{tame torus}}''
instead of ``torus \bT in \bG such that \(\bT_\tamefield\) is split.''  It is common to call an element of \(G\conn\) tame if it belongs to some tame torus.  We shall extend this to the disconnected case by saying that an element \(\gamma\) of \(G\) is \term[element!tame]{tame} if there are a \(\gamma\)-stable, tame torus \bS in \bG such that some prime-to-\(p\) power of \(\gamma\) belongs to \(S\), where \(p\) is the characteristic of \(\sbjat\field 0\), and a \(\gamma\)-stable Borel subgroup \(\bB^+\) of \(\bG_\sepfield\) that contains \(\Cent_\bG(\bS)_\sepfield\).

\subsection{Statements of the hypotheses}
\label{subsec:hyp-statements}

We state below, for convenient reference, the hypotheses we will need in this paper.  Note, however, that we are not yet imposing these hypotheses, only gathering them in one place.  When we mean to impose them, we will say so explicitly.

We have tried to isolate exactly what is needed to make our results `go', but the reader may prefer having practical-to-verify sufficient conditions.  We discuss this in \S\ref{subsec:sufficient}.

Lemma \ref{lem:dual-MP-centre} describes how Moy--Prasad filtrations of the dual Lie algebra interact with the centre-like space \(\Lie^*(G)^{G\conn}\).  It is a slightly souped-up dual-Lie-algebra analogue of \xcite{adler-spice:good-expansions}*{Lemma \xref{lem:depth-in-center}}; see also \xcite{adler-spice:good-expansions}*{Lemma \xref{lem:degen-mod-center}}.  Conceptually, it belongs in \S\ref{subsec:MP}, but it relies on Corollary \ref{cor:dual-centre}.

\begin{lem}
\label{lem:dual-MP-centre}
For every \(r \in \tR\) and every \(x \in \BB(G)\), the intersections with \(\Lie^*(G)^{G\conn}\) of
	\(\sbtl{\Lie^*(G)}x r\),
	\(\sbjtl{\Lie^*(G)}r\),
and	\(\sbjtl{\Lie^*(G)}r + \sbjtl{\Lie^*(G)}r\)
are equal.
\end{lem}

\begin{proof}
Put \(\mf z^* = \Lie^*(G)^{G\conn}\).  The containments
\[
\mf z^* \cap \sbtl{\Lie^*(G)}x r \subseteq
\mf z^* \cap \sbjtl{\Lie^*(G)}r \subseteq
\mf z^* \cap (\sbjtl{\Lie^*(G)}r + \sbjtl{\Lie^*(G)}r)
\]
are obvious, so we need only show that \(\mf z^* \cap (\sbjtl{\Lie^*(G)}r + \sbjtl{\Lie^*(G)}r)\) is contained in \(\sbtl{\Lie^*(G)}x r\).

Fix \(Z^* \in \mf z^* \cap (\sbjtl{\Lie^*(G)}r + \sbjtl{\Lie^*(G)}r)\).  The dual-Lie-algebra version of \cite{adler-debacker:bt-lie}*{Lemma 3.2.1} gives that there are an element \(N^* \in \Lie^*(G)\) and a cocharacter \(\lambda \in \cclat(\bG)\) such that \(Z^*\) belongs to \(\sbjtl{\Lie^*(G)}r + N^*\), and \(\lim_{t \to 0} \Ad^*(\lambda(t))X^*\) equals \(0\) in the analytic topology.  Since \(\sbjtl{\Lie^*(G)}r\) is closed and stabilised by the coadjoint action of \(G\conn\), and since \(Z^*\) is fixed by the coadjoint action of \(G\conn\), we have that \(Z^* = \lim_{t \to 0} \Ad^*(\lambda(t))(Z^* - N^*)\) belongs to \(\sbjtl{\Lie^*(G)}r\).  Let \(y \in \BB(G)\) satisfy \(Z^* \in \sbtl{\Lie^*(G)}y r\), and let \bS be a maximal split torus in \bG such that \(x\) and \(y\) both belong to the apartment of \bS.  Corollary \ref{cor:dual-centre} gives that \(Z^*\) belongs to \(\Lie^*(\Cent_G(\bS))\), hence to \(\sbtl{\Lie^*(G)}y r \cap \Lie^*(\Cent_G(\bS)) = \sbtl{\Lie^*(\Cent_G(\bS))}y r = \sbtl{\Lie^*(\Cent_G(\bS))}x r\), hence to \(\sbtl{\Lie^*(G)}x r\).
\end{proof}

Lemma \ref{lem:MP-cfc} describes an interaction between Moy--Prasad filtrations and the groups defined in \S\ref{sec:funny-centraliser}.  Again, conceptually, it belongs in \S\ref{subsec:MP}, but it could not even be stated until after \S\ref{sec:funny-centraliser}.

\begin{lem}
\label{lem:MP-cfc}
If \(r\) belongs to \(\sbjtl\tR 0\) and \(\gamma\) is a semisimple element of \(\sbjtl G r\)
	(respectively, \(r\) belongs to \tR and \(\gamma\) is a semisimple element of \(\sbjtl{\Lie(G)}r\)),
then \(\CC\bG r(\gamma)\conn\) equals \(\bG\conn\).
\end{lem}

\begin{proof}
The Moy--Prasad domains \(\sbjtl{\Lie(G)}r\) (respectively, \(\sbjtl G r\)) only grow when \field is replaced by a finite, separable extension.
(We run into trouble with Galois descent from non-tame extensions, but there is no such issue for `ascent'.  See, for example, \xcite{adler-spice:good-expansions}*{Lemma \xref{lem:domain-field-ascent}}.)
Thus, upon replacing \field by an appropriate such extension, we may, and do, assume that \(\gamma\) is diagonalisable.
Upon increasing \(r\), which shrinks \(\CC\bG r(\gamma)\), we may, and do, assume that \(r\) is the depth of \(\gamma\), hence belongs to \(\ord(\field\mult)\).

Choose \(x \in \BB(G)\) such that \(\gamma\) belongs to \(\sbtl{\Lie(G)}x r\) (respectively, to \(\sbtl G x r\)).

Write \(\bG'\) for the adjoint quotient of \(\bG\conn\).  Then the adjoint representation \anonmap{\bG'}{\GL(\Lie(\bG))} is faithful.  We have that \(\ad(\gamma)\) (respectively, \(\Ad(\gamma) - 1\)) carries \(\sbtl{\Lie(G)}x i\) into \(\sbtl{\Lie(G)}x{i + r} = \sbjtl\field r\dotm\sbtl{\Lie(G)}x i\) for all \(i \in \R\), so that all eigenvalues of \(\ad(\gamma)\) on \(\Lie(\bG)\) lie in \(\sbjtl\field r\) (respectively, all eigenvalues of \(\Ad(\gamma)\) on \(\Lie(\bG)\) lie in \(\sbjtl{\field\mult}r\)); i.e., \(\CC{\Lie(\bG)}r(\gamma)\) equals \(\Lie(\bG)\).  By Lemma \initref{lem:gp-dfc-unique}\subpref{unique} (respectively, Proposition \initref{prop:Lie-smooth-dfc-facts}\subpref{unique}), we have that \(\CC{\bG'}r(\gamma)\) equals \(\bG'\).  Since \(\CC\bG r(\gamma)\conn\) contains \(\Cent_\bG(\gamma)\conn\), which in turn contains some maximal torus in \(\bG\conn\) and hence contains \(\Zent(\bG\conn)\), and since Proposition \initref{prop:Lie-cfc-facts}\subpref{quotient} (respectively, Proposition \initref{prop:gp-cfc-facts}\subpref{quotient}) gives that \anonmap{\CC\bG r(\gamma)\conn}{\CC{\bG'}r(\gamma)\conn = \bG'} is surjective, it follows that \(\CC\bG r(\gamma)\conn\) equals \(\bG\conn\).
\end{proof}

Because of the improved technology in \S\ref{sec:funny-centraliser}, and a few errors in \cite{spice:asymptotic} (see Appendix \ref{app:errata}), we are able to, and in some cases we must, replace the complicated \xcite{spice:asymptotic}*{Hypotheses \xref{hyp:funny-centraliser} and \xref{hyp:gamma}} by simplified versions.

\begin{defn}
\label{defn:depth-vec}
With the notation of Appendix \ref{app:depth}, if \(\vbG = (\bG^0 \subseteq \dotsb \subseteq \bG^{\ell - 1} \subseteq \bG^\ell = \bG)\) is a nested, tame, twisted Levi sequence in \bG and we put \(S = \tR \cup \sset{-\infty}\) and \(\mc X_0 = \sset{0, \dotsc, \ell}\), then a singly indexed vector \(\vec a \in S^{\mc X_0}\) is called a \term[depth vector|see{vector, singly indexed}]{depth vector} for \vbG.  The singly indexed vector \(\vec a = (a_0, \dotsc, a_\ell)\) is called \term[depth vector!grouplike]{grouplike} if \(a_0\) is positive and \(\vec a\) is concave, in the sense that \(2a_j \ge a_{j'}\) for all \(0 \le j' \le j \le \ell\).
\end{defn}

Most of \xcite{spice:asymptotic}*{Hypothesis \xref{hyp:funny-centraliser}} was actually proven in \S\ref{sec:funny-centraliser}.  In fact, all we are left needing is a statement about buildings.  We formally state Hypothesis \ref{hyp:fc-building} as a replacement for \xcite{spice:asymptotic}*{Hypothesis \xref{hyp:funny-centraliser}}.  Since the statements for group automorphisms and Lie-algebra elements are practically identical, we unify them.  Thus, Hypothesis \ref{hyp:fc-building} involves an element \(\gamma\), which may be either a semisimple automorphism of \bG or a semisimple element of \(\Lie(G)\).

As mentioned in \cite{spice:asymptotic}*{\S3.2, p.~2318}, Hypothesis \ref{hyp:fc-building} holds whenever \(\gamma\) is an element of a tame torus satisfying \xcite{adler-spice:good-expansions}*{Definition \xref{defn:S-is-good}}, as long as \xcite{adler-spice:good-expansions}*{Hypotheses \xref{hyp:reduced}--\xref{hyp:torus-H1-triv}} are satisfied.

More generally, suppose that
\(\gamma\) is a tame element of \(G\) (not necessarily of \(G\conn\)).  Then Lemma \initref{lem:shallow-dfc}\subpref{0+} and \cite{kaletha-prasad:bt-theory}*{Theorem 12.7.1} provide an embedding of buildings \anonmap{\BB(\CCp G 0(\gamma))}{\BB(\CC G 0(\gamma))} and show that part of Hypothesis \ref{hyp:fc-building} is satisfied for that embedding; and one can piece this together with the case of the previous paragraph to show that all of Hypothesis \ref{hyp:fc-building} is satisfied.  We do not go into details here.

The notion of a toral embedding of buildings used in Hypothesis \ref{hyp:fc-building} is as in \cite{landvogt:functorial}*{\S1.3.3}.  Namely, if \bH is a reductive subgroup of \bG and \(\bH'\) is a reductive subgroup of \bH, then we call a map \map i{\BB(H')}{\BB(H)} \term[embedding of buildings!toral]{toral} if it is \(H'(\field)\)-equivariant and, for every maximal split torus \(\bS'\) in \(\bH'\), there is a maximal split torus \bS in \(\Cent_\bH(\bS')\) such that \(i\) carries the apartment of \(S'\), as a subset of \(\BB(H')\), into the apartment of \(S\), as a subset of \(\BB(H)\); and the restriction of \(i\) to the apartment of \(S'\) is equivariant for the translation actions of \(\cclat(\bS') \otimes_\Z \R\) on the apartments of \(S'\) and \(S\).

\label{vGvr-defn}
Remember that
\citelist{
	\cite{adler-spice:good-expansions}*{Definition 5.14}
	\cite{adler-spice:explicit-chars}*{\S1.3, p.~1145}
}
define, for every nested, tame, twisted Levi sequence \(\vbG = (\bG^0 \subseteq \dotsb \subseteq \bG^{\ell - 1} \subseteq \bG^\ell = \bG)\) and depth vector \(\vec a = (a_0, \dotsc, a_{\ell - 1}, a_\ell)\) for \(\vbG\), a sublattice \(\Lie(\vG)_{x, \vec a}\) of \(\Lie(G)\) and, if \(\vec a\) is grouplike, a subgroup \(\vG_{x, \vec a}\) of \(G\).  For consistency with the notation here and in \cite{spice:asymptotic}, we will denote them by \(\sbtl{\Lie(\vG)}x{\vec a}\) and \(\sbtl\vG x{\vec a}\) instead.  If \vbG is \(\gamma\)-stable, then \(\CC\vbG i(\gamma) \ldef (\CC{\bG^0}i(\gamma) \subseteq \dotsb \subseteq \CC{\bG^{\ell - 1}}i(\gamma) \subseteq \CC{\bG^\ell}i(\gamma))\) is a nested, tame, twisted Levi sequence in \(\CC\bG i(\gamma)\), and so we have defined \(\sbtl{\Lie(\CC\vG i(\gamma))}x{\vec a}\) and \(\sbtl{\CC\vG i(\gamma)}x{\vec a}\), for every \(i \in \sbjtl\tR 0\) (if \(\gamma\) is a semisimple automorphism of \bG) or every \(i \in \tR\) (if \(\gamma\) is a semisimple element of \(\Lie(G)\)).

\begin{hyp}
\initlabel{hyp:fc-building}
For every \(i \in \sbjtl\tR 0\) (if \(\gamma\) is a semisimple automorphism of \bG) or every \(i \in \tR\) (if \(\gamma\) is a semisimple element of \(\Lie(G)\)), there is a \(\CCp{G\conn}i(\gamma)\)-equivariant, toral embedding of buildings \anonmap{\BB(\CCp G i(\gamma))}{\BB(\CC G i(\gamma))}
such that, for
every
	\begin{itemize}
	\item \(\gamma\)-stable, nested, tame, twisted Levi sequence \(\vbG = (\bG^0 \subseteq \dotsb \subseteq \bG^{\ell - 1} \subseteq \bG^\ell = \bG)\) in \bG,
	\item \(x \in \BB(\CCp{G^0}i(\gamma))\),
and	\item depth vector \(\vec a\) for \vbG,
	\end{itemize}
the following hold.
	\begin{enumerate}
	\item\sublabel{Lie}
\(\sbtl{\Lie(\CC\vG i(\gamma))}x{\vec a} \cap \Lie(\CCp G i(\gamma))\)
equals
\(\sbtl{\Lie(\CCp\vG i(\gamma))}x{\vec a}\).
	\item\sublabel{gp}
	If \(\vec a\) is grouplike, then
\(\sbtl{\CC\vG i(\gamma)}x{\vec a} \cap \CCp G i(\gamma)\conn\)
equals \(\sbtl{\CCp\vG i(\gamma)}x{\vec a}\).
	\end{enumerate}
\end{hyp}

\begin{rem}
\initlabel{rem:fc-building}
\hfill	\begin{enumerate}
	\item\sublabel{Lie*}
	Dualising
Hypothesis \initref{hyp:fc-building}\subpref{Lie} implies that, with the same notation,
\(\sbtl{\Lie^*(\CC\vG i(\gamma))}x{\vec a} \cap \Lie^*(\CCp\vG i(\gamma))\)
equals
\(\sbtl{\Lie^*(\CCp\vG i(\gamma))}x{\vec a}\).
	\item\sublabel{up-to-G}
If \(\gamma\) is a semisimple element of \(G\), then Lemma \initref{lem:shallow-dfc}\subpref0 gives that \(\CC\bG 0(\gamma)\) is a Levi subgroup of \bG, so that there is a \(\CC{G\conn}0(\gamma)\)-equivariant, toral embedding \anonmap{\BB(\CC G 0(\gamma))}{\BB(G)}.
Chaining this with the various embeddings of Hypothesis \ref{hyp:fc-building} shows that, for each \(i \in \tR\) with \(0 \le i \le r\), we have a \(\CC{G\conn}i(\gamma)\)-equivariant, toral embedding of buildings \anonmap{\BB(\CC G i(\gamma))}{\BB(G)} satisfying the analogues of Hypothesis \initref{hyp:fc-building}(\subref{Lie}, \subref{gp}).
	\end{enumerate}
\end{rem}

Because \xcite{spice:asymptotic}*{Hypothesis \xref{hyp:gamma}} uses the groups \(\CC\bG i(\gamma\pinv)\) with \(i\) negative, Remark \ref{rem:hyp:funny-centraliser} indicates that we need to replace it.  The replacements for \xcite{spice:asymptotic}*{Hypothesis \xref{hyp:gamma}} regarding group elements (Hypothesis \ref{hyp:gp-gamma}) and Lie-algebra elements (Hypothesis \ref{hyp:Lie-gamma}) are similar but have some important differences, so we separate them.

Hypothesis \ref{hyp:gp-gamma} involves a semisimple element \(\gamma \in G\), an element \(x\) of \(\BB(G)\), and a real number \(r\).  Note that we require that \(\gamma\) actually belong to \(G\), not just be a semisimple element of some larger group that normalises \bG; but we do \emph{not} require that \bG be connected.

In the statement of Hypothesis \initref{hyp:gp-gamma}\incpref{building}\subpref0, we use the ``rationalised cocharacter'' \(m_\gamma \in \cclat(\sgen\gamma) \otimes_\Z \Q\) from Notation \ref{notn:ord-gamma}.  By Lemma \initref{lem:shallow-dfc}\subpref0, we have that \(m_\gamma\) lies in \(\cclat(\Zent(\CC\bG 0(\gamma))) \otimes_\Z \Q\), so that the notation \(x - m_\gamma\) makes sense as long as \(x\) belongs to \(\BB(\CC G 0(\gamma))\).  This is the only place in Hypothesis \ref{hyp:gp-gamma} where we require that \(\gamma\) belong to \(G\).  It could probably be avoided by replacing \(m_\gamma\) by some other \(\gamma\)-fixed cocharacter \(\lambda\) of \(\Zent(\CC\bG 0(\gamma))\) such that \(\Cent_\bG(\lambda)\conn\) equals \(\CC\bG 0(\gamma)\conn\), but we have not pursued this.

If the characteristic of \(\sbjat\field 0\) is not \(2\), then,
as mentioned in \xcite{spice:asymptotic}*{\S\xref{sec:depth-matrix} p.~2322}, Hypothesis \ref{hyp:gp-gamma} holds whenever \(\gamma\) belongs to a tame torus satisfying \xcite{adler-spice:good-expansions}*{Definition \xref{defn:S-is-good}}, and \xcite{adler-spice:good-expansions}*{Hypotheses \xref{hyp:reduced}--\xref{hyp:torus-H1-triv}} are satisfied.
Most of Hypothesis \ref{hyp:gp-gamma} is unaffected if we replace \(\gamma\) by a finite power that is relatively prime to the characteristic \(p\) of \(\sbjat\field 0\).  The few exceptions, involving the groups \(\CC\bG i(\gamma)\) with \(i \in \sset{0, \Rp0}\), can be handled by an appeal to Lemma \ref{lem:shallow-dfc}.  Therefore, Hypothesis \ref{hyp:gp-gamma} is also satisfied by \(\gamma\) if it is satisfied by some prime-to-\(p\) power of \(\gamma\); this can show the validity of the hypothesis even for many tame rational points of non-identity components of \bG.

\begin{hyp}
\initlabel{hyp:gp-gamma}
The element \(\gamma\), viewed as an automorphism of \bG, satisfies Hypothesis \ref{hyp:fc-building}.  The real number \(r\) is positive.  For every
	\begin{itemize}
	\item \(i \in \tR\) with \(0 \le i \le r\),
	\item nested, tame, twisted Levi sequence \(\vbG = (\bG^0 \subseteq \dotsb \subseteq \bG^{\ell - 1} \subseteq \bG^\ell = \bG)\) in \bG with \(\gamma \in G^0\) and \(x \in \BB(G^0)\),
and	\item depth vector \(\vec a\) for \vbG,
	\end{itemize}
the following hold.
	\begin{enumerate}
	\item\inclabel{building}
		\begin{enumerate}
		\item\sublabel{pos}
		\(x\) belongs to \(\BB(\CC G r(\gamma))\).
		\item\sublabel0
		\(\gamma\dota x \in \BB(\CC G 0(\gamma))\) equals \(x - m_\gamma\).
		\end{enumerate}
	\item\reinclabel{Lie}
		\(\Ad(\gamma) - 1\) carries \(\sbtl{\Lie(\CC\vG i(\gamma))}x{\vec a}\) into \(\sbtl{\Lie(\CC{\Der\vG}i(\gamma))}x{\vec a + i}\), and induces an isomorphism
\[
\anonmap{\sbat{\Lie(\CC\vG i(\gamma))}x{\vec a}/{\sbat{\Lie(\CCp\vG i(\gamma))}x{\vec a}}}{\sbat{\Lie(\CC\vG i(\gamma))}x{\vec a + i}/{\sbat{\Lie(\CCp\vG i(\gamma))}x{\vec a + i}}}.
\]
	\item\reinclabel{bi-Lie-Lie}
Hypothesis \ref{hyp:fc-building} also holds for \(\gamma^2\).  The map \(\Ad(\gamma) - \Ad(\gamma\inv)\) carries \(\sbtl{\Lie(\CC\vG i(\gamma^2))}x{\vec a}\) into \(\sbtl{\Lie(\CC{\Der\vG}i(\gamma^2))}x{\vec a + i}\), and induces an isomorphism
\[
\anonmap{\sbat{\Lie(\CC\vG i(\gamma^2))}x{\vec a}/{\sbat{\Lie(\CCp\vG i(\gamma^2))}x{\vec a}}}{\sbat{\Lie(\CC\vG i(\gamma^2))}x{\vec a + i}/{\sbat{\Lie(\CCp\vG i(\gamma^2))}x{\vec a + i}}}.
\]
	\item\reinclabel{gp}
		If \(\vec a\) is grouplike, then \(\comm\gamma\anondot\) carries \(\sbtl{\CC\vG i(\gamma)}x{\vec a}\) into \(\sbtl{\CC{\Der\vG}i(\gamma)}x{\vec a + i}\), and induces a bijection
\[
\anonmap{\sbat{\CC\vG i(\gamma)}x{\vec a}/{\sbat{\CCp\vG i(\gamma)}x{\vec a}}}{\sbat{\CC\vG i(\gamma)}x{\vec a + i}/{\sbat{\CCp\vG i(\gamma)}x{\vec a + i}}}.
\]
	\item\reinclabel{orbit}
	If \(g \in \sbtlp{\CC G i(\gamma)}x 0\) is such that
\[
\Int(g)(\gamma\dotm\sbtlp{\CCp G i(\gamma)}x i) \cap
(\gamma\dotm\sbtlp{\CCp G i(\gamma)}x i)
\]
is non-empty, then \(g\) belongs to \(\sbtlp{\CCp G i(\gamma)}x 0\).
	\end{enumerate}
\end{hyp}

\begin{lem}
\label{lem:gp-orbit-0}
Suppose that Hypothesis \ref{hyp:gp-gamma} is satisfied for \((\gamma, x, \Rp0)\).  If \(g \in \sbtlp G x 0\) is such that \(\Int(g)(\gamma\dotm\sbtlp{\CC G 0(\gamma)}x 0) \cap (\gamma\dotm\sbtlp{\CC G 0(\gamma)}x 0)\) is non-empty, then \(g\) belongs to \(\sbtlp{\CC G 0(\gamma)}x 0\).
	\project[Maybe we can also say something about \(\Int(g)\bigl(\gamma\dotm\sbtlp G x 0\bigr) \cap \gamma\dotm\sbtlp G x 0\) \emph{and} \(\Int(g)\bigl(\gamma\dotm\sbtlp G x 0\bigr) \cap \gamma\dotm\sbtlp G x 0\) being non-empty?]
\end{lem}

\begin{proof}
Lemma \initref{lem:shallow-dfc}\subpref0 gives that \(\bM \ldef \CC\bG 0(\gamma)\) is a Levi subgroup of \bG.  Since \(\gamma\) fixes the image of \(x\) in the reduced building of \bM by Hypothesis \initref{hyp:gp-gamma}\incpref{building}\subpref0, we have that \(\gamma\) normalises \(\sbtlp M x 0\).  For every positive integer \(n\), it follows that \(\Int(g)\bigl(\gamma^n\dotm\sbtlp M x 0\bigr) \cap \gamma^n\sbtlp M x 0\) is non-empty; and we have by Lemma \initref{lem:gp-dfc-power}\subpref{contain} (or by Lemma \initref{lem:shallow-dfc}\subpref0 again) that \(\CC\bG 0(\gamma^n)\) equals \(\CC\bG 0(\gamma)\).  We thus may, and do, assume, upon replacing \(\gamma\) by \(\gamma^n\) with \(n = \card{\pi_0(\sgen\gamma)}\), that \(\gamma\) belongs to \(G\conn\).  Since \(\gamma\) also belongs to \(M\), it follows that \(\gamma\) belongs to \(M \cap G\conn\), which equals \(M\conn\) since \bM is a Levi subgroup of \bG.

Write \(g = u^- u^+ m\) with \(u^\pm \in \Prad^\pm \cap \sbtlp G x 0\) and \(m \in \sbtlp M x 0\), and let \(m_1, m_2 \in \sbtlp{\CC G 0(\gamma)}x 0\) be such that \(\Int(g)(\gamma m_1)\) equals \(\gamma m_2\).  Put \(\tilde m_1 
= \comm{\gamma\inv}m\dotm\Int(m)m_1\), which belongs to \(\sbtlp M x 0\).  Then \(\Int(u^+)\bigl(\Int(m)(\gamma m_1)\bigr) = \Int(u^+)(\gamma\tilde m_1)\), which is the product of \(\comm{u^+}{\gamma\tilde m_1} \in \Prad^+\) with \(\gamma\tilde m_1 \in M\), equals \(\Int(u^-)\inv(\gamma m_2)\), which is the product of \(\comm{u\theninv-}{\gamma m_2} \in \Prad^-\) with \(\gamma m_2 \in M\); so \(\comm{u^+}{\Int(m)(\gamma m_1)}\) and \(\comm{u\theninv-}{\gamma m_2}\) both equal \(1\).

In other words, \(\Int(\gamma)\inv u^+\) equals \(\Int(\tilde m_1)u^+\) and \(\Int(m_2)u^-\) equals \(\Int(\gamma)\inv u^-\).  Suppose \(u^+\) is not \(1\).  Then there is some \(a \in \R\) such that \(u^+\) belongs to \(\Prad^+ \cap \sbtl G x a \setminus \sbtlp G x a\).  Since \(\tilde m_1\) belongs to \(\sbtlp M x 0\), also \(\Int(\tilde m_1)u^+\) belongs to \(\Prad^+ \cap \sbtl G x a \setminus \sbtlp G x a\).  On the other hand, we have that \(\Int(\gamma)\inv u^+\) belongs to \(\Prad^+ \cap \sbtl G{\gamma\dota x}a \setminus \sbtlp G{\gamma\dota x}a\), which equals \(\Prad^+ \cap \sbtl G{x - m_\gamma}a \setminus \sbtlp G{x - m_\gamma}a\) by Hypothesis \initref{hyp:gp-gamma}\incpref{building}\subpref0 and so is contained in \(\Prad^+ \cap \sbtlp G x a\).  This is a contradiction.  We reason similarly if \(u^-\) is not \(1\), except now we observe that \(\Int(\gamma)(\Prad^- \cap \sbtl G x b \setminus \sbtlp G x b)\) is contained in \(\Prad^- \setminus \sbtl G x b\).
\end{proof}

Corollary \ref{cor:gp-orbit} is the analogue of \xcite{spice:asymptotic}*{Lemma \initxref{lem:commute-gp}(\subxref{orbit})}.  It follows by applying Lemma \ref{lem:gp-orbit-0} once, and then repeatedly applying Hypothesis \initref{hyp:gp-gamma}\subpref{orbit}.

\begin{cor}
\label{cor:gp-orbit}
Suppose that \((\gamma, x, r)\) satisfies Hypothesis \ref{hyp:gp-gamma}.  If \(g \in \sbtlp G x 0\) is such that \(\Int(g)(\gamma\dotm\sbtl{\CC G r(\gamma)}x r) \cap (\gamma\dotm\sbtl{\CC G r(\gamma)}x r)\) is non-empty, then \(g\) belongs to \(\sbtlp{\CC G r(\gamma)}x 0\).
\end{cor}

Hypothesis \ref{hyp:Lie-gamma} involves a semisimple element \(\gamma\) of \(\Lie(G)\), an element \(x\) of \(\BB(G)\), and a real number \(r\), not necessarily positive.
As for Hypothesis \ref{hyp:gp-gamma}, as long as \xcite{adler-spice:good-expansions}*{Hypotheses \xref{hyp:reduced}--\xref{hyp:torus-H1-triv}} are satisfied, Hypothesis \ref{hyp:Lie-gamma} holds whenever \(\gamma\) belongs to the Lie algebra of a tame torus satisfying \xcite{adler-spice:good-expansions}*{Definition \xref{defn:S-is-good}}.

\begin{hyp}
\initlabel{hyp:Lie-gamma}
For every
	\begin{itemize}
	\item \(i \in \tR\),
	\item nested, tame, twisted Levi sequence \(\vbG = (\bG^0 \subseteq \dotsb \subseteq \bG^{\ell - 1} \subseteq \bG^\ell = \bG)\) in \bG with \(\gamma \in G^0\) and \(x \in \BB(G^0)\),
and	\item depth vector \(\vec a\) for \vbG,
	\end{itemize}
the following hold.
	\begin{enumerate}
	\item\sublabel{building}
	\(x\) belongs to \(\BB(\CC G r(\gamma))\).
	\item\sublabel{Lie}
	\(\ad(\gamma)\) carries \(\sbtl{\Lie(\CC\vG i(\gamma))}x{\vec a}\) into \(\sbtl{\Lie(\CC{\Der\vG}i(\gamma))}x{\vec a + i}\), and induces an isomorphism
\[
\anonmap{\sbat{\Lie(\CC\vG i(\gamma))}x{\vec a}/{\sbat{\Lie(\CCp\vG i(\gamma))}x{\vec a}}}{\sbat{\Lie(\CC\vG i(\gamma))}x{\vec a + i}/{\sbat{\Lie(\CCp\vG i(\gamma))}x{\vec a + i}}}.
\]
	\item\sublabel{gp}
	If \(\vec a\) is grouplike, then the map \((\Ad(\anondot) - 1)\gamma\) carries \(\sbtl{\CC\vG i(\gamma)}x{\vec a}\) into \(\sbtl{\Lie(\CC{\Der\vG}i(\gamma))}x{\vec a + i}\), and induces a bijection
\[
\anonmap{\sbat{\CC\vG i(\gamma)}x{\vec a}/{\sbat{\CCp\vG i(\gamma)}x{\vec a}}}{\sbat{\Lie(\CC\vG i(\gamma))}x{\vec a + i}/{\sbat{\Lie(\CCp\vG i(\gamma))}x{\vec a + i}}}.
\]
	\item\sublabel{orbit} If \(g \in \sbtlp{\CC G i(\gamma)}x 0\) is such that
\[
\Ad(g)(\gamma + \sbtlp{\CCp{\Lie(G)}i(\gamma)}x i)
\cap (\gamma + \sbtlp{\CCp{\Lie(G)}i(\gamma)}x i)
\]
is non-empty, then \(g\) belongs to \(\sbtlp{\CCp G i(\gamma)}x 0\).
	\end{enumerate}
\end{hyp}

%

Lemma \ref{lem:Lie-orbit} is the analogue for \(\gamma \in \Lie(G)\) of \xcite{spice:asymptotic}*{Lemma \initxref{lem:commute-gp}(\subxref{orbit})}.  As with Corollary \ref{cor:gp-orbit}, it follows by applying Hypothesis \initref{hyp:Lie-gamma}\subpref{orbit} repeatedly.

\begin{lem}
\label{lem:Lie-orbit}
Suppose that \((\gamma, x, r)\) satisfies Hypothesis \ref{hyp:Lie-gamma}.  If \(g \in \sbtlp G x 0\) is such that \(\Ad(g)(\gamma + \sbtl{\CC G r(\gamma)}x r) \cap (\gamma + \sbtlp{\CC G r(\gamma)}x r)\) is non-empty, then \(g\) belongs to \(\sbtlp{\CC G r(\gamma)}x 0\).
\end{lem}

Hypothesis \ref{hyp:MP-ad} involves a semisimple element \(\gamma \in \Lie(G)\).  It is an obvious modification of \xcite{spice:asymptotic}*{Hypothesis \xref{hyp:MP-ad}}, and is needed in Proposition \ref{prop:Gauss-to-Weil}, hence also in the remainder of \S\ref{subsec:dist-g-to-g'} and \S\S\ref{subsec:quantitative}, \ref{sec:unwind}.  If \(\gamma\) belongs to the Lie algebra of some tame torus, then, as in the proof of \cite{kaletha-prasad:bt-theory}*{Theorem 13.5.1}, we may piece together the Moy--Prasad isomorphism for tame tori \cite{kaletha-prasad:bt-theory}*{Proposition B.6.9} with root-space-by-root-space computations over a splitting field, and then construct the maps \(\sbat{\ol\mexp}x{\vec\jmath}\) of the hypothesis by tame descent.

\begin{hyp}
\initlabel{hyp:MP-ad}
Let \(\bG\adform\) be the adjoint quotient of \bG.  For each
	\begin{itemize}
	\item nested, tame, twisted Levi sequence \(\vec\bG = (\bG^0 \subseteq \dotsb \subseteq \bG^\ell = \bG)\) in \bG with \(\gamma \in \Lie(G^0)\) and \(x \in \BB(G^0)\),
and	\item grouplike depth vector \(\vec\jmath\) satisfying \(\vec\jmath \cvee \vec\jmath \ge \Rp{\vec\jmath}\),
	\end{itemize}
there is an isomorphism
\[
	\map[\mapisoarrow]{\matnotn{expxj}{\sbat{\ol\mexp}x{\vec\jmath}}}
	{\sbat{\Lie(\vG\adform)}x{\vec\jmath}}
	{\sbat{(\vG\adform)}x{\vec\jmath}}.
\]
	\begin{enumerate}
	\item\sublabel{ad} These isomorphisms satisfy \xcite{spice:asymptotic}*{Hypothesis \initxref{hyp:MP-ad}(\subxref{ad})}.
	\item\sublabel{Ad} For all \(i \in \tR_{\ge 0}\), grouplike depth vectors \(\vec\jmath\) satisfying \(\vec\jmath \cvee \vec\jmath \ge \Rp{\vec\jmath}\), and elements \(Y \in \sbtl{\Lie(\CC{\vG\adform}i(\gamma))}x{\vec\jmath}\), there is an element \(v \in \sbat{\ol\mexp}x{\vec\jmath}(Y) \cap \sbtl{\CC{\vG\adform}i(\gamma)}x{\vec\jmath}\) so that
\[
(\Ad(v) - 1)\gamma - \comm Y\gamma
\qtextq{belongs to}
\sbtlp{\Lie(\vG\adform)}x{i + \vec\jmath}.
\]
	\item\sublabel{refine} For all grouplike depth vectors \(\vec\jmath_m\) satisfying \(\vec\jmath_m \cvee \vec\jmath_m \ge \Rp{\vec\jmath_m}\), for all \(m \in \sset{1, 2}\), if \(\Rp{\vec\jmath_1}\) and \(\Rp{\vec\jmath_2}\) are equal, then the diagram
\[\xymatrix{
\sbat{\Lie(\vG\adform)}x{\max \sset{\vec\jmath_1, \vec \jmath_2}} \ar[r]\ar[d] & \sbat{\Lie(\vG\adform)}x{\vec\jmath_1} \ar[d] \\
\sbat{\Lie(\vG\adform)}x{\vec\jmath_2} \ar[r] & \sbat{(\vG\adform)}x{\min \sset{\vec\jmath_1, \vec\jmath_2}}
}\]
commutes.
	\end{enumerate}
\end{hyp}

Hypothesis \ref{hyp:mexp} involves a Levi subgroup \bM of \bG, a reductive subgroup \bJ of \bM, and a positive real number \(R_{-1}\).  It introduces a mock-exponential map on \(\sbjtl{\Lie(J)}{R_{-1}}\) having certain properties.  This map is needed in \S\S\ref{sec:characters}, \ref{subsec:char-unwind}.

As long as \(p\) also does not divide the order of the fundamental group of \(\Der\bG\), then \citelist{
	\cite{adler-roche:intertwining}*{Proposition 3.2}
	\cite{kaletha:regular-sc}*{Lemma 3.3.2(1, 4)}
} show that Hypothesis \initref{hyp:mexp}\subpref{MP-ad} follows from Hypothesis \initref{hyp:mexp}(\subref{coset}, \subref{elt}).

As mentioned in \xcite{spice:asymptotic}*{\S\xref{sec:dual-blob}, p.~2337}, the exponential map (when it converges on a suitably large domain), the Cayley transform (for classical groups, when the characteristic of \(\sbjat\field 0\) is not \(2\)), and the ``\(1 +{}\)'' map (for general linear groups) all satisfy Hypothesis \ref{hyp:mexp}.

When \field has characteristic \(0\), there is a sufficient condition for the convergence of the exponential map on \(\sbjtlp{\Lie(G)}0\), and hence for the validity of Hypothesis \ref{hyp:mexp} for \((\bM, \bJ, R_{-1}) = (\bG, \bG, \Rp0)\), in \cite{jkim:thesis}*{Proposition 3.1.1}.  We digress to describe it.

Suppose that \anonmap\bG{\GL_n} is a faithful representation of \bG such that, for every \(x \in \BB(G)\), there is some \(y \in \BB(\GL_n, \field)\) such that the representation carries \(\sbtl{\Lie(G)}x r\) into \(\sbtl{\mathfrak{gl}_n(\field)}y r\) for every \(r \in \tR\).  For example, if \bG is adjoint, then we may take the adjoint representation, so that \(n\) equals \(\dim(\bG)\).

For every point \(y \in \BB(\GL_n, \field)\), if \(C\) is a chamber whose closure contains \(y\), then we have the containments
\[
\sbtlp{\mathfrak{gl}_n(\field)}y 0 \subseteq \sbtlp{\mathfrak{gl}_n(\field)}C 0 = \sbtl{\mathfrak{gl}_n(\field)}C{1/n} \subseteq \sbtl{\mathfrak{gl}_n(\field)}C 0 \subseteq \sbtl{\mathfrak{gl}_n(\field)}y 0.
\]
It follows that, for every \(r \in \sbjtlp\tR 0\), the product of \(n\) elements of \(\sbtl{\mathfrak{gl}_n(\field)}y r\) is contained in \(\sbtl{\mathfrak{gl}_n(\field)}C 1 = \sbjtl\field 1\dotm\sbtl{\mathfrak{gl}_n(\field)}C 0 \subseteq \sbjtlp\field 0\dotm\sbtl{\mathfrak{gl}_n(\field)}y 0\), so that the product of \(n + 1\) elements is contained in \(\sbjtlp\field 0\dotm\sbtl{\mathfrak{gl}_n(\field)}y r\).

Therefore, still assuming that \field has characteristic \(0\), \cite{jkim:thesis}*{Proposition 3.1.1} shows that the inequality \(\frac1{n + 1} > \frac{2\ord(p)}{p - 1}\cdot\frac p{p - 1} + \frac1{p - 1}\) is enough to give \cite{jkim-murnaghan:charexp}*{\S1.4, (H\field)}.  An appeal to the Baker--Campbell--Hausdorff formula similar to the one in the proof of \cite{jkim:thesis}*{proof of Proposition 3.1.1, p.~1010}  shows that the same inequality is enough to give Hypothesis \initref{hyp:mexp}(\subref{coset}, \subref{elt}) below.

\begin{hyp}
\initlabel{hyp:mexp}
There is a homeomorphism \map\mexp{\sbjtl{\Lie(J)}{R_{-1}}}{\sbjtl J{R_{-1}}} such that the following hold.
	\begin{enumerate}
	\item\sublabel{coset} We have that \xcite{spice:asymptotic}*{Hypothesis \initxref{hyp:mexp}(\subxref{coset})} holds, with \bJ in place of \bH and \(R_{-1}\) in place of \(r\), for every \(x \in \BB(J)\).
	\item\inclabel{elt}
	The following hold for every index \(i \in \sbjtl\tR{R_{-1}}\) and every semisimple \(Y \in \sbjtl{\Lie(J)}{R_{-1}}\).
		\begin{enumerate}
		\item\sublabel{fc}
		We have that \(\CC{\Lie(\bJ)}i(Y)\) equals \(\CC{\Lie(\bJ)}i(\mexp(Y))\).
		\item\sublabel{ad-Ad}
		For every \((x, a) \in \BB(H) \times \tR\), we have that \(1 + \ad(Y) - \Ad(\mexp(Y))\) carries \(\sbtl{\CC{\Lie(\bJ)}i(Y)}x a\) into \(\sbtlp{\CC{\Lie(\bJ)}i(Y)}x{a + i}\).
		\end{enumerate}
	\item\initnolabel{hyp:mexp}\sublabel{MP-ad}
	We have that \xcite{spice:asymptotic}*{Hypothesis \xref{hyp:MP-ad}} holds (with a few typos corrected as in Remark \ref{rem:spice:asymptotic:hyp:MP-ad}), guaranteeing the existence of certain maps \(\sbat{\ol\mexp}x{\vec\jmath}\),
and the diagram
\[\xymatrix{
\sbtl{\Lie(J)}x r        \ar[r]^\mexp\ar[d]          & \sbtl J x r          \ar[d] \\
\sbat{\Lie(M)}x r        \ar[r]^{\text{Yu}}\ar[d]    & \sbat M x r          \ar[d] \\
\sbat{\Lie(M\adform)}x r \ar[r]^{\sbat{\ol\mexp}x r} & \sbat{(M\adform)}x r
}\]
commutes for all \(x \in \BB(J)\), where the map labelled `Yu' is the one used in \cite{yu:supercuspidal}*{\S9} to discuss representability of characters.
	\end{enumerate}
\end{hyp}

\begin{rem}
\label{rem:mexp-in-hyps}
Together with Lemma \initref{lem:gp-dfc-unique}\subpref{unique} and \initref{prop:Lie-smooth-dfc-facts}\subpref{unique}, Hypothesis \initref{hyp:mexp}\incpref{elt}\subpref{fc} gives that \(\CC\bJ i(Y)\conn\) equals \(\CC\bJ i(\mexp(Y))\conn\) for all \(i \in \sbjtl\tR{R_{-1}}\) and all semisimple \(Y \in \sbjtl{\Lie(J)}{R_{-1}}\).

As mentioned in Remark \ref{rem:spice:asymptotic:hyp:mexp(elt)}, \xcite{spice:asymptotic}*{Hypothesis \initxref{hyp:mexp}(\subxref{elt})} is used only for semisimple \(Y\).  For such elements, our Hypothesis \initref{hyp:mexp}\subpref{elt} is stronger.

Indeed, for \(i\) sufficiently large, we have that \(\Cent_\bJ(Y)\conn\) equals \(\CC\bJ i(Y)\conn\) and \(\CC\bJ i(\mexp(Y))\conn\) equals \(\Cent_\bJ(\mexp(Y))\conn\), so that \(\Cent_\bJ(Y)\conn = \CC\bJ i(Y)\conn\) equals \(\CC\bJ i(\mexp(Y))\conn = \Cent_\bJ(\mexp(Y))\conn\).

Since \(\CC{\Lie(\bJ)}{R_{-1}}(Y)\conn\) equals \(\Lie(\bJ)\) by Lemma \ref{lem:MP-cfc}, we have as in \xcite{debacker-spice:stability}*{Remark \xref{rem:disc:roots}} that the valuation of \(\redD_\bJ(Y)\) is \(\sum_{i \in \sbjtl\R{R_{-1}}} i\dim(\CC{\Lie(\bJ)}i(Y)/\CCp{\Lie(\bJ)}i(Y))\), and similarly for \(\redD_\bJ(\mexp(Y))\), which establishes \xcite{spice:asymptotic}*{Hypothesis \initxref{hyp:mexp}(\incxref{elt})\subxref{disc}}.
\end{rem}

\subsection{Sufficient conditions}
\label{subsec:sufficient}

We discuss a sufficient condition under which both the hypotheses explicitly stated in this paper, and those inherited from the other papers whose results we use, are satisfied.

Write \(p\) for the characteristic of \(\sbjat\field 0\).

Currently, because of our dependence on \cite{jkim-murnaghan:charexp} and our own rather strong Hypothesis \ref{hyp:mexp}, the only sufficient conditions that we know that work for groups of all types begin with assuming that \field has characteristic \(0\), although we emphasise that this is not directly a hypothesis of our paper.  It seems almost certain that, as long as Hypothesis \ref{hyp:mexp} is satisfied---which it is for classical groups and general linear groups---one can drop the assumption on the characteristic of \field once  \(p\) is sufficiently large.

\begin{rem}
\label{rem:black-box}
We emphasise that, throughout this paper, we are treating the Kim--Murnaghan result \cite{jkim-murnaghan:charexp}*{Theorem 3.1.7} as a black box; as long as that theorem, a general result on asymptotic expansions, and the result \xcite{spice:asymptotic}*{Lemma \xref{lem:asymptotic-check}} that allows us to compute the coefficients in an expansion that is already known to exist, are available, we no longer care about the characteristic of \field.

Similarly, our main results, Theorems \ref{thm:orb-asymptotic-exists}, \ref{thm:orb-to-orb'}, \ref{thm:char-asymptotic-exists}, \ref{thm:pi-to-pi'}, \ref{thm:orb-unwind}, and \ref{thm:char-unwind}, will also assume this key Kim--Murnaghan result as a hypothesis.  The significance of stating the result this way, rather than citing \cite{jkim-murnaghan:charexp}*{Theorem 3.1.7(1, 5)} directly, is that the latter imposes some stringent hypotheses (see \cite{jkim-murnaghan:charexp}*{\S1.4}).  Again, since we treat the result as a black box, we are able to avoid directly imposing those hypotheses as long as we know that the concluson holds.
\end{rem}

We assume for the remainder of \S\ref{subsec:sufficient}, but \emph{not} in the rest of the paper, that
\begin{itemize}
\item Hypothesis \ref{hyp:mexp} is satisfied by the exponential map with \((\bM, \bJ, R_{-1}) = (\bG, \bG, \Rp0)\);
\item \(p\) does not divide the order of the 
component group of \bG;
and
\item \(p + 1\) is at least \(272\), and also at least \(6\) times the absolute rank of \bG.
\end{itemize}
Our strong version of Hypothesis \ref{hyp:mexp} implies \cite{jkim-murnaghan:charexp}*{\S1.4, (H\field)}.
Since \(p - 1\) is at least the absolute rank of \(\bG\), it follows that every torus in \bG has dimension strictly less than \(p\), hence is tame.

We first make a note about the formation of the subgroups \(\CC{\bG'}i(\gamma)\) where \(\bG'\) is a twisted Levi subgroup of \(\bG\),
\(\gamma\) is a semisimple automorphism of \(\bG'\), and
\(i\) belongs to \(\sbjtl\tR 0\).
Lemma \ref{lem:shallow-dfc} gives that there is a Levi subgroup \(\bM'\) of \(\bG\primeconn\) such that, for every \(i \in \sset{0, \Rp0}\), the group \(\CC{\bG'}i(\gamma)\conn\) is the identity component of the fixed-point subgroup of an automorphism of \(\bM'\) of prime-to-\(p\) order.
Since \(p\) is not a torsion prime for the absolute root system of \(\bG'\), Remark \ref{rem:concrete-dfc} allows us to run an argument similar to that of \cite{yu:supercuspidal}*{Proposition 7.3} to see that \(\CC{\CCp{\bG'}0(\gamma)\conn}i(\gamma)\) is a twisted Levi subgroup of \(\CCp{\bG'}0(\gamma)\conn\), and, in particular, is connected, for every \(i \in \sbjtlp\tR 0\).
Similar observations apply if \(\gamma\) is a semisimple element of \(\Lie(G')\), except that now we simply conclude that \(\CC{\bG\primeconn}i(\gamma)\) is a 
twisted Levi subgroup of \(\bG\primeconn\) for every \(i \in \tR\).

We note now that certain subgroups of \bG still have component groups of prime-to-\(p\) order.  First, a twisted Levi subgroup of \(\bG\conn\) is connected.  Second, by \cite{steinberg:endomorphisms}*{Theorems 7.5 and 9.1}, if \(\bM'\) is a twisted Levi subgroup of \(\bG\conn\), then the fixed-point group of an automorphism of \(\bM'\) of prime-to-\(p\) order has component group of order dividing the product of the order of the automorphism, and the order of the fundamental group of the absolute root system of \bM.  In particular, this order is \emph{not} divisible by \(p\).
Thus, whenever \(\bG'\) is a twisted Levi subgroup of \bG, \(\gamma\) is a semisimple automorphism of \(\bG'\) (respectively, a semisimple element of \(\Lie(G')\)), and \(i\) belongs to \(\sbjtl\tR 0\) (respectively, to \tR), we have that \(p\) does not divide the order of the component group of \(\CC{\bG'}i(\gamma)\).

The somewhat strange condition on \(p + 1\) that it be at least \(272\) and \(6\) times the absolute rank of \bG is an indirect version of the condition we really want that ensures a similar heredity.  Namely, it is shown in
\cite{waldspurger:endoscopie}*{Lemme 1.1}
that the bound \(p > 3(h - 1)\) is satisfied with \(h\) the Coxeter number not just of the absolute root system of \bG, but of every root system whose rank does not exceed that of the absolute root system of \bG.  This is the condition that we really want.
Note that it imples in turn that \(p\) is greater than \(h\), hence relatively prime to the order of the absolute Weyl group of any connected, reductive subgroup of \bG (say, by \cite{humphreys:reflection}*{Theorem 3.19}).

These two heredity statements together mean that any consequences of our assumptions that we deduce for \bG apply equally well to any subgroup of \bG of the form \(\CC{\bG'}i(\gamma)\), where \(\bG'\) is a twisted Levi subgroup of \bG, \(\gamma\) is a semisimple element of \(G'\) (respectively, \(\Lie(G')\)), and \(i\) belongs to \(\sbjtl\tR 0\) (respectively, \tR).

Since \(p\) does not divide the order of the absolute Weyl group of \(\bG\conn\), we have
by \cite{fintzen:tame-tori}*{Theorem 3.3} that \cite{jkim-murnaghan:charexp}*{\S1.4, (HGT)} is satisfied, and by \cite{fintzen:tame-tori}*{Theorem 3.6} that every maximal torus in \bG satisfies \xcite{adler-spice:good-expansions}*{Definition \xref{defn:S-is-good}}.  It is shown in \xcite{adler-spice:good-expansions}*{Remark \xref{rem:when-hyps-hold}} (also using \cite{fintzen:tame-tori}*{Lemma 3.2(1)}) that \xcite{adler-spice:good-expansions}*{Hypotheses \xref{hyp:reduced}--\xref{hyp:torus-H1-triv}} are satisfied.  As observed in \cite{adler-roche:intertwining}*{\S1, p.~452}, the constant \(k(\bG)\cdot\card{\pi_1(\Der\bG)}\) appearing in \cite{adler-roche:intertwining}*{Proposition 4.1} divides the order of the centre of the simply connected cover of \(\Der\bG\), so that \cite{fintzen:tame-tori}*{Lemma 3.2(2)} implies that the hypotheses of \cite{adler-roche:intertwining}*{Proposition 4.1(2)} are satisfied (unless \(\bG\conn\) is a torus and \(p\) euals \(2\), in which case the conclusion of \cite{adler-roche:intertwining}*{Proposition 4.1} still holds).  That is, \cite{jkim-murnaghan:charexp}*{\S1.4, (HB)} is satisfied.

If \bG were connected, then we would already have shown that every semisimple element of \(G\) was tame (because each such lies in a tame torus).  To handle the disconnected case, recall that the component group of \(\Cent_\bG(\gamma)\) has prime-to-\(p\) order (because it equals \(\CC\bG i(\gamma)\) for \(i\) sufficiently large), so that, in particular, some prime-to-\(p\) power of \(\gamma\) lies in a maximal torus \bS in \(\Cent_\bG(\gamma)\).  In particular, \bS is centralised by \(\gamma\).  We have already observed that any torus in \bG, including \bS, is tame; and Lemma \initref{lem:gamma-large}(\subref{bijection}, \subref{quass}) gives that there is a \(\gamma\)-stable Borel subgroup \(\bB^+\) of \(\bG_\sepfield\) containing \(\Cent_\bG(\bS)_\sepfield\).  That is, \(\gamma\) is tame.

Showing the hypothesis \cite{jkim-murnaghan:charexp}*{\S1.4, (H\(\mathcal N\))} is more involved, since it is the conjunction of \cite{debacker:nilp}*{Hypotheses 4.2.1, 3, 4, 5, 7}.

First, we have by \cite{adler:thesis}*{\S\S1.3 and 1.5, and Proposition 1.6.3} that \cite{debacker:nilp}*{Hypothesis 4.2.7} is satisfied.

For the rest, as observed in \cite{debacker:nilp}*{\S4.2}, many of these hypotheses are implied by the inequality \(p > 3(h - 1)\), where \(h\) is the Coxeter number of the absolute root system of \bG.
Namely, this bound implies the inequality \(p > 2h\), as discussed on \cite{carter:finite}*{\S5.5, p.~154}.  Then \cite{carter:finite}*{Proposition 5.5.2} gives \cite{debacker:nilp}*{Hypothesis 4.2.3}.  Next, since the exponential map \(\exp\) converges on \(\sbjtlp{\Lie(G)}0\) (by Hypothesis \ref{hyp:mexp}), it converges in particular at all nilpotent elements \cite{adler-debacker:bt-lie}*{Lemmas 3.3.2 and 4.1.5}.  It may be taken to be the map denoted by \(\exp_t\) in \cite{debacker:nilp}*{Hypothesis 4.2.4}.  Since \(\exp\) is an injection, and carries \(G\)-conjugacy classes to \(G\)-conjugacy classes, it remains an injection from the set of nilpotent \(G\)-conjugacy classes to the set of unipotent \(G\)-conjugacy classes.  Since these two sets have the same finite cardinality \cite{carter:finite}*{\S1.15}, we have that \(\exp\) actually induces a bijection from the nilpotent to the unipotent conjugacy classes, and hence is itself a surjection onto the set of unipotent elements.  Since finally, as we have already observed, \(\exp\) is an injection, we have that it is a bijection from the set of nilpotent to the set of unipotent elements.  That is, \cite{debacker:nilp}*{Hypothesis 4.2.4} is satisfied.

We can combine \cite{carter:finite}*{Proposition 5.3.2, Proposition 5.5.2, Lemma 5.5.4(iii), the discussion on pp.~154--155, and Proposition 5.5.10} to see that \cite{debacker:nilp}*{Hypothesis 5.2.5} is satisfied.  Finally, \cite{waldspurger:endoscopie}*{Lemme 7.2.1} establishes \cite{debacker:nilp}*{Hypothesis 4.2.1}.

The results of \cite{spice:asymptotic}, and hence the hypotheses it imposes, are needed throughout
the paper, so we show that they, too, follow from the conditions that we have imposed.

As discussed in \S\ref{subsec:hyp-statements}, we have proved most of \xcite{spice:asymptotic}*{Hypotheses \xref{hyp:funny-centraliser} and \xref{hyp:gamma}}, and isolated what remains as Hypotheses \ref{hyp:fc-building} and \ref{hyp:gp-gamma}, which are now satisfied for all semisimple elements \(\gamma \in G\).

We have by \xcite{spice:asymptotic}*{Remark \xref{rem:X*}} that \loccit*{Hypothesis \xref{hyp:X*}} is always satisfied for elements near an element satisfying \xcite{spice:asymptotic}*{Hypothesis \xref{hyp:Z*}}.  When \bG is connected, we have by \cite{yu:supercuspidal}*{Lemmas 8.1 and 8.3} that \xcite{spice:asymptotic}*{Hypothesis \xref{hyp:Z*}} follows from \cite{yu:supercuspidal}*{\S8, \textbf{GE1}}.  Combining the identification of \(\Lie(G)\) with \(\Lie^*(G)\) \textit{via} \cite{jkim-murnaghan:charexp}*{\S1.4, (HB)}, and \cite{fintzen:tame-tori}*{Theorem 3.3}, shows that every Moy--Prasad coset in \(\Lie^*(T)\), where \bT is a maximal torus in \bG, contains an element satisfying \cite{yu:supercuspidal}*{\S8, \textbf{GE1}}.

Our Hypothesis \ref{hyp:mexp} replaces \xcite{spice:asymptotic}*{Hypothesis \xref{hyp:mexp}}.

The condition \xcite{spice:asymptotic}*{Hypothesis \xref{hyp:depth}} concerns a reductive subgroup \bH of \bG, and involves depths of elements in \bH and in \bG, and an analogue for depths of characters.
Our full-strength Hypothesis \ref{hyp:mexp} (with \((\bM, \bJ, R_{-1}) = (\bG, \bG, \Rp)\)) shows that the statement \loccit*{Hypothesis \initxref{hyp:depth}(\subxref{char})} for characters follows from the statement \loccit*{Hypothesis \initxref{hyp:depth}(\subxref{coset})} for elements, so it suffices to prove the latter.  As discussed in \loccit*{\S\xref{sec:dual-blob}, p.~2338}, the part \loccit*{Hypothesis \initxref{hyp:depth}(\subxref{coset})} about depths of elements holds when \bH is a tame, twisted Levi subgroup.  More generally, by combining this case with \cite{kaletha-prasad:bt-theory}*{Theorem 12.7.1}, we see that it holds when \bH is a tame, twisted Levi subgroup of the group of fixed points of an element of \bG that has prime-to-\(p\) order modulo the centre.  Since every \(\CC\bG r(\gamma)\) arises as a tame, twisted Levi subgroup of the fixed-point group in a Levi subgroup, we have that \xcite{spice:asymptotic}*{Hypothesis \initxref{hyp:depth}(\subxref{coset})} is satisfied.

The condition \xcite{spice:asymptotic}*{Hypothesis \xref{hyp:gamma-central}} may be viewed as a restriction on the range in which the main character formula \loccit*{Theorem \xref{thm:asymptotic-pi-to-pi'}} holds.  It essentially says that we need to restrict ourselves to elements near a `most singular' element in a tame torus.  Lemma \initref{lem:shallow-dfc} and \cite{fintzen:tame-tori}*{Proposition 2.1 and Theorem 3.6} show that this is true for arbitrary semisimple elements \(\gamma\).

The condition \xcite{spice:asymptotic}*{Hypothesis \xref{hyp:MP-ad}} involves a semisimple element \(\gamma\), and discusses the existence of a Moy--Prasad map satisfying certain properties.  As discussed before Hypothesis \ref{hyp:MP-ad}, since \(\gamma\) is tame, we can prove \xcite{spice:asymptotic}*{Hypothesis \xref{hyp:MP-ad}} using the argument of \cite{kaletha-prasad:bt-theory}*{Theorem 13.5.1}.
Further, as discussed in \xcite{spice:asymptotic}*{\S\xref{sec:asymptotic}, p.~2343}, the condition \loccit*{Hypothesis \xref{hyp:K-type}} is mainly meant to work around the absence of a suitable Moy--Prasad isomorphism, and, if such an isomorphism exists---for example, if \(\bG_\tamefield\) is split---then the hypothesis is satisfied by a character with an appropriate dual blob.  The same applies to \loccit*{Hypothesis \xref{hyp:phi}}.

\numberwithin{thm}{section}
\section{Compact, open subgroups}
\label{sec:subgps}

Throughout \S\ref{sec:subgps}, we continue to use the
	non-Archimedean local field \field
and	reductive group \bG
from \S\ref{sec:hyps}.
Let
	\matnotn{gamma}\gamma be a semisimple element of \(G\) satisfying Hypothesis \ref{hyp:fc-building},
	\mnotn r a positive real number,
and	\mnotn x a point of \(\BB(\CC G r(\gamma))\),
such that \((\gamma, x, r)\) satisfies Hypothesis \ref{hyp:gp-gamma}.
All the results in this section have natural analogues for \(\gamma\) a semisimple element of \(\Lie(G)\), with Hypothesis \ref{hyp:Lie-gamma} in place of Hypothesis \ref{hyp:gp-gamma} and with the requirement that \(r\) be positive removed, whose proofs are similar but easier; but we prefer to focus on the case \(\gamma \in G\) in order to highlight the difficulties and how we address them.  To be specific, the key difference is that, for \(\gamma\) a semisimple element of \(G\), we must treat differently the compact and non-compact eigenvalues of \(\gamma\); whereas, for \(\gamma\) a semisimple element of \(\Lie(G)\), all eigenvalues can be treated uniformly.  We indicate other differences as they arise.

Put \(\matnotn M\bM = \CC\bG 0(\gamma)\) and \(\matnotn H\bH = \CC\bG r(\gamma)\).
Let \(m_\gamma \in \cclat(\sgen\gamma) \otimes_\Z \Q\) be the ``rationalised cocharacter'' of Notation \ref{notn:ord-gamma}.
Recall from Lemma \initref{lem:shallow-dfc}\subpref0 that \bM equals \(\Cent_\bG(m_\gamma)\), and so is a Levi subgroup of \bG.
Let
\justnotn[U]{\bPrad^+}\justnotn[U]{\bPrad^-}\(\bPrad^\pm\)
be the subgroups \(\operatorname U_\bG(\pm m_\gamma)\) strictly contracted (respectively, dilated) by \(m_\gamma\), so that every weight \(a\) of the split part of \(\Zent(\bM\conn)\) on \(\Lie(\bPrad^+)\) (respectively, \(\Lie(\bPrad^-)\)) satisfies \(\pair a{m_\gamma} > 0\) (respectively, \(\pair a{m_\gamma} < 0\)).
(If \(\gamma\) is a semisimple element of \(\Lie(G)\), then we put \(\bM = \bG\) and \(m_\gamma = 0\), and let \(\bPrad^\pm\) be trivial.)

In \xcite{spice:asymptotic}*{\S\xref{sec:subgps}}, certain compact, open subgroups of \(G\) were constructed using subgroups \(\CC\bG i(\gamma\pinv)\conn\) defined for all \(i \in \tR \cup \sset{-\infty}\).  Unfortunately, such subgroups rarely exist for negative \(i\); see Remark \ref{rem:hyp:funny-centraliser}.  For \(i\) in \(\sbjtl\tR 0\), the subgroups \(\CC\bG i(\gamma\pinv)\) of \xcite{spice:asymptotic}*{Hypothesis \xref{hyp:funny-centraliser}} equal \(\CC\bG i(\gamma)\), and coincide with the subgroups \(\CC\bG i(\gamma)\conn\) defined in Definition \ref{defn:gp-dfc}.

Because the subgroups \(\CC\bG i(\gamma\pinv)\conn\) for negative \(i\) are no longer available, we can no longer use \xcite{spice:asymptotic}*{Definition \xref{defn:vGvr}} directly as written, so we need a substitute, Definition \ref{defn:vGvr}.  Since this error in \xcite{spice:asymptotic}*{\S\xref{sec:depth-matrix}} does not affect the definitions or results of \xcite{spice:asymptotic}*{\S\xref{sec:concave}}, or \xcite{spice:asymptotic}*{Definition \xref{defn:tame-Levi}} (although see Remark \ref{rem:defn:tame-Levi}), we still may, and do, use those freely.
In particular, if \bT is a maximal torus that contains a maximal \tamefield-split torus in \bG and satisfies \(x \in \BB(T)\), and if \map f{\Weight(\bG_\sepfield, \bT_\sepfield)}{\tR \cup \sset{-\infty}} is any function, then \xcite{adler-spice:good-expansions}*{Definition \xref{defn:vGvr}} defines sublattices \(\lsub\bT{\sbtl{\Lie(G)}x f}\) and \(\lsub\bT{\sbtl{\Lie^*(G)}x f}\) of \(\Lie(G)\) and \(\Lie^*(G)\); and, if \(f\) is concave, in the sense of \cite{bruhat-tits:reductive-groups-1}*{\S6.4.3}, a subgroup \(\lsub\bT{\sbtl G x f}\) of \(G\).

Let \(\vbG = (\bG^0 \subseteq \dotsb \subseteq \bG^{\ell - 1} \subseteq \bG^\ell = \bG)\) be a nested, tame, twisted Levi sequence in \bG such that \(\gamma\) belongs to \(G^0\) and \(x\) belongs to \(\BB(G^0)\).

For each index \(i \in \tR\) with \(0 \le i \le r\), Proposition \initref{prop:gp-cfc-facts}\subpref{Levi} gives that \(\matnotn{Cent}{\CC\vbG i(\gamma)} \ldef (\CC{\bG^0}i(\gamma) \subseteq \dotsb \subseteq \CC{\bG^{\ell - 1}}i(\gamma) \subseteq \CC\bG i(\gamma))\) is a \(\gamma\)-stable, nested, tame, twisted Levi sequence in \(\CC\bG i(\gamma)\).
Recall from \S\ref{sec:hyps}, p.~\ref{vGvr-defn}, the definition of
sublattices \(\sbtl{\Lie(\CC\vG i(\gamma))}x{\vec a}\) and
subgroups \(\sbtl{\CC\vG i(\gamma)}x{\vec a}\)
associated to appropriate depth vectors \(\vec a\).

For each index \(j \in \sset{0, \dotsc, \ell}\), put \(\bM^j = \bM \cap \bG^j\) and \(\bH^j = \bH \cap \bG^j\).  (Note that \(\bM^0 = \bM \cap \bG^0\) and \(\bH^0 = \bH \cap \bG^0\) will rarely be the same as the identity components \(\bM\conn\) and \(\bH\conn\) of \bM and \bH.)  Then, in particular, \(\vbM \ldef (\bM^0 \subseteq \dotsb \subseteq \bM^{\ell - 1} \subseteq \bM^\ell = \bM) = \CC\vbG 0(\gamma)\) and \(\vbH \ldef (\bH^0 \subseteq \dotsb \subseteq \bH^{\ell - 1} \subseteq \bH^\ell = \bH) = \CC\vbG r(\gamma)\) are nested, tame, twisted Levi sequences in \bM and \bH.

\begin{defn}
\label{defn:depth-spec}
With the notation of Appendix \ref{app:depth}, put \(S = \tR \cup \sset{-\infty}\), \(\mc X_0 = \sset{0, \dotsc, \ell}\), \(\mc X_1 = \set{i \in \R}{0 \le i \le r}\), and \(\mc X_2 = \R\).
(If \(\gamma\) is a semisimple element of \(\Lie(G)\), then we put \(\mc X_1 = \set{i \in \R}{i \le r}\).)
With these choices, we call a triply indexed vector \(\tau \in S^{\mc X_2 \times \mc X_1 \times \mc X_0}\) a \term[depth specification|see{vector, triply indexed}]{depth specification}.
We say that the depth specification \(\tau = (t_{v i j})_{v, i, j}\) is \term[depth specification!grouplike]{grouplike} if it is monotone and concave, and \(t_{0 r 0}\) is positive.
\end{defn}

Definition \ref{defn:vMvr} specialises to \xcite{spice:asymptotic}*{Definition \xref{defn:vGvr}} when \bG equals \(\bM = \CC\bG 0(\gamma)\).
We describe in Definition \ref{defn:vGvr} how to handle the absence of groups \(\CC\bG i(\gamma\pinv)\) when \(i\) is negative.

\begin{defn}
\label{defn:vMvr}
Let \(\tau = (t_{v i j})_{v, i, j}\) be a depth specification.

We define \(\lsub\gamma{\sbtl{\Lie(\vM)}x\tau}\) to be
\[
\bigoplus_{0 \le i < r} \bigl(
	\Lie(\CCp G i(\gamma))^\perp \cap \sbtl{\Lie(\CC\vG i(\gamma))}x{\vec t_{0 i}}
\bigr) \oplus \sbtl{\Lie(\CC G r(\gamma))}x{\vec t_{0 r}},
\]
where, for each \(0 \le i \le r\), the notation \(\Lie(\CCp G i(\gamma))^\perp\) is as in Definition \ref{defn:embed-dual} (using Remark \ref{rem:embed-dual} to see that the resulting sublattice does not depend on the choice of maximal torus in \(\CCp\bG i(\gamma)\)); we write \(\vec t_{0 i}\) for the vector \((t_{0 i 0}, \dotsc, t_{0 i \ell})\); and the notation \(\sbtl{\Lie(\CC\vG i(\gamma))}x{\vec t_{0 i}}\) is as in \xcite{spice:asymptotic}*{Definition \xref{defn:tame-Levi}}.  Define \(\lsub\gamma{\sbtl{\Lie^*(\vM)}x\tau}\) similarly.

If \(\tau\) is grouplike, then each \(\vec t_{0i}\) is grouplike, in the sense of \xcite{spice:asymptotic}*{Definition \xref{defn:tame-Levi}} (see Remark \ref{rem:depth-ops}),
and we write \(\lsub\gamma{\sbtl\vM x\tau}\) for the group generated by the various \(\sbtl{\CC\vG i(\gamma)}x{\vec t_{0i}}\) with \(i \in \R\) satisfying \(0 \le i \le r\).

(If \(\gamma\) is a semisimple element of \(\Lie(G)\), then we drop the requirement \(i \ge 0\).)
\end{defn}

Choose a maximal split torus \bS in \(\bH^0\) such that \(x\) belongs to the apartment of \bS, and a maximally unramified maximal torus \matnotn T\bT containing \bS.
(That is, we require that \(\bT_\unfield\) contain, not just \(\bS_\unfield\), but also a maximal split torus in \(\bH^0_\unfield\).  Since \(\bH^0_\unfield\) is quasi-split, this implies that \(\bT_E\) contains a maximal split torus in \(\bH^0_E\) for every extension \(E/\unfield\).)
Since \(\gamma\) belongs to \(H^0\), the rationalised cocharacter \(m_\gamma \in \cclat(\sgen\gamma) \otimes_\Z \Q\) belongs to \(\cclat(\bH^0) \otimes_\Z \Q\), so
there is some \(h^0 \in H\thenconn 0\) such that \(m_\gamma\) belongs to \(\cclat(\Int(h^0)\inv\bS) \otimes_\Z \Q\).

Definition \ref{defn:vGvr} now defines \(\sbtl\vG x\tau\) in general, without assuming that \bG equals \(\CC\bG 0(\gamma)\).
(If \(\gamma\) is a semisimple element of \(\Lie(G)\), then \bM equals \bG, and we have already defined \(\sbtl\vM x\tau = \sbtl\vG x\tau\).)

\begin{defn}
\label{defn:vGvr}
Preserve the notation of Definition \ref{defn:vMvr}.
Then we define the function \(f_\tau\) as in Definition \ref{defn:depth-to-concave}, using \(\lambda = \Int(h^0) \circ m_\gamma\) and \(\Root^j = \Root(\bG^j_\sepfield, \bT_\sepfield)\) for each \(j \in \sset{0, \dotsc, \ell}\).
Put
\[
\matnotn[\sbtl{\Lie(\vG)}x\tau]{Gxt}{\lsub\gamma{\sbtl{\Lie(\vG)}x\tau}} = \lsub\gamma{\sbtl{\Lie(\vM)}x\tau} \oplus \bigl(\lsub\bT{\sbtl{\Lie(G)}x{f_\tau}} \cap \Lie(M)^\perp\bigr),
\]
and similarly for \matnotn[\sbtl{\Lie^*(\vG)}x\tau]{Gxt}{\lsub\gamma{\sbtl{\Lie^*(\vG)}x\tau}}.

If \(\tau\) is grouplike, then Lemma \initref{lem:depth-vs-concave-vee}\subpref{concave} gives that \(f_\tau\) is concave, and we write \matnotn[\sbtl\vG x\tau]{Gxt}{\lsub\gamma\sbtl\vG x\tau} for the group generated by \(\lsub\gamma\sbtl\vM x\tau\) and \(\lsub\bT{\sbtl G x{f_\tau}}\).

We will usually omit the pre-subscript \(\gamma\) when it is understood.
\end{defn}

Lemma \ref{lem:filtration} is our version of \xcite{spice:asymptotic}*{Lemma \xref{lem:filtration}}, adjusted to take account of Remark \ref{rem:hyp:funny-centraliser}.
The operation \(\matnotn{concave}\cvvv\) is as in Definition \ref{defn:depth-ops}.

\begin{lem}
\initlabel{lem:filtration}
Let \(\tau = (t_{v i j})_{v, i, j}\) and \(\tau' = (t'_{v i j})_{v, i, j}\) be depth specifications.  Then the following properties hold.
	\begin{enumerate}
	\item\sublabel{Lie-Lie} If \(Y\) belongs to \(\sbtl{\Lie(\vG)}x\tau\) and \(Y'\) belongs to \(\sbtl{\Lie(\vG)}x{\tau'}\), then \(\comm Y{Y'}\) belongs to \(\sbtl{\Lie(\Der\vG)}x{\tau \cvvv \tau'}\).
	\item\sublabel{gp-Lie} If, in addition to \locpref{Lie-Lie}, we have that \(\tau\) is grouplike and \(g\) belongs to \(\sbtl\vG x\tau\), then \((\Ad(g) - 1)Y'\) belongs to \(\sbtl{\Lie(\Der\vG)}x{\tau \cvvv \tau'}\); and analogously on the dual Lie algebra.
	\item\sublabel{gp-gp} If, in addition to \locpref{gp-Lie}, we have that \(\tau'\) is grouplike and \(g'\) belongs to \(\sbtl\vG x{\tau'}\), then \(\comm g{g'}\) belongs to \(\sbtl{\Der\vG}x{\tau \cvvv \tau'}\).
	\end{enumerate}
\end{lem}

\begin{proof}
We give the argument for \locpref{gp-gp}; the others are similar, but easier.

Since \(\tau \cvvv (\tau \cvvv \tau')\) equals \((\tau \cvvv \tau) \cvvv \tau'\) by Lemma \ref{lem:vee-assoc} and therefore majorises \(\tau \cvvv \tau'\) (because \(\tau\) is grouplike), it suffices, as in the proof of \xcite{spice:asymptotic}*{Lemma \xref{lem:der-master-comm}}, to prove the result as \(g\) and \(g'\) range over sets of generators for \(\sbtl\vG x\tau\) and \(\sbtl\vG x{\tau'}\).

There are three cases.
\begin{enumerate}
\item There are indices \(i, i' \in \R\) with \(0 \le i, i' \le r\) such that \(g\) belongs to \(\sbtl{\CC\vbG i(\gamma)}x{\vec t_{0 i}}\) and \(g'\) belongs to \(\sbtl{\CC\vbG{i'}(\gamma)}x{\vec t^{\,\prime}_{0 i'}}\).  We may assume, by swapping \(g\) and \(g'\) if necessary, that \(i\) is greater than or equal to \(i'\).  Then \(\CC\bG{i'}(\gamma)\) contains \(\CC\bG i(\gamma)\) and \(\vec t_{0 i}\) is majorised by \(\vec t_{0 i'}\), by monotonicity; and the singly indexed component of \(\tau \cvvv \tau'\) at \((0, i')\) is majorised by \(\vec t_{0 i'} \cvvv \vec t^{\,\prime}_{0 i'}\), by Definition \ref{defn:depth-ops}; so \xcite{spice:asymptotic}*{Lemma \xref{lem:der-master-comm}} gives that \(\comm g{g'}\) belongs to \(\sbtl{\CC{\Der\vG}{i'}(\gamma)}x{\vec t_{0 i'} \cvvv_2 \vec t^{\,\prime}_{0 i'}} \subseteq \sbtl{\Der\vG}x{\tau \cvvv \tau'}\).
\item\sublabel{outside-M}
\(g\) belongs to \(\lsub\bT{\sbtl G x{f_\tau}}\) and \(g'\) belongs to \(\lsub\bT{\sbtl G x{f_{\tau'}}}\).  We have by \xcite{spice:asymptotic}*{Lemma \xref{lem:der-master-comm}} that \(\comm g{g'}\) belongs to \(\lsub\bT{\sbtl{\Der G}x{f_\tau \vee f_{\tau'}}}\), and by Lemma \ref{lem:depth-vs-concave-vee} that \(\lsub\bT{\sbtl{\Der G}x{f_\tau \vee f_{\tau'}}}\) is contained in \(\lsub\bT{\sbtl{\Der G}x{f_{\tau \cvvv \tau'}}} \subseteq \sbtl{\Der\vG}x{\tau \cvvv \tau'}\).
\item After swapping \(g\) and \(g'\) if necessary, there is an index \(i \in \R\) with \(0 \le i \le r\) such that \(g\) belongs to \(\sbtl{\CC\vG i(\gamma)}x{\vec t_{0 i}}\) and \(g'\) belongs to \(\lsub\bT{\sbtl G x{f_{\tau'}}}\).  Then \(\CC\bG 0(\gamma)\) contains \(\CC\bG i(\gamma)\) and \(\vec t_{0 i}\) is majorised by \(\vec t_{0 0}\), by monotonicity, so \(g\) belongs to \(\lsub\bT{\sbtl G x{f_\tau}}\).  We have now reduced to case \locpref{outside-M}.\qedhere
\end{enumerate}
\end{proof}

\begin{rem}
\label{rem:filtration}
Suppose that \(\tau = (t_{v i j})_{v, i, j}\) is grouplike.
Put \justnotn[Uxtau]{\sbtl{\vPrad^+}x\tau}\justnotn[Uxtau]{\sbtl{\vPrad^-}x\tau}\(\sbtl{\vPrad^\pm}x\tau = \lsub\bT{\sbtl{\Prad^\pm}x{f_\tau}}\).

If \(\tau' = (t'_{v i j})_{v, i, j}\)
is the depth specification defined by \(t'_{v i j} = 
t_{v i j}\) if \(v \ne 0\), and \(t'_{0 i j} = t_{0 r j}\),
for all \(i \in \R\) with \(0 \le i \le r\) and all \(j \in \sset{0, \dotsc, \ell}\), then \(\tau'\) is grouplike and satisfies \(\tau' \le \tau
\).

We have that \(\lsub\bT{\sbtl{\Prad^\pm}x{f_\tau}} = \sbtl{\vPrad^\pm}x\tau\) equals \(\lsub\bT{\sbtl{\Prad^+}x{f_{\tau'}}} = \sbtl{\vPrad^+}x{\tau'}\); \(\lsub\bT{\sbtl M x{f_\tau}}\) equals \(\sbtl\vM x{\vec t_{0 0}} \subseteq \sbtl\vM x\tau\); and \(\sbtl\vM x\tau\) is contained in \(\sbtl\vM x{\vec t_{0 r}} = \lsub\bT{\sbtl M x{f_{\tau'}}}\).
In particular,
\(\lsub\bT{\sbtl G x{f_{\tau'}}}\) contains \(\sbtl\vM x\tau\) and \(\lsub\bT{\sbtl G x{f_\tau}}\), and therefore contains \(\sbtl\vG x\tau\);
and
\(\sbtl\vG x\tau\) is generated by \(\sbtl{\vPrad^-}x\tau\), \(\sbtl\vM x\tau\), and \(\sbtl{\vPrad^+}x\tau\).  See Lemma \ref{lem:gen-Iwahori-factorisation} for a sharpened version of this last statement.
\end{rem}

\begin{rem}
\label{rem:fc-Iwahori}
Suppose that \(\tau\) is grouplike.  Then Lemma \ref{lem:filtration} shows that \(\sbtl\vM x\tau\) is the set product \(\prod_{0 \le i \le r} \sbtl{\CC\vG i(\gamma)}x{\vec t_{0 i}}\).  In particular, for every \(i \in \R\) with \(0 \le i \le r\), we have that \(\sbtl\vM x\tau \cap \CC G i(\gamma)\conn\) is contained in
\begin{multline*}
\Bigl(\prod_{0 \le i' \le i} \sbtl{\CC\vG{i'}(\gamma)}x{\vec t_{0 i'}} \cap \CC G i(\gamma)\conn\Bigr)\dotm\prod_{0 \le i < i' \le r} \sbtl{\CC\vG{i'}(\gamma)}x{\vec t_{0 i'}} \\
\subseteq \bigl(\sbtl\vG x{\vec t_{0 i}} \cap \CC G i(\gamma)\conn\bigr)\dotm\prod_{0 \le i < i' \le r} \sbtl{\CC\vG{i'}(\gamma)}x{\vec t_{0 i'}},
\end{multline*}
which, by Remark \initref{rem:fc-building}\subpref{up-to-G}, equals
\[
\sbtl{\CC\vG i(\gamma)}x{\vec t_{0 i}}\dotm\prod_{0 \le i < i' \le r} \sbtl{\CC\vG{i'}(\gamma)}x{\vec t_{0 i'}} = \sbtl{\CC\vG i(\gamma)}x\tau.
\]
The reverse containment is clear, so we have equality.

Working with eigenspace decompositions relative to \(\gamma\) shows that, similarly, \(\sbtl{\Lie(\vG)}x\tau \cap \Lie(\CC G i(\gamma))\) equals \(\sbtl{\Lie(\CC\vG i(\gamma))}x\tau\) and \(\sbtl{\Lie^*(\vG)}x\tau \cap \Lie^*(\CC G i(\gamma))\) equals \(\sbtl{\Lie^*(\CC\vG i(\gamma))}x\tau\), whether or not \(\tau\) is grouplike.

(If \(\gamma\) is a semisimple element of \(\Lie(G)\), then we drop the requirements \(i \ge 0\) and \(i' \ge 0\).)
\end{rem}

\begin{lem}
\label{lem:gen-Iwahori-factorisation}
If \(\tau\) is grouplike, then the multiplication map \anonmap{\sbtl{\vPrad^+}x\tau \times \sbtl\vM x\tau \times \sbtl{\vPrad^-}x\tau}{\sbtl\vG x\tau} is a bijection.
\end{lem}

\begin{proof}
Since the multiplication map \anonmap{\Prad^+ \times M \times \Prad^-}G is a bijection onto its image, it suffices to show that the image of \(\sbtl{\vPrad^+}x\tau \times \sbtl\vM x\tau \times \sbtl{\vPrad^-}x\tau\) is \(\sbtl\vG x\tau\).

By Remark \ref{rem:filtration}, it suffices to show that the image of \(\sbtl{\vPrad^+}x\tau \times \sbtl\vM x\tau \times \sbtl{\vPrad^-}x\tau\) is a group, for which it suffices to show that it is closed under multiplication.

Let \(\tau'\) be as in Remark \ref{rem:filtration}.
We have by \xcite{spice:asymptotic}*{Lemma \xref{lem:der-master-comm}} and Lemma \ref{lem:depth-vs-concave-vee} that the commutator of \(\sbtl{\vPrad^\pm}x\tau = \lsub\bT{\sbtl{\Prad^\pm}x{f_{\tau'}}}\) with \(\sbtl\vM x\tau \subseteq \lsub\bT{\sbtl M x{f_{\tau'}}}\) lies in \(\lsub\bT{\sbtl{\Prad^\pm}x{f_{\tau'}}} = \sbtl{\vPrad^\pm}x\tau\).  Therefore,
\[
\sbtl{\vPrad^+}x\tau\dotm\sbtl\vM x\tau\dotm\sbtl{\vPrad^-}x\tau\dotm\sbtl{\vPrad^+}x\tau\dotm\sbtl\vM x\tau\dotm\sbtl{\vPrad^-}x\tau
\]
equals
\[
\sbtl{\vPrad^+}x\tau\dotm\sbtl{\vPrad^-}x\tau\dotm\sbtl{\vPrad^+}x\tau\dotm\sbtl\vM x\tau\dotm\sbtl{\vPrad^-}x\tau
\]
and so is contained in
\[
\lsub\bT{\sbtl G x{f_\tau}}\dotm\sbtl\vM x\tau\dotm\sbtl{\vPrad^-}x\tau,
\]
which, by \xcite{adler-spice:good-expansions}*{Lemma \xref{lem:gen-iwahori-factorization}}, equals
\begin{multline*}
\sbtl{\vPrad^+}x\tau\dotm\sbtl\vM x{\vec t_{0 0}}\dotm\sbtl{\vPrad^-}x\tau\dotm\sbtl\vM x\tau\dotm\sbtl{\vPrad^-}x\tau \\
= \sbtl{\vPrad^+}x\tau\dotm\sbtl\vM x{\vec t_{0 0}}\dotm\sbtl\vM x\tau\dotm\sbtl{\vPrad^-}x\tau
= \sbtl{\vPrad^+}x\tau\dotm\sbtl\vM x\tau\dotm\sbtl{\vPrad^-}x\tau.
\end{multline*}
This shows that \(\sbtl{\vPrad^+}x\tau\dotm\sbtl\vM x\tau\dotm\sbtl{\vPrad^-}x\tau\) is closed under multiplication, and so completes the proof.
\end{proof}

\begin{rem}
\label{rem:vGvr-Mperp}
Preserve the notation of Definition \ref{defn:vGvr}.


The doubly indexed vector \(T_v = (t_{v i j})_{i, j}\) does not affect \(\lsub\gamma\sbtl{{\Lie(\vG)}}x\tau\), \(\lsub\gamma\sbtl{{\Lie^*(\vG)}}x\tau\), or \(\lsub\gamma\sbtl{\vG}x\tau\) unless \(v\) is the valuation of some weight of \(\gamma\) on \(\Lie(\bG)\).  If \(\tau\) is grouplike and we define \(\tilde t_{v i j} = t_{v i j}\) or \(\tilde t_{v i j} = \infty\) according as \(v\) is, or is not, the valuation of a weight of \(\gamma\) on \(\Lie(\bG)\), then \((\tilde t_{v i j})_{v, i, j}\) is still grouplike.
(If \(\gamma\) is a semisimple element of \(\Lie(G)\), then the doubly indexed vector \(T_v\) does not affect \(\lsub\gamma\sbtl{{\Lie(\vG)}}x\tau\) unless \(v\) equals \(0\).  Therefore, we may carry out the analogous construction, now setting \(\tilde t_{v i j} = \infty\) for all non-\(0\) \(v\).)

If \(v \in \R\) is non-\(0\), then only \(\sup \set{\vec t_{v i}}{i \in \mc X_1}\), not any individual \(\vec t_{v i}\), affects the definitions of \(\lsub\gamma\sbtl{{\Lie(\vG)}}x\tau\), \(\lsub\gamma\sbtl{{\Lie^*(\vG)}}x\tau\), and \(\lsub\gamma\sbtl{\vG}x\tau\).  In this sense, our depth specifications carry too much information, and could be pruned down to just \(f\supd_v = \sup \set{\vec t_{v i}}{i \in \mc X_1}\) and \(f\supc_v = \sup \set{\vec t_{(-v)i}}{i \in \mc X_1}\) for every \(v \in \mc X_2\) with \(v < 0\), as well as \(f\supn_i = \vec t_{0i}\) for every \(i \in \mc X_1\).  This connects our notion of a depth specification to the notion of a depth matrix in \xcite{spice:asymptotic}*{Definition \xref{defn:vGvr}}.  The extra information in a depth specification could be useful to define a finer filtration associated to \emph{two} commuting, semisimple elements \(\gamma_1\) and \(\gamma_2\) satisfying certain conditions, but we do not have any use for such a filtration here, and so do not pursue this.

Therefore, every group of the form \(\sbtl\vG x\tau\) can, if desired, be realised with \(\tau\) skew-constant, in the sense of Definition \ref{defn:depth-ops}.  Lemma \ref{lem:vee-monotone} shows that the restriction to skew-constant \(\tau\) interacts well with the operation \(\cvvv\) on depth specifications.
%
\end{rem}

\begin{notn}
\label{notn:ord-gamma-spec}
Put \(\matnotn{ord}{\ord_{\gamma - 1}} = (a_{v i j})_{v, i, j}\),
\(\matnotn{ord}{\ord_{\gamma\inv - 1}} = (a_{(-v)i j})_{v, i, j}\),
\(\matnotn{ord}{\ord_\gamma} = (b_{v i j})_{v, i, j}\), and
\(\matnotn{ord}{\ord_{\gamma\inv}} = (b_{(-v)i j})_{v, i, j}\),
where
\(a_{v i j} = 0\), \(a_{(-v)i j} = -v\), and \(b_{v i j} = v\) for all \(v \in \sbjtlp\R 0\), and
\(a_{0 i j} = i\) and \(b_{0 i j} = 0\),
for all \(i \in \R\) with \(0 \le i \le r\) and all \(j \in \sset{0, \dotsc, \ell}\).

(If \(\gamma\) is a semisimple element of \(\Lie(G)\), then instead put \(\matnotn{ord}{\ord_\gamma} = (a_{v i j})_{v, i, j}\), where, for all \(i \in \R\) with \(i \le r\) and all \(j \in \sset{0, \dotsc, \ell}\), we put \(a_{0 i j} = i\) and \(a_{v i j} = \infty\) for all non-\(0\) \(v \in \R\).)
\end{notn}

\begin{rem}
\label{rem:ord-gamma-spec}
Use Notation \ref{notn:ord-gamma-spec}.

The function \(f_{\ord_\gamma}\) of Definitions \ref{defn:vGvr} and \ref{defn:depth-to-concave} is, up to conjugation, just pairing with the cocharacter \(m_\gamma\) of Notation \ref{notn:ord-gamma}.  Note that \(\ord_\gamma\) is monotone and concave, and not grouplike, although \(\tau + \ord_\gamma\) is grouplike whenever \(\tau\) is; and that \(\ord_{\gamma - 1}\) is monotone but not concave, although there are easy modifications (setting \(a_{v i j} = \infty\) whenever \(v\) is negative, or else whenever \(v\) is positive) that will convert it to a grouplike depth specification.

If we write \(\ord_{\gamma - 1} = (a_{v i j})_{v, i, j}\) and \(\ord_{\gamma\inv - 1} = (a'_{v i j})_{v, i, j}\), then, for all \(i \in \R\) with \(0 \le i \le r\) and all \(j \in \sset{0, \dotsc, \ell}\), we have that \(a_{0 i j} = a'_{0 i j}\), and \(\max \sset{a_{v i j}, a'_{v i j}}\) equals \(0\) for all non-\(0\) \(v \in \R\).
\end{rem}

Lemma \ref{lem:commute-gp} is the analogue of \xcite{spice:asymptotic}*{Lemma \xref{lem:commute-gp}}, except for \xcite{spice:asymptotic}*{Lemma \initxref{lem:commute-gp}(\subxref{orbit})}, which is our Corollary \ref{cor:gp-orbit}.
(If \(\gamma\) is a semisimple element of \(\Lie(G)\), then all occurrences of \(\ord_{\gamma - 1}\) and \(\ord_{\gamma\inv - 1}\) should be replaced by \(\ord_\gamma\).  The result Lemma \initref{lem:commute-gp}\subpref{bi-up} has no analogue in this case.)

\begin{lem}
\initlabel{lem:commute-gp}
Let \(\tau\) and \(\tilde\tau\) be depth specifications.
	\begin{enumerate}
	\item\sublabel{down-Lie}
	\(\Ad(\gamma) - 1\) carries \(\sbtl{\Lie(\vG)}x\tau\) into \(\sbtl{\Lie(\Der\vG)}x{\tau + \ord_{\gamma - 1}}\).
	\item\sublabel{down-gp}
	Suppose \(\tau\) and \(\tilde\tau\) are grouplike, and we have \(\tilde\tau \le \tau \cvvv \tilde\tau\) and \(\tilde\tau \le \tau + \ord_{\gamma - 1}\).  Then \(\comm\gamma\anondot\) carries \(\sbtl\vG x\tau\) into \(\sbtl{\Der\vG}x{\tilde\tau}\).
	\item\sublabel{up-Lie}
	The pre-image of \(\sbtl{\Lie(\vG)}x{\tau + \ord_{\gamma - 1}} + \Lie(H)\) under \(\Ad(\gamma) - 1\) is contained in \(\sbtl{\Lie(\vG)}x\tau + \Lie(H)\).
	\item\sublabel{up-gp}
	If \(\tau\) is grouplike, then the pre-image in \(\sbtlp M x 0\) of \(\sbtl\vM x{\tau + \ord_{\gamma - 1}}\dotm\sbtlp H x 0\) under \(\comm\gamma\anondot\) is contained in \(\sbtl\vM x\tau\dotm\sbtlp H x 0\).
	\item\sublabel{bi-up}
	If the characteristic of \(\sbjat\field 0\) is not \(2\), then the pre-image of \(\Lie(\CC G r(\gamma)) + \sbtl{\Lie(\vG)}x{\tau + \min \sset{\ord_{\gamma - 1}, \ord_{\gamma\inv - 1}}}\) under \(\Ad(\gamma) - \Ad(\gamma)\inv\) is contained in \(\Lie(\CC G r(\gamma^2)) + \sbtl{\Lie(\vG)}x\tau\).
	\item\sublabel{onto}
	Suppose that \(\tau\) and \(\tilde\tau\) are grouplike, and satisfy the inequalities \(\tilde\tau \le \tau \cvvv \tilde\tau\) and \(\tilde\tau \ge \tau + \max \sset{\ord_{\gamma - 1}, \ord_{\gamma\inv - 1}}\).  Fix \(h \in \sbtl\vH x\tau \cap \sbtl H x r\).  Then \(\Int(\sbtl{(\vM, \vG)}x{(\tau, \tilde\tau)})(\sbtl\vH x{\tilde\tau}\dotm h\gamma)\) contains \(\sbtl\vG x{\tilde\tau}\dotm h\gamma\sbtl\vG x{\tilde\tau}\).
	\end{enumerate}
\end{lem}

\begin{proof}
Although we usually do not use the depth specification \(\ord_\gamma\) (defined as in Notation \ref{notn:ord-gamma-spec}) when \(\gamma\) is a semisimple automorphism of \(G\), rather than a semisimple element of \(\Lie(G)\), we will need it a few times in this proof.  The key points are that, by Hypothesis \initref{hyp:gp-gamma}\incpref{building}\subpref0, conjugation by \(\gamma\) carries \(\sbtl\vG x\tau\) to \(\sbtl\vG{x - m_\gamma}\tau = \sbtl\vG x{\tau + \ord_\gamma}\), and similarly on the Lie algebra and for other depth specifications; and that replacing \(\ord_{\gamma - 1}\) by \(\ord_\gamma\), or conversely, has no effect on subgroups of \(\vPrad^-\).

As in the proof of Lemma \ref{lem:filtration}, using the inequality \(\tilde\tau \le \tau \cvvv \tilde\tau\) for \locpref{down-gp}, it suffices to check (\locref{down-Lie}, \locref{down-gp}) on sets of generators for \(\sbtl{\Lie(\vG)}x\tau\) and \(\sbtl\vG x\tau\).

For each \(i \in \R\) with \(0 \le i \le r\), we have by Hypothesis \initref{hyp:gp-gamma}\subpref{Lie} that \(\Ad(\gamma) - 1\) carries \(\Lie(\CCp G i(\gamma))^\perp \cap \sbtl{\Lie(\CC\vG i(\gamma))}x{\vec t_{0 i}}\) into \(\Lie(\CCp G i(\gamma))^\perp \cap \sbtl{\Lie(\CC{\Der\vG}i(\gamma))}x{\vec t_{0 i} + i} \subseteq \sbtl{\Lie(\Der\vG)}x{\tau + \ord_{\gamma - 1}}\).
We have by Hypothesis \initref{hyp:gp-gamma}\incpref{building}\subpref0 that \(\Ad(\gamma)\) restricts to isomorphisms
\anonmap
	{\sbtl{\Lie(\vPrad^\pm)}x\tau}
	{
	\sbtl{\Lie(\vPrad^\pm)}x{\tau + \ord_\gamma}},
so that \(\Ad(\gamma) - 1\) carries \(\sbtl{\Lie(\vPrad^+)}x\tau\) into \(\sbtl{\Lie(\vPrad^+)}x\tau = \sbtl{\Lie(\vPrad^+)}x{\tau + \ord_{\gamma - 1}}\) and \(\sbtl{\Lie(\vPrad^-)}x\tau\) into \(\sbtl{\Lie(\vPrad^-)}x{\tau + \ord_\gamma} = \sbtl{\Lie(\vPrad^-)}x{\tau + \ord_{\gamma - 1}}\), both of which are contained in \(\sbtl{\Lie(\Der\vG)}x{\tau + \ord_{\gamma - 1}}\).
This establishes \locpref{down-Lie}.

Similarly, using Hypothesis \initref{hyp:gp-gamma}\subpref{gp} in place of Hypothesis \initref{hyp:gp-gamma}\subpref{Lie}, we have that \(\comm\gamma\anondot\) carries \(\sbtl{\vPrad^+}x\tau\) into \(\sbtl{\vPrad^+}x\tau = \sbtl{\vPrad^+}x{\tau + \ord_{\gamma - 1}}\) and \(\sbtl{\vPrad^-}x\tau\) into \(\sbtl{\vPrad^-}x{\tau + \ord_\gamma} = \sbtl{\vPrad^-}x{\tau + \ord_{\gamma - 1}}\), both of which are contained in \(\sbtl{\Der\vG}x{\tilde\tau}\), and each \(\sbtl{\CC\vG i(\gamma)}x{\vec t_{0 i}}\) into \(\sbtl{\CC{\Der\vG}i(\gamma)}x{\vec t_{0 i} + i}\), which is contained in \(\sbtl{\Der G \cap \vM}x{\tau + \ord_{\gamma - 1}}\), and hence in \(\sbtl{\Der\vG}x{\tilde\tau}\) by monotonicity of \(\tau\).  This establishes \locpref{down-gp}.

For \locpref{up-gp}, suppose that \(g \in \sbtlp M x 0\) satisfies \(\comm\gamma g \in \sbtl\vM x{\tau + \ord_{\gamma - 1}}\dotm\sbtlp H x 0\), but \(g\) does not belong to \(\sbtl\vM x\tau\dotm\sbtlp H x 0\).  Let \(i_0\) be the greatest index \(i\) such that \(g\) belongs to \(\sbtl\vM x\tau\dotm\sbtlp{\CC G i(\gamma)}x 0\), let \(j_0\) be the least index \(j\) such that \(g\) belongs to \(\sbtl\vM x\tau\dotm\sbtlp{\CC{G^j}{i_0}(\gamma)}x 0\), and let \(a_0\) be the greatest index \(a\) such that \(g\) belongs to \(\sbtl\vM x\tau\dotm\sbtl{\CC\vG{i_0}(\gamma)}x{(\Rp0, \dotsc, \Rp0, a, \infty, \dotsc, \infty)}\) (where \(a\) is in the \(j\)th slot).  Then \(i_0\) is strictly less than \(r\) and \(a_0\) is strictly less than \(t_{0 i_0 j_0}\).  Put \(\vec a = (\Rp0, \dotsc, \Rp0, a_0, \infty, \dotsc, \infty)\), and write \(g = g_+g_-\), with \(g_+ \in \sbtl\vM x\tau\) and \(g_- \in \sbtl{\CC\vG{i_0}(\gamma)}x{\vec a}\).  By \locpref{down-gp}, we have that \(\comm\gamma{g_+}\) belongs to \(\sbtl\vM x{\tau + \ord_{\gamma - 1}}\); and then Lemma \initref{lem:filtration}\subpref{gp-gp} gives that
\begin{multline*}
\comm\gamma{g_-} = \Int(g_+)\inv(\comm\gamma{g_+}\inv\comm\gamma{g_+ g_-})
\quad\text{belongs to}\quad \\
\sbtlp{\CC{G^{j_0}}{i_0}(\gamma)}x 0 \cap \sbtl\vM x{\tau + \ord_{\gamma - 1}}\dotm\sbtlp H x 0,
\end{multline*}
which is contained in \(\sbtl{\CC{\vG \cap G^{j_0}}{i_0}(\gamma)}x{\vec t_{0 i_0} + i_0}\dotm\sbtlp H x 0\) by Remark \ref{rem:fc-Iwahori}.  Thus \(\comm\gamma{g_-}\) belongs to \(\sbtlp{\CC\vG{i_0}(\gamma)}x{\vec a + i_0}\dotm\sbtlp H x 0\).  Hypothesis \initref{hyp:gp-gamma}\subpref{gp} gives that \(g_-\) belongs to \(\sbtlp{\CC\vG{i_0}(\gamma)}x{\vec a}\dotm\sbtlp H x 0\), which is a contradiction of the maximality of \(a_0\).  This establishes \locpref{up-gp}.

For \locpref{up-Lie}, suppose instead that \(X \in \Lie(G)\) satisfies \((\Ad(\gamma) - 1)X \in \sbtl{\Lie(\vM)}x{\tau + \ord_{\gamma - 1}} + \Lie(H)\).  Reasoning as in the group case shows that the \bM-equivariant projection of \(X\) on \(\Lie(M)\) belongs to \(\sbtl{\Lie(\vM)}x\tau + \Lie(H)\).  Write \(X^\pm\) for the \bM-equivariant projection of \(X\) on \(\Lie(\Prad^\pm)\).  The \bM-equivariant projection \((\Ad(\gamma) - 1)X^-\) of \((\Ad(\gamma) - 1)X\) on \(\Lie(\Prad^-)\) lies in \(\sbtl{\Lie(\vPrad^-)}x{\tau + \ord_{\gamma - 1}}\).  Suppose that \(X^-\) does not belong to \(\sbtl{\Lie(\vPrad^-)}x\tau\).  Let \(j_0\) be the least index \(j\) such that \(X^-\) belongs to \(\sbtl{\Lie(\vPrad^-)}x\tau + \Lie(\vPrad^- \cap G^j)\).  For each \(v_0 \in \R\) and \(a \in \tR\), define a depth specification \(\tau'_{v_0 a} = (t'_{v i j})_{v, i, j}\) by \(t'_{v i j} = -\infty\) if \(v \ne v_0\) or \(j \ne j_0\), and \(t'_{v_0 i j_0} = a\), for all \(i \in \R\) with \(0 \le i \le r\).
Let \(a_0\) be the greatest index \(a\) such that \(X^-\) belongs to \(\sbtl{\Lie(\vPrad^-)}x\tau + \sbtl{\Lie(\vPrad^-)}x{\tau'_{v_0 a}}\).  Then \(a_0\) is strictly less than \(t_{v_0 i j_0}\) for some \(v_0 \in \R\) and \(i \in \R\) with \(0 \le i \le r\).  Write \(X^- = X^-_+ + X^-_-\), with \(X^-_+ \in \sbtl{\Lie(\vPrad^-)}x\tau\) and \(X^-_- \in \sbtl{\Lie(\vPrad^-)}x{\tau'_{v_0 a_0}}\).  Then \locpref{down-Lie} gives that \((\Ad(\gamma) - 1)X^-_+\) lies in \(\sbtl{\Lie(\vPrad^-)}x{\tau + \ord_{\gamma - 1}}\), so
\begin{multline*}
(\Ad(\gamma) - 1)X^-_- = (\Ad(\gamma) - 1)X^- - (\Ad(\gamma) - 1)X^-_+
\quad\text{lies in}\quad \\
\sbtl{\Lie(\vPrad^-)}x{\tau + \ord_{\gamma - 1}} = \sbtl{\Lie(\vPrad^-)}x{\tau + \ord_\gamma} = \sbtl{\Lie(\vPrad^-)}{\gamma\dota x}\tau.
\end{multline*}
Since, for every index \(v \in \R\) and \(j \in \sset{0, \dotsc, \ell}\), we have that \((\tau'_{v_0 a_0})_{v i j}\) is strictly less than \(\tau_{v i j}\) for some \(i \in \R\) with \(0 \le i \le r\), we have by Definitions \ref{defn:vGvr} and \ref{defn:depth-to-concave} that \(\sbtl{\Lie(\vPrad^-)}{\gamma\dota x}\tau\), and so \((\Ad(\gamma) - 1)X^-_-\), is contained in \(\sbtlp{\Lie(\vPrad^-)}{\gamma\dota x}{\tau'_{v_0 a_0}}\).  Since \(\Ad(\gamma) - 1\) agrees with \(\Ad(\gamma)\), and is therefore an isomorphism, as a map \anonmap{\sbat{\Lie(\vPrad^-)}x{\tau'_{v_0 a_0}}}{\sbat{\Lie(\vPrad^-)}{\gamma\dota x}{\tau'_{v_0 a_0}} = \sbat{\Lie(\vPrad^-)}x{\tau'_{v_0 a_0} + \ord_\gamma}}, the fact that \((\Ad(\gamma) - 1)X^-_-\) belongs to \(\sbtlp{\Lie(\vPrad^-)}{\gamma\dota x}{\tau'_{v_0 a_0}}\) means that \(X^-_-\) belongs to \(\sbtlp{\Lie(\vPrad^-)}x{\tau'_{v_0 a_0}}\), which equals \(\sbtl{\Lie(\vPrad^-)}x{(\tau'_{v_0 a_0}) + \varepsilon} = \sbtl{\Lie(\vPrad^-)}x{\tau'_{v_0(a_0 + \varepsilon)}}\) for all \(\varepsilon > 0\) sufficiently small.  This is a contradiction of the maximality of \(a_0\).

The argument for \(\Lie(\Prad^+)\) is similar, except that now \(\Ad(\gamma) - 1\) equals \(-1\), not \(\Ad(\gamma)\), as maps \anonmap{\sbat{\Lie(\vPrad^+)}x{\tau_{v_0 a_0}}}{\sbat{\Lie(\vPrad^+)}x{\tau_{v_0 a_0}}}.  This establishes \locpref{up-Lie}.

The argument for \locpref{bi-up} is as in \locpref{up-Lie}, using Hypothesis \initref{hyp:gp-gamma}\subpref{bi-Lie-Lie} instead of Hypothesis \initref{hyp:gp-gamma}\subpref{Lie}.  (We use the fact that the residual characteristic of \field is not \(2\) to see that, for every eigenvalue \(\lambda\) of \(\Ad(\gamma)\) on \(\Lie(\bG)_\sepfield\) that belongs to \(\sbjtlp{\field\twosup\sepsup\multsup}0\), the eigenvalue \(\lambda - \lambda\inv\) of \(\Ad(\gamma) - \Ad(\gamma)\inv\) has the same depth as \(\lambda - 1\) in the filtration of \(\sepfield\).)

For \locpref{onto}, our strategy is simpler conceptually than notationally, so we outline it before beginning.  We deal with the pieces (in terms of the decomposition described in Lemma \ref{lem:gen-Iwahori-factorisation}) in \(\Prad^-\), \(\Prad^+\), and \(M\) separately.  For the \(\Prad^-\) and \(\Prad^+\) pieces, working in turn, we show inductively that we can conjugate our original element \(g\) in such a way that that piece is pushed deeper and deeper, until, in the limit, we get rid of it entirely.  The difference for \(M\) is that there is no separate ``\(\CC\vG i(\gamma)\) piece'' and ``\(\CCp\vG i(\gamma)\) piece''; so, rather than getting rid of the \(M\) piece entirely, we push it closer and closer to \(\sbtl\vH x{\tilde\tau}\), until, in the limit, it lies in \(\sbtl\vH x{\tilde\tau}\).
We will use Lemma \initref{lem:filtration}\subpref{gp-gp} repeatedly, without explicit mention.

We begin by using \xcite{adler-spice:good-expansions}*{Lemma \xref{lem:concave-in-R}} to replace \(\tau\) and \(\tilde\tau\) by real-valued functions.
Note that, since \(\max \sset{\ord_{\gamma - 1}, \ord_{\gamma\inv - 1}}\) is everywhere positive, we have the inequality \(\tilde\tau \ge \tau\).

Fix \(h \in \sbtl\vG x\tau \cap \sbtl H x r\) and \(g \in \sbtl\vG x{\tilde\tau}\dotm h\gamma\sbtl\vG x{\tilde\tau}\).
Note that \(h\) normalises \(\sbtl{\vPrad^+}x{\tilde\tau}\), \(\sbtl{\vec M}x{\tilde\tau}\), and \(\sbtl{\vPrad^-}x{\tilde\tau}\) by Lemma \initref{lem:filtration}\subpref{gp-gp} and the inequality \(\tilde\tau \le \tau \cvvv \tilde\tau\); and that \(\Int(\gamma)\sbtl{\vPrad^+}x{\tilde\tau}\) is contained in \(\sbtl{\vPrad^+}x{\tilde\tau}\), \(\Int(\gamma)\sbtl{\vec M}x{\tilde\tau}\) equals \(\sbtl{\vec M}x{\tilde\tau}\), and \(\Int(\gamma)\inv\sbtl{\vPrad^-}x{\tilde\tau}\) is contained in \(\sbtl{\vPrad^-}x{\tilde\tau}\) by Hypothesis \initref{hyp:gp-gamma}\incpref{building}\subpref0.  Therefore, Lemma \ref{lem:gen-Iwahori-factorisation} gives that \(g\) belongs to \(\sbtl{\vPrad^+}x{\tilde\tau}\dotm\sbtl\vM x{\tilde\tau}\dotm h\gamma\sbtl{\vPrad^-}x{\tilde\tau}\).

We get rid of the \(\Prad^-\) piece first.
Starting with the base case \(\sbjhd{u^-}0 = 1\), suppose that \(a\) is a non-negative real number, and that there are \(\sbjhd{u^-}a \in \sbtl{\vPrad^-}x{\tilde\tau}\), \(m \in \sbtl\vM x{\tilde\tau}\), and \(\sbjat{u^-}a \in \sbtl{\vPrad^-}x{\tilde\tau + a}\) such that \(\Int(\sbjhd{u^-}a)g\) belongs to \(\sbtl{\vPrad^+}x{\tilde\tau}\dotm m h\gamma\sbjat{u^-}a\).  Put \(\sbjhdp{u^-}a = \sbjat{u^-}a\sbjhd{u^-}a\).  Then
\begin{multline*}
\Int(\sbjhdp{u^-}a)g
\quad\text{belongs to}\quad \\
\Int(\sbjat{u^-}a)(\sbtl{\vPrad^+}x{\tilde\tau}\dotm m h\gamma\sbjat{u^-}a)
= \Int(\sbjat{u^-}a)\sbtl{\vPrad^+}x{\tilde\tau}\dotm m h\gamma\dotm\Int(m h\gamma)\inv\sbjat{u^-}a.
\end{multline*}
The inequality \(\tilde\tau + a \le (\tau \cvvv \tilde\tau) + a = \tau \cvvv (\tilde\tau + a)\) gives that \(\sbtl\vM x\tau\) normalises \(\sbtl{\vPrad^-}x{\tilde\tau + a}\), so that \(\Int(m h)\inv\sbjat{u^-}a\) is contained in \(\sbtl{\vPrad^-}x{\tilde\tau + a}\).  Then Hypothesis \initref{hyp:gp-gamma}\incpref{building}\subpref0 gives that \(\Int(m h\gamma)\inv\sbjat{u^-}a\) is contained in \(\sbtl{\vPrad^-}x{\tilde\tau - \ord_\gamma + a} \subseteq \sbtlp{\vPrad^-}x{\tilde\tau + a}\).
The inequality \(\tilde\tau \ge \tau + \ord_{\gamma - 1}\) implies that \(\sbjat{u^-}a \in \sbtl{\vPrad^-}x{\tilde\tau + a}\) is contained in \(\sbtlp{\vPrad^-}x{\tau + a}\), so that its commutator with \(\sbtl{\vPrad^+}x{\tilde\tau}\) is contained in \(\sbtlp\vG x{(\tau + a) \cvvv \tilde\tau} \subseteq \sbtlp\vG x{\tilde\tau + a}\).  This gives that
\(\Int(\sbjat{u^-}a)\sbtl{\vPrad^+}x{\tilde\tau}\) is contained in \(\sbtl{\vPrad^+}x{\tilde\tau}\dotm\sbtlp\vG x{\tilde\tau + a}\).
Another application of Lemma \ref{lem:gen-Iwahori-factorisation} shows that \(\sbtlp\vG x{\tilde\tau + a}\) equals \(\sbtlp{\vPrad^+}x{\tilde\tau + a}\dotm\sbtlp\vM x{\tilde\tau + a}\dotm\sbtlp{\vPrad^-}x{\tilde\tau + a}\), so that \(\sbtlp\vG x{\tilde\tau + a}\dotm m h\gamma\) is contained in \(\sbtlp{\vPrad^+}x{\tilde\tau + a}\dotm\sbtlp\vM x{\tilde\tau + a}\dotm m h \gamma\sbtlp{\vPrad^-}x{\tilde\tau + a}\) and hence
\[
\Int(\sbjhdp{u^-}a)g
\quad\text{belongs to}\quad
\sbtl{\vPrad^+}x{\tilde\tau}\dotm\sbtl\vM x{\tilde\tau}\dotm h\gamma\sbtlp{\vPrad^-}x{\tilde\tau + a}.
\]
We replace \(\sbjhd{u^-}a\) by \(\sbjhdp{u^-}a\), and then \(a\) by \(a + \varepsilon\) for \(\varepsilon\) a sufficiently small positive real number, and iterate.  This process yields a Cauchy sequence in \(\sbtl{\vPrad^-}x{\tilde\tau}\), which has a limit by completeness, and whose limit \(u^-\) satisfies \(\Int(u^-)g \in \sbtl{\vPrad^+}x{\tilde\tau}\dotm\sbtl\vM x{\tilde\tau}\dotm h\gamma\).
The same reasoning, with \(\Prad^+\) in place of \(\Prad^-\), shows that there is an element \(u^+ \in \sbtl{\vPrad^+}x{\tilde\tau}\) such that \(\Int(u^+ u^-)g\) belongs to \(\sbtl\vM x{\tilde\tau}\dotm h\gamma\).

Now we may, and do, replace \(g\) by \(\Int(u^+ u^-)g\), and so assume that \(g\) belongs to \(\sbtl\vM x{\tilde\tau}\dotm h\gamma\).

We are at the final stage.  Starting with the base case \(\sbjhd m 0 = 1\), suppose that we have a non-negative real number \(a\) and an element \(\sbjhd m a \in \sbtl\vM x\tau\) such that \(\Int(\sbjhd m a)g\) belongs to \(\sbtl\vM x{\tilde\tau + a}\dotm\sbtl\vH x{\tilde\tau}\dotm h\gamma\).  If \(\Int(\sbjhd m a)g\) belongs to \(\sbtlp\vM x{\tilde\tau + a}\dotm\sbtl\vH x{\tilde\tau}\dotm h\gamma\), then put \(\sbjhdp m a = \sbjhd m a\).  Otherwise, we may, and do,  assume, after replacing \(\sbjhd m a\) by a suitable left \(\sbtl\vM x{\tau + a}\)-translate if needed, that there is some index \(i_0 \in \sbjtl\R 0\) such that
\(\Int(\sbjhd m a)g\) belongs to
\(\sbtlp\vM x{\tilde\tau + a}\dotm\sbtl\vH x{\tilde\tau}\dotm\sbtl{\CC\vG{i_0}(\gamma)}x{\tilde\tau + a}\dotm h\gamma\),
but there is no \(\sbjhdp m a \in \sbtl\vM x{\tau + a}\dotm\sbjhd m a\) such that \(\Int(\sbjhdp m a)g\) belongs to \(\sbtlp\vM x{\tilde\tau + a}\dotm\sbtl\vH x{\tilde\tau}\dotm\sbtl{\CCp\vG{i_0}(\gamma)}x{\tilde\tau + a}\dotm h\gamma\).  We will use this to obtain a contradiction.  (It therefore seems that this iteration, unlike the previous ones, never makes progress, since we either put \(\sbjhdp m a = \sbjhd m a\) or obtain a contradiction.  The progress is made when we adjust our choice of representative of the left \(\sbtl\vM x{\tau + a}\)-coset through \(\sbjhd m a\) to maximise the index \(i_0\).)

The inequalities \(\tilde\tau \ge \tau + \ord_{\gamma - 1}\) and \(\tilde\tau \le \tau \cvvv \tilde\tau\) imply that \(\tilde\tau + a - i_0\) is a grouplike depth specification for \(\CC\vbG{i_0}(\gamma)\) (though it might not be so for all of \bG).  Therefore, by Hypothesis \initref{hyp:gp-gamma}\subpref{gp}, there is an element \(\sbjat m a\) in \(\sbtl{\CC\vG{i_0}(\gamma)}x{\tilde\tau + a - i_0}\) such that
\[
\Int(\sbjhd m a)g\quad\text{belongs to}\quad
\sbtlp\vM x{\tilde\tau + a}\dotm\sbtl\vH x{\tilde\tau}\dotm\sbtl{\CCp\vG{i_0}(\gamma)}x{\tilde\tau + a}\dotm\comm\gamma{\sbjat m a\inv}\inv\dotm h\gamma.
\]
Put \(\sbjhdp m a = \sbjat m a\dotm\sbjhd m a\).

Another appeal to the inequality \(\tilde\tau \ge \tau + \ord_{\gamma - 1}\) gives that \(\sbjat m a\) belongs to \(\sbtl{\CC\vG{i_0}(\gamma)}x{\tau + a}\), hence to \(\sbtl\vM x{\tau + a}\).  In particular, \(\sbjat m a\) normalises \(\sbtlp\vM x{\tilde\tau + a}\), \(\sbtl\vH x{\tilde\tau}\), and \(\sbtl{\CCp\vG{i_0}(\gamma)}x{\tilde\tau + a}\).  Similarly, since \(i_0\) is less than \(r\), we have that the commutator of \(\sbjat m a \in \sbtl{\CC\vG{i_0}(\gamma)}x{\tilde\tau + a - i_0}\) with \(h \in \sbtl H x r\) lies in \(\sbtl{\CC\vG{i_0}(\gamma)}x{\tilde\tau + a - i_0 + r}\), hence in \(\sbtlp\vM x{\tilde\tau + a}\).  Therefore
\begin{multline*}
\Int(\sbjhdp m a)g = \Int(\sbjat m a)(\Int(\sbjhd m a)g)
\quad\text{belongs to}\quad \\
\sbtlp\vM x{\tilde\tau + a}\dotm\sbtl\vH x{\tilde\tau}\dotm\sbtl{\CCp\vG{i_0}(\gamma)}x{\tilde\tau + a}\dotm\comm\gamma{\sbjat m a}\dotm\sbtlp\vM x{\tilde\tau + a}h\dotm\Int(\sbjat m a)\gamma.
\end{multline*}
Since \(\comm\gamma{\sbjat m a}\) and \(\sbtl\vH x{\tilde\tau}\dotm\sbtl{\CCp\vG{i_0}(\gamma)}x{\tilde\tau + a}\) normalise \(\sbtlp\vM x{\tilde\tau + a}\), and since the commutator of \(\comm\gamma{\sbjat m a}\) with \(h\) lies in \(\sbtlp\vM x{\tilde\tau + a}\), we have that
\begin{multline*}
\Int(\sbjhdp m a)g\quad\text{belongs to}\quad \\
\sbtlp\vM x{\tilde\tau + a}\dotm\sbtl\vH x{\tilde\tau}\dotm\sbtl{\CCp\vG{i_0}(\gamma)}x{\tilde\tau + a}\dotm h\dotm\comm\gamma{\sbjat m a}\dotm\Int(\sbjat m a)\gamma,
\end{multline*}
that is, to
\(\sbtlp\vM x{\tilde\tau + a}\dotm\sbtl\vH x{\tilde\tau}\dotm\sbtl{\CCp\vG{i_0}(\gamma)}x{\tilde\tau + a}\dotm h\gamma\).  This is a contradiction of our assumption regarding the maximality of \(i_0\).

Thus, after all, we were in the case where we defined an element \(\sbjhdp m a\) such that \(\Int(\sbjhdp m a)g\) belonged to \(\sbtlp\vM x{\tilde\tau + a}\dotm\sbtl\vH x{\tilde\tau}\dotm h\gamma\).  Once again, we replace \(\sbjhd m a\) by \(\sbjhdp m a\) and then \(a\) by \(a + \varepsilon\), for \(\varepsilon\) a sufficiently small positive real number, and iterate.  This process yields a Cauchy sequence in \(\sbtl\vM x\tau\), which has a limit by completeness, and whose limit \(m\) satisfies \(\Int(m)g \in \sbtl\vH x{\tilde\tau}\dotm h\gamma\).
\end{proof}

\section{Asymptotic expansions of orbital integrals}
\label{sec:orbits}
\def\xtopic{spice:asymptotic}

Throughout \S\ref{sec:orbits}, we continue to use the
	non-Archimedean local field \field
and	reductive group \bG
from \S\ref{sec:hyps}.

Because our focus in \S\ref{sec:orbits} is almost entirely on the Lie algebra \(\Lie(\bG)\), we assume throughout \S\ref{sec:orbits} that \bG is connected.
	\project[But it makes perfectly good sense to consider orbital integrals for disconnected groups.  Worth saying?]

Because of issues such as the one identified in Example \ref{exa:not-most-singular}, we need to assume that \(\bG_\tamefield\) is split.

We assume that the characteristic of \(\sbjat\field 0\) is not \(2\),
although it seems that the only obstruction to handling the case of residual characteristic \(2\) is the requirement in \xcite{spice:asymptotic}*{Definition \xref{defn:Weil-index}} that \(L^\bullet\) be contained in \(2L\), and its effect on our main tool Proposition \ref{prop:lattice-orth} for dealing with Weil indices.

\numberwithin{thm}{subsection}
\subsection{Existence of asymptotic expansions}
\label{subsec:orbital-exists}

We begin by stating some Lie-algebra analogues of group results in \cite{spice:asymptotic}.  The proofs of the relevant group results in \cite{spice:asymptotic} proceeded by reduction to the Lie algebra, so the proofs in the Lie-algebra case are the same, \textit{mutatis mutandis}, where they are not easier.  Where there is no significant difference in the proof, we indicate only the statement.  Where the proof involves a significant simplification, we sketch that simplification.  The one exception is Theorem \ref{thm:sample-orb-to-orb'}, which is analogous to \xcite{spice:asymptotic}*{Theorem \xref{thm:isotypic-pi-to-pi'} and Corollary \xref{cor:isotypic-pi-to-pi'}} but, as directly involving Fourier transforms of orbital integrals rather than characters, requires a geometric rather than spectral approach.  We therefore give its proof in full.

We choose
	\begin{itemize}
	\item a real number \mnotn r,
	\item an element \(\matnotn{Gamma}\Gamma \in \Lie^*(G)\),
and	\item a semisimple element \(\matnotn{gamma}\gamma \in \Lie(G)\).
	\end{itemize}

We require throughout \S\ref{sec:orbits} that \(\gamma\) satisfy Hypothesis \ref{hyp:fc-building}, and that \xcite{spice:asymptotic}*{Hypothesis \initxref{hyp:Z*}(\subxref{central}, \subxref{good})} be satisfied with \(\Gamma\) in place of \(Z^*_o\).  This gives, in particular, that \(\matnotn{Gprime}{\bG\primeconn} \ldef \Cent_\bG(\Gamma)\conn\) is a tame, twisted Levi subgroup of \(\bG\conn\).  The same group is denoted by \(\bG'\) in \loccit, but we prefer to save that notation for a possibly larger group admitting \(\bG\primeconn\) as its identity component.  See Definition \ref{defn:G'}.  With an eye towards that eventual notation, we shall feel free to use \(\bG'\) in any context where the dependence is only through the identity component; for example, we shall write \(\Lie^*(\bG')\) in place of \(\Lie^*(\bG\primeconn)\).

Put \(\bH = \CC\bG r(\gamma)\).
We require throughout \S\ref{subsec:orbital-exists} that \(\bH\conn\) equal \(\Cent_\bG(\gamma)\conn\), but we will drop this assumption for much of \S\ref{subsec:dist-g-to-g'};
and
we require throughout \S\S\ref{subsec:orbital-exists}, \ref{subsec:dist-g-to-g'} that \(r\) be positive, but we will drop this assumption in \S\ref{subsec:quantitative}.

\begin{rem}
\label{rem:Lie-symmetry}
Since \(\Lie(\bH)\) equals \(\Cent_{\Lie(\bG)}(\gamma)\), and \(\Lie(\bG')\) equals \(\Cent_{\Lie(\bG)}(\Gamma)\), the containments \(\gamma \in \Lie(G')\) and \(\Gamma \in \Lie^*(H)\) are equivalent to \(\ad^*(\gamma)\Gamma = 0\), hence to each other.
\end{rem}

\begin{rem}
\label{rem:Gamma=0}
We may choose \(\Gamma = 0\), in which case \(\bG'\) equals \bG and \(r\) can be any (positive, in \S\S\ref{subsec:orbital-exists}, \ref{subsec:dist-g-to-g'}) real number for which \(\CC\bG r(\gamma)\conn\) equals \(\Cent_\bG(\gamma)\conn\).
\end{rem}

Note that \xcite{spice:asymptotic}*{Hypothesis \initxref{hyp:Z*}(\subxref{central}, \subxref{good})} are `algebraic', i.e., their validity is unaffected by base change to an algebraic closure.  Since \(\bG_\tamefield\), hence \(\bG\primeconn_\tamefield\), is split, Lemma \initref{lem:dual-MP-ascent}\subpref{tame-or-split} shows that any lower bounds on \(\Gamma + \sbjtlpp{\Lie^*(G')}{-r}\) are also algebraic, in the sense that this domain only grows if we replace \field by an algebraic extension of finite ramification degree.  We thus can, and often do, replace \field by a convenient extension whenever this does not affect the conclusion or the other hypotheses of a result.

\begin{lem}
\initlabel{lem:ss-of-dual-blob}
Fix \(X^* \in \Gamma + \sbjtlpp{\Lie^*(G')}{-r}\), and write \((X^*\semi, X^*\nilp)\) for a Jordan decomposition of \(X^*\), in the sense of Definition \ref{defn:dual-Jordan}.
\begin{enumerate}
\item\sublabel{full-rk}
\(X^*\semi\) is a semisimple element of \(\Lie^*(\bG')(\algfield)\).
\item\sublabel{fc}
\(\Cent_{\bG_\algfield}(X^*\semi)\conn\) is contained in \(\bG\primeconn_\algfield\).
\item\sublabel{roots}
There is a finite extension \(E/\field\) such that
	\begin{itemize}
	\item \(\Cent_{\bG\primeconn_E}(X^*\semi)\) contains a split, maximal torus \bT in \(\bG_E\),
	\item \(X^*\semi\) belongs to \(\Gamma + \sbjtlpp{\Lie^*(\bT)(E)}{-r}\),
and	\item \(\Root(\bG'_E, \bT)\) equals \(\set{\alpha \in \Root(\bG_E, \bT)}{\pair{X^*\semi}{\upd\alpha^\vee(1)} \in \sbjtlpp E{-r}}\).
	\end{itemize}
\end{enumerate}
\end{lem}

\begin{proof}
We may, and do, replace \field by finite extensions as appropriate.

We begin by showing that there is \emph{some} Jordan decomposition of \(X^*\) satisfying the desired properties, and then that \emph{all} semisimple parts of Jordan decompositions of \(X^*\) are \(\bG\primeconn(\algfield)\)-conjugate.

Put \(X\primedual = X^* - \Gamma\), so that \(X\primedual\) belongs to \(\sbjtlpp{\Lie^*(G')}{-r}\).  By
\cite{spice-tsai:jordan}*{Remark 3.4},
there is a Jordan decomposition \((X\primedual\semi, X\primedual\nilp)\) of \(X\primedual\) with \(X\primedual\semi\) a semisimple element of \(\Lie^*(\bG')(\algfield)\), so that \locpref{full-rk} holds.  In particular, \(\Cent_{\bG'_\algfield}(X\primedual\semi)\) contains a torus \(\bT_1\) that is maximal in \(\bG'_\algfield\), hence in \(\bG_\algfield\).  We may, and do, assume that \(X\primedual\semi\) belongs to \(\Lie^*(G')\), and that \(\bT_1\) descends to a split torus \bT over \field.

Lemma \ref{lem:dual-centre} gives that \((\Gamma + X\primedual\semi, X\primedual\nilp)\) is a Jordan decomposition of \(X^*\).  By Lemma \ref{lem:dual-contract}, there is a cocharacter \(\lambda\) of \(\Cent_{\bG\primeconn}(X\primedual\semi)_\algfield\) such that the limit \(\lim_{t \to 0} \Ad^*(\lambda(t))X\primedual\nilp\) in the Zariski topology equals \(0\).  We may, and do, assume that \(\lambda\) is a cocharacter of \(\Cent_{\bG\primeconn}(X\primedual\semi)\).  Then also the limit \(\lim_{t \to 0} \Ad^*(\lambda(t))X\primedual\nilp\) in the analytic topology equals \(0\).
Thus the limit \(\lim_{t \to 0} \Ad^*(\lambda(t))X\primedual\) in the analytic topology equals \(X\primedual\semi\).

In particular, \(X\primedual\semi\) belongs to the closure, in the analytic topology, of the \(\Cent_{G\primeconn}(X\primedual)\)-coadjoint orbit of \(X\primedual\).  Then the proof of \cite{adler-debacker:bt-lie}*{Lemma 3.3.8} carries through unchanged, after replacing \cite{adler-debacker:bt-lie}*{Corollary 3.2.6} by its dual-Lie-algebra analogue, to show that the depth of \(X\primedual\semi\) equals that of \(X\primedual\), and hence is strictly greater than \(-r\).  That is, \(X\primedual\semi\) belongs to \(\sbjtlpp{\Lie^*(G')}{-r}\).

Since \bT is a Levi subgroup of \(\bG\primeconn\), the dual-Lie-algebra version of \cite{adler-debacker:bt-lie}*{Lemma 3.5.3} gives that \(\Lie^*(T) \cap \sbjtlpp{\Lie^*(G')}{-r}\), which contains \(X\primedual\semi\), equals \(\sbjtlpp{\Lie^*(T)}{-r}\).

For every \(\alpha \in \Root(\bG, \bT)\), Lemma \initref{lem:dual-MP-by-cclat} gives that \(\pair{X\primedual\semi}{\upd\alpha^\vee(-1)}\) belongs to \(\sbjtlpp\field{-r}\).  It follows that the set \(\set{\alpha \in \Root(\bG, \bT)}{\pair{\Gamma + X\primedual\semi}{\upd\alpha^\vee(1)} \in \sbjtlpp\field{-r}}\), and its analogue constructed using \(\Gamma\) in place of \(X^*\semi\), are equal; but \xcite{spice:asymptotic}*{Hypothesis \initxref{hyp:Z*}(\subxref{good})} says that this latter set is precisely \(\Root(\bG', \bT)\).  This shows \locpref{roots} for our specific choice of Jordan decomposition of \(X^*\).

Finally, since \cite{borel:linear}*{Proposition 13.20} says that \(\Cent_\bG(X^*\semi)\conn\) is generated by \bT and those root subgroups correspoding to roots in \(\Root(\Cent_\bG(X^*\semi), \bT)\), and analogously for \(\bG\primeconn\), and since \locpref{roots} gives that \(\Root(\Cent_\bG(X^*\semi), \bT)\) is contained in \(\Root(\bG', \bT)\), we have shown \locpref{fc} for our specific choice of Jordan decomposition of \(X^*\).

In particular, \locpref{fc} shows that \(\Cent_\bG(\Gamma + X\primedual\semi)\conn\) is centralized by \(\Zent(\bG\primeconn)\), and \locpref{roots} shows that \(\Cent_\bG(\Gamma + X\primedual\semi)\conn\) contains \bT, hence \(\Zent(\bG\primeconn)\).  Thus \(\Zent(\bG\primeconn)\) is contained in \(\mc Z \ldef \Zent(\Cent_\bG(\Gamma + X\primedual\semi)\conn)\), so that \(\Cent_\bG(\mc Z)\conn\) is contained in \(\Cent_\bG(\Zent(\bG\primeconn))\conn = \bG\primeconn\).

So far we have worked with a preferred Jordan decomposition \((\Gamma + X\primedual\semi, X\primedual\nilp)\) of \(X^*\).  Now we need to handle the arbitrary Jordan decomposition \((X^*\semi, X^*\nilp)\).
We have by
\cite{spice-tsai:jordan}*{Theorem 8.4} that \(X^*\semi\) is \(\Cent_\bG(\mc Z)\conn(\algfield)\)-conjugate, hence \(\bG\primeconn(\algfield)\)-conjugate, to \(\Gamma + X\primedual\semi\); and it is easy to see that the desired properties are unaffected by such conjugacy.
\end{proof}

\begin{lem}
\label{lem:unique-approx*}
Suppose that \(\Gamma'\) is another element of \(\Lie^*(G)\) satisfying \xcite{spice:asymptotic}*{Hypothesis \initxref{hyp:Z*}(\subxref{central}, \subxref{good})}, with associated group \(\bG\twosup{\prime\prime}\connsup \ldef \Cent_\bG(\Gamma')\conn\).  If \((\Gamma + \sbjtlpp{\Lie^*(G')}{-r}) \cap (\Gamma' + \sbjtlpp{\Lie^*(G'')}{-r})\) is non-empty, then \(\bG\primeconn\) equals \(\bG\twosup{\prime\prime}\connsup\).
\end{lem}

\begin{proof}
Fix \(X^* \in (\Gamma + \sbjtlpp{\Lie^*(G')}{-r}) \cap (\Gamma' + \sbjtlpp{\Lie^*(G'')}{-r})\).
We may, and do, assume, upon replacing \field by a finite extension, that there is a Jordan decomposition \((X^*\semi, X^*\nilp)\) of \(X^*\) such that \(X^*\semi\) belongs to \(\Lie^*(G)\) and \(\Cent_\bG(X^*\semi)\) contains a maximal split torus \bT.
Lemma \initref{lem:ss-of-dual-blob}(\subref{fc}, \subref{roots}) shows that \bT is contained in \(\bG\primeconn \cap \bG\twosup{\prime\prime}\connsup\), and that \(\Root(\bG', \bT)\) and \(\Root(\bG'', \bT)\) both equal \(\set{\alpha \in \Root(\bG, \bT)}{\pair{X^*\semi}{\upd\alpha^\vee(1)} \in \sbjtlpp\field{-r}}\).  Since \(\bG\primeconn\) and \(\bG\twosup{\prime\prime}\connsup\) are both generated by \bT and the root subgroups for \bT in \bG corresponding to roots in this common set \cite{borel:linear}*{Proposition 13.20}, they are equal, as desired.
\end{proof}

Lemma \ref{lem:Lie*-orbit} discusses analogues of \xcite{spice:asymptotic}*{Hypothesis \initxref{hyp:Z*}(\subxref{orbit})}.

\begin{lem}
\initlabel{lem:Lie*-orbit}\hfill
\begin{enumerate}
\item\sublabel{ad*}
If \(g \in \Lie(G)\) is such that
\[
\ad^*(g)(\Gamma + \sbjtlpp{\Lie^*(G')}{-r}) \cap (\Gamma + \sbjtlpp{\Lie^*(G')}{-r})
\]
is non-empty, then \(g\) belongs to \(\Lie(G')\).
\item\sublabel{Ad*}
The groups
\[
\gNorm_G(\Gamma + \sbjtlpp{\Lie^*(G')}{-r})
\qandq
\gNorm_G(\Lie^*(G')^{G\primeconn} \cap (\Gamma + \sbjtlpp{\Lie^*(G')}{-r}))
\]
are both contained in \(\gNorm_G(\bG\primeconn)\), and both equal the set of elements \(g \in G\) such that
\[
\Ad^*(g)(\Gamma + \sbjtlpp{\Lie^*(G')}{-r}) \cap (\Gamma + \sbjtlpp{\Lie^*(G')}{-r})
\]
is non-empty.
\end{enumerate}
\end{lem}

\begin{proof}
Put \(\UU' = \Gamma + \sbjtlpp{\Lie^*(G')}{-r}\) and \(\mf z\primedual = \Lie^*(G')^{G\primeconn}\).

For \locpref{Ad*}, let us say that \(g \in G\) intertwines \(\UU'\) when \(\Ad^*(g)\UU' \cap \UU'\) is non-empty.  Certainly, every element of \(\gNorm_G(\UU')\) and every element of \(\gNorm_G(\mf z\primedual \cap \UU')\) intertwines \(\UU'\).

If \(g \in G\) intertwines \(\UU'\), then we have by Lemma \ref{lem:unique-approx*} that \(\Int(g)\bG\primeconn = \Cent_\bG(\Ad^*(g)\Gamma)\conn\) equals \(\bG\primeconn\); that is, \(g\) belongs to \(\gNorm_G(\bG\primeconn)\), and so normalises \(\sbjtlpp{\Lie^*(G')}{-r}\) and \(\mf z\primedual\).  Therefore, \((\Ad^*(g)\Gamma + \sbjtlpp{\Lie^*(G')}{-r}) \cap (\Gamma + \sbjtlpp{\Lie^*(G')}{-r}) = \Ad^*(g)\UU' \cap \UU'\) is non-empty, so \(\Ad^*(g)\Gamma - \Gamma \in \mf z\primedual\) belongs to \(\sbjtlpp{\Lie^*(G)}{-r} + \sbjtlpp{\Lie^*(G')}{-r}\).  In particular, for every \(x \in \BB(G')\), we have by Lemma \ref{lem:dual-MP-centre} that \(\Ad^*(g)\Gamma - \Gamma\) belongs to \(\sbtlpp{\Lie^*(G')}x{-r}\), and hence that \(\Ad^*(g)\Gamma + \sbtlpp{\Lie^*(G')}x{-r}\) equals \(\Gamma + \sbtlpp{\Lie^*(G')}x{-r}\).  Therefore, \(\Ad^*(g)\Gamma + \sbjtlpp{\Lie^*(G')}{-r} = \Ad^*(g)\UU'\) equals \(\UU'\).  That is, \(g\) belongs to \(\gNorm_G(\UU')\), and therefore, since it normalises \(\mf z\primedual\), also to \(\gNorm_G(\mf z\primedual \cap \UU')\).
This shows \locpref{Ad*}.

For \locpref{ad*}, we may, and do, assume, upon passing to a tame extension, that \(\bG\primeconn\) is a Levi subgroup of \(\bG\conn\).
Fix \(X^*\) in \(\Gamma + \sbjtlpp{\Lie^*(G')}{-r}\) such that \(\ad^*(g)X^*\) also belongs to \(\Gamma + \sbjtlpp{\Lie^*(G')}{-r}\).  In particular, \(\ad^*(g)X^*\) belongs to \(\Lie^*(G')\).

Write \(g = u^+ + g' + u^-\), where \(u^\pm\) belongs to \(\Lie(\Prad^\pm)\).  Then \(\ad^*(u^\pm)(\Gamma + \sbjtlpp{\Lie^*(G')}{-r})\) is contained in \(\Lie^*(\Prad^\pm)\), but \(\ad^*(u^+)X^* + \ad^*(g')X^* + \ad^*(u^-)X^* \in \ad^*(u^+)X^* + \ad^*(u^-)X^* + \Lie^*(G')\) belongs to \(\Lie^*(G')\); so \(\ad^*(u^\pm)X^*\) equal \(0\).

Suppose that \(u^+\) does not equal \(0\).  Choose \(x \in \BB(G')\) such that \(X^*\) belongs to \(\Gamma + \sbtlpp{\Lie^*(G')}x{-r}\), and write \(d = \depth_x(u^+)\) for the depth of \(u^+\) with respect to \(x\).  Then \(\ad^*(u^+)X^* = 0\) belongs to \(\ad^*(u^+)\Gamma + \sbtlpp{\Lie^*(\Prad^+)}x{d - r}\), so that \(\ad^*(u^+)\Gamma\) belongs to \(\sbtlpp{\Lie^*(\Prad^+)}x{d - r}\).  Remember that \(\bG'\) is split.  Let \bT be a split, maximal torus in \(\bG'\).  For every \(\alpha \in \Root(\bP^+, \bT) \setminus \Root(\bG', \bT)\), we have that the \(x\)-depth of the \(\alpha\)-component of \(u^+\) is the \(x\)-depth of the (\(-\alpha\))-component of \(\ad^*(u^+)\Gamma\), which is strictly greater than \(d - r\), minus the valuation of \(\pair\Gamma{\upd\alpha(1)}\), which is precisely \(-r\) by \xcite{spice:asymptotic}*{Hypothesis \initxref{hyp:Z*}(\subxref{good})}.  That is, each component of \(u^+\) lives at \(x\)-depth greater than \(d\), so \(u^+\) itself lives at \(x\)-depth greater than \(d\), which is a contradiction.

A similar argument shows that \(u^-\) equals \(0\).
\end{proof}

\begin{cor}
\label{cor:connect-G'}
The group \(G\primeconn\) stabilises \(\Gamma + \sbjtlpp{\Lie^*(G')}{-r}\).
\end{cor}

Lemma \ref{lem:disconnect-G'} provides the key to defining a suitable group \(\bG'\) whose identity component is \(\Cent_\bG(\Gamma)\conn\).  In fact, the actual definition (Definition \ref{defn:G'}) is anti-climactic, and just defines \(\bG'\) to be the group \(\bJ'\) already defined in Lemma \ref{lem:disconnect-G'}.  The only reason to wait to make the definition is that we want to know for sure that the identity component of the group that we want to denote by \(\bG'\) is the group already denoted by \(\bG\primeconn\).

\begin{lem}
\label{lem:disconnect-G'}
Let \bT be a maximal torus in \(\bG\primeconn\), and write \(\bar\Gamma\) for the element of \(\Hom_\Z(\cclat(\bT_\sepfield), \sepfield/\sbjtlpp{\field\sep}{-r})\) defined by \(\pair{\bar\Gamma}\lambda = \pair\Gamma{\upd\lambda(1)}\) for all \(\lambda \in \cclat(\bT_\sepfield)\).  There is a natural action of \(\Weyl(\bG, \bT)(\sepfield)\) on \(\Hom_\Z(\cclat(\bT_\sepfield), \sepfield/\sbjtlpp{\field\sep}{-r})\) by pre-composition.  Let \(\bJ'\) be the descent to \field of the subgroup of \(\bG_\sepfield\) generated by \(\bG\primeconn_\sepfield\) and the pullback to \(\gNorm_\bG(\bT)(\sepfield)\) of the subgroup of \(\Weyl(\bG, \bT)(\sepfield)\) that fixes \(\bar\Gamma\).

Then \(\bJ'\) is independent of the choice of \bT, has identity component \(\bG\primeconn\), and satisfies \(J' = \gNorm_G(\Gamma + \sbjtlpp{\Lie^*(G')}{-r})\).
\end{lem}

\begin{proof}
Independence of \bT follows from the fact that any two choices are \(\bG\primeconn(\sepfield)\)-conjugate.

Since everything is defined by descent from \sepfield, we may, and do, assume, upon replacing \field by a finite, separable extension, that \bT is split.  Then every element of \(\Weyl(\bG, \bT)(\sepfield)\) admits a lift to \(\gNorm_G(\bT)\).

Lemma \ref{lem:dual-MP-by-cclat} gives that the fixer in \(\gNorm_G(\bT)\) of \(\bar\Gamma\) is the stabiliser there of \(\Gamma + \sbjtlpp{\Lie^*(T)}{-r}\).

Put \(\UU' = \Gamma + \sbjtlpp{\Lie^*(G')}{-r}\).  Since \(\sbjtlpp{\Lie^*(T)}{-r}\) is contained in \(\sbjtlpp{\Lie^*(G')}{-r}\), and since Corollary \ref{cor:dual-centre} implies that \(\UU'\) equals \((\Lie^*(\bG')^{\bG\primeconn}(\field) \cap (\Gamma + \sbjtlpp{\Lie^*(G')}{-r})) + \sbjtlpp{\Lie^*(T)}{-r}\), Lemma \initref{lem:Lie*-orbit}\subpref{Ad*} gives that \(n \in \gNorm_G(\bT)\) fixes \(\bar\Gamma\) if and only if it belongs to \(\gNorm_G(\UU')\), in which case \(n\) belongs to \(\gNorm_G(\bG\primeconn)\).

That is, \(\bJ'\) is generated by \(\gNorm_G(\UU')\), which is contained in \(\gNorm_G(\bG\primeconn)\), and \(\bG\primeconn\).  In particular, \(\bJ'\) is contained in \(\gNorm_\bG(\bG\primeconn)\).  Since this latter group admits \(\bG\primeconn\) as its identity component (because \(\bG\primeconn\) contains a maximal torus in \bG), so does \(\bJ'\).  Thus every component of \(\bJ'\) contains an element of \(\gNorm_G(\UU')\), so \(J'\) is generated by \(\gNorm_G(\UU')\) and \(H\primeconn = G\primeconn\).  Finally, Corollary \ref{cor:connect-G'} gives that \(J'\) equals \(\gNorm_G(\UU')\).
\end{proof}

For almost all of \S\ref{sec:orbits}, we do not need the group \(\bG'\) defined in Definition \ref{defn:G'}; it is usually used only through its identity component \(\bG\primeconn = \Cent_\bG(\Gamma)\conn\).  Its only use here is for ensuring (thanks to Lemma \ref{lem:disconnect-G'}) that sums like the one appearing in Theorem \ref{thm:orb-to-orb'} have no `overlap'; that is, in the notation of that theorem, that no \(H\conn\)-orbit can occur as a summand corresponding to two different \((G', H\conn)\)-double cosets.

\begin{defn}
\label{defn:G'}
Write \(\bG'\) for the group \(\bJ'\) defined in Lemma \ref{lem:disconnect-G'}.
\end{defn}

In many situations, the element \(\bar\Gamma\) of Lemma \ref{lem:disconnect-G'} will take values in \(\sbjat{\field\sep}{-r}\).  Example \ref{exa:GE2} shows that this need not always happen.

\begin{exa}
\label{exa:GE2}
Let \(p\) be a prime, possibly \(p = 2\).  Let \field be an unramified extension of \(\Q_p\) that is sufficiently large that there are \(p\) elements \(a_1, \dotsc, a_p\) in \(\sbjtl\field 0\) whose sum is \(-1\), and whose images in \(\sbjat\field 0\) are distinct.  (If \(p\) equals \(2\), then we can take \(\field = \Q_2\).  Otherwise, any proper unramified extension will do.)  Put \(E = \field[\sqrt[p]p]\).

Label the weights of the diagonal maximal torus in \(\GL_{p, E}\) in the defining representation by \(\epsilon_1, \dotsc, \epsilon_p\), in the obvious order, and write \(\epsilon_1^\vee, \dotsc, \epsilon_p^\vee\) for the elements of the dual basis of cocharacters of the diagonal maximal torus.

Now write \bG for the Weil restriction \(\operatorname{Res}_{E/\field}\PGL_{p, E}\), \(\bB^+\) for the upper-triangular Borel subgroup of \bG, and \bT for the diagonal maximal torus in \bG.  Upon choosing an almost-simple component \(\bG_1\) of \(\bG_E\), we may, and do, regard \(\epsilon_1^\vee, \dotsc, \epsilon_p^\vee\) as cocharacters of \(\bG_1 \cap \bT_E\), no longer \Z-linearly independent but now satisfying \(\epsilon_1^\vee + \dotsb + \epsilon_p^\vee = 0\), that generate \(\cclat(\bT_E)\) as a \(\Z[\Gal(E/\field)]\)-module.

There is a unique element \(\Gamma\) of \(\Lie^*(T)\) such that \(\pair\Gamma{\upd\epsilon_i^\vee(1)}\) equals \(p\inv + a_i\) for all \(i \in \sset{1, \dotsc, p}\).  Then \(\pair\Gamma{\upd(\epsilon_i^\vee - \epsilon_j^\vee)(1)}\) equals \(a_i - a_j\), and hence has valuation \(0\), for every \(i, j \in \sset{1, \dotsc, p}\) with \(i \ne j\).  That is, \(\Gamma\) satisfies \xcite{spice:asymptotic}*{Hypothesis \initxref{hyp:Z*}(\subxref{central}, \subxref{good})}, with associated group \(\bG\primeconn = \bT\) (with \(r = 0\), or this example can be adapted so that \(r\) is strictly positive by replacing \(\Gamma\) by an appropriate scalar multiple).

Write \(\tr_{E/\field}\) for the natural map \anonmap{\Lie(\bT)(E)}{\Lie(T)}.
Since \bT is induced, there is only one choice of admissible filtration \cite{kaletha-prasad:bt-theory}*{Proposition B.10.5}.  Namely, for all \(d \in \tR\), we have that \(\sbjtl{\Lie^*(T)}d\) is spanned, as a \(\sbjtl\field 0\)-module, by \(\bigcup_{i = 1}^p \tr_{E/\field}(\upd\epsilon_i^\vee(\sbjtl E d))\).  In particular, since \(\pair\Gamma{\tr_{E/\field}(\upd\epsilon_i^\vee(t))}\) equals \(\tr_{E/\field}((p\inv + a_i)t)\), which lies in \(\sbjtlp\field 0\) for all \(t \in \sbjtlp E 0\) and equals \(1 + a_i p \notin \sbjtlp\field 0\) for \(t = 1 \in \sbjtl E 0\), we have that \(\Gamma\) belongs to \(\sbjtl{\Lie^*(T)}0 \setminus \sbjtlp{\Lie^*(T)}0\).

There is no contradiction of \cite{yu:supercuspidal}*{Lemma 8.1} in the fact that \(\Gamma\) does not also satisfy \cite{yu:supercuspidal}*{\S8, \textbf{GE2}}, since the residual characteristic \(p\) of \field is a torsion prime for the dual absolute root datum of \bG.  However, we note that \cite{yu:supercuspidal}*{\S8, \textbf{GE2}} cannot even be \emph{stated}, since, upon taking \(\varpi_r = 1\), we have that the image in \(\clat(\bT_\sepfield) \otimes_{\mathbb Z} \field\) of \(\varpi_r\Gamma = \Gamma \in \Lie^*(T)\) does not belong to \(\clat(\bT_\sepfield) \otimes_{\mathbb Z} \sbjtl{\field\alg}0\) (because, for all \(i \in \sset{1, \dotsc, p}\), we have that \(\pair\Gamma{\upd\epsilon_i^\vee(1)} = p\inv + a_i\) does not lie in \(\sbjtl{\field\alg}0\)).
\end{exa}

\begin{rem}
\label{rem:GE2}
Preserve the notation of Lemma \ref{lem:disconnect-G'}.
If \(\bar\Gamma\) takes values in \(\sbjat{\field\sep}{-r}\) (for example, if \bG is simply connected), then we may choose a valuation-\(r\) element \(\varpi_r \in \sepfield\), and consider \(\varpi_r\bar\Gamma\) as an element of \(\Hom_\Z(\cclat(\bT_\sepfield), \sbjat{\field\sep}0)\).  Then we are in the setting of \cite{yu:supercuspidal}*{\S8}, so that, for example, \loccit*{\textbf{GE2}} holds if and only if \(\bJ'\) is connected.
\end{rem}

%

For Theorem \ref{thm:orb-asymptotic-exists}, we will have to assume that
Hypothesis \ref{hyp:Lie-gamma} is satisfied by \((\gamma, x, r)\)
for \emph{all} \(x \in \BB(H)\).
Until then, we fix
a point \(\mnotn x \in \BB(H)\) such that
\((\gamma, x, r)\) satisfies Hypothesis \ref{hyp:Lie-gamma}.

\begin{lem}[\xcite{spice:asymptotic}*{Lemma \xref{lem:fa-explicit}}]
\initlabel{lem:fa-explicit}
For every
	\begin{itemize}
	\item nested, tame, twisted Levi sequence \(\vbG = (\bG^0 \subseteq \dotsb \subseteq \bG^{\ell - 1} \subseteq \bG^\ell = \bG)\) in \bG such that \(\gamma\) belongs to \(\Lie(G^0)\) and \(x\) belongs to \(\BB(G^0)\),
	\item concave depth vector \(\vec a\),
and	\item \(h \in \sbtl{\Lie(H)}x r\),
	\end{itemize}
we have that
\begin{align*}
\uint_{\gamma + h + \sbtl{\Lie(\vG)}x{\vec a + r}} &F(g)\upd g \\
\intertext{equals}
\uint_{\sbtl{(\vH, \vG)}x{(\Rp{\vec a}, \vec a + r - \ord_\gamma)}} \uint_{\gamma + h + \sbtl{\Lie(\vH)}x{\vec a + r}} &F(\Int(g)m)\upd m\,\upd g
\end{align*}
for all \(F \in \Hecke(\Lie(G))\).
\end{lem}

Although we are assuming in \S\ref{sec:orbits} that \bG is connected, we state Lemma \ref{lem:centre} in terms of \(G\conn\) in order better to parallel \xcite{spice:asymptotic}*{Lemma \xref{lem:centre}}.
The notation \(\sbtl{\Lie(\vH, \vG)}x{(\Rp{\vec a}, \vec a) + r}\) in Lemma \ref{lem:centre} indicates the sublattice associated to the depth specification \((t_{v i j})_{v, i, j}\), where, for all \(v \in \R\) and \(j \in \sset{0, \dotsc, \ell}\), we put \(t_{v i j} = a_j\) for all \(i \in \R\) with \(i < r\), and \(t_{v r j} = \Rp{a_j}\).

\begin{lem}[\xcite{spice:asymptotic}*{Lemma \xref{lem:centre}}]
\initlabel{lem:centre}
If \(T\) is a \(G\conn\)-invariant distribution on \(\Lie(G)\), then there is a unique \(H\conn\)-invariant distribution \(T_\gamma\) on \(\Lie(H)\), supported by \(\gamma + \sbjtl{\Lie(H)}r\), such that, for all tame, twisted Levi sequences \vbG containing \(\gamma\) and all concave depth vectors \(\vec a\), we have that
\[
T_\gamma\bigl(\gamma + h + \chrc{\sbtlp{\Lie(\vH)}x{\vec a + r}}\bigr)
\qeqq
T\bigl(\gamma + g + h + \chrc{\sbtlp{\Lie(\vG)}x{\vec a + r}}\bigr)
\]
for all \(h \in \sbtl{\Lie(H)}x r\) and \(g \in \sbtl{\Lie(\vH, \vG)}x{(\Rp{\vec a}, \vec a) + r}\).
\end{lem}

\begin{cor}[\xcite{spice:asymptotic}*{Corollary \xref{cor:centre}}]
\label{cor:centre}
With the notation of Lemma \ref{lem:centre}, for every \(a \in \tR\) with \(r \le a < \infty\) and every \(X^* \in \sbtl{\Lie^*(H)}x{-a}\), we have that
\[
T_\gamma\bigl(\gamma + h + \chrc{\sbtl{\Lie(H)}x a, \contra{\AddChar_{X^*}}}\bigr)
\qeqq
T\bigl(\gamma + h + \chrc{\sbtl{\Lie(G)}x a, \contra{\AddChar_{X^*}}}\bigr)
\]
for all \(h \in \sbtl{\Lie(H)}x r\).
\end{cor}

\begin{defn}[\xcite{spice:asymptotic}*{Definition \xref{defn:centred-char}}]
\label{defn:centred-orbit}
If \(\xi\) belongs to \(\Lie^*(H)\), then we define \matnotn{Ocheck}{\widecheck\muhat^G_{\xi, \gamma}} and \matnotn{Ocheck}{\widecheck\muhat^G_{\xi, \gamma, \Gamma}} to be the distributions on \(\Lie^*(H)\) given for all \(f^* \in \Hecke(\Lie^*(H))\) by
\begin{align*}
\widecheck\muhat^G_{\xi, \gamma}(f^*)        & {}= \muhat^G_{\xi, \gamma}\bigl(\gamma + \widecheck f^*) \\
\intertext{and}
\widecheck\muhat^G_{\xi, \gamma, \Gamma}(f^*) & {}= \widecheck\muhat^G_{\xi, \gamma}(f^*_\Gamma),
\end{align*}
where we have introduced the \textit{ad hoc} notation \(f^*_\Gamma\) for the function that vanishes outside \(\Lie^*(H) \cap \Ad^*(H\conn)(\Gamma + \sbjtlpp{\Lie^*(G')}{-r})\), and agrees with \(f^*\) on that domain; and where \(\muhat^G_{\xi, \gamma}\) is the distribution \(T_\gamma\) deduced from \(T = \muhat^G_\xi\) in Lemma \ref{lem:centre}.
\end{defn}

\begin{rem}
\label{rem:wrong-descent}
With the notation of Definition \ref{defn:centred-orbit}, it might seem natural to guess that \(T_\gamma = \Ohat^G_{\xi, \gamma}\) should equal \(\Ohat^{H\conn}_\xi\), possibly up to some minor adjustments, but this is not true.  Notice that \(\Ohat^G_{\xi, \gamma}\) is represented on the set of regular, semisimple elements in \(\gamma + \sbjtl{\Lie(H)}r\) by the restriction of the function representing \(\Ohat^G_\xi\).  If, for example, \(\bG = \operatorname{SL}_2\) and \(\gamma\) is a split, regular, semisimple element, so that \bH is a split, maximal torus in \bG, then \(\Ohat^H_\xi\) is represented by \(\AddChar_\xi\).  The explicit formul{\ae} of \cite{spice:sl2-mu-hat}*{Theorem 11.3} (which is a special case of our Theorem \ref{thm:orb-unwind}) show that there is no obvious modification of the representing function for \(\Ohat^G_{\xi, \gamma}\) that will make them agree.
\end{rem}

\begin{lem}[\xcite{spice:asymptotic}*{Lemma \xref{lem:sample}}]
\label{lem:orb-sample}
With the notation of Definition \ref{defn:centred-orbit}, if
	\(a \in \tR\) satisfies \(r \le a < \infty\)
and	
	\(X^*\) belongs to \(\sbtl{\Lie^*(H)}x{-a}\),
then
\begin{align*}
&\widecheck\muhat^G_{\xi, \gamma}\bigl(X^* + \chrc{\sbtlpp{\Lie^*(H)}x{-r}}\bigr) \\
\intertext{equals}
\meas(\sbtl{\Lie(H)}x a)\sum_{h \in \sbtl{\Lie(H)}x r/{\sbtl{\Lie(H)}x a}} \contra{\AddChar_{X^*}}(h)&
\muhat^G_\xi(\gamma + h + \chrc{\sbtl{\Lie(G)}x a, \contra{\AddChar_{X^*}}}).
\end{align*}
\end{lem}

For the remainder of \S\ref{subsec:orbital-exists}, we impose Hypothesis \ref{hyp:mexp}, with \((\bM, \bJ, R_{-1}) = (\bG, \bG, \Rp0)\); and
\xcite{spice:asymptotic}*{Hypothesis \xref{hyp:depth}}.

\begin{cor}[\xcite{spice:asymptotic}*{Corollary \xref{cor:sample}}]
\initlabel{cor:orb-sample}
With the notation and hypotheses of Lemma \ref{lem:orb-sample}, if \(\xi\) lies in \(\Gamma + \sbjtlpp{\Lie^*(G)}{-r}\), then
\[
\widecheck\muhat^G_{\xi, \gamma}\bigl(X^* + \chrc{\sbtlpp{\Lie^*(H)}x{-r}}\bigr)
\]
vanishes unless \(\Ad^*(G)\xi\) intersects \(X^* + \sbtlpp{\Lie^*(H)}x{-a}\); in particular, unless we have \(a > r\) and \(X^* + \sbtlpp{\Lie^*(H)}x{-a}\) is degenerate, or we have \(a = r\) and there is some \(g \in G\) so that \(X^*\) belongs to \(\Ad^*(g)\inv\Gamma + \sbjtlpp{\Lie^*(H \cap \Int(g)\inv G')}{-r}\).
\end{cor}

\begin{proof}
For each \(h \in \sbjtl{\Lie(H)}r\), the Fourier transform of \(\gamma + h + \chrc{\sbtl{\Lie(G)}x a, \contra{\AddChar_{X^*}}}\) is 
supported by \(X^* + \sbtlpp{\Lie^*(G)}x{-a}\); so, by Lemma \ref{lem:orb-sample}, the sampled value vanishes unless \(\Ad^*(G)\xi\) intersects \(X^* + \sbtlpp{\Lie^*(H)}x{-a}\).

If \(a\) is strictly greater than \(r\), then, since \(\xi\) lies in \(\sbjtl{\Lie^*(G)}{-r}\), we have by \xcite{spice:asymptotic}*{Remark \xref{rem:depth}} that \(X^* + \sbtlpp{\Lie^*(H)}x{-a}\) is degenerate.  If \(a\) equals \(r\), then there is some \(g \in G\) such that \(X^* + \sbtlpp{\Lie^*(G)}x{-r}\) intersects \(\Ad^*(g)\inv(\Gamma + \sbjtlpp{\Lie^*(G)}{-r})\), and it follows exactly as in the proof of \xcite{spice:asymptotic}*{Corollary \xref{cor:sample}} (see also Remark \ref{rem:asymptotic:cor:sample}) that \(X^*\) belongs to \(\Ad^*(g)\inv\Gamma + \sbjtlpp{\Lie^*(H \cap \Int(g)\inv G')}{-r}\).
\end{proof}

We now drop our choice of a particular point \(x \in \BB(H)\)
made before Lemma \ref{lem:fa-explicit}, and require that
Hypothesis \ref{hyp:Lie-gamma} is satisfied by \((\gamma, x, r)\)
for \emph{all} \(x \in \BB(H)\).

Theorem \ref{thm:orb-asymptotic-exists} is, among other things, a generalisation of the Shalika-germ expansion \cite{shalika:germs}*{Theorem 2.1.1} that simultaneously makes precise its range of validity, and allows us to re-centre it at semisimple elements other than the origin.  Indeed, we may take \(\Gamma = 0\) and \(\bG' = \bG\) to find that, for \(\xi \in \sbjtlpp{\Lie^*(G)}{-r}\), the Shalika-germ expansion for \(\Ohat^G_\xi\) centred at \(\gamma\) is valid on \(\sbjtl{\Lie(H)}r\).  The case \(\gamma = 0\) of Theorem \ref{thm:orb-asymptotic-exists} is \cite{jkim-murnaghan:gamma-asymptotic}*{Corollary 9.2.6}.

As with \xcite{spice:asymptotic}*{Theorem \xref{thm:asymptotic-exists}}, and Theorems \ref{thm:orb-to-orb'}, \ref{thm:char-asymptotic-exists}, \ref{thm:pi-to-pi'}, \ref{thm:orb-unwind}, and \ref{thm:char-unwind} later, we impose \cite{jkim-murnaghan:charexp}*{Theorem 3.1.7(1, 5)} as a hypothesis of Theorem \ref{thm:orb-asymptotic-exists}.  See Remark \ref{rem:black-box}.

\begin{thm}[\xcite{spice:asymptotic}*{Theorem \xref{thm:asymptotic-exists}}]
\label{thm:orb-asymptotic-exists}
Suppose that
\cite{jkim-murnaghan:charexp}*{Theorem 3.1.7(1, 5)} is satisfied, and
all of the relevant orbital integrals converge.  For every \(\xi \in \Gamma + \sbjtlpp{\Lie^*(G')}{-r}\), there is a finitely supported, \(\OO^{H\conn}(\Ad^*(G)\Gamma)\)-indexed vector \(b(\xi, \gamma)\) so that
\[
\Ohat^G_\xi(\gamma + Y) \qeqq
\sum_{\OO \in \OO^{H\conn}(\Ad^*(G)\Gamma)}
	b_\OO(\xi, \gamma)\Ohat^{H\conn}_\OO(Y)
\]
for all \(Y \in \Lie(H)\rss \cap \sbjtl{\Lie(H)}r\).
\end{thm}

\begin{rem}
\label{rem:which-orb}
Put \(\bH' = \bG' \cap \bH\) and \(\UU\primedual = \Gamma + \sbjtlpp{\Lie^*(H')}{-r}\).

For a given finitely supported \(\OO^{H\conn}(\Ad^*(G)\UU\primedual)\)-indexed vector \(b(\xi, \gamma)\) and element \(Y \in \Lie(H)\rss \cap \sbjtl{\Lie(H)}r\), the equations
\[
\Ohat^G_\xi(\gamma + Y)
= \sum_{\OO \in \OO^{H\conn}(\Ad^*(G)\UU\primedual)}
	b_\OO(\xi, \gamma)\Ohat^{H\conn}_\OO(Y)
\]
and
\begin{multline*}
\abs{\Disc_{G/H}(\gamma)}^{1/2}\abs{\redD_G(\xi)}^{1/2}\muhat^G_\xi(\gamma + Y) \\
= \sum_{\substack
	{g \in G'\bslash G/H\conn \\
	\Ad^*(g)\inv\Gamma \in \Lie^*(H)}
} \sum_{\OO \in \OO^{H\conn}(\Ad^*(g)\inv\UU\primedual)}
	\abs{\redD_H(\OO)}^{1/2}
	b_\OO(\xi, \gamma)\muhat^{H\conn}_\OO(Y)
\end{multline*}
are equivalent.
(This is one of the few occasions when we genuinely need to work with the group \(\bG'\) of Definition \ref{defn:G'}, not just \(\bG\primeconn = \Cent_\bG(\Gamma)\conn\), to ensure that there is no `overlap' in the latter sum.)

We thus need not distinguish between the \emph{existence} of asymptotic expansions of \(\Ohat^G_\xi\) and of \(\muhat^G_\xi\) (although, of course, the coefficients will differ).
\end{rem}

\begin{proof}
We use Remark \ref{rem:which-orb} to allow us to work with \(\muhat^G_\xi\) and \(\muhat^{H\conn}_\OO\) rather than \(\Ohat^G_\xi\) and \(\Ohat^{H\conn}_\OO\) (with different coefficients).

As in, and with the notation of, the proof of \xcite{spice:asymptotic}*{Theorem \xref{thm:asymptotic-exists}} (and Definition \ref{defn:centred-orbit}), we use Corollary \ref{cor:orb-sample} to show that
\[
\widecheck\muhat^{G\,0}_{\xi, \gamma}
\ldef \widecheck\muhat^G_{\xi, \gamma}
- \sum_{\substack
	{g \in G'\bslash G/H\conn \\
	\Ad^*(g)\inv\Gamma \in \Lie^*(H)}
}
	\widecheck\muhat^G_{\xi, \gamma, \Ad^*(g)\inv\Gamma}
\]
vanishes on the function space \(\mc D^{-r}_{\Rpp{-r}}\) of \cite{jkim-murnaghan:charexp}*{Definition 3.1.1}, adapted from \(\Lie(G)\) to \(\Lie^*(H)\); and that \(\widecheck\muhat^G_{\xi, \gamma}\) and each \(\widecheck\muhat^G_{\xi, \gamma, \Ad^*(g)\inv\Gamma}\) lie in the distribution space \(\mc J^{\Ad^*(g)\inv\Gamma}_{\Rpp{-r}}\) of \cite{jkim-murnaghan:charexp}*{Definition 3.1.2(2)}, similarly adapted.
(As in Remark \ref{rem:which-orb}, we genuinely need to work with the group \(\bG'\) of Definition \ref{defn:G'}, not just \(\bG\primeconn = \Cent_\bG(\Gamma)\conn\), to ensure that there is no `overlap' in the sum defining \(\widecheck\muhat^{G\,0}_{\xi,\gamma}\).)
We then apply \cite{jkim-murnaghan:charexp}*{Theorem 3.1.7(5)} to find that each \(\widecheck\muhat^G_{\xi, \gamma, \Ad^*(g)\inv\Gamma}\) has an asymptotic expansion in terms of orbital integrals \(\mu^{H\conn}_\OO\) on the function space \(\mc D_{\Rpp{-r}}\) of \cite{jkim-murnaghan:charexp}*{Definition 3.1.1}, and then \cite{jkim-murnaghan:charexp}*{Theorem 3.1.7(1)} to conclude that \(\widecheck\muhat^G_{\xi, \gamma}\) equals the sum of these various asymptotic expansions.  Then, again, exactly as in the proofs of \cite{debacker:homogeneity}*{Theorem 3.5.2} and \cite{jkim-murnaghan:charexp}*{Theorem 5.3.1}, the desired equality of representing functions follows upon taking Fourier transforms.
\end{proof}

\subsection{Matching distributions on algebras and subalgebras}
\label{subsec:dist-g-to-g'}

Throughout \S\ref{subsec:dist-g-to-g'}, we continue to use the
	\begin{itemize}
	\item real number \(r\),
	\item element \(\Gamma \in \Lie^*(G)\) satisfying \xcite{spice:asymptotic}*{Hypothesis \initxref{hyp:Z*}(\subxref{central}, \subxref{good})},
	\item subgroup \(\bG'\) of Definition \ref{defn:G'},
	\item semisimple element \(\gamma \in \Lie(G)\) satisfying Hypothesis \ref{hyp:fc-building},
	\item group \(\bH = \CC\bG r(\gamma)\),
and	\item point \(x \in \BB(H)\) such that \((\gamma, x, r)\) satisfies Hypothesis \ref{hyp:Lie-gamma}
	\end{itemize}
from \S\ref{subsec:orbital-exists}.
Put \(\mnotn s = r/2\) and \(\mnotn{\bH'} = \bH \cap \bG'\).

Although we will eventually (in \S\ref{subsec:quantitative}) remove the requirement, the strategy of argument from \cite{spice:asymptotic}*{\S\xref{sec:dist}} requires that \(r\) be positive beginning with Lemma \ref{lem:r-to-s+}.  In Theorem \ref{thm:sample-orb-to-orb'}, we will also require that \(\gamma\) lie in \(\sbtl{\Lie(G)}x 0\).  Both of these assumptions will be removed in \S\ref{subsec:quantitative}.  We will soon (before Lemma \ref{lem:dist-g-to-g'}) require that \(\bH\conn\) equal \(\Cent_\bG(\gamma)\conn\), but, unlike in \S\ref{subsec:orbital-exists}, we do not do so yet.

Also fix an element \(X^* \in \Lie^*(G)\).  We will eventually (before Remark \ref{rem:disc-Gamma}) require that \(\gamma\) belong to \(\Lie(G')\) and \(X^*\) and \(x\) satisfy \xcite{spice:asymptotic}*{Hypothesis \xref{hyp:X*}}, but we do not do so yet.

\begin{notn}[\xcite{spice:asymptotic}*{Notation \xref{notn:Gauss}}]
\label{notn:Gauss}
We define
\[
\matnotn{bXgamma}{b_{X^*, \gamma}}(Y_1, Y_2) = \pair[\big]{X^*}{\comm{Y_1}{\comm{Y_2}\gamma}}
\qandq
\matnotn{qXgamma}{q_{X^*, \gamma}}(Y) = b_{X^*, \gamma}(Y, Y)
\]
for all \(Y, Y_1, Y_2 \in \Lie(G)\).  We write \matnotn{Gauss}{\Gauss_G(X^*, \gamma)} for the Weil index of the pairing \(q_{X^*, \gamma}\) on \(\Lie(G)\) \xcite{spice:asymptotic}*{Definition \xref{defn:Weil-index}}.  If \(X^*\) belongs to \(\Lie^*(H)\), then we put \(\matnotn{Gauss}{\Gauss_{G/H}(X^*, \gamma)} = \Gauss_G(X^*, \gamma)/\Gauss_H(X^*, \gamma)\).
\end{notn}

\begin{rem}
\label{rem:Lie-Gauss}
We have that
\begin{multline*}
b_{X^*, \gamma}(Y_1, Y_2) - b_{X^*, \gamma}(Y_2, Y_1)
\qeqq \\
\pair[\big]{X^*}{\comm{Y_1}{\comm{Y_2}\gamma}}
	+ \pair[\big]{X^*}{\comm{Y_2}{\comm\gamma{Y_1}}}
= \pair[\big]{\ad^*(\gamma)X^*}{\comm{Y_1}{Y_2}}
\end{multline*}
for all \(Y_1, Y_2 \in \Lie(G)\).  In particular, unlike in the group case \xcite{spice:asymptotic}*{Definition \xref{notn:Gauss}}, we have that \(b_{X^*, \gamma}\) is symmetric if \(\gamma\) centralises \(X^*\) (\ie, if \(\ad^*(\gamma)X^*\) equals \(0\)).  In general, we have that
\begin{multline*}
b_{X^*, \gamma}(Y_1, \anondot) + b_{X^*, \gamma}(\anondot, Y_1)
\qeqq \\
\bigl(-\ad^*(\gamma)\ad^*(Y_1) + \ad^*(\comm{Y_1}\gamma)\bigr)X^*
= \bigl(-2\ad^*(\gamma)\ad^*(Y_1) + \ad^*(Y_1)\ad^*(\gamma)\bigr)X^*
\end{multline*}
for all \(Y_1 \in \Lie(G)\).
\end{rem}

For the remainder of \S\ref{subsec:dist-g-to-g'}, we require that
\(\gamma\) belong to \(\Lie(G')\), and
\((X^*, x)\) satisfy \xcite{spice:asymptotic}*{Hypothesis \xref{hyp:X*}}.
The requirement that \(\gamma\) belong to \(\Lie(G')\) will be dropped in \S\ref{subsec:quantitative}.

\begin{rem}
\label{rem:disc-Gamma}
Let \(\xi\) be any element satisfying \xcite{spice:asymptotic}*{Hypothesis \xref{hyp:X*}}; for example, \(X^*\), or any element of \(\Gamma + \sbjtlpp{\Lie^*(G)}{-r}\) \xcite{spice:asymptotic}*{Remark \xref{rem:X*}}, including \(\Gamma\) itself.

For every \(a \in \R\), we have by \xcite{spice:asymptotic}*{Hypothesis \initxref{hyp:X*}(\subxref{Lie}) and Remark \xref{rem:X*}} that \(\bigl(\ad^*(\anondot)\Gamma\bigr)\inv \circ \bigl(\ad^*(\anondot)\xi\bigr)\), viewed as an endomorphism of \(\sbat{\Lie(G)}x a/{\sbat{\Lie(G')}x a}\), is the identity; so a standard limiting argument, as in, for example, \cite{jkim-murnaghan:charexp}*{Lemma 2.3.4}, shows that \(\bigl(\ad^*(\anondot)\Gamma\bigr)\inv \circ \bigl(\ad^*(\anondot)\xi\bigr)\), now viewed as an endomorphism of the lattice \(\sbtl{\Lie(G)}x 0/{\sbtl{\Lie(G')}x 0}\), is an isomorphism, and hence that its determinant is a unit in \(\sbjtl\field 0\).  Thus, by Lemma \ref{lem:disc-as-det}, we have that \(\abs{\Disc_{G/G'}(\xi)}\) equals \(\abs{\Disc_{G/G'}(\Gamma)}\), which, by \xcite{spice:asymptotic}*{Hypothesis \initxref{hyp:Z*}(\subxref{good})}, equals \(q_{G/G'}^r\), where we put \(q_G = q^{\dim(\bG)}\), \(q_{G'} = q^{\dim(\bG')}\), and \(q_{G/G'} = q_G/q_{G'}\), as before \xcite{spice:asymptotic}*{Lemma \initxref{lem:index-to-disc-X*}}.  In particular, \(\Disc_{G/G'}(\xi)\) is non-\(0\), and so equals \(\redD_G(\xi)\redD_{G'}(\xi)\inv\).  Similarly, \(\Disc_{H/H'}(\xi)\) equals \(\redD_H(\xi)/\redD_{H'}(\xi)\inv\), and has absolute value \(q_{H/H'}^r\), with the analogous notation.


We will use this to re-write results in \xcite{spice:asymptotic} involving quantities such as \(q_G^s\) and \(q_{G'}^s\), or \(q_H^s\) and \(q_{H'}^s\), in terms of the analogous quantities \(\abs{\redD_G(\xi)}^{1/2}\) and \(\abs{\redD_{G'}(\xi)}^{1/2}\), or \(\abs{\redD_H(\xi)}^{1/2}\) and \(\abs{\redD_{H'}(\xi)}^{1/2}\).  (Note that we are \emph{not} asserting that \(q_G^s\) equals \(\abs{\redD_G(\xi)}^{1/2}\), or the analogous fact on \(G'\), only that the ratios \(q_{G/G'}^s\) and \(\abs{\redD_G(\xi)}^{1/2}\abs{\redD_G(\xi)}\inv[1/2]\) are equal, and similarly for \(H\) and \(H'\).)

Keeping in mind that we may use these equalities for any element \(\xi\) satisfying \xcite{spice:asymptotic}*{Hypothesis \xref{hyp:X*}}, including \(\xi = \Gamma\) and \(\xi = X^*\),
compare
Lemma \ref{lem:r-to-s+} to \xcite{spice:asymptotic}*{Lemma \xref{lem:r-to-s+}};
Proposition \ref{prop:Gauss-appears} to \xcite{spice:asymptotic}*{Proposition \xref{prop:Gauss-appears}};
Proposition \ref{prop:dist-r-to-s+} to \xcite{spice:asymptotic}*{Proposition \xref{prop:dist-r-to-s+}};
Theorem \ref{thm:dist-g-to-g'} to \xcite{spice:asymptotic}*{Theorem \xref{thm:dist-G-to-G'}};
and
Lemma \ref{lem:dist-g-to-g'} to \xcite{spice:asymptotic}*{Lemma \xref{lem:dist-G-to-G'}}.
\end{rem}

\begin{prop}[\xcite{spice:asymptotic}*{Proposition \xref{prop:lattice-orth}}]
\label{prop:lattice-orth}
We have that \(\sbtl{\Lie(H, G)}x{(\Rp0, (r - \ord_\gamma)/2)}\) pairs \via \(b_{X^*, \gamma}\) with itself into \(\sbjtl\field 0\), and with \(\sbtl{\Lie(H, G)}x{(\Rp0, \Rpp{(r - \ord_\gamma)/2})}\) on the right or left into \(\sbjtlp\field 0\).

Further, the \(b_{X^*, \gamma}\)-orthogonal modulo \(\sbjtlp\field 0\) of
\[
\sbtl{\Lie(H, G', G)}x{(\infty, \infty, (r - \ord_\gamma)/2)}
\]
in \(\Lie(H)^\perp \cap \Lie(G')^\perp\) is
\[
\sbtlp{\Lie(H, G', G)}x{(\infty, \infty, (r - \ord_\gamma)/2)}.
\]
\end{prop}

\begin{cor}[\xcite{spice:asymptotic}*{Corollary \xref{cor:lattice-orth}}]
\label{cor:lattice-orth}
We have that
\begin{gather*}
\frac{\Gauss_{G/H}(X^*, \gamma)}{\Gauss_{G'/H'}(X^*, \gamma)} \\
\intertext{equals}
\card{\sbat{\Lie(H, G', G)}x{(\Rp0, \Rp0, (r - \ord_\gamma)/2)}}^{-1/2}\sum_Y \AddChar_{1/2}(q_{X^*, \gamma}(Y)),
\end{gather*}
where the sum over \(Y\) runs over
\[
\sbtl{\Lie(H, G', G)}x{(\Rp0, r - \ord_\gamma, (r - \ord_\gamma)/2)}
	/
{\sbtl{\Lie(H, G', G)}x{(\Rp0, r - \ord_\gamma, \Rpp{(r - \ord_\gamma)/2})}}.
\]
\end{cor}

\begin{cor}[\xcite{spice:asymptotic}*{Corollary \xref{cor:Gauss-const}}]
\label{cor:Gauss-const}
The quantity
\[
\frac{\Gauss_{G/H}(X^*, \gamma)}{\Gauss_{G'/H'}(X^*, \gamma)}
\]
does not change if we replace \(X^*\) by a translate under \(\sbtlpp{\Lie^*(H')}x{-r}\) and \(\gamma\) by a translate under \(\sbtl{\Lie(H')}x r\).
\end{cor}

We now impose Hypothesis \ref{hyp:MP-ad}.  This is needed in Proposition \ref{prop:Gauss-to-Weil}, and so the remainder of \S\ref{subsec:dist-g-to-g'} and \S\S\ref{subsec:quantitative}, \ref{subsec:orb-unwind}.

\begin{prop}[\xcite{spice:asymptotic}*{Proposition \xref{prop:Gauss-to-Weil}}]
\label{prop:Gauss-to-Weil}
We have that
\begin{align*}
&\card{\sbat{(H, G', G)}x{(\Rp0, \Rp0, (r - \ord_\gamma)/2)}}^{1/2}\times{} \\
&\qquad\uint_{\sbtl{(H, G', G)}x{(\Rp0, r - \ord_\gamma, (r - \ord_\gamma)/2)}}
	\AddChar_{X^*}((\Ad(v) - 1)\gamma)\upd v \\
\intertext{equals}
&\frac{\Gauss_{G/H}(X^*, \gamma)}{\Gauss_{G'/H'}(X^*, \gamma)}.
\end{align*}
\end{prop}

For the remainder of \S\ref{subsec:dist-g-to-g'}, we require that \(r\) be positive.

\begin{lem}[\xcite{spice:asymptotic}*{Lemma \xref{lem:r-to-s+}}]
\label{lem:r-to-s+}
We have that
\begin{align*}
\abs{\redD_G(X^*)}\inv[1/2]
\card{\sbat{\Lie(G)}x 0}^{1/2}
\card{\sbat{\Lie(G)}x s}^{1/2}
&\chrc{\sbtl{\Lie(G)}x r, \contra{\AddChar_{X^*}}} \\
\intertext{equals}
\abs{\redD_{G'}(X^*)}\inv[1/2]
\card{\sbat{\Lie(G')}x 0}^{1/2}
\card{\sbat{\Lie(G')}x s}^{1/2}
&\uint_{\sbtlp G x 0} g\inv\chrc{\sbtl{\Lie(G', G)}x{(r, \Rp s)}, \contra{\AddChar_{X^*}}}g\,\upd g.
\end{align*}
\end{lem}

\begin{prop}[\xcite{spice:asymptotic}*{Proposition \xref{prop:Gauss-appears}}]
\label{prop:Gauss-appears}
For every \(f \in \Hecke(\Lie(G)\sslash{\sbtl{\Lie(G', G)}x{(r, \Rp s)}, \AddChar_{X^*}})\), we have that
\begin{align*}
&\frac{\abs{\redD_G(X^*)}^{1/2}}{\abs{\redD_H(X^*)}^{1/2}}
\abs{\Disc_{G/H}(\gamma)}^{1/2}
\indx{\sbat{\Lie(G)}x 0}{\sbat{\Lie(H)}x 0}\inv[1/2]
\Gauss_{G/H}(X^*, \gamma)\inv\times{} \\
&\qquad\uint_{\sbtlp G x 0} f(\Int(g)\gamma)\upd g \\
\intertext{equals}
&\abs{\Disc_{G'/H'}(\gamma)}^{1/2}
\indx{\sbat{\Lie(G')}x 0}{\sbat{\Lie(H')}x 0}^{-1/2}
\Gauss_{G'/H'}(X^*, \gamma)\inv\times{} \\
&\qquad\uint_{\sbtl{(G', G)}x{(\Rp0, s)}} f(\Int(j)\gamma)\upd j.
\end{align*}
\end{prop}

\begin{prop}[\xcite{spice:asymptotic}*{Proposition \xref{prop:dist-r-to-s+}}]
\label{prop:dist-r-to-s+}
If \(T\) is an invariant distribution on \(\Lie(G)\), then
\begin{align*}
&\abs{\redD_H(X^*)}\inv[1/2]
\abs{\Disc_{G/H}(\gamma)}^{1/2}
\card{\sbat{\Lie(H)}x 0}^{1/2}
\card{\sbat{\Lie(G)}x s}^{1/2}
\Gauss_{G/H}(X^*, \gamma)\inv\times{} \\
&\qquad T\bigl(\gamma + \chrc{\sbtl{\Lie(G)}x r, \contra{\AddChar_{X^*}}}\bigr) \\
\intertext{equals}
&\abs{\redD_{H'}(X^*)}\inv[1/2]
\abs{\Disc_{G'/H'}(\gamma)}^{1/2}
\card{\sbat{\Lie(H')}x 0}^{1/2}
\card{\sbat{\Lie(G')}x s}^{1/2}
\Gauss_{G'/H'}(X^*, \gamma)\inv\times{} \\
&\qquad T\bigl(\gamma + \chrc{\sbtl{\Lie(G', G)}x{(r, \Rp s)}, \contra{\AddChar_{X^*}}}\bigr).
\end{align*}
\end{prop}

\begin{thm}[\xcite{spice:asymptotic}*{Theorem \xref{thm:dist-G-to-G'}}]
\label{thm:dist-g-to-g'}
Suppose that
	\begin{itemize}
	\item \(T'\) is an invariant distribution on \(\Lie(G')\),
	\item \(c(T', \gamma)\) is a finitely supported, \(\OO^{H\primeconn}(\Lie^*(H'))\)-indexed vector of complex numbers such that \(\abs{\Disc_{G'/H'}(\gamma)}^{1/2}T'(\gamma + \chrc{\sbtl{\Lie(G')}x r, \contra{\AddChar_{X^*}}})\) equals
\[
\abs{\redD_{H'}(\Gamma)}^{1/2}\Gauss_{G'/H'}(X^*, \gamma)
\sum_{\OO' \in \OO^{H\primeconn}(\Lie^*(H'))}
	c_{\OO'}(T', \gamma)\muhat^{H\primeconn}_{\OO'}\bigl(\chrc{\sbtl{\Lie(H')}x r, \AddChar_{X^*}}\bigr),
\]
and	\item \(T\) is an invariant distribution on \(\Lie(G)\) such that
\begin{align*}
&
\card{\sbat{\Lie(G)}x s}\inv[1/2]
T\bigl(\gamma + \chrc{\sbtl{\Lie(G', G)}x{(r, \Rp s)}, \contra{\AddChar_{X^*}}}\bigr) \\
\intertext{equals}
&
\card{\sbat{\Lie(G')}x s}\inv[1/2]
T'\bigl(\gamma + \chrc{\sbtl{\Lie(G')}x r, \contra{\AddChar_{X^*}}}\bigr).
\end{align*}
	\end{itemize}
Then \(\abs{\Disc_{G/H}(\gamma)}^{1/2}T(\gamma + \chrc{\sbtl{\Lie(G)}x r, \contra{\AddChar_{X^*}}})\) equals
\[
\abs{\redD_H(\Gamma)}^{1/2}
\Gauss_{G/H}(X^*, \gamma)
\sum_{\OO' \in \OO^{H\primeconn}(\Lie^*(H'))}
	c_{\OO'}(T', \gamma)\muhat^{H\conn}_{\OO'}\bigl(\chrc{\sbtl{\Lie(H)}x r, \AddChar_{X^*}}\bigr).
\]
\end{thm}

We now re-impose the requirement from \S\ref{subsec:orbital-exists} that \(\bH\conn\) equals \(\Cent_\bG(\gamma)\conn\).

\begin{lem}[\xcite{spice:asymptotic}*{Lemma \xref{lem:dist-G-to-G'}}]
\label{lem:dist-g-to-g'}
Put \(\UU\primedual = \Gamma + \sbjtlpp{\Lie^*(H')}{-r})\).  Suppose that
	\begin{itemize}
	\item \(T'\) is an invariant distribution on \(\Lie(G')\),
	\item \(c(T', \gamma)\) is a finitely supported, \(\OO^{H\primeconn}(\Ad^*(G')\UU\primedual)\)-indexed vector of complex numbers such that \(\abs{\Disc_{G'/H'}(\gamma)}^{1/2}T'(\gamma + \chrc{\sbtl{\Lie(G')}x r, \contra{\AddChar_{X^*}}})\) equals
\[
\sum_{\OO' \in \OO^{H\primeconn}(\Ad^*(G')\UU\primedual)}
{}	\Gauss_{G'/H'}(\OO', \gamma)
	c_{\OO'}(T', \gamma)
	\abs{\redD_{H'}(\OO')}^{1/2}
	\muhat^{H\primeconn}_{\OO'}\bigl(\chrc{\sbtl{\Lie(H')}x r, \contra{\AddChar_{X^*}}}\bigr),
\]
and	\item \(T\) is an invariant distribution on \(\Lie(G)\) such that
\begin{align*}
&
\card{\sbat{\Lie(G)}x s}\inv[1/2]
T\bigl(\gamma + \chrc{\sbtl{\Lie(G', G)}x{(r, \Rp s)}, \contra{\AddChar_{X^*}}}\bigr) \\
\intertext{equals}
&
\card{\sbat{\Lie(G')}x s}\inv[1/2]
T'\bigl(\gamma + \chrc{\sbtl{\Lie(G')}x r, \contra{\AddChar_{X^*}}}\bigr).
\end{align*}
	\end{itemize}
Then \(\abs{\Disc_{G/H}(\gamma)}^{1/2}T(\gamma + \chrc{\sbtl{\Lie(G)}x r, \contra{\AddChar_{X^*}}})\) equals
\[
\sum_{\OO' \in \OO^{H\primeconn}(\UU\primedual)}
	\Gauss_{G/H}(\OO', \gamma)
	c_{\OO'}(T', \gamma)
	\abs{\redD_H(\OO')}^{1/2}
	\muhat^{H\conn}_{\OO'}\bigl(\chrc{\sbtl{\Lie(H)}x r, \contra{\AddChar_{X^*}}}\bigr).
\]
\end{lem}

In addition to being the geometric analogue of the spectral result \xcite{spice:asymptotic}*{Corollary \xref{cor:isotypic-pi-to-pi'}}, Theorem \ref{thm:sample-orb-to-orb'} generalises \xcite{spice:asymptotic}*{Lemma \xref{lem:mu-G-to-G'}} (or, rather, the combination of that lemma and Proposition \ref{prop:dist-r-to-s+}).

We have been assuming since Lemma \ref{lem:r-to-s+} that \(r\) is positive, which, given that all root values of \(\Gamma\) have depth \(-r\) \xcite{spice:asymptotic}*{Hypothesis \initxref{hyp:Z*}(\subxref{good})}, is a way of saying that \(\Gamma\) is not too close to the identity.  In Theorem \ref{thm:sample-orb-to-orb'}, we also impose the requirement that \(\gamma\) is not too \emph{far} from the identity.  We will remove this requirement in Theorem \ref{thm:orb-to-orb'}.

\begin{thm}[\xcite{spice:asymptotic}*{Theorem \xref{thm:isotypic-pi-to-pi'} and Corollary \xref{cor:isotypic-pi-to-pi'}}]
\initlabel{thm:sample-orb-to-orb'}
Suppose that \(\gamma\) lies in \(\sbtl{\Lie(G)}x 0\).  Then we have for every \(\xi \in \Gamma + \sbjtlpp{\Lie^*(G')}{-r}\) that
\begin{align*}
\card{\sbat{\Lie(G)}x s}\inv[1/2]
&\abs{\redD_G(\xi)}^{1/2}
\muhat^G_\xi(\gamma + \chrc{\sbtl{\Lie(G', G)}x{(r, \Rp s)}, \contra{\AddChar_{X^*}}})
\intertext{equals}
\card{\sbat{\Lie(G')}x s}\inv[1/2]&
\abs{\redD_{G'}(\xi)}^{1/2}
\muhat^{G'}_\xi(\gamma + \chrc{\sbtl{\Lie(G')}x r, \contra{\AddChar_{X^*}}}).
\end{align*}
\end{thm}

\begin{proof}
Put \(\mc C = \set{g \in G}{\Ad^*(g)\xi \in X^* + \sbtl{\Lie^*(G', G)}x{(\Rpp{-r}, -s)}}\) and \(\mc C' = \set{g' \in G'}{\Ad^*(g')\xi \in X^* + \sbtlpp{\Lie^*(G')}x{-r}}\).  It is clear that \(\sbtl G x s\dotm\mc C'\) is contained in \(\mc C\).  For the reverse containment, if \(g\) belongs to \(\mc C\), then combining \xcite{spice:asymptotic}*{Hypothesis \initxref{hyp:X*}(\subxref{gp})} with a successive-approximation argument, as in \cite{adler:thesis}*{Lemma 2.3.2}, allows us to assume, after replace \(g\) by an element of \(\sbtl G x s\), that \(\Ad^*(g)\xi\) belongs to \(X^* + \sbtlpp{\Lie^*(G')}x{-r}\); and then, since \(\xi\) also belongs to \(X^* + \sbjtlpp{\Lie^*(G')}{-r}\), \xcite{spice:asymptotic}*{Hypothesis \initxref{hyp:X*}(\subxref{orbit})} gives that \(g\) belongs to \(G'\), hence to \(\mc C'\).

Since the Fourier transform of \(\gamma + \chrc{\sbtl{\Lie(G', G)}x{(r, \Rp s)}, \contra{\AddChar_{X^*}}}\) is
\[
\meas(\sbtl{\Lie^*(G', G)}x{(\Rpp{-r}, -s)})\AddChar_{X^*}(\gamma)(X^* + \chrc{\sbtl{\Lie^*(G', G)}x{(\Rpp{-r}, -s)}, \AddChar_\gamma}),
\]
where we write \(\AddChar_\gamma\) for the additive character \anonmapto{Y^*}{\AddChar(\pair{Y^*}\gamma)} of \(\Lie^*(G)\), we have that
\[
\muhat^G_\xi(\gamma + \chrc{\sbtl{\Lie(G', G)}x{(r, \Rp s)}, \contra{\AddChar_{X^*}}})
\]
equals
\begin{align}
\tag{$*$}
\sublabel{eq:mu-G-pre}
& \meas(\sbtl{\Lie^*(G', G)}x{(\Rpp{-r}, -s)})\AddChar_{X^*}(\gamma)\mu^G_\xi(X^* + \chrc{\sbtl{\Lie^*(G', G)}x{(\Rpp{-r}, -s)}, \AddChar_\gamma}) \\
\nonumber
={} & \sum_{g_0' \in \sbtl G x s\bslash\mc C/{\Cent_G(\xi)}}
	\int_{\sbtl G x s g_0'\Cent_G(\xi)/{\Cent_G(\xi)}} \AddChar_\gamma(\Ad^*(g)\xi)\upd g.
\end{align}
Similarly,
\[
\muhat^{G'}_\xi(\gamma + \chrc{\sbtl{\Lie(G')}x r, \AddChar_{X^*}})
\]
equals
\begin{equation}
\tag{$**$}
\sublabel{eq:mu-G'}
\sum_{g_0' \in \sbtl{G'}x s\bslash\mc C'/{\Cent_{G'}(\xi)}}
	\meas(\Int(g_0')\inv\sbtl{G'}x s \cap \Cent_{G'}(\xi))\inv
	\int_{\sbtl{G'}x s} \AddChar_\gamma(\Ad^*(h'g_0')\xi)\upd{h'}.
\end{equation}

In particular, there is nothing to prove (since both sides vanish) unless \(\mc C'\) is non-empty, in which case, upon replacing \(\xi\) by its conjugate under some element of \(\mc C'\), we may, and do, assume that \(\xi\) lies in \(X^* + \sbtlpp{\Lie^*(G')}x{-r}\).  In particular, we have that \(\Cent_\bG(\xi)\) equals \(\Cent_{\bG'}(\xi)\) \xcite{spice:asymptotic}*{Hypothesis \initxref{hyp:X*}(\subxref{orbit})}.  Then \(\mc C'\) is closed under left translation by \(\sbtl{G'}x s\), and so, since \(\sbtl G x s \cap G'\) equals \(\sbtl{G'}x s\), the natural map \(\anonmap{\sbtl{G'}x s\bslash\mc C'/{\Cent_{G'}(\xi)}}{\sbtl G x s\bslash\mc C/{\Cent_G(\xi)}}\) is a bijection.
Thus, we may re-write \loceqref{eq:mu-G-pre} as
\begin{equation}
\tag{$*'$}
\sublabel{eq:mu-G}
\sum_{g_0' \in \sbtl{G'}x s\bslash\mc C'/{\Cent_{G'}(\xi)}}
	\meas(\Int(g_0')\inv\sbtl G x s \cap \Cent_G(\xi))\inv
	\int_{\sbtl G x s} \AddChar_\gamma(\Ad^*(h g_0')\xi)\upd h.
\end{equation}

Now consider the summand in \loceqref{eq:mu-G} corresponding to the \((\sbtl{G'}x s, \Cent_{G'}(\xi))\)-double coset in \(\mc C'\) containing \(g_0'\).  Note that \(\Int(g_0')\inv\sbtl G x s \cap \Cent_G(\xi)\) equals \(\Int(g_0')\inv\sbtl G x s \cap \Cent_{G'}(\xi) = \Int(g_0')\inv\sbtl{G'}x s \cap \Cent_{G'}(\xi)\).

We claim that \(\Ad^*(h^\perp g_0')\xi\) lies in \(\Ad^*(g_0')\xi + (\Lie^*(G')^\perp \cap \sbtl{\Lie^*(G)}x{-s}) + \sbtl{\Lie^*(G)}x 0\) for every \(h^\perp \in \sbtl{(G', G)}x{(r, s)}\).  Indeed, it suffices to prove this after passing to a tame extension; so, by \xcite{spice:asymptotic}*{Hypothesis \xref{hyp:Z*}}, we may, and do, assume that \(\bG'\) is a Levi subgroup of \bG.  Notice that the set of \(h^\perp \in \sbtl{(G', G)}x{(r, s)}\) for which the statement holds is a semigroup.  Let \(\bN^\pm\) be the unipotent radicals of two opposite parabolic subgroups of \bG with Levi subgroup \(\bG'\).  We have that \(\Ad^*(\sbtl{G'}x r g_0')\xi\) is contained in \(\Ad^*(g_0')\xi + \sbtl{\Lie(G')}x 0\), and \(\Ad^*((N^\pm \cap \sbtl G x s)g_0')\xi\) is contained in \(\Ad^*(g_0')\xi + (\mf n\thendual\mp \cap \sbtl{\Lie^*(G)}x{-s})\), i.e., that the claim holds if \(h^\perp\) is in \(\sbtl{G'}x r\) or \(N^\pm \cap \sbtl G x s\).  Since \(\sbtl{(G', G)}x{(r, s)}\) is generated as a semigroup by \(\sbtl{G'}x r\) and \(N^\pm \cap \sbtl G x s\) \xcite{adler-spice:good-expansions}*{Proposition \xref{prop:heres-a-gp}}, the claim follows in general.

Therefore, we have for every \(h' \in \sbtl{G'}x s\) and \(h^\perp \in \sbtl{(G', G)}x{(r, s)}\) that \(\Ad^*(h'h^\perp g_0')\xi\) lies in \(\Ad^*(h'g_0')\xi + \Lie^*(G')^\perp + \sbtl{\Lie^*(G)}x 0\).  Since \(\gamma\) lies in \(\sbtl{\Lie(G)}x r \subseteq \sbtlp{\Lie^*(G)}x 0\), so that \(\pair{\Lie^*(G')^\perp}\gamma\) is trivial and \(\pair{\sbtl{\Lie^*(G)}x 0}\gamma\) is contained in \(\sbjtlp\field 0\), we have that \(\AddChar_\gamma(\Ad^*(h'h^\perp g_0')\xi)\) equals \(\AddChar_\gamma(\Ad^*(h'g_0')\xi)\).  Since \(\sbtl G x s\) equals \(\sbtl{G'}x s\dotm\sbtl{(G', G)}x{(r, s)}\) \xcite{adler-spice:good-expansions}*{Proposition \xref{prop:heres-a-gp}}, we have that
\[
\int_{\sbtl G x s} \AddChar_\gamma(\Ad^*(h g_0')\xi)\upd h
= \int_{\sbtl{G'}x s\bslash\sbtl G x s} \int_{\sbtl{G'}x s}
	\AddChar_\gamma(\Ad^*(h'h^\perp g_0')\xi)\upd{h'}\upd{h^\perp}
\]
equals
\[
\frac{\meas(\sbtl G x s)}{\meas(\sbtl{G'}x s)}\int_{\sbtl{G'}x s}
	\AddChar_\gamma(\Ad^*(h'g_0')\xi)\upd{h'}.
\]
Since we are the Haar measure which assigns to \(\sbtlp G x 0\) and \(\sbtlp{G'}x 0\) masses \(\card{\sbat{\Lie(G)}x 0}^{-1/2}\) and \(\card{\sbat{\Lie(G')}x 0}^{-1/2}\), respectively, we have that \(\frac
	{\meas(\sbtl G x s)}
	{\meas(\sbtl{G'}x s)}\) equals
\[
\dfrac
	{\card{\sbat{\Lie(G)}x 0}^{-1/2}}
	{\card{\sbat{\Lie(G')}x 0}^{-1/2}}
\Bigl(\dfrac
	{\indx{\sbtlp G x 0}{\sbtl G x s}}
	{\indx{\sbtlp{G'}x 0}{\sbtl G x s}}
\Bigr)\inv,
\]
which, by \xcite{spice:asymptotic}*{Lemma \xref{lem:index-to-disc-X*}} and Remark \ref{rem:disc-Gamma}, equals
\begin{equation}
\tag{$\dag$}
\sublabel{eq:factor}
\dfrac
	{q_G\inv[s]}
	{q_{G'}\inv[s]}\dotm
\dfrac
	{\card{\sbat{\Lie(G)}x s}^{1/2}}
	{\card{\sbat{\Lie(G')}x s}^{1/2}}
= \dfrac
	{\abs{\redD_G(\xi)}\inv[1/2]}
	{\abs{\redD_{G'}(\xi)}\inv[1/2]}
\dotm\dfrac
	{\card{\sbat{\Lie(G)}x s}^{1/2}}
	{\card{\sbat{\Lie(G')}x s}^{1/2}}.
\end{equation}

That is, each summand in \loceqref{eq:mu-G} equals \loceqref{eq:factor} times the corresponding summand in \loceqref{eq:mu-G'}.  Since \(\muhat^G_\xi(\gamma + \chrc{\sbtl{\Lie(G', G)}x{(r, \Rp s)}, \contra{\AddChar_{X^*}}})\) equals \loceqref{eq:mu-G} and \(\muhat^{G'}_\xi(\gamma + \chrc{\sbtl{\Lie(G')}x r, \contra{\AddChar_{X^*}}})\) equals \loceqref{eq:mu-G'}, we are done.
\end{proof}

\subsection{Computation of asymptotic expansions}
\label{subsec:quantitative}

Throughout \S\ref{subsec:quantitative}, we continue to use the
	\begin{itemize}
	\item real number \(r\),
	\item element \(\Gamma \in \Lie^*(G)\) satisfying \xcite{spice:asymptotic}*{Hypothesis \initxref{hyp:Z*}(\subxref{central}, \subxref{good})},
	\item subgroup \(\bG'\) of Definition \ref{defn:G'},
	\item semisimple element \(\gamma \in \Lie(G)\) satisfying Hypothesis \ref{hyp:fc-building},
and	\item group \(\bH = \CC\bG r(\gamma)\)
	\end{itemize}
from \S\ref{subsec:orbital-exists}.
We require that
Hypothesis \ref{hyp:MP-ad} be satisfied,
\(\bH\conn\) equal \(\Cent_\bG(\gamma)\conn\),
and
\((\gamma, x, r)\) satisfy Hypothesis \ref{hyp:Lie-gamma} for \emph{every} point \(x \in \BB(H)\).

In \S\ref{subsec:quantitative}, we require neither that \(r\) be positive nor that \(\gamma\) belong to \(\sbjtl{\Lie^*(G)}0\).  We thank Cheng-Chiang Tsai for his observation that an appropriate scaling, as in the proof of Theorem \ref{thm:orb-to-orb'}, allows us to remove these requirements simultaneously while still using the results of \S\ref{subsec:dist-g-to-g'}, and particularly Theorem \ref{thm:sample-orb-to-orb'}.
Unlike in \S\ref{subsec:dist-g-to-g'}, it is important for our vanishing result Corollary \ref{cor:orb-vanish} that we do \emph{not} require that \(\gamma\) belong to \(\Lie(G')\) (although our proof of Theorem \ref{thm:orb-to-orb'} proceeds by reduction to this case).

Fix \(\matnotn{xi}\xi \in \Gamma + \sbjtlpp{\Lie^*(G')}{-r}\).

\begin{thm}[\xcite{spice:asymptotic}*{Theorem \xref{thm:asymptotic-pi-to-pi'}}]
\label{thm:orb-to-orb'}
For every \(g \in G\), put \(\xi_g = \Ad^*(g)\inv\xi\), \(\Gamma_g = \Ad^*(g)\inv\Gamma\), \(\bG'_g = \Int(g)\inv\bG'\), \(\UU\primedual_g = \Gamma_g + \sbjtlpp{\Lie^*(G'_g)}{-r}\), and, if \(\Gamma_g\) belongs to \(\Lie^*(H)\), also \(\bH'_g = \bH \cap \bG'_g\).

Suppose that all of the relevant orbital integrals converge and that, for every \(g \in G'\bslash G/H\conn\) for which \(\Gamma_g\) belongs to \(\Lie^*(H)\), \cite{jkim-murnaghan:charexp}*{Theorem 3.1.7(1, 5)} is satisfied, and there is a finitely supported, \(\OO^{H\primeconn_g}(\UU\primedual_g)\)-indexed vector \(c(\xi_g, \gamma)\) such that
\[
\Ohat^{G'_g}_{\xi_g}(\gamma + Y')
\qeqq
\sum_{\OO' \in \OO^{H\primeconn_g}(\UU\primedual_g)}
	\Gauss_{G'_g/H'_g}(\OO', \gamma)c_{\OO'}(\xi_g, \gamma)\Ohat^{H\primeconn_g}_{\OO'}(Y')
\]
for all \(Y' \in \Lie(H'_g)\rss \cap \sbjtl{\Lie(H'_g)}r\).  Then we have that
\[
\Ohat^G_\xi(\gamma + Y) \qeqq
\sum_{\substack{
	g \in G'\bslash G/H\conn \\
	\Gamma_g \in \Lie^*(H)
}} \sum_{\OO'}
	\Gauss_{G/H}(\OO', \gamma)c_{\OO'}(\xi_g, \gamma)\Ohat^{H\conn}_{\OO'}(Y)
\]
for all \(Y \in \Lie(H)\rss \cap \sbjtl{\Lie(H)}r\).
\end{thm}

\begin{note}
See Remark \ref{rem:black-box} regarding our imposition of \cite{jkim-murnaghan:charexp}*{Theorem 3.1.7(1, 5)} as a hypothesis.

The group \(\bH'_g\) usually does \emph{not} equal \(\Int(g)\inv(\bH \cap \bG')\).
\end{note}

\begin{proof}
Since replacing \((\Gamma, \gamma, \xi, Y', Y)\) by \((\lambda\inv\Gamma, \lambda\gamma, \lambda\inv\xi, \lambda Y', \lambda Y)\), where \(\lambda\) belongs to \(\field\mult\), and hence each \(\xi_g\) by \(\lambda\inv\xi_g\) and each \(\OO' \in \OO^{H\primeconn_g}(\UU\primedual_g)\) by \(\lambda\inv\OO'\), changes none of \(\Ohat^{G'_g}_{\xi_g}(\gamma + Y')\), \(\Gauss_{G'_g/H'_g}(\OO', \gamma)\), \(\Ohat^{H\primeconn_g}_{\OO'}(Y')\), \(\Ohat^G_\xi(\gamma + Y)\), \(\Gauss_{G/H}(\OO', \gamma)\), and \(\Ohat^{H\conn}_{\OO'}(Y)\), and since this transformation replaces \(r\) by \(r + \ord(\lambda)\), we may assume, upon choosing \(\lambda\) sufficiently close to \(0\), that \(r\) is positive and \(\gamma\) belongs to \(\sbtl{\Lie^*(G)}x 0\).  We do so now.

By Theorem \ref{thm:orb-asymptotic-exists}, there is \emph{some} asymptotic expansion for \(\Ohat^G_\xi\) in terms of \(\OO^{H\conn}(\Ad^*(G)\Gamma) \subseteq \OO^{H\conn}(\UU\primedual)\).  By Remark \ref{rem:which-orb} and \xcite{spice:asymptotic}*{Lemma \xref{lem:asymptotic-check}}, to check that the proposed expansion in the statement is correct, it suffices to show that, whenever \(g_0 \in G\) is such that \(\Ad^*(g_0)\inv\Gamma\) belongs to \(\Lie^*(H)\), we have for all \(x \in \BB(H'_{g_0})\) and \(X^* \in \sbtl{\Lie^*(H)}x{-r} \cap \UU_{g_0}\primedual\) that
\[
\sum_{\OO'}
	\abs{\redD_H(\OO')}^{1/2}
	\Gauss_{G/H}(\OO', \gamma)
	c_{\OO'}(\xi_{g_0}, \gamma)
	\muhat^{H\conn}_{\OO'}(\chrc{\sbtl{\Lie(H)}x r, \contra{\AddChar_{X^*}}})
\]
equals
\[
\int_{\Lie(H)}
	\abs{\Disc_{G/H}(\gamma)}^{1/2}
	\abs{\redD_G(\xi)}^{1/2}
	\muhat^G_\xi(\gamma + Y)
	\chrc{\sbtl{\Lie(H)}x r, \contra{\AddChar_{X^*}}}(Y)
\upd Y,
\]
which, by Lemma \ref{lem:orb-sample}, equals
\[
\abs{\Disc_{G/H}(\gamma)}^{1/2}\abs{\redD_G(\xi)}^{1/2}\muhat^G_\xi(\gamma + \chrc{\sbtl{\Lie(G)}x r, \contra{\AddChar_{X^*}}}).
\]

We assume for notational convenience that \(g_0\) is the identity.  Put \(\bH' = \bH'_{g_0}\) and \(\UU' = \UU'_{g_0}\).  By the Lie-algebra analogue of \xcite{spice:asymptotic}*{Lemma \xref{lem:in-H'}} and the fact that \(\bH\conn\) equals \(\Cent_\bG(\gamma)\conn\), if there is anything to test, which is to say if \(\Lie^*(H) \cap \UU\primedual\) is non-empty, then \(\gamma\) belongs to \(\Lie(H')\) and \(\Gamma\) belongs to \(\Lie^*(H')\).

By Remark \ref{rem:which-orb}, (the other direction of) \xcite{spice:asymptotic}*{Lemma \xref{lem:asymptotic-check}}, and Lemma \ref{lem:orb-sample}, we have the desired equality ``with primes''; that is to say, we know that
\[
\sum_{\OO' \in \OO^{H\primeconn}(\UU')}
	\abs{\redD_{H'}(\OO')}^{1/2}
	\Gauss_{G'/H'}(\OO', \gamma)
	c_{\OO'}(\xi, \gamma)
	\muhat^{H\primeconn}_{\OO'}(\chrc{\sbtl{\Lie(H')}x r, \contra{\AddChar_{X^*}}})
\]
equals
\[
\abs{\Disc_{G'/H'}(\gamma)}^{1/2}
\abs{\redD_{G'}(\xi)}^{1/2}
\muhat^{G'}_\xi(\gamma + \chrc{\sbtl{\Lie(G')}x r, \contra{\AddChar_{X^*}}}).
\]
The desired equality thus follows from Theorem \ref{thm:sample-orb-to-orb'} and Lemma \ref{lem:dist-g-to-g'}.
\end{proof}

\begin{cor}
\label{cor:onestep-orb-vanish}
With the notation and hypotheses of Theorem \ref{thm:orb-to-orb'}, the function \(\Ohat^G_\xi\) vanishes on \(\gamma + (\Lie(H)\rss \cap \sbjtl{\Lie(H)}r)\) unless \(\gamma\) lies in \(\Ad(G)\Lie(G')\).
\end{cor}

\section{Supercuspidal representations}
\label{sec:cuspidal}

Throughout \S\ref{sec:cuspidal}, we continue to use the
	non-Archimedean local field \field
and	reductive group \bG from \S\ref{sec:hyps}.
As in \S\ref{sec:orbits}, we assume that \bG is connected, \(\bG_\tamefield\) is split, and the characteristic of \(\sbjat\field 0\) is not \(2\).

\numberwithin{thm}{subsection}
\subsection{Notation}
\label{subsec:cuspidal-notn}

This subsection has no new results; it only fixes for later use the notation, which is as in \xcite{spice:asymptotic}*{\S\xref{sec:cuspidal}} except that it considers a full Yu datum rather than the `one-step' version there.  Specifically, we will use this notation in some, but not all, of \S\ref{sec:characters}, and in all of \S\ref{subsec:char-unwind}, but not in \S\ref{subsec:orb-unwind}.

Suppose we have a Yu datum \((\vbG, o, \vec r, \rho_0', \vec\phi)\) as in \cite{yu:supercuspidal}*{\S3}, which, for us, constitutes
	\begin{itemize}
	\item a nested, tame, twisted Levi sequence \justnotn[G]{\bG^j}\(\vbG = (\bG^0 \subseteq \dotsb \subseteq \bG^{\ell - 1} \subseteq \bG^\ell = \bG)\),
	\item a point \mnotn o in \(\BB(G^0)\),
	\item a vector \justnotn[r]{r_j}\((r_0 < \dotsb < r_{\ell - 1} \le r_\ell)\) of real numbers,
	\item a representation \(\rho_0'\) of the stabiliser \matnotn{Go}{G^0_o} of the image of \(o\) in the reduced building of \(G^0\),
and	\item a vector \justnotn[phi]{\phi_j}\(\vec\phi = (\phi_0, \dotsc, \phi_{\ell - 1}, \phi_\ell)\), with \(\phi_j\) a character of \(G^j\) for each index \(j\).
	\end{itemize}
For each index \(j \in \sset{0, \dotsc, \ell}\), we write \matnotn{pij}{\pi_j} for the representation of \(G^j\) constructed from the truncated datum \(((\bG^0 \subseteq \dotsb \subseteq \bG^j), o, (r_0 < \dotsb < r_{j - 1} \le r_\ell), \rho_0', (\phi_0, \dotsc, \phi_{j - 1}, \prod_{i = j}^\ell \phi_i))\) in \cite{yu:supercuspidal}*{\S4}, modified as in \xcite{fintzen-kaletha-spice:twist}*{\S\xref{sub:Yu}}.  In particular, note that the inducing datum for \(\pi_j\) \emph{includes} all characters \(\phi_i\) with \(i \ge j\), and hence has depth \(r_\ell\); so, for example, \(\pi_0\) equals \(\Ind_{G^0_x}^{G^0} \rho_0' \otimes \prod_{j = 0}^\ell \phi_j\) and is usually of positive depth, not of depth \(0\).  Put \(\matnotn{pi}\pi = \pi_\ell\) and \(\matnotn{piprime}{\pi'} = \pi_{\ell - 1}\).

For each index \(j \in \sset{0, \dotsc, \ell - 1}\), we write \matnotn{phihat}{\hat\phi_j} for the character of \(\sbtl{G^{j + 1}}o{(r_j, \Rpp{r_j/2})}\) obtained by extending \(\phi_j\) trivially across \(\sbtlp{(G^j, G^{j + 1})}o{(r_j, r_j/2)}\).

The components of the datum satisfy certain conditions \cite{yu:supercuspidal}*{\S3, p.~590}, which we mostly do not re-capitulate here; but a crucial one is that, for each index \(j < \ell\), there is a \(\bG^j\)-fixed element \(\matnotn{Gammaj}{\Gamma^j} \in \sbjtl{\Lie^*(G^j)}{-r_j}\) that represents \(\phi_j\) on \(\sbtl{G^j}o{r_j}\), in the sense of \cite{yu:supercuspidal}*{\S5, p.~594}, and that is generic in the sense of \cite{yu:supercuspidal}*{\S8, \textbf{GE1} and \textbf{GE2}}.  Since \bG is connected, this genericity condition implies that \xcite{spice:asymptotic}*{Hypothesis \xref{hyp:Z*}} is satisfied with \((\bG^{j + 1}, \Gamma^j)\) in place of \((\bG, Z^*_o)\).  Because \bG is tame, so that Moy--Prasad isomorphisms are always available, \xcite{spice:asymptotic}*{Hypothesis \xref{hyp:K-type}} with \((\sbtl{G^{j + 1}}o{r_j}, \hat\phi_j)\) in place of \((\sbtl G o r, \hat\phi_o)\) follows from the definition of ``representing'' \cite{yu:supercuspidal}*{\S5}.

For every index \(j \in \sset{0, \dotsc, \ell - 1}\) and every \(x \in \BB(G^j)\), we have that \(\phi_j\) equips \(V^j_x \ldef \sbtl{(G^j, G^{j + 1})}x{(r_j, r_j/2)}/\sbtl{(G^j, G^{j + 1})}x{(r_j, \Rpp{r_j/2})}\) with the structure of a symplectic space over \(\mathbb F_p\), where \(p\) is the residual characteristic of \field.  Write \(G^j_x\) for the stabiliser in \(G^j\) of the image of \(x\) in the reduced building of \bG.  Then \cite{yu:supercuspidal}*{Proposition 11.4} constructs a canonical symplectic action, in the sense of \cite{yu:supercuspidal}*{p.~601}, of \(G^j_x \ltimes \sbtl{(G^j, G^{j + 1})}x{(r_j, r_j/2)}/{\ker(\hat\phi_j)}\) on \(V^j_x\), and we write \matnotn[\tilde\phi_x]{phitilde}{(\phi_j)\sptilde_x} for the pull-back to \(G^j_x \ltimes \sbtl{(G^j, G^{j + 1})}x{(r_j, r_j/2)}\) of the Weil representation of the Weil--Heisenberg group of \(V^j_x\) whose central character is the restriction to \(\sbtl{(G^j, G^{j + 1})}x{(r_j, \Rpp{r_j/2})}\) of \(\hat\phi_j\).  Note that we consider this Weil-type representation for \emph{every} point \(x\) of \(\BB(G^j)\), not just for \(x = o\).  However, unlike in \cite{yu:supercuspidal}*{\S4}, it is \emph{not} the bare Weil representation \(\tilde\phi_x\), but its twist by the character \(\epsilon^{G^{j + 1}/G^j}_x\) constructed in \xcite{fintzen-kaletha-spice:twist}*{Theorem \xref{thm:main}}, that we use when producing inducing data for the supercuspidal \(\pi_j\).

\subsection{The twisted Weil representation}
\label{subsec:Weil-rep}

We continue using the notation of \ref{subsec:cuspidal-notn}, except that we put \(\matnotn{Gprime}{\bG'} = \bG^{\ell - 1}\) and will often drop a subscript \(\ell - 1\), so that \(\matnotn{phi}\phi\) means \(\phi_{\ell - 1}\), \(\matnotn{phihat}{\hat\phi}\) means \(\hat\phi_{\ell - 1}\), and, for every \(x \in \BB(G')\), \(\matnotn{phitilde}{\tilde\phi_x}\) means \((\phi_{\ell - 1})\sptilde_x\).

We need a small technical result on the interaction between the character of the Weil representation and the Jordan decomposition in a symplectic group.  Lemma \ref{lem:Weil-Jordan} is probably well known, but we were not able to find a reference.

\begin{lem}
\label{lem:Weil-Jordan}
Let \((V, \anonpair)\) be a finite-dimensional, non-degenerate symplectic space over \(\sbjat\field 0\), and suppose that \(s, u \in \Sp(V)\) are commuting elements such that \(s\) is semisimple and \(u\) is unipotent.  Then the character of the Weil representation of \(V\) at \(s u\) is \(\card{\Cent_V(s)}\inv[1/2]\) times the product of the characters of the Weil representations of \(V\) at \(s\), and of \(\Cent_V(s)\) at \(u\).
\end{lem}

\begin{proof}
There is a unique \(s\)-stable complement \(\Cent_V(s)^\perp\) to \(\Cent_V(s)\) in \(V\), and \(\Cent_V(s)\) is orthogonal to \(\Cent_V(s)^\perp\).  Therefore, the character of the Weil representation of \(V\) at \(s u\) is the product of the characters of the Weil representations of \(\Cent_V(s)^\perp\) and of \(\Cent_V(s)\) at \(s u\), and similarly for \(s\); and the character of the Weil representation of \(\Cent_V(s)\) at \(s\) is \(\card{\Cent_V(s)}^{1/2}\).  Since \(s u\) and \(u\) have the same action on \(\Cent_V(s)\), it remains only to show that the character of the Weil representation of \(\Cent_V(s)^\perp\) has the same values at \(s\) and \(s u\).  Using, and with the notation of, \cite{thomas:weil}*{\S1.1.1, p.~221, and Remark 1.3}, we have that the relevant characters are \(\gamma(1)^{\dim(\Cent_V(s)^\perp)}\det(s - 1)\) and \(\gamma(1)^{\dim(\Cent_V(s)^\perp)}\det(s u - 1)\), which are equal since \(s - 1\) is the semisimple part of \(s u - 1\).
\end{proof}

Suppose that \(\gamma\) is a semisimple element of \(G'\) that is compact modulo the centre (of \bG or of \(\bG'\) does not matter, by \cite{yu:supercuspidal}*{\S3, p.~590, \textbf{D1}}), that \(\gamma\) satisfies Hypothesis \ref{hyp:fc-building}, and that \((\gamma, x, \Rp0)\) satisfies Hypothesis \ref{hyp:gp-gamma} for every \(x \in \BB(\CCp{\bG'}0(\gamma))\).

Notation \ref{notn:e-} is temporary, introduced only to recall a collection of notation from other papers that we need to compare in Proposition \ref{prop:e-}.

\begin{notn}
\label{notn:e-}
Fix a maximal tame torus \bT in \(\CCp{\bG'}0(\gamma)\), and a point \(x \in \BB(T)\).

Write \(G'_x\) for the stabiliser of the image of \(x\) in the reduced building of \bG, and let \(\epsilon_x^{G/G'}\) be the character of \(G'_x\) constructed in \xcite{fintzen-kaletha-spice:twist}*{Theorem \xref{thm:main} and Lemma \xref{lem:diff}} using the element \(\Gamma\) of \(\Cent_{\Lie^*(G')}(\bG')\) \xcite{fintzen-kaletha-spice:twist}*{Remark \xref{rem:generic}}.

Recall the functions \(\tilde e(G/G', T, \anondot)\) and \(\tilde e(\CCp G 0(\gamma)/\CCp{G'}0(\gamma), \anondot)\), defined in \xcite{debacker-spice:stability}*{Definition \xref{defn:signs}}, and the function \(\epsilon_{\sharp, x}^{G/G'}\) on \((\bT/{\Zent(\bG)})(\field)\subb\), defined in \xcite{fintzen-kaletha-spice:twist}*{Definition \xref{dfn:lstk}}.
(Technically, \(\tilde e(G, T, \anondot)\) is defined only on \(T\subb\); but its definition is in terms only of the projections of root values to \(\sbjat\field 0\), hence makes sense as a function on \((\bT/{\Zent(\bG)})(\field)\subb\).)
\end{notn}

\begin{rem}
\label{rem:G'-tame}
By Lemma \initref{lem:shallow-dfc}\subpref{0+}, Remark \ref{rem:shallow-dfc}, and \xcite{spice:jordan}*{Corollary \xref{cor:abs-F-ss-tame}}, every maximal tame torus in \(\CCp{\bG'}0(\gamma)\) is maximal tame in \(\bG'\).  Since \(\bG'\) is a tame, twisted Levi subgroup of \bG, every maximal tame torus in \(\bG'\) is maximal tame in \bG; and, since \bG splits over a tame extension of \field, in fact the maximal tame tori in \bG are maximal tori in \bG.
\end{rem}

Although the formula for \(\tr \tilde\phi_x(\gamma)\) in Proposition \ref{prop:e-} is perhaps of some interest on its own, the main interest for us is in the independence of \(\epsilon_x^{G/G'}(\gamma)\frac{\tr \tilde\phi_x(\gamma)}{\abs{\tr \tilde\phi_x(\gamma)}}\) from the parameters that seem to go into its construction.

\begin{prop}
\label{prop:e-}
Use Notation \ref{notn:e-} and the notation \(\sbjat\gamma 0\) of Remark \ref{rem:shallow-dfc}.
For every maximal tame torus \bT in \(\CCp{\bG'}0(\gamma)\)
(which is a maximal torus in \bG, by Remark \ref{rem:G'-tame})
and every \(x \in \BB(T)\), we have that
\[
\tr \tilde\phi_x(\gamma)
\qeqq
\tr \tilde\phi_x(\sbjat\gamma 0)
= \indx{\sbat{\CC G 0(\gamma)}x s}{\sbat{\CC{G'}0(\gamma)}x s}^{1/2}
\tilde e(\pi', \sbjat\gamma 0)
\epsilon_{\sharp, x}^{G/G'}(\sbjat\gamma 0);
\]
and
\(\epsilon_x^{G/G'}(\gamma)
\frac
	{\tr \tilde\phi_x(\gamma)}
	{\abs{\tr \tilde\phi_x(\gamma)}}
\) depends only on \bG, \(\bG'\), \(r\), and \(\sbjat\gamma 0\), not on \bT, \(x\), \(\gamma\), or \(\Gamma\).
\end{prop}

\begin{proof}
Write \(V_x = \sbtl{(G', G)}x{(r, s)}/\sbtl{(G', G)}x{(r, \Rp s)}\) for the symplectic space underlying the Weil representation \(\tilde\phi_x\).

We continue to denote the images of \(\gamma\), \(\sbjat\gamma 0\), and \(\sbjtlp\gamma 0 \ldef \sbjat\gamma 0\inv\gamma\) in \(\Sp(V_x)\) by the same symbols.  Since (the images of) \(\sbjat\gamma 0\) and \(\sbjtlp\gamma 0\) have prime-to-\(p\) and \(p\)-power order, where \(p\) is the residual characteristic of \field, it follows that they are the semisimple and unipotent parts of (the image of) \(\gamma\).

As in the proof of \xcite{adler-spice:explicit-chars}*{Proposition \xref{prop:theta-tilde-phi}}, and using Lemma \initref{lem:shallow-dfc}\subpref{0+}, we find that
\[
\anonmap
	{
	\sbtl{(\CCp{G'}0(\gamma), \CCp G 0(\gamma))}x{(r, s)}/\sbtl{(\CCp{G'}0(\gamma), \CCp G 0(\gamma))}x{(r, \Rp s)}
	}
	{\Cent_{V_x}(\sbjat\gamma 0)}
\]
is an isomorphism.  In particular, by Hypothesis \initref{hyp:gp-gamma}\subpref{gp}, we have that the action of \(\gamma\) on \(\Cent_{V_x}(\sbjat\gamma 0)\) is trivial, so that the character of the Weil representation of \(\Cent_{V_x}(\sbjat\gamma 0)\) at \(\gamma\) is \(\card{\Cent_{V_x}(\sbjat\gamma 0)}^{1/2}\).  By Lemma \ref{lem:Weil-Jordan}, we have that \(\tr \tilde\phi_x(\gamma)\) equals \(\tr \tilde\phi_x(\sbjat\gamma 0)\), which, in the notation of \xcite{adler-spice:explicit-chars}*{Proposition \xref{prop:theta-tilde-phi}}, is the product of
\begin{multline*}
\card{\sbtl{(\CCp{G'}0(\gamma), \CCp G 0(\gamma))}x{(r, s)}/\sbtl{(\CCp{G'}0(\gamma), \CCp G 0(\gamma))}x{(r, \Rp s)}}^{1/2} \\
= \indx{\sbat{\CCp G 0(\gamma)}x s}{\sbat{\CCp{G'}0(\gamma)}x s}^{1/2}
\end{multline*}
with \(\varepsilon(\phi, \sbjat\gamma 0)\).

In the notation of \xcite{adler-spice:explicit-chars}*{Definition \xref{defn:gauss}} (\emph{not} of \xcite{spice:asymptotic}*{Notation \xref{notn:Gauss}} or Notation \ref{notn:Gauss}), we have that \(\mf G(\phi, \sbjat\gamma 0)\) equals \(1\), since \(\sbjat\gamma 0\) is absolutely semisimple; and \xcite{debacker-spice:stability}*{Proposition \xref{prop:root}} gives, in the notation of \xcite{debacker-spice:stability}*{Definition \xref{defn:signs}}, that \(\varepsilon(\phi, \sbjat\gamma 0) = \varepsilon(\phi, \sbjat\gamma 0)\mf G(\phi, \sbjat\gamma 0)\) equals \(\tilde e(\pi', \sbjat\gamma 0)\dotm\varepsilon^\text{ram}(\pi', \sbjat\gamma 0)\varepsilon_{\text{symm}, \text{ram}}(\pi', \sbjat\gamma 0)\).  As observed in \xcite{debacker-spice:stability}*{proof of Proposition \xref{prop:root}, p.~20} (specifically, the demonstration that \(\Xi_\text{symm} = \Xi_{\text{symm}, \text{unram}}\)), it follows again from the fact that \(\sbjat\gamma 0\) is absolutely semisimple that \(\varepsilon_{\text{symm}, \text{ram}}(\pi', \sbjat\gamma 0)\) also equals \(1\).  By \xcite{fintzen-kaletha-spice:twist}*{Definition \xref{dfn:lstk}} (see also \cite{fintzen-kaletha-spice:twist}*{Remark \xref{rem:compare-notn}}), the character \(\varepsilon^\text{ram}(\pi', \anondot)\) of \cite{debacker-spice:stability} is the same as the character \(\epsilon_{\sharp, x}^{G/G'}\) of \cite{fintzen-kaletha-spice:twist}.  We have thus established the desired equality.

By construction, \(\epsilon_x^{G/G'}(\gamma)\frac{\tr \tilde\phi_x(\gamma)}{\card{\tr \tilde\phi_x(\gamma)}}\) depends only on \bG, \(\bG'\), \(\gamma\), \(\Gamma\), and \(x\), not on \bT; and the explicit computation of \(\tr \tilde\phi_x(\gamma)\), together with the fact (as observed in \xcite{fintzen-kaletha-spice:twist}*{Lemma \xref{lem:p-vs-2}}) that \(\epsilon_x^{G/G'}(\sbjtlp\gamma 0)\) is trivial, show that its dependence on \(\gamma\) is only \textit{via} \(\sbjat\gamma 0\).  On the other hand, \xcite{fintzen-kaletha-spice:twist}*{Definition \xref{dfn:lstk} and Theorem \xref{thm:main}} and \xcite{debacker-spice:stability}*{Remark \xref{rem:signs}} show that \(\epsilon_x^{G/G'}(\gamma)\epsilon_{\sharp, x}^{G/G'}(\gamma)\tilde e(\pi', \gamma)\) depends only on \bG, \(\bG'\), \(r\), \(\gamma\), and \bT, not on \(x\) or \(\Gamma\).  (In fact the details of the proof of \xcite{fintzen-kaletha-spice:twist}*{Theorem \xref{thm:main}} show that also \(\epsilon_x^{G/G'}\) does not depend on \(\Gamma\), but we do not need this.)
\end{proof}

Notation \ref{notn:twisted-Gauss} does not really specify \(\wtilde\Gauss_G(X^*, \gamma)\) and \(\wtilde\Gauss_{G'}(X^*, \gamma)\) individually; instead, it specifies them up to simultaneous scaling by a non-\(0\) complex number.  This is perfectly acceptable in Theorem \ref{thm:pi-to-pi'} (and analogously in Theorem \ref{thm:char-unwind}), where all we are trying to do is to compare character formul{\ae} involving these two quantities.  There would be no harm in taking, for example, \(\wtilde\Gauss_{G'}(X^*, \gamma)\) to be \(1\) and \(\wtilde\Gauss_G(X^*, \gamma)\) to be \(\wtilde\Gauss_{G/G'}(X^*, \gamma)\).  However, we phrase our definition in this form in the hopes of suggesting an `absolute' definition of \(\wtilde\Gauss_G(X^*, \gamma)\), analogous to that of \(\Gauss_G(X^*, \gamma)\) in \xcite{spice:asymptotic}*{Notation \xref{notn:Gauss}}, from which our \(\wtilde\Gauss_{G/G'}(X^*, \gamma)\) can be recovered.  We do not yet know such a definition.

\begin{notn}
\label{notn:twisted-Gauss}
With the notation of Proposition \ref{prop:e-}, for every \(X^* \in \Gamma + \sbjtlpp{\Lie^*(G')}{-r}\), put
\[
\wtilde\Gauss_{G/G'}(X^*, \gamma)
= \epsilon_x^{G/G'}(\gamma)
\frac
	{\tr \tilde\phi_x(\gamma)}
	{\abs{\tr \tilde\phi_x(\gamma)}}
\Gauss_{\CCp G 0(\gamma)/{\CCp{G'}0(\gamma)}}(X^*, \gamma)
\]
for any \(x \in \BB(\CCp{G'}0(\gamma))\), and let \matnotn{Gauss}{\wtilde\Gauss_G(X^*, \gamma)} and \(\wtilde\Gauss_{G'}(X^*, \gamma)\) be any two non-\(0\) complex numbers whose quotient is \(\wtilde\Gauss_{G/G'}(X^*, \gamma)\).
\end{notn}

\numberwithin{thm}{section}
\section{Asymptotic expansions of characters redux}
\label{sec:characters}

Throughout \S\ref{sec:characters}, we continue to use the
	non-Archimedean local field \field
and	reductive group \bG
from \S\ref{sec:hyps}.  We do \emph{not} yet impose the notation and all the assumptions of \S\ref{sec:cuspidal}, but we \emph{do} assume that \(\bG_\tamefield\) is split and the residual characteristic of \field is not \(2\).

In \S\ref{sec:characters}, we re-cast the one-step results of \xcite{spice:asymptotic}*{\S\xref{sec:cuspidal}}, which relate positive-depth, supercuspidal characters of \bG to smaller-depth, supercuspidal characters of tame, twisted Levi subgroups, in a way that can be iterated to relate positive-depth, supercuspidal characters of \bG to depth-\(0\), supercuspidal characters of tame, twisted Levi subgroups.  We will carry out the `unwinding' step in \S\ref{sec:unwind}.

One consequence of our approach is that, instead of specifying a `centre' \(\gamma\) and evaluating characters at elements of the form \(\gamma\dotm\mexp(Y)\) where \(Y\) is an element that commutes with \(\gamma\) and lies sufficiently close to the origin, instead we let our `centre' be of the form \(\gamma\dotm\mexp(\sbjhd Y r)\), and evaluate characters at elements of the form \(\gamma\dotm\mexp(\sbjhd Y r + \sbjtl Y r)\), where \(\sbjtl Y r\) commutes with \(\gamma\) and \(\sbjtl Y r\) and lies sufficiently close to the origin.
If \field has characteristic \(0\) and \mexp is the exponential map, or more generally if \mexp is multiplicative on sums of commuting elements, then there is no need for the new results of this section, since in that case \(\gamma\dotm\mexp(\sbjhd Y r + \sbjtl Y r)\) equals \(\gamma\dotm\mexp(\sbjhd Y r)\mexp(\sbjtl Y r)\) for all \(\sbjtl Y r \in \sbjtl{\Lie(H)}r\), and our new `re-centred' expansion depending on \(\gamma\) and \(\sbjhd Y r\)
(as in Theorem \ref{thm:char-asymptotic-exists})
is the same as the expansion of \xcite{spice:asymptotic}*{Theorem \xref{thm:asymptotic-pi-to-pi'}} centred at \(\gamma\dotm\mexp(\sbjhd Y r)\).  Since the effort required to deal with the extra generality afforded by a `mock' exponential map in place of the genuine exponential map is very small, we have chosen to do so.

We replace \xcite{spice:asymptotic}*{Hypothesis \xref{hyp:mexp}} by Hypothesis \ref{hyp:mexp}, and change the statements of some results of \xcite{spice:asymptotic}*{\S\xref{sec:cuspidal}} slightly.  The proofs require no new techniques, only minor, obvious changes, which we leave to the reader.

To be precise, we choose
	\begin{itemize}
	\item numbers \(\matnotn{Rminus}{R_{-1}} \in \sbjtlp\tR 0\) and \(\mnotn r \in \sbjtlp\R 0\) with \(R_{-1} \le r\),
	\item a semisimple element \(\matnotn{gamma}\gamma \in G\) satisfying Hypothesis \ref{hyp:fc-building},
	\item an element \(\matnotn{Gamma}\Gamma \in \Lie^*(G)\) satisfying \xcite{spice:asymptotic}*{Hypothesis \xref{hyp:Z*}},
and	\item a semisimple element \(\matnotn{Yr}{\sbjhd Y r} \in \Lie(\CC\bG{R_{-1}}(\gamma))\) satisfying Hypothesis \ref{hyp:fc-building}, with \(\CC\bG{R_{-1}}(\gamma)\) in place of \bG and \(\sbjhd Y r\) in place of the Lie-algebra element \(\gamma\) in that hypothesis.
	\end{itemize}
In particular, we have by \xcite{spice:asymptotic}*{Hypothesis \xref{hyp:Z*}} that \(\matnotn{Gprime}{\bG'} \ldef \Cent_\bG(\Gamma)\) is a tame, twisted Levi subgroup of \bG.  (Because we have imposed the full force of \xcite{spice:asymptotic}*{Hypothesis \xref{hyp:Z*}} here, including \loccit*{Hypothesis \initxref{hyp:Z*}(\subxref{orbit})}, we do not need the more delicate construction of \(\bG'\) from Definition \ref{defn:G'}.)
Note that we can choose \(R_{-1} = \Rp0\).

We will eventually (before Theorem \ref{thm:pi-to-pi'}) assume that \(\gamma\) is compact modulo centre, but we do not do so yet.
Put \(\mnotn s = r/2\), \(\matnotn J\bJ = \CC\bG{R_{-1}}(\gamma)\), and \(\matnotn H\bH = \CC\bJ r(\sbjhd Y r)\).  We require that
	\begin{itemize}
	\item \((\gamma, x, R_{-1})\) satisfy Hypothesis \ref{hyp:gp-gamma} for all \(x \in \BB(J)\),
	\item \(\gamma\) centralise \(\bJ\conn\),
	\item \((\sbjhd Y r, x, r)\) satisfy Hypothesis \ref{hyp:Lie-gamma} for all \(x \in \BB(H)\),
	\item \(\sbjhd Y r\) belong to \(\sbjtl{\Lie(H)}{R_{-1}}\),
	\item \(\sbjhd Y r\) centralise \(\Lie(\bH)\).
	\end{itemize}
In particular, \(\bJ\conn\) equals \(\Cent_\bG(\gamma)\conn\) and \(\bH\conn\) equals \(\Cent_\bJ(\sbjhd Y r)\conn = \Cent_\bG(\gamma, \sbjhd Y r)\conn\).

We impose Hypothesis \ref{hyp:mexp} for \((\CC\bG 0(\gamma), \bJ, R_{-1})\), which guarantees the existence of a homeomorphism \map{\matnotn{exp}\mexp}{\sbjtl{\Lie(J)}{R_{-1}}}{\sbjtl J{R_{-1}}}.  Write \matnotn{log}\mlog for the inverse homeomorphism \anonmap{\sbjtl J{R_{-1}}}{\sbjtl{\Lie(J)}{R_{-1}}}.

\begin{rem}
\label{rem:mexp:steps}
By Hypothesis \ref{hyp:fc-building},
which implies that \(\sbjtl{\Lie(J)}{R_{-1}}\) is contained in \(\sbjtl{\Lie(G)}{R_{-1}}\) and \(\sbjtl J{R_{-1}}\) is contained in \(\sbjtl G{R_{-1}}\), and Lemma \ref{lem:MP-cfc}, we have that \(\CC\bG{R_{-1}}(\sbjhd Y r)\conn\) and \(\CC\bG{R_{-1}}(\mexp(\sbjhd Y r))\conn\) equal \bG.

If \(i\) is strictly less than \(R_{-1}\), then Lemma \initref{lem:gp-dfc-nearby}\subpref{equal} gives that \(\CC\bG i(\gamma\dotm\mexp(\sbjhd Y r))\) equals \(\CC\bG i(\gamma)\).
If \(i\) is at least \(R_{-1}\), then Lemma \initref{lem:gp-dfc-unique}\subpref{nearby-sub} gives that \(\CC\bG i(\gamma\dotm\mexp(\sbjhd Y r))\conn\) is the identity component of \(\Cent_\bG(\gamma) \cap \CC\bG i(\mexp(\sbjhd Y r))\), hence also of 
\(\bJ \cap \CC\bG i(\sbjhd Y r) = \CC\bJ i(\sbjhd Y r)\).

In particular, \(\gamma\dotm\mexp(\sbjhd Y r)\) satisfies Hypothesis \ref{hyp:fc-building}, and \((\gamma\dotm\mexp(\sbjhd Y r), x, r)\) satisfies Hypothesis \ref{hyp:gp-gamma} for all \(x \in \BB(H)\); and \(\bH\conn\) is the identity component of \(\Cent_\bG(\gamma) \cap \CC\bG r(\sbjhd Y r)\), hence equals \(\CC\bJ r(\sbjhd Y r)\conn = \Cent_\bG(\gamma, \sbjhd Y r)\conn = \Cent_\bG(\gamma, \mexp(\sbjhd Y r))\conn = \Cent_\bG(\gamma\dotm\mexp(\sbjhd Y r))\conn\).
%
\end{rem}

\begin{rem}
\label{rem:gp-symmetry}
Since \(\bJ\conn\) equals \(\Cent_\bG(\gamma)\conn\), and \(\bG'\) equals \(\Cent_\bG(\Gamma)\), the containments \(\gamma \in \Lie(G')\) and \(\Gamma \in \Lie^*(J)\) are equivalent to \(\ad^*(\gamma)\Gamma = 0\), hence to each other.

Similarly, the containments \(\Gamma \in \Lie^*(H)\) and \(\gamma\dotm\mexp(\sbjhd Y r) \in G'\) are equivalent to each other, hence, by Remark \ref{rem:mexp:steps}, to the conjunction of the containments \(\gamma \in G'\) and \(\sbjhd Y r \in \Lie(G')\).
\end{rem}

Lemma \ref{lem:gp-Gauss-combine} is not needed until the proof of Theorem \ref{thm:char-unwind}, but it seems natural to put it here as an illustration of the use of Hypothesis \ref{hyp:mexp}.

\begin{lem}
\label{lem:gp-Gauss-combine}
Suppose that
	\begin{itemize}
	\item \(\gamma\) belongs to \(G'\),
	\item \(\sbjhd Y r\) belongs to \(\Lie(G')\),
and	\item \(\CCp\bG 0(\gamma)\conn\) equals \bG.
	\end{itemize}
Put \(\bJ' = \bJ \cap \bG'\) and \(\bH' = \bH \cap \bG'\).
For every \(X^* \in \Gamma + \sbtlpp{\Lie^*(G')}x{-r}\), we have that
{
\newcommand\WQpiece[3]{\Gauss_{#1/#2}(X^*, #3)}
\newcommand\WQ[3]{\frac{\WQpiece{#1}{#2}{#3}}{\WQpiece{#1'}{#2'}{#3}}}
\[
\WQ G H{\gamma\dotm\mexp(\sbjhd Y r)}
\qeqq
\WQ G J\gamma\dotm\WQ J H{\sbjhd Y r}.
\]
}
\end{lem}

\begin{proof}
We have that \(\bJ'\) is a tame, twisted Levi subgroup of \bJ by Proposition \initref{prop:gp-cfc-facts}\subpref{Levi}, and so that \(\bH' = \bH \cap \bJ'\) is a tame, twisted Levi subgroup of \bH by Proposition \initref{prop:Lie-cfc-facts}\subpref{Levi}.  In particular, we have that \(\sbjtl{\Lie(H)}{R_{-1}} \cap \Lie(J')\), which contains \(\sbjhd Y r\), equals \(\sbjtl{\Lie(H')}{R_{-1}}\).  Choose \(x \in \BB(H')\) such that \(\sbjhd Y r\) belongs to \(\sbtl{\Lie(H')}x{R_{-1}}\).
	(In fact, since \(\sbjhd Y r\) is central in \(\Lie(H)\), hence in \(\Lie(H')\), \emph{any} \(x \in \BB(H')\) will do.)

Put \(\delta = \gamma\dotm\mexp(\sbjhd Y r)\).  Recall the definition of the pairing
\[
\mapto{b_{X^*, \gamma}}{(Z_1, Z_2)}{\pair[\big]{X^*}{\comm{Z_1}{(1 - \Ad(\gamma))Z_2}}}
\]
from \xcite{spice:asymptotic}*{Notation \xref{notn:Gauss}}, and the pairing
\[
\mapto{b_{X^*, \sbjhd Y r}}{(Z_1, Z_2)}{\pair[\big]{X^*}{\comm{Z_1}{\comm{Z_2}{\sbjhd Y r}}}}
\]
from Notation \ref{notn:Gauss}.  We define \(b_{X^*, \delta}\) analogously to \(b_{X^*, \gamma}\).

Our strategy is as follows.  The space
\begin{multline*}
\sbtl{\Lie(\CC G r(\delta), \CCp{G'}0(\delta), \CCp G 0(\delta))}x{(\Rp0, r - \ord_{\delta - 1}, (r - \ord_{\delta - 1})/2)}/{} \\
{\sbtl{\Lie(\CC G r(\delta), \CCp{G'}0(\delta), \CCp G 0(\delta))}x{(\Rp0, r - \ord_{\delta - 1}, \Rpp{(r - \ord_{\delta - 1})/2})}}
\end{multline*}
of \xcite{spice:asymptotic}*{Corollary \xref{cor:lattice-orth}} is spanned by the various
\[
\sbat{\Lie(\CC G i(\delta))}x{(r - i)/2}/{\sbat{\Lie(\CCp G i(\delta), \CC{G'}i(\delta), \CC G i(\delta))}x{((r - i)/2, (r - i)/2, \Rpp{(r - i)/2})}},
\]
which are orthogonal for the pairing \(b_{X^*, \delta}\) (viewed as a \(\sbjat\field 0\)-valued, not a \field-valued, pairing), and analogously for \(\gamma\) and \(\sbjhd Y r\).
By \cite{ranga-rao:weil}*{Theorem A.2(3, 5)}, we have that the Weil index of the orthogonal sum is the product of the Weil indices for the images of each \(\sbtl{\Lie(\CC G i(\delta))}x{(r - i)/2}\).  We will show that the contribution from \(\delta\) at index \(i\) matches up with the contribution from \(\gamma\) or from \(Y\) according as \(i\) is, or is not, less than \(R_{-1}\).
Since \(\CC\bJ i(Y)\conn\) equals \(\CCp\bJ i(Y)\conn\) for every \(i\) strictly less than \(R_{-1}\) by Lemma \ref{lem:MP-cfc},
and \(\CC\bG i(\gamma)\conn\) and \(\CCp\bG i(\gamma)\conn\) both equal \(\Cent_\bG(\gamma)\conn\) for every \(i\) greater than or equal to \(R_{-1}\), this will establish the desired equality.


Fix \(i \in \R\) with \(0 < i < r\).  Put \(s_i^\pm = (r \pm i)/2\) and \(L_i = \sbtl{\Lie(\CC G i(\delta))}x{s_i^-}\).

If \(i\) is strictly less than \(R_{-1}\), then Remark \ref{rem:mexp:steps} gives that \(\CC\bG i(\delta)\) equals \(\CC\bG i(\gamma)\), and in particular contains \bJ.  Since \(\sbjhd Y r\) lies in \(\sbtl{\Lie(J)}x{R_{-1}} \subseteq \sbtlp{\Lie(\CC G i(\delta))}x i\), it follows from \xcite{spice:asymptotic}*{Hypothesis \initxref{hyp:mexp}(\incxref{coset})\subxref{coset}} that \(\mexp(\sbjhd Y r)\) lies in \(\sbtl J x{R_{-1}} \subseteq \sbtl{\CCp G i(\delta))}x i\).  Thus \((1 - \Ad(\mexp(\sbjhd Y r)))L_i\) lies in \(\sbtlpp{\Lie(\CC G i(\delta))}x{s_i^+}\).  Hypothesis \initref{hyp:gp-gamma}\incpref{building}\subpref0 gives that \(\gamma\) fixes the image of \(x\) in the reduced building of \(\CC\bG i(\gamma)\), and so normalises \(\sbtlpp{\Lie(\CC G i(\gamma))}x{s_i^+} = \sbtlpp{\Lie(\CC G i(\delta))}x{s_i^+}\).  Thus, \((1 - \Ad(\delta))Z_2\) lies in \((1 - \Ad(\gamma))Z_2 + \sbtlpp{\Lie(\CC G i(\delta))}x{s_i^+}\), so that \(b_{X^*, \delta}(Z_1, Z_2)\) lies in \(b_{X^*, \gamma}(Z_1, Z_2) + \sbjtlp\field 0\), for all \(Z_1, Z_2 \in L_i\).

If \(i\) is greater than or equal to \(R_{-1}\), then Remark \ref{rem:mexp:steps} gives that \(L_i\) equals \(\sbtl{\Lie(\CC J i(\sbjhd Y r))}x{s_i^-}\), and in particular is contained in \(\Lie(J) = \Cent_{\Lie(G)}(\gamma)\), so that \((1 - \Ad(\delta))Z_2\) equals \((1 - \Ad(\mexp(\sbjhd Y r)))Z_2\).  Hypothesis \initref{hyp:mexp}\incpref{elt}\subpref{ad-Ad} gives that \(1 - \ad(\sbjhd Y r) - \Ad(\mexp(\sbjhd Y r))\) carries \(L_i = \sbtl{\Lie(\CC J i(Y))}x{s_i^-}\) into \(\sbtlpp{\Lie(\CC J i(Y))}x{s_i^+} = \sbtlpp{\Lie(\CC G i(\delta))}x{s_i^+}\).  Thus, \((1 - \Ad(\delta))Z_2 = (1 - \Ad(\mexp(\sbjhd Y r)))Z_2\) lies in \(-\ad(\sbjhd Y r)Z_2 + \sbtlpp{\Lie(\CC G i(\delta))}x{s_i^+}\), so that \(b_{X^*, \delta}(Z_1, Z_2)\) lies in \(b_{X^*, Y}(Z_1, Z_2) + \sbjtlp\field 0\), for all \(Z_1, Z_2 \in L_i\).
\end{proof}

Let \mnotn o be a point of \(\BB(G')\), and \matnotn{phio}{\phi_o} a character of \(\sbat{G'}o r\).  Write \matnotn{phihato}{\hat\phi_o} for the extension of \(\phi_o\) trivially across \(\sbtlp{(G', G)}o{(r, s)}\) to \(\sbtl{(G', G)}o{(r, \Rp s)}\).  We require that \xcite{spice:asymptotic}*{Hypothesis \xref{hyp:K-type}} be satisfied by \(\hat\phi_o\), with \(\Gamma\) in place of \(Z^*_o\).  Recall from the discussion preceding \xcite{spice:asymptotic}*{Hypothesis \xref{hyp:K-type}} that this hypothesis is automatically satisfied if \((\sbtl{G'}o r, \phi_o)\) is the character with dual blob \(\Gamma + \sbtlpp{\Lie^*(G')}o{-r}\).
We will eventually (before Theorem \ref{thm:pi-to-pi'}) require this.  Although we do not make that requirement yet, we see that, for the purposes of the main results of \S\ref{sec:characters} (Theorem \ref{thm:pi-to-pi'}) and of this paper (Theorems \ref{thm:orb-unwind} and \ref{thm:char-unwind}), the assumption that \((\hat\phi_o, \Gamma)\) satisfies \xcite{spice:asymptotic}*{Hypothesis \xref{hyp:K-type}} need not be separately imposed.

\begin{defn}
\label{defn:centred-char}
If \(\pi\) is an admissible representation of \(G\), then we define \matnotn{Thetacheck}{\widecheck\Theta_{\pi, \gamma, \sbjhd Y r}} and \matnotn{Thetacheck}{\widecheck\Theta_{\pi, \gamma, \sbjhd Y r, \Gamma}} to be the distributions on \(\Lie^*(H)\) given for every \(f^* \in \Hecke(\Lie^*(H))\) by
\begin{align*}
\widecheck\Theta_{\pi, \gamma, \sbjhd Y r}(f^*)        & {}= \Theta_{\pi, \gamma\dotm\mexp(\sbjhd Y r)}\bigl(\check f^*_{\gamma, \sbjhd Y r}) \\
\intertext{and}
\widecheck\Theta_{\pi, \gamma, \sbjhd Y r, \Gamma}(f^*) & {}= \widecheck\Theta_{\pi, \gamma, \sbjhd Y r}(f^*_\Gamma),
\end{align*}
where we have introduced the \textit{ad hoc} notations
	\begin{itemize}
	\item \(\check f^*_{\gamma, \sbjhd Y r}\) for the function that vanishes outside \(\gamma\dotm\mexp(\sbjhd Y r)\sbjtl H r = \gamma\dotm\mexp(\sbjhd Y r + \sbjtl{\Lie(H)}r)\) \xcite{spice:asymptotic}*{Hypothesis \initxref{hyp:mexp}(\incxref{coset})\subxref{coset}}, and is given on that domain by \anonmapto{\gamma\dotm\mexp(\sbjhd Y r + \sbjtl Y r)}{\check f^*(\sbjtl Y r)},
and	\item \(f^*_\Gamma\) for the function that vanishes outside \(\Lie^*(H) \cap \Ad^*(H\conn)(\Gamma + \sbjtlpp{\Lie^*(G')}{-r})\), and agrees with \(f^*\) on that domain;
	\end{itemize}
and where
	\begin{itemize}
	\item \(\Theta_{\pi, \gamma\dotm\mexp(\sbjhd Y r)}\) is the distribution \(T_{\gamma\dotm\mexp(\sbjhd Y r)}\) deduced from \(T = \Theta_\pi\) in \xcite{spice:asymptotic}*{Lemma \xref{lem:centre}}.
	\end{itemize}
\end{defn}

Lemma \ref{lem:char-sample} is of interest mainly because of its use in Corollary \ref{cor:char-sample}.  Recall the notion of the dual blob of a character from \xcite{spice:asymptotic}*{Definition \xref{defn:dual-blob}}.

\begin{lem}[\xcite{spice:asymptotic}*{Lemma \xref{lem:sample}}]
\label{lem:char-sample}
With the notation of Definition \ref{defn:centred-char}, if
	\begin{itemize}
	\item \(a \in \tR\) satisfies \(r \le a < \infty\),
and	\item \(X^*\) belongs to \(\sbtl{\Lie^*(H)}x{-a}\),
	\end{itemize}
then
\begin{align*}
&\widecheck\Theta_{\pi, \gamma, \sbjhd Y r}\bigl(X^* + \chrc{\sbtlpp{\Lie^*(H)}x{-r}}\bigr) \\
\intertext{equals}
\meas(\sbtl{\Lie(H)}x a)\sum_{\sbjtl Y r \in \sbtl{\Lie(H)}x r/{\sbtl{\Lie(H)}x a}} \contra{\AddChar_{X^*}}(\sbjtl Y r)&
\tr \pi(\gamma\dotm\mexp(\sbjhd Y r + \sbjtl Y r)\chrc{\sbtl G x a, \contra{\hat\phi}}),
\end{align*}
where \(\hat\phi\) is the character of \(\sbtl G x a\) with dual blob \(X^* + \sbtlpp{\Lie^*(G)}x{-a}\).
\end{lem}

\begin{proof}
Note that, for every \(Z \in \sbtl{\Lie(H)}x r\) and \(W \in \sbtl{\Lie(H)}x a\), we have that \(\mexp(\sbjhd Y r + Z)\inv\mexp(\sbjhd Y r + Z + W)\) lies in \(\mexp(W)\sbtlp H x a\) by \xcite{spice:asymptotic}*{Hypothesis \initxref{hyp:mexp}(\incxref{coset})\subxref{iso}}, so that \(\contra\phi(\mexp(\sbjhd Y r + Z)\inv\mexp(\sbjhd Y r + Z + W))\) equals \(\contra\phi(\mexp(W)) = \contra{\AddChar_{X^*}}(W)\).  That is, the pullback of \(Z + \chrc{\sbtl{\Lie(H)}x a, \contra{\AddChar_{X^*}}}\) \textit{via} the map \(\anonmapto{\gamma\dotm\mexp(\sbjhd Y r + \sbjtl Y r)}{\sbjtl Y r}\) is \(\gamma\dotm\mexp(\sbjhd Y r + Z)\chrc{\sbtl H x a, \contra\phi}\).

Since \((X^* + \chrc{\sbtlpp{\Lie^*(H)}x{-r}})\spcheck\) equals
\[
\meas(\sbtl{\Lie(H)}x a)\sum_{\sbjtl Y r \in \sbtl{\Lie(H)}x r/{\sbtl{\Lie(H)}x a}} \contra{\AddChar_{X^*}}(\sbjtl Y r)(\sbjtl Y r + \chrc{\sbtl{\Lie(H)}x a, \contra{\AddChar_{X^*}}}),
\]
its pullback \textit{via} the above map is
\[
\meas(\sbtl H x a)\sum_{\sbjtl Y r \in \sbtl{\Lie(H)}x r/{\sbtl{\Lie(H)}x a}} \contra{\AddChar_{X^*}}(\sbjtl Y r)\bigl(\gamma\dotm\mexp(\sbjhd Y r + \sbjtl Y r)\chrc{\sbtl H x a, \contra\phi}\bigr).
\]
The result now follows from \xcite{spice:asymptotic}*{Corollary \xref{cor:centre}}.
\end{proof}

We now impose \xcite{spice:asymptotic}*{Hypothesis \xref{hyp:depth}}.  Corollary \ref{cor:char-sample} is the key tool for dealing with our `re-centred' expansions, depending on both \(\gamma\) and \(\sbjhd Y r\), in Theorem \ref{thm:char-asymptotic-exists}.

\begin{cor}[\xcite{spice:asymptotic}*{Corollary \xref{cor:sample}}]
\label{cor:char-sample}
With the notation and hypotheses of Lemma \ref{lem:char-sample}, suppose further that \(\pi\) is irreducible and contains \((\sbtl G o r, \hat\phi_o)\).  Then
\[
\widecheck\Theta_{\pi, \gamma, \sbjhd Y r}\bigl(X^* + \chrc{\sbtlpp{\Lie^*(H)}x{-r}}\bigr)
\]
vanishes unless we have \(a > r\) and \(X^* + \sbtlpp{\Lie^*(H)}x{-a}\) is degenerate, or we have \(a = r\) and there is some \(g \in G\) so that \(X^*\) belongs to \(\Ad^*(g)\inv\Gamma + \sbjtlpp{\Lie^*(H \cap \Int(g)\inv G')}{-r}\).
\end{cor}

\begin{thm}[\xcite{spice:asymptotic}*{Theorem \xref{thm:asymptotic-exists}}]
\label{thm:char-asymptotic-exists}
Suppose that \cite{jkim-murnaghan:charexp}*{Theorem 3.1.7(1, 5)} is satisfied, and all of the relevant orbital integrals converge.  For every irreducible representation \(\pi\) containing \((\sbtl G o r, \hat\phi_o)\), there is a finitely supported, \(\OO^{H\conn}(\Ad^*(G)\Gamma)\)-indexed vector \(b(\pi, \gamma, \sbjhd Y r)\) so that
\[
\Phi_\pi(\gamma\dotm\mexp(\sbjhd Y r + \sbjtl Y r))
\qeqq
\sum_{\OO \in \OO^{H\conn}(\Ad^*(G)\Gamma)}
	b_\OO(\pi, \gamma, \sbjhd Y r)\Ohat^{H\conn}_\OO(\sbjtl Y r)
\]
for all \(\sbjtl Y r \in \Lie(H)\rss \cap \sbjtl{\Lie(H)}r\).
\end{thm}

\begin{note}
See Remark \ref{rem:black-box} regarding our imposition of \cite{jkim-murnaghan:charexp}*{Theorem 3.1.7(1, 5)} as a hypothesis.
\end{note}

We now adopt the notation and assumptions of \S\ref{sec:cuspidal}.  In particular, we assume that \bG is connected and that \(\pi\) is supercuspidal, and define a supercuspidal representation \(\pi'\) of \(G'\) (that \emph{includes} a twist by the character \(\phi_{\ell - 1}\phi_\ell\), so that the depth of \(\pi'\) is \(r_\ell\), not \(r_{\ell - 1}\), unless \(r_{\ell - 1}\) equals \(r_\ell\)).

Since our goal is to compute supercuspidal characters (Theorem \ref{thm:pi-to-pi'}), and since the character of a supercuspidal representation is supported on elements that are compact modulo the centre \cite{deligne:support}*{p.~156, Th\'eor\`eme}, there is no loss of generality in requiring that \(\gamma\) be compact modulo centre, so we do so.

The statement of \xcite{spice:asymptotic}*{Theorem \xref{thm:asymptotic-pi-to-pi'}} imposed a condition on asymptotic expansions (that all orbits occurring in the asymptotic expansion had semisimple part conjugate to some fixed element \(X^*_o\)).  This was relevant for the intended applications of that result, but was never used in the proof.  For our applications, we need to remove the condition from the statement of Theorem \ref{thm:pi-to-pi'}.  This does not affect the proof.

\begin{thm}[\xcite{spice:asymptotic}*{Theorem \xref{thm:asymptotic-pi-to-pi'}}]
\label{thm:pi-to-pi'}
For every \(g \in G\), put \(\Gamma_g = \Ad^*(g)\inv\Gamma\), \(\bG'_g = \Int(g)\inv\bG'\), \(\UU\primedual_g = \Gamma_g + \sbjtlpp{\Lie^*(G'_g)}{-r}\), and, if \(\Gamma_g\) belongs to \(\Lie^*(H)\), also \(\bH'_g = \bH \cap \bG'_g\).

Suppose that, for every \(g \in G'\bslash G/H\conn\) for which \(\Gamma_g\) belongs to \(\Lie^*(H)\), \cite{jkim-murnaghan:charexp}*{Theorem 3.1.7(1, 5)} is satisfied, and there is a finitely supported, \(\OO^{H\primeconn_g}(\UU\primedual_g)\)-indexed vector \(c(\pi\primethen g, \gamma, \sbjhd Y r)\) such that
\begin{multline*}
\phi_\ell^g(\mexp(\sbjhd Y r + \sbjtl{Y'}r))\inv
\Phi_{\pi\primethen g}(\gamma\dotm\mexp(\sbjhd Y r + \sbjtl{Y'}r)) \qeqq \\
\sum_{\OO' \in \OO^{H\primeconn_g}(\UU\primedual_g)}
	\wtilde\Gauss_{G'_g/H'_g}(\OO', \gamma\dotm\mexp(\sbjhd Y r))
	c_{\OO'}(\pi^{\prime\,g}, \gamma, \sbjhd Y r)
	\Ohat^{H\primeconn_g}_{\OO'}(\sbjtl{Y'}r)
\end{multline*}
for all \(\sbjtl{Y'}r \in \Lie(H'_g)\rss \cap \sbjtl{\Lie(H'_g)}r\).  Then we have that
\begin{multline*}
\phi_\ell(\mexp(\sbjhd Y r + \sbjtl Y r))\inv
\Phi_\pi(\gamma\dotm\mexp(\sbjhd Y r + \sbjtl Y r)) \qeqq \\
\sum_{\substack{
	g \in G'\bslash G/H\conn \\
	\Gamma_g \in \Lie^*(H)
}}
\sum_{\OO'}
{}
	\wtilde\Gauss_{G/H}(\OO', \gamma\dotm\mexp(\sbjhd Y r))
	c_{\OO'}(\pi^{\prime\,g}, \gamma, \sbjhd Y r)
	\Ohat^{H\conn}_{\OO'}(\sbjtl Y r)
\end{multline*}
for all \(\sbjtl Y r \in \Lie(H)\rss \cap \sbjtl{\Lie(H)}r\), where \(\wtilde\Gauss\) is as in Notation \ref{notn:twisted-Gauss}.
\end{thm}

\begin{note}
See Remark \ref{rem:black-box} regarding our imposition of \cite{jkim-murnaghan:charexp}*{Theorem 3.1.7(1, 5)} as a hypothesis.

The group \(\bH'_g\) usually does \emph{not} equal \(\Int(g)\inv(\bH \cap \bG')\).

The representation \(\pi'\) of \xcite{spice:asymptotic}*{Theorem \xref{thm:asymptotic-pi-to-pi'}} is our \(\pi' \otimes \phi_\ell\inv\).

If \(\phi_\ell(\mexp(\sbjhd Y r + \sbjtl Y r))\) equals \(\phi_\ell(\mexp(\sbjhd Y r))\phi_\ell(\mexp(\sbjtl Y r))\) for all \(\sbjtl Y r \in \Lie(H)\rss \cap \sbjtl{\Lie(H)}r\), then, in Theorem \ref{thm:pi-to-pi'}, we can absorb the factor \(\phi_\ell(\mexp(\sbjhd Y r))\) into the constants \(c_{\OO'}(\pi\primethen g, \gamma, \sbjhd Y r)\).
\end{note}

\begin{rem}
\label{rem:represent-orbit}
We thank Sandeep Varma for pointing out that, if \(\Ohat^G_\xi\) (respectively, \(\Phi_\pi\)) is representable on all of \(\Lie(G)\) (respectively, \(G\)), and each \(\Ohat^{H\conn}_\OO\) is representable on all of \(\Lie(H)\)---as happens, for example, when \field has characteristic \(0\) \cite{hc:queens}*{Theorems 4.4 and 16.3}---then we may take the representing functions to vanish off their respective regular semisimple sets; so, since \(\gamma + Y\) is regular semisimple in \(\Lie(G)\) (respectively, \(\gamma\dotm\mexp(Y)\) is regular semisimple in \(G\)) if and only if \(Y\) is regular semisimple in \(\Lie(H)\), the equalities involving \(\Ohat^G_\xi\) in Theorems \ref{thm:orb-asymptotic-exists} and \ref{thm:orb-to-orb'} (respectively, those involving \(\Phi_\pi\) in Theorems \ref{thm:char-asymptotic-exists} and \ref{thm:pi-to-pi'}) hold for all \(Y \in \sbjtl{\Lie(H)}r\).
\end{rem}

\section{Unwinding the induction}
\label{sec:unwind}

Throughout \S\ref{sec:unwind}, we continue to use the
	local, non-Archimedean field \field
and	reductive group \bG
from \S\ref{sec:hyps}.
As in \S\S\ref{sec:orbits}--\ref{sec:characters}, we assume that \(\bG_\tamefield\) is split, and the characteristic of \(\sbjat\field 0\) is not \(2\).

Since Theorem \ref{thm:orb-unwind} applies only to orbital integrals on connected groups, and since the results of \cite{yu:supercuspidal} that we use in \S\ref{subsec:char-unwind} also require connectedness, we assume that \bG is connected.

Let \matnotn{Gamma}\Gamma be an element of \(\Lie^*(G)\).  Unlike in \S\S\ref{sec:orbits}, \ref{sec:characters}, we do not assume that \(\Gamma\) itself satisfies \xcite{spice:asymptotic}*{Hypothesis \xref{hyp:Z*}}.  Instead, we suppose that we have
	a non-negative integer \(\ell\),
	a sequence \(r_{-1} < r_0 < \dotsb < r_{\ell - 1}\) of real numbers,
and	a decomposition \(\Gamma = \sum_{j = -1}^{\ell - 1} \matnotn{Gammaj}{\Gamma^j}\).  Then we put \(\bG^\ell = \bG\), assume inductively for all \(j \in \sset{-1, \dotsc, \ell - 1}\) that \(\Gamma^j\) satisfies \xcite{spice:asymptotic}*{Hypothesis \initxref{hyp:Z*}(\subxref{central}, \subxref{good})}, with \(r_j\) and \(\bG^{j + 1}\) in place of \(r\) and \bG, and write \(\matnotn{Gj}{\bG^j} = \Cent_{\bG^{j + 1}}(\Gamma^j)\) for the associated tame, twisted Levi subgroup of \(\bG^{j + 1}\).  Thus, \(\matnotn{Gvec}\vbG = (\bG^0 \subseteq \dotsb \subseteq \bG^{\ell - 1} \subseteq \bG^\ell = \bG)\) is a nested, tame, twisted Levi sequence in \bG.  Note that there is a (potentially) still smaller tame, twisted Levi subgroup \(\bG^{-1}\) of \bG that we could have included, but did not include, in \vbG.

Put \(\mnotn r = r_{\ell - 1}\) and \(\matnotn{Gprime}{\bG'} = \bG^{\ell - 1}\).

\label{page:common:sec:unwind}
We refer to all this as the common notation of \S\ref{sec:unwind}.

\begin{rem}
\label{rem:Gamma-1=0}
As in Remark \ref{rem:Gamma=0}, we may, but need not, take \(\Gamma^{-1} = 0\) and \(r_{-1} = r_0\).
\end{rem}

\numberwithin{thm}{subsection}
\subsection{Orbital integrals}
\label{subsec:orb-unwind}

In addition to the common notation of \S\ref{sec:unwind}, let \(\gamma\) be a semisimple element of \(\Lie(G)\) satisfying Hypothesis \ref{hyp:fc-building}.

We require that Hypothesis \ref{hyp:MP-ad} be satisfied.

Put \(\matnotn J\bJ = \CC\bG{r_{-1}}(\gamma)\).  We require that
	\begin{itemize}
	\item \((\gamma, x, r_{-1})\) satisfy Hypothesis \ref{hyp:Lie-gamma} for all \(x \in \BB(J)\)
and	\item \(\gamma\) centralise \(\Lie(\bJ)\).
	\end{itemize}
Put \(\sbjhd Y{r_{-1}} = 0\).  For \(j \in \sset{0, \dotsc, \ell - 1}\), we now inductively choose a semisimple element \(Y^{(j)}\) of \(\Lie(\CC J{r_{j - 1}}(\sbjhd Y{r_{j - 1}}))\) satisfying Hypothesis \ref{hyp:fc-building}, with \(\CC\bJ{r_{j - 1}}(\sbjhd Y{r_{j - 1}})\) in place of \bG, put \(\matnotn{Yr}{\sbjhd Y{r_j}} = \sbjhd Y{r_{j - 1}} + Y^{(j)}\), and require that
	\begin{itemize}
	\item \((Y^{(j)}, x, r_j)\) satisfy Hypothesis \ref{hyp:Lie-gamma} for all \(x \in \BB(\CC J{r_j}(\sbjhd Y{r_j}))\),
	\item \(Y^{(j)}\) belong to \(\sbjtl{\Lie(\CC J{r_j}(\sbjhd Y{r_j}))}{r_{j - 1}}\),
and	\item \(Y^{(j)}\) centralise \(\Lie(\CC\bJ{r_j}(\sbjhd Y{r_j}))\).
	\end{itemize}
Note the rough similarity of these conditions to the definition of a normal approximation in \xcite{adler-spice:good-expansions}*{Definition \xref{defn:r-approx}}.
Put \(\matnotn H\bH = \CC\bJ r(\sbjhd Y r)\).

\begin{rem}
\label{rem:deep-Lie-symmetry}
The natural analogues of Remarks \ref{rem:Lie-symmetry} and Remark \ref{rem:mexp:steps} (without reference to \mexp) hold; for example, we have that \(\bH\conn\) equals both \(\Cent_\bG(\gamma + \sbjhd Y r)\conn\) and \(\Cent_\bG(\gamma, Y^{(0)}, \dotsc, Y^{(\ell - 1)})\conn\), and (so) that \(\Gamma^{\ell - 1}\) belongs to \(\Lie^*(H)\) if and only if \(\gamma\) and each \(Y^{(j)}\) belong to \(\Lie(G')\).
\end{rem}

The proof of Lemma \ref{lem:Lie-Gauss-combine} is similar to, but easier than, that of Lemma \ref{lem:gp-Gauss-combine}, with \xcite{spice:asymptotic}*{Corollary \xref{cor:Gauss-const}} replaced by Corollary \ref{cor:Gauss-const}, so we omit it.  Note that we do not require that \(\CCp\bG 0(\gamma)\) equal \bG.

\begin{lem}
\label{lem:Lie-Gauss-combine}
Suppose that \(\gamma\) and \(\sbjhd Y r\) belong to \(\Lie(G')\).  Put \(\bJ' = \bJ \cap \bG'\) and \(\bH' = \bH \cap \bG'\).  Then, for every \(X^* \in \Gamma + \sbjtlpp{\Lie^*(G')}{-r}\), we have that
{
\newcommand\WQpiece[3]{\Gauss_{#1/#2}(X^*, #3)}
\newcommand\WQ[3]{\frac{\WQpiece{#1}{#2}{#3}}{\WQpiece{#1'}{#2'}{#3}}}
\[
\WQ G H{\gamma + \sbjhd Y r}
\qeqq
\WQ G J\gamma\dotm\WQ J H{\sbjhd Y r}.
\]
}
\end{lem}

%
%

Fix \(\matnotn{xi}\xi \in \Gamma + \sbjtlpp{\Lie^*(G^{-1})}{-r_{-1}}\).

We have `partially' unwound the induction in Theorem \ref{thm:orb-unwind}, in the sense that there is a (potentially) still smaller tame, twisted Levi subgroup \(\bG^{-1}\) that we could use, but have not used, in our unwinding.  The fully unwound result follows just from shifting the indexing, but we find it convenient for our intended application in Theorem \ref{thm:char-unwind} to state the result in this form.

\begin{thm}
\initlabel{thm:orb-unwind}
For every \(g \in G\) and every \(j \in \sset{-1, \dotsc, \ell - 1}\), put \(\xi_g = \Ad^*(g)\inv\xi\), \(\Gamma_g = \Ad^*(g)\inv\Gamma\) and \(\Gamma^j_g = \Ad^*(g)\inv\Gamma^j\), \(\bG^j_g = \Int(g)\inv\bG^j\), \(\UU\thendual{-1}_g = \Gamma_g + \sbjtlpp{\Lie^*(G^{-1}_g)}{-r_{-1}}\), and, if \(\Gamma^j_g\) belongs to \(\Lie^*(J)\), also \(\bJ^j_g = \bJ \cap \bG^j_g\).


Suppose that all of the relevant orbital integrals converge; and, for every \(g \in G^{-1}\bslash G/J\conn\) for which \(\Gamma_g\) belongs to \(\Lie^*(J)\) and every \(j \in \sset{0, \dotsc, \ell - 1}\), that \cite{jkim-murnaghan:charexp}*{Theorem 3.1.7(1, 5)} holds with \((\bG^{j + 1}_g, \Gamma^j_g)\), \((\bJ^{j + 1}_g, \Gamma^j_g)\), or \((\bJ^j_g, \Gamma^j_g)\) in place of \((\bG, \Gamma)\), and there is a finitely supported, \(\OO^{J\thenconn 0_g}(\UU\thendual{-1}_g)\)-indexed vector \(c(\xi_g, \gamma)\) such that
\begin{multline*}
\Ohat^{G^0_g}_{\xi_g}(\gamma + \sbjtl{Y^0}{r_{-1}}) \qeqq \\
\sum_{\substack{
	g^0 \in G^{-1}_g\bslash G^0_g/J\thenconn 0_g \\
	\Gamma_{g g^0} \in \Lie^*(J^0_g)
}}
\sum_{\OO^0 \in \OO^{J\thenconn 0_g}(\UU\thendual{-1}_{g g^0})}
	\Gauss_{G^0_g/J^0_g}(\OO^0, \gamma)c_{\OO^0}(\xi_{g g^0}, \gamma)\Ohat^{J\thenconn 0_g}_{\OO^0}(\sbjtl{Y^0}{r_{-1}})
\end{multline*}
for all \(\sbjtl{Y^0}{r_{-1}} \in \Lie(J^0_g)\rss \cap \sbjtl{\Lie(J^0_g)}{r_{-1}}\).
Then we have that
\[
\Ohat^G_\xi(\gamma + \sbjhd Y r + \sbjtl Y r) \qeqq
\sum_{\substack{
	g \in G^{-1}\bslash G/J\conn \\
	\Gamma_g \in \Lie^*(J)
}} \sum_{\OO^0}
{}
	\Gauss_{G/J}(\OO^0, \gamma)c_{\OO^0}(\xi_g, \gamma)\Ohat^{J\conn}_{\OO^0}(\sbjhd Y r + \sbjtl Y r)
\]
for all \(\sbjtl Y r \in \Lie(H)\rss \cap \sbjtl{\Lie(H)}r\).
\end{thm}

\begin{note}
See Remark \ref{rem:black-box} regarding our imposition of \cite{jkim-murnaghan:charexp}*{Theorem 3.1.7(1, 5)} as a hypothesis.  Note that, for the pair \((\bJ^j_g, \Gamma^j_g)\), the result \cite{jkim-murnaghan:charexp}*{Theorem 3.1.7(1, 5)} becomes \cite{debacker:homogeneity}*{Theorem 2.1.5(1, 3)}, since \(\Gamma^j_g\) is centralised by \(\bJ^j_g\).

The group \(\bJ^0_g = \bJ \cap \Int(g)\inv\bG^0\) usually does \emph{not} equal \(\Int(g)\inv(\bJ \cap \bG^0)\); and we must distinguish between \(\bJ^0_g = \CC{\bG^0_g}{r_{-1}}(\gamma)\), which is a possibly disconnected subgroup of \(\bG^0_g\), and \(\bJ\conn\), which is the identity component of \bJ.
\end{note}

\begin{proof}
We proceed by induction on the height \(\ell\) of the twisted Levi sequence \vbG.
If \(\ell\) equals \(0\), then the hypothesis is the conclusion, and we are done.
Now suppose that \(\ell\) is positive, and that we have proven the result for \((\bG', \Gamma, \xi)\) instead of \((\bG, \Gamma, \xi)\), hence also for every \(G\)-conjugate of \((\bG', \Gamma, \xi)\).


Suppose that \(g \in G^{-1}\bslash G/H\conn\) is such that \(\Gamma^{\ell - 1}_g\) belongs to \(\Lie^*(H)\).  Put \(\bG'_g = \bG^{\ell - 1}_g = \Int(g)\inv\bG'\).

By Remark \ref{rem:deep-Lie-symmetry}, we have that \(\gamma + \sbjhd Y r\) belongs to \(\Lie(G'_g)\).  Put \(\bJ'_g = \bJ^{\ell - 1}_g = \bJ \cap \bG'_g = \CC{\bG'_g}{r_{-1}}(\gamma)\) and \(\bH'_g = \bH \cap \bG'_g = \CC{\bG'_g}r(\gamma + \sbjhd Y r)\).  Then we have by induction that \(\Ohat^{G'_g}_{\xi_g}(\gamma + \sbjhd Y r + \sbjtl{Y'}r)\) equals
\[
\sum_{\substack{
	g' \in G^{-1}_g\bslash G'_g/J\primeconn_g \\
	\Gamma_{g g'} \in \Lie^*(J'_{g g'})
}}
\sum_{\OO^0 \in \OO^{J\thenconn 0_{g g'}}(\UU\thendual{-1}_{g g'})}
	\Gauss_{G'_g/J'_g}(\OO^0, \gamma)
	c_{\OO^0}(\xi_{g g'}, \gamma)
	\Ohat^{J\primeconn_g}_{\OO^0}(\sbjhd Y r + \sbjtl{Y'}r)
\]
for all \(\sbjtl{Y'}r \in \Lie^*(H'_g)\rss \cap \sbjtl{\Lie^*(H'_g)}r\).  Since \(\Gamma_{g g'} = \Ad^*(g')\inv\Ad^*(g)\inv\Gamma\),
which certainly belongs to \(\Ad^*(g')\inv\Lie^*(G'_g) = \Lie^*(G'_g)\),
belongs to \(\Lie^*(J'_{g g'}) = \Lie^*(J) \cap \Lie^*(G'_{g g'})\) if and only if it belongs to \(\Lie^*(J)\), we have that
\begin{multline}
\tag{$*$}
\sublabel{eq:mu-G'g}
\Ohat^{G'_g}_{\xi_g}(\gamma + \sbjhd Y r + \sbjtl{Y'}r) \qeqq \\
\sum_{\substack{
	g' \in G^{-1}_g\bslash G'_g/J\primeconn_g \\
	\Gamma_{g g'} \in \Lie^*(J)
}}
\sum_{\OO^0 \in \OO^{J\thenconn 0_{g g'}}(\UU\thendual{-1}_{g g'})}
	\Gauss_{G'_g/J'_g}(\OO^0, \gamma)
	c_{\OO^0}(\xi_{g g'}, \gamma)
	\Ohat^{J\primeconn_g}_{\OO^0}(\sbjhd Y r + \sbjtl{Y'}r)
\end{multline}
for all \(\sbjtl{Y'}r \in \Lie(H'_g)\rss \cap \sbjtl{\Lie(H'_g)}r\).  Now remember that \(\Gamma^{\ell - 1}_g = \Ad^*(g)\inv\Gamma^{\ell - 1}\) is centralised by \(\Int(g)\inv\bG' = \bG'_g\).  For every \(g' \in G^{-1}_g\bslash G'_g/J\primeconn_g\)
such that \(\Gamma_{g g'}\) belongs to \(\Lie^*(J)\),
we have by Theorem \ref{thm:orb-asymptotic-exists}, applied with \((\bJ\primeconn_{g g'} = \bJ\primeconn_g, \xi_{g g'}, \Gamma^{\ell - 1}_{g g'} = \Gamma^{\ell - 1}_g)\) in place of \((\bG, \xi, \Gamma)\), that, for every \(\OO^0 \in \OO^{J\thenconn 0_{g g'}}(\UU\thendual{-1}_{g g'})\), there is some finitely supported vector \(c^{J\primeconn_g}(\OO^0, \sbjhd Y r)\) indexed by \(\OO^{H\primeconn_g}(\Ad^*(J\primeconn_g)\inv\Gamma^{\ell - 1}_g) = \OO^{H\primeconn_g}(\Gamma^{\ell - 1}_g)\) such that
\begin{equation}
\tag{$**$}
\sublabel{eq:mu-J'g}
\Ohat^{J\primeconn_g}_{\OO^0}(\sbjhd Y r + \sbjtl{Y'}r)
\qeqq
\sum_{\OO' \in \OO^{H\primeconn_g}(\Gamma^{\ell - 1}_g)}
	\Gauss_{J'_g/H'_g}(\OO', \sbjhd Y r)
	c^{J\primeconn_g}_{\OO'}(\OO^0, \sbjhd Y r)
	\Ohat^{H\primeconn_g}_{\OO'}(\sbjtl{Y'}r)
\end{equation}
for all \(\sbjtl{Y'}r \in \Lie(H'_g)\rss \cap \sbjtl{\Lie(H'_g)}r\).  Thus, combining \loceqref{eq:mu-G'g} and \loceqref{eq:mu-J'g} gives that
\begin{multline}
\tag{$\dag$}
\sublabel{eq:mu-G'g-deep}
\Ohat^{G'_g}_{\xi_g}(\gamma + \sbjhd Y r + \sbjtl{Y'}r) \qeqq \\
\sum_{\OO' \in \OO^{H\primeconn_g}(\Gamma^{\ell - 1}_g)}
	\Gauss_{G'_g/H'_g}(\OO', \gamma + \sbjhd Y r)
	c_{\OO'}(\xi_g, \gamma + \sbjhd Y r)
	\Ohat^{H\primeconn_g}_{\OO'}(\sbjtl{Y'}r)
\end{multline}
for all \(\sbjtl{Y'}r \in \Lie(H'_g)\rss \cap \sbjtl{\Lie(H'_g)}r\), where, for every \(\OO' \in \OO^{H\primeconn_g}(\Gamma^{\ell - 1}_g)\), we define \(c_{\OO'}(\xi_g, \gamma + \sbjhd Y r)\) to be
\[
\sum_{\substack{
	g' \in G^{-1}_g\bslash G'_g/J\primeconn_g \\
	\Gamma_{g g'} \in \Lie^*(J)
}}
\sum_{\OO^0 \in \OO^{J\thenconn 0_{g g'}}(\UU\thendual{-1}_{g g'})}
	\frac
		{\Gauss_{G'_g/J'_g}(\OO^0, \gamma)\Gauss_{J'_g/H'_g}(\OO', \sbjhd Y r)}
		{\Gauss_{G'_g/H'_g}(\OO', \gamma + \sbjhd Y r)}
	c_{\OO^0}(\xi_{g g'}, \gamma)
	c^{J\primeconn_g}_{\OO'}(\OO^0, \sbjhd Y r).
\]
Now fix \(\OO' \in \OO^{H\primeconn_g}(\Gamma^{\ell - 1}_g)\).  By the definition of \(\OO^{H\primeconn_g}(\Gamma^{\ell - 1}_g)\) (see \S\ref{subsec:p-adic}, p.~\pageref{OO-defn}), we have that \(\OO'\) contains an element of \(\Gamma^{\ell - 1}_g + \sbjtlpp{\Lie^*(H')}{-r}\).  Then Corollary \ref{cor:Gauss-const} gives that
\begin{multline*}
\frac
	{\Gauss_{J/H}(\OO', \sbjhd Y r)}
	{\Gauss_{G/H}(\OO', \gamma + \sbjhd Y r)}
\Bigl(\frac
	{\Gauss_{J'_g/H'_g}(\OO', \sbjhd Y r)}
	{\Gauss_{G'_g/H'_g}(\OO', \gamma + \sbjhd Y r)}
\Bigr)\inv
\qeqq \\
\frac
	{\Gauss_{J/H}(\Gamma^{\ell - 1}_g, \sbjhd Y r)}
	{\Gauss_{G/H}(\Gamma^{\ell - 1}_g, \gamma + \sbjhd Y r)}
\Bigl(\frac
	{\Gauss_{J'_g/H'_g}(\Gamma^{\ell - 1}_g, \sbjhd Y r)}
	{\Gauss_{G'_g/H'_g}(\Gamma^{\ell - 1}_g, \gamma + \sbjhd Y r)}
\Bigr)\inv;
\end{multline*}
and, since \(\OO^0\) contains an element of \(\Lie^*(J) \cap \UU\thendual{-1}_{g g'}\), which equals \(\Gamma_{g g'} + \sbjtlpp{\Lie(J_{g g'} \cap G^{-1}_{g g'})}{-r_{-1}}\) by \cite{adler-debacker:bt-lie}*{Lemma 3.5.3} (and tame descent) and hence is contained in
\[
\Gamma^{\ell - 1}_g + \sbjtlpp{\Lie(\Zent(J_{g g'} \cap G^{-1}_{g g'}))}{-r} + \sbjtlpp{\Lie(J_{g g'} \cap G^{-1}_{g g'})}{-r_{-1}} \subseteq \Gamma^{\ell - 1}_g + \sbjtlpp{\Lie(J'_{g g'})}{-r},
\]
that
\[
\frac
	{\Gauss_{G/J}(\OO^0, \gamma)}
	{\Gauss_{G'_g/J'_g}(\OO^0, \gamma)}
\qeqq
\frac
	{\Gauss_{G/J}(\Gamma^{\ell - 1}_{g g'}, \gamma)}
	{\Gauss_{G'_g/J'_g}(\Gamma^{\ell - 1}_{g g'}, \gamma)}.
\]
It follows that
\begin{gather*}
\frac
	{\Gauss_{G/J}(\OO^0, \gamma)\Gauss_{J/H}(\OO', \sbjhd Y r)}
	{\Gauss_{G/H}(\OO', \gamma + \sbjhd Y r)}
\Bigl(\frac
	{\Gauss_{G'_g/J'_g}(\OO^0, \gamma)\Gauss_{J'_g/H'_g}(\OO', \sbjhd Y r)}
	{\Gauss_{G'_g/H'_g}(\OO', \gamma + \sbjhd Y r)}
\Bigr)\inv
\intertext{equals}
\frac
	{\Gauss_{G/J}(\Gamma^{\ell - 1}_g, \gamma)\Gauss_{J/H}(\Gamma^{\ell - 1}_{g g'}, \sbjhd Y r)}
	{\Gauss_{G/H}(\Gamma^{\ell - 1}_g, \gamma + \sbjhd Y r)}
\Bigl(\frac
	{\Gauss_{G'_g/J'_g}(\Gamma^{\ell - 1}_g, \gamma)\Gauss_{J'_g/H'_g}(\Gamma^{\ell - 1}_{g g'}, \sbjhd Y r)}
	{\Gauss_{G'_g/H'_g}(\Gamma^{\ell - 1}_g, \gamma + \sbjhd Y r)}
\Bigr)\inv,
\end{gather*}
which, by Lemma \ref{lem:Lie-Gauss-combine}, is \(1\).  Thus, \(c_{\OO'}(\xi_g, \gamma + \sbjhd Y r)\) equals
\begin{equation}
\tag{$\S$}
\sublabel{eq:cO}
\sum_{\substack{
	g' \in G^{-1}_g\bslash G'_g/J\primeconn_g \\
	\Gamma_{g g'} \in \Lie^*(J)
}}
\sum_{\OO^0 \in \OO^{J\thenconn 0_{g g'}}(\UU\thendual{-1}_{g g'})}
	\frac
		{\Gauss_{G/J}(\OO^0, \gamma)\Gauss_{J/H}(\OO', \sbjhd Y r)}
		{\Gauss_{G/H}(\OO', \gamma + \sbjhd Y r)}
	c_{\OO^0}(\xi_{g g'}, \gamma)
	c^{J\primeconn_g}_{\OO'}(\OO^0, \sbjhd Y r).
\end{equation}

Now fix \(\sbjtl Y r \in \Lie(H)\rss \cap \sbjtl{\Lie(H)}r\), and, for convenience, write \(Y\) for \(\sbjhd Y r + \sbjtl Y r\).  It follows from Theorem \ref{thm:orb-to-orb'} and \loceqref{eq:mu-G'g-deep} that \(\Ohat^G_\xi(\gamma + Y)\) equals
\[
\sum_{\substack{
	g \in G'\bslash G/H\conn \\
	\Gamma^{\ell - 1}_g \in \Lie^*(H)
}}
\sum_{\OO' \in \OO^{H\primeconn_g}(\Gamma^{\ell - 1}_g)}
	\Gauss_{G/H}(\OO', \gamma + \sbjhd Y r)
	c_{\OO'}(\xi_g, \gamma + \sbjhd Y r)
	\Ohat^{H\conn}_{\OO'}(\sbjtl Y r),
\]
which, by \loceqref{eq:cO}, equals
\begin{align*}
\sum_g
\sum_{\OO'}
\sum_{\substack{
	g' \in G^{-1}_g\bslash G'_g/J\primeconn_g \\
	\Gamma_{g g'} \in \Lie^*(J'_g)
}}
\sum_{\OO^0 \in \OO^{J\thenconn 0_{g g'}}(\UU\thendual{-1}_{g g'})}
{}	& \Gauss_{G/J}(\OO', \gamma)c_{\OO^0}(\xi_{g g'}, \gamma)\times{} \\
	& \quad\times\Gauss_{J/H}(\OO^0, \sbjhd Y r)c^{J\primeconn_g}_{\OO'}(\OO^0, \sbjhd Y r)
	\Ohat^{H\conn}_{\OO'}(\sbjtl Y r).
\end{align*}

The next step looks complicated, but is mostly just re-ordering and re-grouping the sums.
Recall that the sums over \(\OO'\) and \(\OO^0\) are finitely supported.  Upon
	\begin{itemize}
	\item moving the sum over \(\OO'\) to the inside,
	\item re-writing
\begin{align*}
\sum_{\substack{
	g_1 \in G'\bslash G/H\conn \\
	\Gamma^{\ell - 1}_{g_1} \in \Lie^*(H)
}}
\sum_{\substack{
	g'_1 \in G^{-1}_{g_1}\bslash G'_{g_1}/J\primeconn_{g_1} \\
	\Gamma_{g_1g'_1} \in \Lie^*(J)
}}
={} & \sum_{g_2 \in G'\bslash G/J\conn}
\sum_{\substack{
	j_2 \in G'_{g_2}\bslash J\conn/H\conn \\
	\Gamma^{\ell - 1}_{g_2j_2} \in \Lie^*(H)
}}
\sum_{\substack{
	g'_1 \in G^{-1}_{g_2j_2}\bslash G'_{g_2j_2}/J\primeconn_{g_2j_2} \\
	\Gamma_{g_2j_2g'_1} \in \Lie^*(J)
}} \\
={} & \sum_{g_2 \in G'\bslash G/J\conn}
\sum_{\substack{
	j_2 \in G'_{g_2}\bslash J\conn/H\conn \\
	\Gamma^{\ell - 1}_{g_2j_2} \in \Lie^*(H)
}}
\sum_{\substack{
	g'_3 \in G^{-1}_{g_2}\bslash G'_{g_2}/J\primeconn_{g_2} \\
	\Gamma_{g_2g'_3} \in \Lie^*(J)
}}
= \sum_{\substack{
	g_4 \in G^{-1}\bslash G/J\conn \\
	\Gamma_{g_4} \in \Lie^*(J)
}}
\sum_{\substack{
	j_2 \in G'_{g_4}\bslash J\conn/H\conn \\
	\Gamma^{\ell - 1}_{g_4j_2} \in \Lie^*(H)
}},
\end{align*}
where \(g_1 = g_2j_2\), \(g'_3 = \Int(j_2)g'_1\), and \(g_4 = g_2g'_3\), so that \(g_1g'_1 = g_4j_2\),
	\item noting that \(\bJ\primeconn_{g_1}\) equals \(\bJ\primeconn_{g_1 g'_1} = \bJ\primeconn_{g_4 j_2} = \Int(j_2)\inv\bJ\primeconn_{g_4}\) and similarly that \(\bH\primeconn_{g_1}\) equals \(\Int(j_2)\inv\bH\primeconn_{g_4}\) and \(\Gamma^{\ell - 1}_{g_1}\) equals \(\Gamma^{\ell - 1}_{g_4j_2}\), and that \(\bJ\thenconn 0_{g_1g'_1}\) equals \(\Int(j_2)\inv\bJ\thenconn 0_{g_4}\), \(\xi_{g_1g'_1}\) equals \(\Ad^*(j_2)\inv\xi_{g_4}\), and \(\UU\thendual{-1}_{g_1g'_1}\) equals \(\Ad^*(j_2)\inv\UU\thendual{-1}_{g_4}\), so that \(\OO^{H\primeconn_{g_4j_2}}(\Gamma^{\ell - 1}_{g_4j_2})\) equals \(\OO^{H\primeconn_{g_1}}(\Gamma^{\ell - 1}_{g_1})\) and the map \anonmap{\OO^{J\conn i_{g_4}}(\UU\thendual{-1}_{g_4})}{\OO^{J\thenconn 0_{g_1g'_1}}(\UU\thendual{-1}_{g_1g'_1})} given by \anonmapto{\OO^0_4}{\Ad^*(j_2)\inv\OO^0_4} is a bijection,
	\item switching the order of the sums over \(j_2\) and \(\OO^0\),
and	\item using the fact that \(\gamma\) centralises \(\Lie(J)\), hence is centralised by \(J\conn\), to show that \(\Gauss_{G/J}(\Ad^*(j_2)\inv\OO^0_4, \gamma)\) equals \(\Gauss_{G/J}(\OO^0_4, \Ad(j_2)\gamma) = \Gauss_{G/J}(\OO^0_4, \gamma)\) and \(c_{\Ad^*(j_2)\inv\OO^0_4}(\xi_{g_1g'_1}, \gamma)\) equals \(c_{\OO^0_4}(\Ad^*(j_2)\xi_{g_1g'_1}, \Ad(j_2)\gamma) = c_{\OO^0_4}(\xi_{g_4}, \gamma)\),
	\end{itemize}
we see that \(\Ohat^G_\xi(\gamma + Y)\) equals
\begin{align*}
\sum_{\substack{
	g \in G^{-1}\bslash G/J\conn \\
	\Gamma_g \in \Lie^*(J)
}}
\sum_{\OO^0 \in \OO^{J\thenconn 0_g}(\UU\thendual{-1}_g)}
{}	& \Gauss_{G/J}(\OO^0, \gamma)
	c_{\OO^0}(\xi_g, \gamma)\times{} \\
        &\quad\times\underbrace{
		\sum_{\substack{
			j \in J\primeconn_g\bslash J\conn/H\conn \\
			\Gamma^{\ell - 1}_{g j} \in \Lie^*(H)
		}}
		\sum_{\OO' \in \OO^{H\primeconn_{g j}}(\Gamma^{\ell - 1}_{g j})}
			\Gauss_{J/H}(\Ad^*(j)\inv\OO^0, \sbjhd Y r)\times{}
	}_{(\P)} \\
        &\quad\hphantom{
		\sum_{\substack{
			j \in J\primeconn_g\bslash J\conn/H\conn \\
			\Gamma^{\ell - 1}_{g j} \in \Lie^*(H)
		}}
		\sum_{\OO' \in \OO^{H\primeconn_{g j}}(\Gamma^{\ell - 1}_{g j})}
	}\quad\times\underbrace{
			c^{J\primeconn_{g j}}_{\OO'}(\Ad^*(j)\inv\OO^0, \sbjhd Y r)
			\Ohat^{H\conn}_{\OO'}(\sbjtl Y r)
	}_{(\P)}.
\end{align*}
Finally, Theorem \ref{thm:orb-to-orb'}, with \((\bJ\conn, \OO^0, \Gamma^{\ell - 1}_g)\) in place of \((\bG, \xi, \Gamma)\), and \loceqref{eq:mu-J'g} show that (\(\P\)) equals \(\Ohat^{J\conn}_{\OO^0}(Y)\), so we are done.
\end{proof}

\begin{cor}
\label{cor:orb-vanish}
Preserve the notation of Theorem \ref{thm:orb-unwind}.
We have that \(\Ohat^G_\xi(\gamma + \sbjhd Y r + \sbjtl Y r)\) vanishes for all \(\sbjtl Y r \in \sbjtl{\Lie(H)}r\) unless \(\gamma\) lies in \(\Ad(G)\Lie(G^{-1})\).
\end{cor}

\begin{proof}
The set of \(g \in G^{-1}\backslash G/J^\circ\) for which \(\Gamma_g\) belongs to \(\Lie^*(J)\) is empty, so the hypotheses of Theorem \ref{thm:orb-unwind} are vacuously satisfied; and the sum computing \(\widehat O^G_\xi(\gamma + \sbjhd Y r + \sbjtl Y r)\) is over an empty set and hence vanishes.
\end{proof}

Corollary \ref{cor:orb-lc} may be viewed as a local constancy result for Fourier transforms of orbital integrals.
Part \initnoref{cor:orb-lc}\incpref{general} generalises and quantifies \cite{waldspurger:loc-trace-form}*{Proposition VIII.1} by dropping the requirement that \(\xi\) (there \(X\)) be regular, and re-centring at a possibly non-\(0\) element \(\gamma\);
and
part \initnoref{cor:orb-lc}\incpref{near-0} generalises \cite{jkim-murnaghan:gamma-asymptotic}*{Corollary 9.2.4} by dropping the requirement that \(\xi\) (there \(\Gamma'\)) be regular or elliptic.

\begin{cor}
\initlabel{cor:orb-lc}
Preserve the notation and hypotheses of Theorem \ref{thm:orb-unwind}.
Suppose further that \(\Cent_\bG(\xi)\conn\) equals \(\bG^0\).
\begin{enumerate}
\item\sublabel{general}
We have that
\[
\Ohat^G_\xi(\gamma + \sbjhd Y r + \sbjtl Y r)
\qeqq
\sum_{\substack{
	g \in G^0\bslash G/J\conn \\
	\Gamma_g \in \Lie^*(J)
}}
	\frac
		{\Gauss_{G/J}(\Gamma_g, \gamma)}
		{\Gauss_{G^0_g/J^0_g}(\Gamma_g, \gamma)}
	\AddChar_{\xi_g}(\gamma)
	\Ohat^{J\conn}_{\Gamma_g}(\sbjhd Y r + \sbjtl Y r)
\]
for all \(\sbjtl Y r \in \Lie(H)\rss \cap \sbjtl{\Lie(H)}r\).
\item\sublabel{near-0}
If \(\gamma\) equals \(0\), then \(\Ohat^G_\xi(\sbjhd Y r + \sbjtl Y r)\) equals \(\Ohat^G_\Gamma(\sbjhd Y r + \sbjtl Y r)\) for all \(\sbjtl Y r \in \Lie(H)\rss \cap \sbjtl{\Lie(H)}r\).
\item\sublabel{regular}
If \(\gamma\) lies in \(\Lie(G)\rss\), then
\[
\Ohat^G_\xi(\gamma + \sbjhd Y r + \sbjtl Y r)
\qeqq
\sum_{\substack{
	g \in G^0\bslash G/J\conn \\
	\Gamma_g \in \Lie^*(J)
}}
	\Gauss_{G/G^0_g}(\Gamma_g, \gamma)
	\AddChar_{\xi_g}(\gamma + \sbjhd Y r + \sbjtl Y r)
\]
for all \(\sbjtl Y r \in \Lie(H)\rss \cap \sbjtl{\Lie(H)}r\).
\end{enumerate}
\end{cor}

\begin{proof}
Note that \(\Cent_\bG(\xi)\conn = \Cent_{\bG^0}(\xi)\conn\), which equals \(\bG^0\) by hypothesis, is contained in \(\Cent_{\bG^0}(\Gamma)\) by \xcite{spice:asymptotic}*{Hypothesis \initxref{hyp:Z*}(\subxref{orbit})}, so also \(\bG^{-1} = \Cent_\bG(\Gamma) = \Cent_{\bG^0}(\Gamma)\) equals \(\bG^0\).  Fix \(g \in G\) such that \(\Gamma_g\) belongs to \(\Lie^*(J^0)\).
Since \(\xi_g\) and \(\Gamma_g\) are central in \(\bG^0_g\),
we have that
\[
\Ohat^{G^0_g}_{\xi_g}(\gamma + \sbjtl{Y^0}{r_{-1}}) = \AddChar_{\xi_g}(\gamma + \sbjtl{Y^0}{r_{-1}})
\qeqq
\AddChar_{\xi_g}(\gamma)\AddChar_{\Gamma_g}(\sbjtl{Y^0}{r_{-1}}) = \AddChar_{\xi_g}(\gamma)\Ohat^{G^0_g}_{\Gamma_g}(\sbjtl{Y^0}{r_{-1}})
\]
for all \(\sbjtl{Y^0}{r_{-1}} \in \Lie(J^0_g)\rss \cap \sbjtl{\Lie(J^0_g)}{r_{-1}}\).
Since \(g \in G\) 
was arbitrary, \locpref{general} follows from Theorem \ref{thm:orb-unwind}, and then \locpref{near-0} is an immediate consequence.

For \locpref{regular}, note that \(\bJ = \Cent_\bG(\gamma)\conn\) is a torus, so, for each \(g \in G\) with \(\Gamma_g \in \Lie^*(J)\), the tame, twisted Levi subgroup \(\bJ^0_g\) of \bJ is all of \bJ, and hence \(\Gauss_{G/J}(\Gamma_g, \gamma)\Gauss_{G^0_g/J^0_g}(\Gamma_g, \gamma)\inv\) equals \(\Gauss_{G/G^0_g}(\Gamma_g, \gamma)\).
\end{proof}

\subsection{Characters}
\label{subsec:char-unwind}


Our notation in \S\ref{subsec:char-unwind} is analogous to, but in important ways different from, that of \S\ref{subsec:orb-unwind}.  In addition to the common notation of \S\ref{sec:unwind} (see p.~\pageref{page:common:sec:unwind}), let \matnotn{gamma}\gamma be a semisimple element of \(G\) satisfying Hypothesis \ref{hyp:fc-building}.

We require that \(r_{-1}\) be non-negative, and put \(\matnotn{Rminus}{R_{-1}} = \max \sset{r_{-1}, \Rp0}\).
Put \(\matnotn J\bJ = \CC\bG{R_{-1}}(\gamma)\).  We require that
	\begin{itemize}
	\item \((\gamma, x, R_{-1})\) satisfy Hypothesis \ref{hyp:gp-gamma} for all \(x \in \BB(J)\)
and	\item \(\gamma\) centralise \(\Lie(\bJ)\).
	\end{itemize}
Put \(\sbjhd Y{r_{-1}} = 0\).  For \(j \in \sset{0, \dotsc, \ell - 1}\), we now inductively choose a semisimple element \(Y^{(j)}\) of \(\Lie(\CC J{r_{j - 1}}(\sbjhd Y{r_{j - 1}}))\) satisfying Hypothesis \ref{hyp:fc-building}, with \(\CC\bJ{r_{j - 1}}(\sbjhd Y{r_{j - 1}})\) in place of \bG, put \(\matnotn{Yr}{\sbjhd Y{r_j}} = \sbjhd Y{r_{j - 1}} + Y^{(j)}\), and require that
	\begin{itemize}
	\item \((Y^{(j)}, x, r_j)\) satisfy Hypothesis \ref{hyp:Lie-gamma} for all \(x \in \BB(\CC J{r_j}(\sbjhd Y{r_j}))\),
	\item \(Y^{(j)}\) belong to \(\sbjtl{\Lie(\CC J{r_j}(\sbjhd Y{r_j}))}{\max \sset{r_{j - 1}, \Rp0}}\),
and	\item \(Y^{(j)}\) centralise \(\Lie(\CC\bJ{r_j}(\sbjhd Y{r_j}))\).
	\end{itemize}
(Since \(r_{j - 1}\) is strictly positive when \(j\) is not \(0\), we can replace \(\max \sset{r_{j - 1}, \Rp0}\) by \(r_{j - 1}\) in that case; but we write it this way to have a uniform statement even when \(j\) equals \(0\).)
Put \(\matnotn H\bH = \CC\bJ r(\sbjhd Y r)\).

We now re-impose the notation of \S\ref{subsec:cuspidal-notn}, and Hypothesis \ref{hyp:mexp} for \((\CC\bG 0(\gamma), \bJ, R_{-1})\), as in \S\ref{sec:characters}.
In particular, we have inverse homeomorphisms \map{\matnotn{exp}\mexp}{\sbjtl{\Lie(J)}{R_{-1}}}{\sbjtl J{R_{-1}}} and \map{\matnotn{log}\mlog}{\sbjtl J{R_{-1}}}{\sbjtl{\Lie(J)}{R_{-1}}}, and supercuspidal representations \(\pi_j\) of each \(G^j\), so that \(\pi = \pi_\ell\) is a supercuspidal representation of \(G\).
We now seem to have two competing meanings for the elements \(\Gamma^j\) and the groups \(\bG^j\), but we require that they be compatible.  See Theorem \ref{thm:char-unwind}.

\begin{rem}
\label{rem:deep-gp-symmetry}
The natural analogues of Remarks \ref{rem:gp-symmetry} and Remark \initref{rem:mexp:steps} hold; for example, we have that \(\bH\conn\) equals both \(\Cent_\bG(\gamma\dotm\mexp(\sbjhd Y r))\conn\) and \(\Cent_\bG(\gamma, Y^{(0)}, \dotsc, Y^{(\ell - 1)})\conn\), and (so) that \(\Gamma^{\ell - 1}\) belongs to \(\Lie^*(H)\) if and only if \(\gamma\) belongs to \(G'\) and each \(Y^{(j)}\) belongs to \(\Lie(G')\).
\end{rem}

The strong representability requirement on \(\phi_j \circ \Int(g) \circ \mexp\) in Theorem \ref{thm:char-unwind} is closely related to \cite{jkim-murnaghan:gamma-asymptotic}*{Definition 4.1.3(2)}.  It is explained in \cite{fintzen-kaletha-spice:twist}*{proof of Lemma 4.3.3} that this representability can be arranged as long as each \(\phi_j\) is trivial on \(\Der G^j\) and the exponential map converges on \(\sbjtlp{(\bG^j/\Der\bG^j)(\field)}0\).

\begin{thm}
\initlabel{thm:char-unwind}
For every \(g \in G\) and every \(j \in \sset{-1, \dotsc, \ell - 1}\), put \(\Gamma_g = \Ad^*(g)\inv\Gamma\) and \(\Gamma^j_g = \Ad^*(g)\inv\Gamma_j\), \(\bG^j_g = \Int(g)\inv\bG^j\), \(\UU\thendual{-1}_g = \Gamma_g + \sbjtlpp{\Lie^*(G^{-1}_g)}{-r_{-1}}\), and, if \(\Gamma^j_g\) belongs to \(\Lie^*(J)\), also \(\bJ^j_g = \bJ \cap \bG^j_g\).

Suppose first,
for every index \(j \in \sset{0, \dotsc, \ell - 1}\) and every \(g \in G\) such that \(\Gamma^j_g\) belongs to \(\Lie^*(J)\), that \(\phi_j \circ \Int(g) \circ \mexp\) equals \(\AddChar_{\Gamma^j_g}\) on \(\sbjtl{\Lie(J^j_g)}{R_{-1}}\).


Suppose next that all relevant orbital integrals converge; and, for every \(g \in G^{-1}\bslash G/J\conn\) for which \(\Gamma_g\) belongs to \(\Lie^*(J)\) and every \(j \in \{0, \dotsc, \ell - 1\}\), that \cite{jkim-murnaghan:charexp}*{Theorem 3.1.7(1, 5)} holds with \((\bG^{j + 1}_g, \Gamma^j_g)\), \((\bJ^{j + 1}_g, \Gamma^j_g)\), or \((\bJ^j_g, \Gamma^j_g)\) in place of \((\bG, \Gamma)\), and there is a finitely supported, \(\OO^{J\thenconn 0_g}(\UU\thendual{-1}_g)\)-indexed vector \(c(\pi_0^g, \gamma)\) such that
\begin{multline*}
\phi_\ell^g(\mexp(\sbjtl{Y^0}{R_{-1}}))\inv
\Phi_{\pi_0^g}(\gamma\dotm\mexp(\sbjtl{Y^0}{R_{-1}})) \qeqq \\
\sum_{\substack{
	g^0 \in G^{-1}_g\backslash G^0_g/J\thenconn 0_g \\
	\Gamma_{g g^0} \in \Lie^*(J^0_g)
}}
\sum_{\OO^0 \in \OO^{J\thenconn 0_g}(\UU\thendual{-1}_{g g^0})}
	\wtilde\Gauss_{G^0_g/J^0_g}(\OO^0, \gamma)
	c_{\OO^0}(\pi_0^g, \gamma)
	\Ohat^{J\thenconn 0_g}_{\OO^0}(\sbjtl{Y^0}{R_{-1}})
\end{multline*}
for all \(\sbjtl{Y^0}{R_{-1}} \in \Lie(J^0_g)\rss \cap \sbjtl{\Lie(J^0_g)}{R_{-1}}\).  Then we have that
\begin{multline*}
\phi_\ell(\mexp(\sbjhd Y r + \sbjtl Y r))\inv
\Phi_\pi(\gamma\dotm\mexp(\sbjhd Y r + \sbjtl Y r)) \qeqq \\
\sum_{\substack{
	g \in G^{-1}\bslash G/J\conn \\
	\Gamma_g \in \Lie^*(J)
}} \sum_{\OO^0}
{}
	\wtilde\Gauss_{G/J}(\OO^0, \gamma)
	c_{\OO^0}(\pi_0^g, \gamma)
	\Ohat^{J\conn}_{\OO^0}(\sbjhd Y r + \sbjtl Y r)
\end{multline*}
for all \(\sbjtl Y r \in \Lie(H)\rss \cap \sbjtl{\Lie(H)}r\), where \(\wtilde\Gauss\) is as in Notation \ref{notn:twisted-Gauss}.
\end{thm}

\begin{note}
See Remark \ref{rem:black-box} regarding our imposition of \cite{jkim-murnaghan:charexp}*{Theorem 3.1.7(1, 5)} as a hypothesis.  As in Theorem \ref{thm:orb-unwind}, for the pair \((\bJ^j_g, \Gamma^j_g)\), the result \cite{jkim-murnaghan:charexp}*{Theorem 3.1.7(1, 5)} becomes \cite{debacker:homogeneity}*{Theorem 2.1.5(1, 3)}, since \(\Gamma^j_g\) is centralised by \(\bJ^j_g\).
\end{note}

\begin{proof}
Our proof is similar to Theorem \ref{thm:orb-unwind}, with some, but not all, references to Fourier transforms of orbital integrals replaced by references to characters of supercuspidal representations.

We proceed by induction on the height \(\ell\) of the twisted Levi sequence \vbG.
If \(\ell\) equals \(0\), then the hypothesis is the conclusion, and we are done.
Now suppose that \(\ell\) is positive,
and put \(\Gamma' = \Gamma - \Gamma^{\ell - 1}\).  Note that replacing the Yu datum \(((\bG^0 \subseteq \dotsb \subseteq \bG^{\ell - 1} \subseteq \bG^\ell), o, (r_0 < \dotsb < r_{\ell - 1} \le r_{\ell}), \rho_0', (\phi_0, \dotsc, \phi_{\ell - 1}, \phi_\ell))\) by \((\bG^0 \subseteq \dotsb \subseteq \bG^{\ell - 1}), o, (r_0 < \dotsb < r_{\ell - 1}), \rho_0', (\phi_0, \dotsc, \phi_{\ell - 2}, \phi_{\ell - 1}\phi_\ell))\) replaces \(\pi\) by \(\pi'\) and \(\Gamma\) by \(\Gamma'\).
Suppose that we have proven the result for \((\bG', \Gamma', \pi')\) instead of \((\bG, \Gamma, \pi)\), hence also for every \(G\)-conjugate of \((\bG', \Gamma', \pi')\).

Suppose that \(g \in G^{-1}\bslash G/H\conn\) is such that \(\Gamma^{\ell - 1}_g\) belongs to \(\Lie^*(H)\).  Put \(\bG'_g = \bG^{\ell - 1}_g = \Int(g)\inv\bG'\) and \(\pi' = \pi_{\ell - 1}\).

By Remark \ref{rem:deep-gp-symmetry}, we have that \(\gamma\dotm\mexp(\sbjhd Y r)\) belongs to \(G'_g\).  Put \(\bH'_g = \bH \cap \bG'_g = \CC{\bG'_g}r(\gamma\dotm\mexp(\sbjhd Y r))\).
For each \(g' \in \bG'_g\), we have that \(\Gamma^{\ell - 1}_g = \Gamma^{\ell - 1}_{g g'}\) is centralised by \(\bG'_{g g'}\), and in particular by \(\bJ\thenconn 0_{g g'}\), so that there is a bijection \anonmap{\OO^{J\thenconn 0_{g g'}}(\UU\thendual{-1}_{g g'})}{\OO^{J\thenconn 0_{g g'}}(\Ad^*(g g')\inv(\Gamma' + \sbjtlpp{\Lie^*(G^{-1})}{-r_{-1}}))} given by \anonmapto{\OO^0}{\OO^0 - \Gamma^{\ell - 1}_g};
and, if \(\Gamma_{g g'}\) belongs to \(\Lie^*(J^0_{g g'})\), then
\[
\wtilde\Gauss_{G^0_{g g'}/J^0_{g g'}}(\OO^0, \gamma)
\qeqq
\wtilde\Gauss_{G^0_{g g'}/J^0_{g g'}}(\OO^0 - \Gamma^{\ell - 1}_g, \gamma)
\]
and
\begin{multline*}
\phi_{\ell - 1}^{g g'}(\mexp(\sbjtl{Y^0}{R_{-1}}))\inv
\Ohat^{J\thenconn 0_{g g'}}_{\OO^0}(\sbjtl{Y^0}{R_{-1}})
\qeqq \\
\AddChar_{\Gamma^{\ell - 1}_{g g'}}(\sbjtl{Y^0}{R_{-1}})\inv
\Ohat^{J\thenconn 0_{g g'}}_{\OO^0}(\sbjtl{Y^0}{R_{-1}})
= \Ohat^{J\thenconn 0_{g g'}}_{\OO^0 - \Gamma^{\ell - 1}_g}(\sbjtl{Y^0}{R_{-1}})
\end{multline*}
for all \(\sbjtl{Y^0}{R_{-1}} \in \Lie(J^0_g)\rss \cap \sbjtl{\Lie(J^0_g)}{R_{-1}}\).
Thus, after multiplying both sides of our formula for
\(\phi_\ell^{g g'}(\mexp(\sbjhd{Y^0}{R_{-1}}))\inv\Phi_{\pi_0^{g g'}}(\gamma\dotm\mexp(\sbjhd{Y^0}{R_{-1}}))\)
by \(\phi_{\ell - 1}^{g g'}(\mexp(\sbjhd{Y^0}{R_{-1}}))\inv\) and applying the result for \((\bG'_g, \pi\primethen g, \Gamma'_g)\),
we find that
\begin{multline}
\tag{$*'$}
\sublabel{eq:too-twisted}
(\phi_{\ell - 1}\phi_\ell)^g(\mexp(\sbjhd Y{r_{\ell - 2}} + \sbjtl{Y'}{r_{\ell - 2}}))\inv\Phi_{\pi\primethen g}(\gamma\dotm\mexp(\sbjhd Y{r_{\ell - 2}} + \sbjtl{Y'}{r_{\ell - 2}}))
\qeqq \\
\sum_{\substack{
	g' \in G^{-1}_g\bslash G'_g/J\primeconn_g \\
	\Gamma_{g g'} \in \Lie^*(J'_{g g'})
}} \sum_{\OO^0}
{}
	\wtilde\Gauss_{G'_g/J'_g}(\OO^0 - \Gamma^{\ell - 1}_g, \gamma)
	c_{\OO^0}(\pi_0^{g g'}, \gamma)
	\Ohat^{J\primeconn_g}_{\OO^0 - \Gamma^{\ell - 1}_g}(\sbjhd Y{r_{\ell - 2}} + \sbjtl{Y'}{r_{\ell - 2}})
\end{multline}
for all \(\sbjtl{Y'}{r_{\ell - 2}} \in \Lie(\CC\bJ{r_{\ell - 2}}(\sbjhd Y{r_{\ell - 2}}))\rss \cap \sbjtl{\Lie(\CC\bJ{r_{\ell - 2}}(\sbjhd Y{r_{\ell - 2}}))}{r_{\ell - 2}}\).
For every \(\sbjtl{Y'}r \in \Lie(H'_g)\rss \cap \sbjtl{\Lie(H'_g)}r\), we can
	put \(\sbjtl{Y'}{r_{\ell - 2}} = Y^{(\ell - 1)} + \sbjtl{Y'}r\),
	multiply both sides of \loceqref{eq:too-twisted} by \(\phi_{\ell - 1}^{g g'}(\mexp(\sbjhd Y r + \sbjtl{Y'}r)) = \phi_{\ell - 1}^g(\mexp(\sbjhd Y r + \sbjtl{Y'}r))\),
and	use that
\[
\Gauss_{G'_g/J'_g}(\OO^0 - \Gamma^{\ell - 1}_g, \gamma)
\qeqq
\Gauss_{G'_g/J'_g}(\OO^0, \gamma)
\]
and
\begin{multline*}
\phi_{\ell - 1}^g(\mexp(\sbjhd Y r + \sbjtl{Y'}r))
\Ohat^{J\primeconn_g}_{\OO^0 - \Gamma^{\ell - 1}_g}(\sbjhd Y r + \sbjtl{Y'}r)
\qeqq \\
\AddChar_{\Gamma^{\ell - 1}_g}(\sbjhd Y r + \sbjtl{Y'}r)
\Ohat^{J\primeconn_g}_{\OO^0 - \Gamma^{\ell - 1}_g}(\sbjhd Y r + \sbjtl{Y'}r)
= \Ohat^{J\primeconn_g}_{\OO^0}(\sbjhd Y r + \sbjtl{Y'}r)
\end{multline*}
to find that
\begin{multline*}
\phi_\ell^g(\mexp(\sbjhd Y r + \sbjtl{Y'}r))\inv\Phi_{\pi\primethen g}(\gamma\dotm\mexp(\sbjhd Y r + \sbjtl{Y'}r))
\qeqq \\
\sum_{\substack{
	g' \in G^{-1}_g\bslash G'_g/J\primeconn_g \\
	\Gamma_{g g'} \in \Lie^*(J'_{g g'})
}} \sum_{\OO^0}
{}
	\wtilde\Gauss_{G'_g/J'_g}(\OO^0, \gamma)
	c_{\OO^0}(\pi_0^{g g'}, \gamma)
	\Ohat^{J\primeconn_g}_{\OO^0}(\sbjhd Y r + \sbjtl{Y'}r)
\end{multline*}
for all \(\sbjtl{Y'}r \in \Lie(H'_g)\rss \cap \sbjtl{\Lie(H'_g)}r\).
Thus, exactly as in the proof of Theorem \ref{thm:orb-unwind}, but now using \xcite{spice:asymptotic}*{Corollary \xref{cor:Gauss-const}} instead of Corollary \ref{cor:Gauss-const}, and Proposition \ref{prop:e-} and Lemma \ref{lem:gp-Gauss-combine} instead of Lemma \ref{lem:Lie-Gauss-combine}, we have that
\begin{multline}
\tag{$\dag$}
\sublabel{eq:pi'-G'g-deep}
\phi_\ell(\mexp(\sbjhd Y r + \sbjtl{Y'}r))\inv\Phi_{\pi\primethen g}(\gamma\dotm\mexp(\sbjhd Y r + \sbjtl{Y'}r)) \qeqq \\
\sum_{\OO' \in \OO^{H\primeconn_g}(\Gamma^{\ell - 1}_g)}
	\wtilde\Gauss_{G'_g/H'_g}(\OO', \gamma\dotm\mexp(\sbjhd Y r))
	\phi_{\ell - 1}^g(\gamma\dotm\mexp(\sbjhd Y r))c_{\OO'}(\pi\primethen g, \gamma, \sbjhd Y r)
	\Ohat^{H\primeconn_g}_{\OO'}(\sbjtl{Y'}r)
\end{multline}
for all \(\sbjtl{Y'}r \in \Lie(H'_g)\rss \cap \sbjtl{\Lie(H'_g)}r\), where, for every \(\OO' \in \OO^{H\primeconn_g}(\Gamma^{\ell - 1}_g)\), we define \(c_{\OO'}(\pi\primethen g, \gamma, \sbjhd Y r)\) to be
\begin{equation}
\tag{$\S$}
\sublabel{eq:cO}
\sum_{\substack{
	g' \in G^{-1}_g\bslash G'_g/J\primeconn_g \\
	\Gamma_{g g'} \in \Lie^*(J)
}}
\sum_{\OO^0 \in \OO^{J\thenconn 0_{g g'}}(\UU\thendual{-1}_{g g'})}
	\frac
		{\wtilde\Gauss_{G/J}(\OO^0, \gamma)\Gauss_{J/H}(\OO', \sbjhd Y r)}
		{\wtilde\Gauss_{G/H}(\OO', \gamma\dotm\mexp(\sbjhd Y r))}
	c_{\OO^0}(\pi_0^{g g'}, \gamma)
	c^{J\primeconn_g}_{\OO'}(\OO^0, \sbjhd Y r).
\end{equation}

Now fix \(\sbjtl Y r \in \Lie(H)\rss \cap \sbjtl{\Lie(H)}r\), and, for convenience, write \(Y\) for \(\sbjhd Y r + \sbjtl Y r\).
Put \(\sbjtl\phi 0 = \prod_{j = 0}^\ell \phi_j\).
It follows from Theorem \ref{thm:pi-to-pi'} and \loceqref{eq:pi'-G'g-deep} that \(\phi_\ell(\mexp(Y))\inv\Phi_\pi(\gamma\dotm\mexp(Y))\) equals
\[
\sum_{\substack{
	g \in G'\bslash G/H\conn \\
	\Gamma^{\ell - 1}_g \in \Lie^*(H)
}}
\sum_{\OO' \in \OO^{H\primeconn_g}(\Gamma^{\ell - 1}_g)}
	\wtilde\Gauss_{G/H}(\OO', \gamma\dotm\mexp(\sbjhd Y r))
	\phi_\ell(\gamma)c_{\OO'}(\pi\thenprime g, \gamma\dotm\mexp(\sbjhd Y r))
	\Ohat^{H\conn}_{\OO'}(\sbjtl Y r),
\]
which, by \loceqref{eq:cO}, equals
\begin{align*}
\sum_g
\sum_{\OO'}
\sum_{\substack{
	g' \in G^{-1}_g\bslash G'_g/J\primeconn_g \\
	\Gamma_{g g'} \in \Lie^*(J'_g)
}}
\sum_{\OO^0 \in \OO^{J\thenconn 0_{g g'}}(\UU\thendual{-1}_{g g'})}
{}	& \wtilde\Gauss_{G/J}(\OO', \gamma)\sbjtl\phi 0^{g g'}(\gamma)c_{\OO^0}(\pi_0^{g g'}, \gamma)\times{} \\
	& \quad\times\Gauss_{J/H}(\OO^0, \sbjhd Y r)c^{J\primeconn_g}_{\OO'}(\OO^0, \sbjhd Y r)
	\Ohat^{H\conn}_{\OO'}(\sbjtl Y r).
\end{align*}
By re-ordering and re-grouping the nested sums and using Theorem \ref{thm:orb-to-orb'} (\emph{not} Theorem \ref{thm:pi-to-pi'}!\@) with \((\bJ\conn, \OO^0, \Gamma^{\ell - 1}_g)\) in place of \((\bG, \xi, \Gamma)\), again exactly as in the proof of Theorem \ref{thm:orb-unwind}, we see that \(\phi_\ell(\mexp(Y))\inv\Phi_\pi(\gamma\dotm\mexp(Y))\) equals
\begin{align*}
\sum_{\substack{
	g \in G^{-1}\bslash G/J\conn \\
	\Gamma_g \in \Lie^*(J)
}}
\sum_{\OO^0 \in \OO^{J\thenconn 0_g}(\UU\thendual{-1}_g)}
{}	& \wtilde\Gauss_{G/J}(\OO^0, \gamma)
	\sbjtl\phi 0^g(\gamma)c_{\OO^0}(\pi_0^g, \gamma)\times{} \\
        &\quad\times\underbrace{
		\sum_{\substack{
			j \in J\primeconn_g\bslash J\conn/H\conn \\
			\Gamma^{\ell - 1}_{g j} \in \Lie^*(H)
		}}
		\sum_{\OO' \in \OO^{H\primeconn_{g j}}(\Gamma^{\ell - 1}_{g j})}
			\Gauss_{J/H}(\Ad^*(j)\inv\OO^0, \sbjhd Y r)\times{}
	}_{(\P)} \\
	&\quad\hphantom{
		\sum_{\substack{
			j \in J\primeconn_g\bslash J\conn/H\conn \\
			\Gamma^{\ell - 1}_{g j} \in \Lie^*(H)
		}}
		\sum_{\OO' \in \OO^{H\primeconn_{g j}}(\Gamma^{\ell - 1}_{g j})}
	}\quad\times\underbrace{
			c^{J\primeconn_{g j}}_{\OO'}(\Ad^*(j)\inv\OO^0, \sbjhd Y r)
			\Ohat^{H\conn}_{\OO'}(\sbjtl Y r)
	}_{(\P)},
\end{align*}
and that (\(\P\)) equals \(\Ohat^{J\conn}_{\OO^0}(Y)\), so we are done.
\end{proof}

\begin{cor}
\label{cor:char-vanish}
Preserve the notation of Theorem \ref{thm:char-unwind}.
We have that \(\Phi_\pi(\gamma\dotm\mexp(\sbjhd Y r + \sbjtl Y r))\) vanishes for all \(\sbjtl Y r \in \sbjtl{\Lie(H)}r\) unless \(\gamma\) belongs to \(\Int(G)(G^{-1})\).
\end{cor}

\begin{proof}
The set of \(g \in G^{-1}\backslash G/J^\circ\) for which \(\Gamma_g\) belongs to \(\Lie^*(J)\) is empty, so the hypotheses of Theorem \ref{thm:char-unwind} are vacuously satisfied; and the sum computing \(\Phi_\pi(\gamma\dotm\mexp(\sbjhd Y r + \sbjtl Y r))\) is over an empty set and hence vanishes.
\end{proof}

Corollary \ref{cor:char-lc} is a particularly easy-to-describe special case of Theorem \ref{thm:char-unwind}.  See \xcite{fintzen-kaletha-spice:twist}*{\S\xref{sub:char}} for the more involved and very important case where \(\pi\) is a regular supercuspidal representation.
Corollary \initref{cor:char-lc}\subpref{regular} is related to \cite{kaletha:regular-sc}*{\S4.4} and \cite{chan-oi:unramified-sc}*{Proposition 9.4}, although its hypotheses are more restrictive than either.

\begin{cor}
\initlabel{cor:char-lc}
Preserve the notation and hypotheses of Theorem \ref{thm:char-unwind}.
Suppose that \(\Cent_\bG(\Gamma)\conn\) equals \(\bG^0\).
\begin{enumerate}
\item\sublabel{near-1}
If \(\gamma\) equals \(1\), then
\begin{multline*}
\phi_\ell(\mexp(\sbjhd Y r + \sbjtl Y r))\inv\Phi_\pi(\mexp(\sbjhd Y r + \sbjtl Y r))
\qeqq \\
\sum_{\OO^0 \in \OO^{J\thenconn 0}(\Gamma)} c_{\OO^0}(\pi_0, 1)\Ohat^{J\conn}_{\OO^0}(\sbjhd Y r + \sbjtl Y r)
\end{multline*}
for all \(\sbjtl Y r \in \Lie(H)\rss \cap \sbjtl{\Lie(H)}r\).
\item\sublabel{regular}
If \(\bG^0\) is a torus and \(\gamma\) lies in \(G\rss\), then
\begin{multline*}
\phi_\ell(\mexp(\sbjhd Y r + \sbjtl Y r))\inv\Phi_\pi(\gamma\dotm\mexp(\sbjhd Y r + \sbjtl Y r))
\qeqq \\
\sum_{\substack{
	g \in G^0\bslash G/J\conn \\
	\Gamma_g \in \Lie^*(J)
}}
	\wtilde\Gauss_{G/G^0_g}(\Gamma_g, \gamma)
	\theta^g(\gamma)
	\AddChar_{\Gamma_g}(\sbjhd Y r + \sbjtl Y r)
\end{multline*}
for all \(\sbjtl Y r \in \Lie(H)\rss \cap \sbjtl{\Lie(H)}r\),
where \(\theta\) is the linear character \(\pi_0\) of \(G^0\).
\end{enumerate}
\end{cor}

\begin{proof}
The proof is similar to that of Corollary \ref{cor:orb-lc}.  As in that proof, we have that \(\bG^{-1}\) equals \(\bG^0\).

For \locpref{near-1}, we have that \(\bJ\) equals \(\bG\), so that \(G^{-1}\bslash G/J\conn\) is a singleton.  Our `base' expansion of \(\phi_\ell\inv\Phi_{\pi_0}\), used as input to Theorem \ref{thm:char-unwind}, is the local character expansion \cite{debacker:homogeneity}*{Theorem 3.5.2}, shifted by the \(\bG^0\)-fixed element \(\Gamma \in \Lie^*(G^0)\).  For \locpref{regular}, we observe that \(\pi_0 = \theta\) is a \(1\)-dimensional representation and \(\bJ\conn = \bH\conn = \bG^0_g = \bJ\thenconn 0_g\) is a torus for all \(g \in G\) such that \(\Gamma_g\) belongs to \(\Lie^*(J)\), and so use the equalities
\begin{gather*}
\phi_\ell^g(\mexp(\sbjtl{Y^0}{R_{-1}}))\inv\Phi_{\pi_0^g}(\gamma\dotm\mexp(\sbjtl{Y^0}{R_{-1}}))
= \theta^g(\gamma)\AddChar_{\Gamma_g}(\sbjtl{Y^0}{R_{-1}})
= \theta^g(\gamma)\Ohat^{J\thenconn 0_g}_{\Gamma_g}(\sbjtl{Y^0}{R_{-1}})
\intertext{and}
\Ohat^{J\conn}_{\Gamma_g}(\sbjhd Y r + \sbjtl Y r)
= \AddChar_{\Gamma_g}(\sbjhd Y r + \sbjtl Y r)
\end{gather*}
for all such \(g\), and all \(\sbjtl{Y^0}{R_{-1}} \in \Lie(J^0_g)\rss \cap \sbjtl{\Lie(J^0_g)}{R_{-1}}\) and \(\sbjtl Y r \in \Lie(H)\rss \cap \sbjtl{\Lie(H)}r\).
\end{proof}

\appendix

\numberwithin{thm}{section}
\section{Errata to \cite{spice:asymptotic}}
\label{app:errata}

\begin{rem}[\xcite{spice:asymptotic}*{Definition \xref{defn:tame-Levi}}]
\label{rem:defn:tame-Levi}
\hfill\begin{itemize}
\item When \bG is disconnected, it is too restrictive to require a tame Levi subgroup of \bG to satisfy \cite{digne-michel:non-connexe}*{D\'efinition 1.4}.  For example, the intersection of the subgroups contracted and dilated by an element \(\gamma \in G\) \cite{deligne:support}*{\S1, p.~151} need not be a Levi subgroup.
(We thank Stephen DeBacker for clarifying this for us.)  Instead, we follow \S\ref{sec:notation} in requiring of a parabolic \bP and a Levi subgroup \bM of \bG only that their intersections with \(\bG\conn\) are a parabolic and a Levi subgroup in the usual sense.  In particular, \bP belongs to \(\gNorm_\bG(\bP \cap \bG\conn)\), hence is contained in a parabolic subgroup of \bG in the sense of \cite{digne-michel:non-connexe}*{D\'efinition 1.4}; and \bM belongs to \(\gNorm_\bG(\bP \cap \bG\conn, \bM \cap \bG\conn)\), hence is contained in a Levi subgroup of \bG in the sense of \cite{digne-michel:non-connexe}*{D\'efinition 1.4}.
\item The definition of a tame Levi sequence \(\vbG = (\bG^0, \dotsc, \bG^\ell)\) incorrectly omits that the sequence be nested, with \(\bG^j \subseteq \bG^{j + 1}\) for all \(j < \ell\).  All tame, twisted Levi sequences actually used in \cite{spice:asymptotic} do satisfy this condition.  In order for \xcite{spice:asymptotic}*{Lemma \xref{lem:filtration}} to work, we need to impose the nesting requirement.  With this requirement in place, there is no need to require separately that there be a maximal \tamefield-split torus in \bG containing the \tamefield-split part of each \(\Zent(\bG\thenconn j)\), as every maximal \tamefield-split torus in \(\bG^0\) satisfies this requirement.
\end{itemize}
\end{rem}

\begin{rem}[\xcite{spice:asymptotic}*{Hypothesis \initxref{hyp:funny-centraliser}}]
\label{rem:hyp:funny-centraliser}
\hfill\begin{itemize}
\item In \xcite{spice:asymptotic}*{Hypothesis \initxref{hyp:funny-centraliser}(\subxref{basic})}, the statement that \(\CCp\bG 0(\gamma) \cap \bG\conn\) is the centraliser in \(\bG\conn\) of the absolutely-semisimple-modulo-\(\Zent(\CC\bG 0(\gamma)\conn)\) part of \(\gamma\) does not make sense literally as written.  It should say that \(\CCp\bG 0(\gamma) \cap \bG\conn\) is the fixed-point group of the absolutely semisimple part of the automorphism of \(\CC\bG 0(\gamma)\conn\) induced by \(\gamma\).  See Lemma \initref{lem:shallow-dfc}\subpref{0+}.
\item In \xcite{spice:asymptotic}*{Hypothesis \initxref{hyp:funny-centraliser}(\subxref{more-vGvr-facts})}, the equality
\[
\sbtl\vG x{\vec a} \cap \CC G i(\gamma\pinv)
\qeqq
\sbtl{(\CC{G^0}i(\gamma\pinv), \dotsc, \CC{G^\ell}i(\gamma\pinv))}x{\vec a}
\]
should involve only \(\CC G i(\gamma\pinv)\conn\) instead of \(\CC G i(\gamma\pinv)\), reading
\[
\sbtl\vG x{\vec a} \cap \CC G i(\gamma\pinv)\conn
\qeqq
\sbtl{(\CC{G^0}i(\gamma\pinv), \dotsc, \CC{G^\ell}i(\gamma\pinv))}x{\vec a}.
\]
(In addition, it should require that \(i\) be non-negative, and so speak of groups \(\CC{\bG^j}i(\gamma)\) instead of \(\CC{\bG^j}i(\gamma\pinv)\).  See the next item.)
\item More importantly, groups \(\CC\bG i(\gamma\pinv)\) and \(\CC\bG i(\gamma)\) with the required Lie algebras \xcite{spice:asymptotic}*{Hypothesis \initxref{hyp:funny-centraliser}(\subxref{Lie})} will rarely exist when \(i < 0\).  We thus regard \(\CC\bG i(\gamma)\) as defined only for those \(i \in \tR\) with \(0 \le i \le r\).  See Definitions \ref{defn:gp-dfc} and \ref{defn:vGvr}.
This seems to affect \xcite{spice:asymptotic}*{Proposition \xref{prop:lattice-orth}}, but actually that object deals only with the Lie-algebra analogues and their duals, where the appropriate vector spaces (not subalgebras) can easily be defined in terms of weight spaces, so that no change is necessary.
Note that \(\sbtl{\Lie(\CC G r(\gamma), G)}x{(\Rp0, (r - \ord_{\gamma^{\pm1}})/2)}\) there would be written in our notation (Notation \ref{notn:ord-gamma-spec}) as \(\sbtl{(H, G)}x{(\Rp0, (r - \max \sset{\ord_{\gamma - 1}, \ord_{\gamma\inv - 1}})/2)}\);
\begin{gather*}
\sbtl{\Lie(\CC G r(\gamma), N^+, G)}x{(\Rp0, (r - \ord_{\gamma\inv})/2, \Rpp{(r - \ord_\gamma)/2})}
\intertext{as, for example,}
\sbtlp{\Lie(H, M)}x{(0, (r - \ord_{\gamma - 1})/2)} + \sbtl{\Lie(\Prad^+)}x{(r - \ord_{\gamma\inv - 1})/2} + \sbtlpp{\Lie(\Prad^-)}x{(r - \ord_{\gamma - 1})/2};
\end{gather*}
and similarly for the other notation.
\end{itemize}
The objects whose existence we had to hypothesise in \xcite{spice:asymptotic}*{Hypothesis \xref{hyp:funny-centraliser}} are now explicitly constructed in \S\ref{sec:funny-centraliser}.  Thus, we are left having only to hypothesise good behaviour of the building, which we do in Hypothesis \ref{hyp:fc-building}.
\end{rem}

\begin{rem}[\xcite{spice:asymptotic}*{Hypothesis \xref{hyp:gamma}}]
In addition to the issues with groups \(\CC\bG i(\gamma\pinv)\) for \(i\) negative (see Remark \ref{rem:hyp:funny-centraliser}), this hypothesis should require that \(x\) belong to \(\BB(\CC{G^0}r(\gamma))\), not just to \(\BB(\vec G)\).  Both of these issues have been corrected in our analogues, Hypotheses \ref{hyp:gp-gamma} and \ref{hyp:Lie-gamma}.
\end{rem}

\begin{rem}[\xcite{spice:asymptotic}*{Lemma \xref{lem:commute-gp}}]
\label{rem:lem:commute-gp}
In addition to the issues with groups \(\CC\bG i(\gamma\pinv)\) for \(i\) negative (see Remark \ref{rem:hyp:funny-centraliser}), we need to impose stronger hypotheses in \xcite{spice:asymptotic}*{Lemma \initxref{lem:commute-gp}(\subxref{onto})}.  In the notation of that result, we should have required \(F < f \vee F\); \(F^\mp \ge f^\mp\)  and \(F\supn \ge f\supn + \ord_\gamma\) instead of just \(F^\mp \ge f^\mp + \ord_{\gamma^{\pm1}}\); and that \(h\) belong to \(\sbtl{\CC G r(\gamma)}x r\) as well as to \(\sbtl{\CC\vG r(\gamma)}x f \cap \sbtl{\CC G r(\gamma)}x r\).  These requirements have been corrected in our analogue, Lemma \initref{lem:commute-gp}\subpref{onto}, where, in the language of depth specifications, they become \(\tilde\tau < \tau \cvvv \tilde\tau\); \(\tilde\tau \ge \tau + \max \sset{\ord_\gamma, \ord_{\gamma\inv}}\); and that \(h\) belongs to \(\sbtl H x r\) as well as \(\sbtl\vH x\tau\).

The affected result \xcite{spice:asymptotic}*{Lemma \initxref{lem:commute-gp}(\subxref{onto})} is used in only two places.  It is used twice in \xcite{spice:asymptotic}*{Lemma \xref{lem:H-perp}}, where we are working inside \(\CC\bG 0(\gamma)\) anyway, and where the extra hypotheses are satisfied, so no harm is done.  It is also used in \xcite{spice:asymptotic}*{Lemma \xref{lem:fa-explicit}}, this time ``outside'' \(\CC\bG 0(\gamma)\).  We must therefore replace the groups \(\mc H_{\vec b}\) there by, in our notation, \(\sbtl{(\vH, \vM, \vG)}x{(\Rp{\vec b}, \vec b + r - \ord_{\gamma - 1}, \vec b + r)}\).  As argued in the proof of \xcite{spice:asymptotic}*{Lemma \xref{lem:fa-explicit}}, this change does not affect \(\indx{\mc H_{\vec a}}{\mc H_{\vec b}\dotm\sbtlp\vH x{\vec a}}\), so the conclusion is unaffected.  (Note that our analogous result Lemma \ref{lem:fa-explicit}, stated for \(\gamma\) a semisimple element of \(\Lie(G)\), does \emph{not} require this modification.)
\end{rem}

\begin{rem}
\label{rem:explicit-tame}
When we wrote \cite{spice:asymptotic}, we were not aware of issues with the dual Moy--Prasad sublattices such as the one described in Example \ref{exa:not-most-singular}.  To be safe, \xcite{spice:asymptotic}*{\S\S\xref{sec:qualitative}, \xref{sec:quantitative}} should have assumed that \(\bG_\tamefield\) is split.  Although we had gone to some length to avoid having to impose this condition directly, in practice the hypotheses we impose are most easily verified in the presence of tameness anyway (see \S\ref{subsec:sufficient}).
\end{rem}

\begin{rem}[\xcite{spice:asymptotic}*{Hypothesis \initxref{hyp:mexp}(\subxref{elt})}]
\label{rem:spice:asymptotic:hyp:mexp(elt)}
The conditions imposed on all \(Y \in \sbjtl{\Lie(H)}r\) in \xcite{spice:asymptotic}*{Hypothesis \initxref{hyp:mexp}(\subxref{elt})} are only needed when \(Y\) is regular semisimple.
\end{rem}

\begin{rem}[\xcite{spice:asymptotic}*{Corollary \xref{cor:sample}}]
\label{rem:asymptotic:cor:sample}
The last sentence of the proof of \xcite{spice:asymptotic}*{Corollary \xref{cor:sample}} claims only, in the notation there, that \(X^* + \sbtlpp{\Lie^*(H')}x{-r}\) intersects \(Z^*_o + \sbjtlpp{\Lie^*(H')}{-r}\).  This is true, but weaker than the condition that \(X^*\) belongs to \(Z^*_o + \sbjtlpp{\Lie^*(H')}{-r}\) claimed in the statement.  Fortunately, the fix is easy; the degeneracy of \((X^* - Z^*_o) + \sbtlpp{\Lie^*(H')}x{-r}\) shown in the penultimate sentence does indeed imply, by the dual-Lie-algebra analogue of \cite{adler-debacker:bt-lie}*{Corollary 3.2.6}, that \(X^* - Z^*_o\) belongs to \(\sbjtlpp{\Lie^*(H')}{-r}\).
\end{rem}

\begin{rem}[\xcite{spice:asymptotic}*{Hypotheses \initxref{hyp:MP-ad}(\subxref{refine}) and \xref{hyp:phi}, and Proposition \xref{prop:Gauss-to-Weil}}]
\label{rem:spice:asymptotic:hyp:MP-ad}
The map \(\sbat{\ol\mexp}x j\) of \xcite{spice:asymptotic}*{Hypothesis \xref{hyp:MP-ad}} should be from \(\sbat{\Lie(\vec M\adform)}x{\vec\jmath}\), not \(\sbat{\Lie(\vec M\adform)}x j\), to \(\sbat{(\vec M\adform)}x{\vec\jmath}\).

Both occurrences of \(\CC{\vec M}i(\gamma)\) in \xcite{spice:asymptotic}*{Hypothesis \initxref{hyp:MP-ad}(\subxref{ad})} should be \(\CC{\vec M\adform}i(\gamma)\).
All occurrences of \(\sbat M x{\vec\jmath}\) and \(\sbat{\Lie(M)}x{\vec\jmath}\) in \xcite{spice:asymptotic}*{Hypothesis \initxref{hyp:MP-ad}(\subxref{refine})} should be \(\sbat{(\vec M\adform)}x{\vec\jmath}\) and \(\sbat{\Lie(\vec M\adform)}x{\vec\jmath}\).
These have been corrected in our Lie-algebra analogue, Hypothesis \ref{hyp:MP-ad}.

Both occurrences of \(\CCp G 0(\gamma)\) in \xcite{spice:asymptotic}*{Hypothesis \xref{hyp:phi}} should be \(\CCp{M\adform}0(\gamma)\).

The proof of \xcite{spice:asymptotic}*{Proposition \xref{prop:Gauss-to-Weil}} should observe, after re-writing \(\uint_{\mc V^\perp} \hat\phi(\mathfrak Q_\gamma(v))\upd v\) as \(\prod_{0 < i < r} \uint_{\sbtl{(\CC{G'}i(\gamma), \CC G i(\gamma))}x{(r - i, (r - i)/2)}} \hat\phi(\mathfrak Q_\gamma(v))\upd v\), that, for all \(i \in \R\) with \(0 < i < r\) and all \(v \in \sbtl{(\CC{G'}i(\gamma), \CC G i(\gamma))}x{(r - i, (r - i)/2)}\), the quantity \(\hat\phi(\mathfrak Q_\gamma(v))\) depends only on the coset of \(v\) modulo \(\sbtl{(\CC{G'}i(\gamma), \CC G i(\gamma))}x{(r - i, \Rpp{(r - i)/2})}\); and that the natural map from
\[
\sbtl{(\CC{G'}i(\gamma), \CC G i(\gamma))}x{(r - i, (r - i)/2)}/\sbtl{(\CC{G'}i(\gamma), \CC G i(\gamma))}x{(r - i, \Rpp{(r - i)/2})}
\]
to
\[
\sbtl{(\CC{M'\adform}i(\gamma), \CC{M\adform}i(\gamma))}x{(r - i, (r - i)/2)}/\sbtl{(\CC{M'\adform}i(\gamma), \CC{M\adform}i(\gamma))}x{(r - i, \Rpp{(r - i)/2})}
\]
is an isomorphism.
\end{rem}

\begin{rem}[\xcite{spice:asymptotic}*{Definition \xref{defn:rho-pi}}]
Contrary to what is suggested by the notation, the quantity denoted by \(e^-_{\CCp G 0(\gamma^2)/{\CCp G 0(\gamma)}}(\gamma)\) in \xcite{spice:asymptotic}*{Definition \xref{defn:rho-pi}}, which, using Notation \ref{notn:e-}, equals
\[
\epsilon_x^{G/G'}(\gamma)\frac{\tr \tilde\phi_x(\gamma)}{\abs{\tr \tilde\phi_x(\gamma)}}\dotm\frac{\Gauss_{G/{\CCp G 0(\gamma^2)}}(X^*, \gamma)}{\Gauss_{G'/{\CCp{G'}0(\gamma^2)}}(X^*, \gamma)},
\]
can be non-trivial even if \(\CCp\bG 0(\gamma)\) equals \(\CCp\bG 0(\gamma^2)\).  In the notation of Remark \ref{rem:shallow-dfc} and \xcite{fintzen-kaletha-spice:twist}*{Definition \xref{dfn:lstk}}, we \emph{do} have that
\[
\epsilon_{\sharp, x}^{G/G'}(\sbjat\gamma 0)\frac{\tr \tilde\phi_x(\gamma)}{\abs{\tr \tilde\phi_x(\gamma)}}\dotm\frac{\Gauss_{G/{\CCp G 0(\gamma^2)}}(X^*, \gamma)}{\Gauss_{G'/{\CCp{G'}0(\gamma^2)}}(X^*, \gamma)}
\]
depends only on \(\CCp\bG 0(\gamma^2)\), and is trivial if \(\CCp\bG 0(\gamma^2)\) equals \(\CCp\bG 0(\gamma)\).
\end{rem}

\section{Depth specifications and concave functions}
\label{app:depth}

{\let\sstar=*
Throughout Appendix \ref{app:depth}, we let \((S, \sstar)\) be a semigroup (i.e., set with associative binary operation), and \(\le\) a partial ordering that is compatible with the semigroup operation (i.e., for which \(a' \le a\) implies \(a' \sstar b \le a \sstar b\) for all \(a, a', b \in S\)) with respect to which all non-empty sets have infima.  We will say for short that \((S, \sstar, \le)\) is a complete, partially ordered semigroup.

Remark \ref{rem:hyp:funny-centraliser} indicates the need for a replacement of the notion of depth matrices discussed in \xcite{spice:asymptotic}*{Definition \ref{defn:vGvr}}.  Our new notion is called a \noterm{depth specification}, although it might be more suggestive just to call it a triply indexed vector of depths.  The definition can be phrased quite generally, but the conditions that we impose on it look rather bizarre without context, so it will help to have in mind the situation of \S\ref{sec:subgps} while reading Appendix \ref{app:depth}.

\numberwithin{thm}{subsection}
\subsection{Depth specifications}
\label{subapp:spec}

Although the new notion of a depth specification (Definition \ref{defn:depth-spec}), which we just call a triply indexed vector in our generality (Definition \ref{defn:depth-ops}), is not much simpler in itself than that of a depth matrix, it does provide a pleasant simplification of the description of the operation \(\cvvv\) that is used to describe the interaction of two compact, open subgroups or sublattices attached to a depth specification.  See Lemma \ref{lem:filtration}.

Lemma \ref{lem:pre-assoc} is a technical result whose only purpose for us is to streamline the proof of Lemma \ref{lem:vee-assoc}.
Since there are two different orders simultaneously under consideration in Lemma \ref{lem:pre-assoc}, and since distinguishing them during the proof is particularly important, we use different symbols for them.  After Lemma \ref{lem:pre-assoc}, to avoid having to come up with a plethora of new order symbols (at least 4 different ones), we will simply denote all orders by \(\le\), and leave it to context to distinguish them.

\newcommand\vstar{\mathbin{\vec\sstar}}
\newcommand\vle{\mathrel{\vec\le}}
{
\makeatletter
\let\@inf=\inf
\edef\inf{\@inf\nolimits}
\let\@max=\max
\edef\max{\@max\nolimits}
\let\up@rrow\uparrow
\def\uparrow{\mathbin\up@rrow}
\makeatother
\newcommand\uuparrow{\mathbin{\uparrow\uparrow}}
\begin{lem}
\initlabel{lem:pre-assoc}
Let \((\mc X, \preceq)\) be a totally ordered set.  Define \(\vstar\) on \(S^{\mc X}\) by putting \((a_x)_{x \in \mc X} \vstar (b_x)_{x \in \mc X} = (c_x)_{x \in \mc X}\), where
\[
c_x = \inf_\le \sett{a_{x_1} \vstar b_{x_2}}{\(x \preceq x_1 = x_2\) or \(x_2 \preceq x = x_1\) or \(x_1 \preceq x_2 = x\)}
\]
for all \(x \in \mc X\).  Define the partial order \(\vle\) on \(S^{\mc X}\) by declaring, for all \(\vec a = (a_x)_{x \in \mc X}, \vec b = (b_x)_{x \in \mc X} \in S^{\mc X}\), that we have \(\vec a \vle \vec b\) if and only if \(a_x \le b_x\) for all \(x \in \mc X\).
Then \((S^{\mc X}, \vstar, \vle)\) is a complete, partially ordered semigroup.  If \(\sstar\) is commutative, then so is \(\vstar\).
\end{lem}

\begin{proof}
It is clear that all non-empty subsets of \(S^{\mc X}\) have infima with respect to \(\vle\), that \(\vle\) is compatible with \(\vstar\), and that \(\vstar\) is commutative if \(\sstar\) is.  (We do \emph{not} assume for the remainder of the proof that \(\sstar\) is commutative.)

For all \(x \in \mc X\), write \(x\uparrow\) for the set of ordered pairs \((x_1, x_2)\) of elements of \mc X such that we have \(x \preceq x_1 = x_2\) or \(x_2 \preceq x = x_1\) or \(x_1 \preceq x_2 = x\), and then \(x\uuparrow\) for the set of ordered triples \((x_1, x_2, x_3)\) of elements of \mc X such that, for some \(x_{23} \in \mc X\), we have \((x_1, x_{23}) \in x\uparrow\) and \((x_2, x_3) \in x_{23}\uparrow\).
Let us call a triple \((x_1, x_2, x_3)\) \emph{redundantly maximised} if at least two of the components are maximal; that is, if \(x_1 \preceq x_2 = x_3\) or \(x_3 \preceq x_1 = x_2\) or \(x_2 \preceq x_1 = x_3\).

If \(\vec a^1\), \(\vec a^2\), and \(\vec a^3\) are three vectors in \(S^{\mc X}\), and if we put \(\vec b^{23} = \vec a^2 \vstar \vec a^3\) and \(\vec c = \vec a^1 \vstar \vec b^{23}\), then, by definition, we have
\begin{align*}
c_x
& {}= \inf_\le \set{a^1_{x_1} \sstar b^{23}_{x_{23}}}
	{(x_1, x_{23}) \in x\uparrow} \\
& {}= \inf_\le \set{a^1_{x_1} \sstar a^2_{x_2} \sstar a^3_{x_3}}{(x_1, x_2, x_3) \in x\uuparrow}
\end{align*}
for all \(x \in \mc X\).  Applying the same definition to the opposite semigroup \((S, \sstar\textsup{op}, \le)\) shows that, for every \(x \in \mc X\), the \(x\)-component of \((\vec a^1 \vstar \vec a^2) \vstar \vec a^3 = \vec a^3 \vstar\textsup{op} (\vec a^2 \vstar\textsup{op} \vec a^1)\) equals
\[
\inf_\le \set{a^1_{x_1} \sstar a^2_{x_2} \sstar a^3_{x_3}}{(x_3, x_2, x_1) \in x\uuparrow}.
\]
It thus suffices to show that \(x\uuparrow\) is stable under the involution \anonmapto{(x_1, x_2, x_3)}{(x_3, x_2, x_1)}.

To show this, it suffices to show that, for all \(x \in \mc X\), the set \(x\uuparrow\) contains precisely those triples \((x_1, x_2, x_3)\) such that
	we have \(x \preceq \max_\preceq \sset{x_1, x_2, x_3}\),
and	\(x = \max_\preceq \sset{x_1, x_2, x_3}\) or \((x_1, x_2, x_3)\) is redundantly maximised.

Fix \(x, x_1, x_2, x_3 \in \mc X\), and put \(m = \max_\preceq \sset{x_1, x_2, x_3}\).

First suppose that \((x_1, x_2, x_3)\) belongs to \(x\uuparrow\), and choose \(x_{23} \in \mc X\) such that \((x_1, x_{23})\) belongs to \(x\uparrow\) and \((x_2, x_3)\) belongs to \(x_{23}\uparrow\).  In particular, we have \(x_{23} \preceq \max_\preceq \sset{x_2, x_3}\) and \(x \preceq \max_\preceq \sset{x_1, x_{23}} \preceq m\).  Suppose that \((x_1, x_2, x_3)\) is not redundantly maximised.

If we have \(x_2, x_3 \prec x_1 = m\), then we have \(x_{23} \preceq \max_\preceq \sset{x_2, x_3} \prec x_1 = m\).  This means that both \(x \preceq x_1 = x_{23}\) and \(x_1 \preceq x_{23} = x\) are false, so we must have \(x_{23} \preceq x = x_1\).  In particular, \(x\) equals \(m\).

If we have \(x_1 \prec m\), and \(x_2 \prec x_3 = m\) or \(x_3 \prec x_2 = m\), then, in the first case, both \(x_{23} \preceq x_2 = x_3\) and \(x_3 \preceq x_{23} = x_2\) are false, so that we have \(x_2 \preceq x_3 = x_{23}\); and, in the second case, similarly we must have \(x_3 \preceq x_{23} = x_2\).  In either case, \(x_{23}\) equals \(m\).  Again, this means that both \(x \preceq x_1 = x_{23}\) and \(x_{23} \preceq x = x_1\) are false, so we must have \(x_1 \preceq x_{23} = x\).  In particular, \(x\) equals \(m\).

Now suppose instead that we have \(x \preceq m\), and that \(x\) equals \(m\) or \((x_1, x_2, x_3)\) is redundantly maximised.  We shall find some \(x_{23} \in \mc X\) such that \((x_1, x_{23})\) belongs to \(x\uparrow\) and \((x_2, x_3)\) belongs to \(x_{23}\uparrow\).

If \(x\) equals \(m\), then we put \(x_{23} = \max_\preceq \sset{x_2, x_3}\).  This suffices since we have \(x_3 \preceq x_{23} = x_2\) or \(x_2 \preceq x_3 = x_{23}\); and since, because of the equality \(x = m = \max_\preceq \sset{x_1, x_{23}}\), we also have \(x_{23} \preceq x = x_1\) or \(x_1 \preceq x_{23} = x\).

If \(x_2\) and \(x_3\) equal \(m\), then we put \(x_{23} = \max_\preceq \sset{x, x_1}\).  This suffices since we have \(x_{23} \preceq x_2 = x_3\), and also \(x \preceq x_1 = x_{23}\) or \(x_1 \preceq x_{23} = x\).

If \(x_1\), and \(x_2\) or \(x_3\), equals \(m\), then we put \(x_{23} = m = x_1 = \max_\preceq \sset{x_2, x_3}\).  This suffices since we have \(x_3 \preceq x_{23} = x_2\) or \(x_2 \preceq x_3 = x_{23}\), and also \(x \preceq x_1 = x_{23}\).
\end{proof}
}

During Appendix \ref{app:depth}, we will have occasion to work with many different orders.  In Lemma \ref{lem:pre-assoc}, we used different symbols for different orders; but now so many different symbols would be needed that, rather than trying for orthographic creativity, we simply denote them all by \(\le\), with one exception.

Let \((\mc X_0, \le)\) and \((\mc X_1, \ge)\) be totally ordered sets, \((\mc X_2, +)\) a monoid (i.e., semigroup with identity), and \(\le\) a total order on \(\mc X_2\) compatible with \(+\).  We will typically denote elements of \(\mc X_0\) by \(j\), elements of \(\mc X_1\) by \(i\), and elements of \(\mc X_2\) by \(v\).
Note that we write the order for \(\mc X_1\) opposite to those for \(\mc X_0\) and \(\mc X_2\); this is not required in our abstract setting, but is relevant for our intended application in \S\ref{sec:subgps}.
Namely, if we have \(i_- \le i_+\) and \(j_- \le j_+\), then, in the notation of \S\ref{sec:subgps}, we have that \(\CC\bG{i_-}(\gamma)\) is a \emph{larger} group than \(\CC\bG{i_+}(\gamma)\), whereas \(\bG^{j_-}\) is a \emph{smaller} group than \(\bG^{j_+}\).
It may help to use Remark \ref{rem:depth-ops} to see the effect of this mirrored notation.

Definition \ref{defn:depth-ops} is an adaptation of \xcite{spice:asymptotic}*{Definition \xref{defn:concave-vee}} to our setting.  It gives a few operations on multiply indexed vectors.
Of these, only \(\cvvv\) is of interest in most of the paper, but we need the others to define it.

\begin{defn}
\label{defn:depth-ops}
Suppose that the operation \(\sstar\) on \(S\) is commutative.

%
A
\noterm{singly indexed vector} is an element of \(S_1 \ldef S^{\mc X_0}\); a
\noterm{doubly indexed vector} is an element of \(S_2 \ldef S_1^{\mc X_1}\);
and a
\noterm{triply indexed vector}
is an element of \(S_2^{\mc X_2}\).
We denote the orders on \(S_1\) and \(S_2\) again by \(\le\).

We say that a doubly indexed vector \(A = (\vec a_i)_{i \in \mc X_1}\) is \term[vector!doubly indexed!monotone]{monotone} if, for every \(i_-, i_+ \in \mc X_1\) with \(i_- \le i_+\), we have \(\vec a_{i_-} \ge \vec a_{i_+}\); and that \(A\) is \term[vector!doubly indexed!constant]{constant} if, for every \(i_-, i_+ \in \mc X_1\), we have \(\vec a_{i_-} = \vec a_{i_+}\).

We say that a triply indexed vector \(\rho = (A_v)_{v \in \mc X_2}\) is \term[vector!triply indexed!monotone]{monotone} if every doubly indexed component \(A_v\) is monotone; and that \(\rho\) is \term[vector!triply indexed!skew-constant]{skew-constant} if every doubly indexed component \(A_v\) corresponding to some \(v \ne 0\) is constant.

We equip singly indexed vectors with the operation \(\cvvv_1 \ldef \vstar\) deduced from \(\sstar\) using Lemma \ref{lem:pre-assoc}, and doubly indexed vectors with the operation \(\cvvv_2 \ldef \overrightarrow{\cvvv_1}\) deduced from \(\cvvv_1\) using Lemma \ref{lem:pre-assoc} (but remembering that we write the order on \(\mc X_1\) in the opposite direction from that on \(\mc X_0\)).
We say that a singly indexed vector \(\vec a\) is \term[vector!singly indexed!concave]{concave} if the inequality \(\vec a \le \vec a \cvvv_1 \vec a\) holds.

We equip triply indexed vectors with a different sort of operation \(\cvvv = \cvvv_3\).
If \(\rho = (A_v)_{v \in \R}\) and \(\sigma = (B_v)_{v \in \R}\) are triply indexed vectors, then we define \(\rho \cvvv_3 \sigma\), or just \matnotn{vee}{\rho \cvvv \sigma}, to be \((C_v)_{v \in \R}\), where \(C_v = \inf \set{A_{v_1} \cvvv_2 B_{v_2}}{v_1 + v_2 = v}\) for all \(v\).
We say that \(\rho\) is \term[vector!triply indexed!concave]{concave} if the inequality \(\rho \le \rho \cvvv \rho\) holds.
\end{defn}

\begin{rem}
\label{rem:depth-ops}
Although we find it helpful, specifically for Lemma \ref{lem:vee-assoc}, to define \(\cvvv\) in stages, it may help to think of all the infima in Definition \ref{defn:depth-ops} as being taken at once.  Namely, using the notation there, if we write \(\rho = (a_{v i j})_{v, i, j}\) and \(\sigma = (b_{v i j})_{v, i, j}\), then, for all \(v\), \(i\), and \(j\), the \((v, i, j)\) entry of \(\rho \cvvv \sigma\) is the infimum of all \(a_{v_1 i_1 j_1} * b_{v_2 i_2 j_2}\), where
	\(v = v_1 + v_2\);
	\(i \ge i_1 = i_2\) or \(i_2 \ge i = i_1\) or \(i_1 \ge i_2 = i\);
and	\(j \le j_1 = j_2\) or \(j_2 \le j = j_1\) or \(j_1 \le j_2 = j\).  In particular, the doubly indexed component of \(\rho \cvvv \sigma\) at \(v\) is majorised by \(A_v \cvvv_2 B_v\), and the singly indexed component of \(\rho \cvvv \sigma\) at \((v, i)\) is majorised by \(\vec a_{v i} \cvvv_1 \vec b_{v i}\).
\end{rem}

Unlike for the three operations \(\bowtie\), \(\rtimes\), and \(\vee\) defined for concave functions in \xcite{spice:asymptotic}*{Lemma \xref{lem:der-master-comm}}, and for depth matrices in \xcite{spice:asymptotic}*{Definition \xref{defn:concave-vee}}, it turns out that we need only the single operation \(\cvvv\) for depth specifications.  We make this precise in Lemma \ref{lem:filtration}, Corollary \ref{cor:matrix-vs-spec-vee}, and Lemma \ref{lem:depth-vs-concave-vee}.

First we prove a simple result, Lemma \ref{lem:vee-monotone}, involving monotonicity and skew-constancy.  This lemma is not directly needed in the other material, but is relevant in the context of Remark \ref{rem:vGvr-Mperp}.

\begin{lem}
\label{lem:vee-monotone}
Preserve the notation and terminology of Definition \ref{defn:depth-ops}.
\begin{enumerate}
\item If \(A\) and \(B\) are monotone, then so is \(A \cvvv_2 B\).  If \(A\) or \(B\) is constant, then so is \(A \cvvv_2 B\), with every singly indexed component equal to \(\inf \set{\vec a_{i_1} + \vec b_{i_2}}{i_1, i_2 \in \mc X_1}\).
\item If \(\rho\) and \(\sigma\) are monotone, then so is \(\rho \cvvv_3 \sigma\).  If \(\rho\) or \(\sigma\) is skew-constant, then, for every \(v \ne 0\), the doubly indexed component of \(\rho \cvvv_3 \sigma\) at \(v\) is \(\inf \sett{\vec a_{v_1 i_1} + \vec a_{v_2 i_2}}{\(v_1 + v_2 = v\) and \(i_1, i_2 \in \mc X_1\)}\).
\end{enumerate}
\end{lem}

\begin{proof}
Put \(A \cvvv_2 B = (\vec c_i)_{i \in \mc X_1}\).

Suppose first that \(A\) and \(B\) are monotone, and fix \(i_-, i_+ \in \mc X_1\) with \(i_- \le i_+\).  Fix \(i_{1-}, i_{2-} \in \mc X_1\) satisfying \(i_- \ge i_{1-} = i_{2-}\), respectively \(i_{2-} \ge i_- = i_{1-}\) or \(i_{1-} \ge i_{2-} = i_-\).  If we put \((i_{1+}, i_{2+}) = (i_{1-}, i_{2-})\), respectively \((i_{1+}, i_{2+}) = (\max \sset{i_{1-}, i_+}, \max \sset{i_{2-}, i_+})\), then we have \(i_{1-} \le i_{1+}\) and \(i_{2-} \le i_{2+}\) (in both cases) and \(i_+ \ge i_{1+} = i_{2+}\), respectively \(i_{2+} \ge i_+ = i_{1+}\) or \(i_{1+} \ge i_{2+} = i_+\).  It follows that we have \(\vec c_{i_+} \le \vec a_{v i_{1+}} + \vec b_{v i_{2+}} \le \vec a_{v i_{1-}} + \vec b_{v i_{2-}}\).  Taking the infimum over all \((i_{1-}, i_{2-})\) shows that we have \(\vec c_{i_+} \le \vec c_{i_-}\).  That is, \(A \cvvv_2 B\) is monotone.  It follows immediately that, if \(\rho\) and \(\sigma\) are monotone, then so is \(\rho \cvvv_3 \sigma\).

Now suppose that \(A\) or \(B\) is constant.  Since \(\cvvv_2\) is commutative, we may, and do, suppose that \(A\) is.  Then, for every \(i \in \mc X_1\), we have \(\vec c_i \le \vec a_{\max \sset{i, i_2}} + \vec b_{i_2} = \vec a_{i_1} + \vec b_{i_2}\) for every \(i_1, i_2 \in \mc X_1\), so that \(\vec c_i \le \inf \set{\vec a_{i_1} + \vec b_{i_2}}{i_1, i_2 \in \mc X_1}\).  Since the reverse inequality is obvious, we have equality.

Finally, suppose that \(\rho\) or \(\sigma\) is skew-constant.  Again, since \(\cvvv_3\) is commutative, we may, and do, suppose that \(\rho\) is.  Put \(\rho \cvvv_3 \sigma = (C_v)_{v \in \mc X_2}\).  Fix \(v \in \mc X_2\) with \(v \ne 0\).  For every pair \(v_1, v_2 \in \mc X_2\) with \(v_1 + v_2 = v\), at least one of \(A_{v_1}\) or \(B_{v_2}\) is constant, and so every component of \(A_{v_1} \cvvv_2 B_{v_2}\) equals \(\inf \set{\vec a_{v_1 i_1} + \vec b_{v_2 i_2}}{i_1, i_2 \in \mc X_1}\); so every component of \(C_v = \inf \set{A_{v_1} \cvvv_2 B_{v_2}}{v_1 + v_2 = v}\) equals \(\inf \sett{\vec a_{v_1 i_1} + \vec b_{v_2 i_2}}{\(v_1 + v_2 = v\) and \(i_1, i_2 \in \mc X_1\)}\).
\end{proof}

\begin{lem}
\label{lem:vee-assoc}
Preserve the notation of Definition \ref{defn:depth-ops}.  The operations \(\cvvv_1\), \(\cvvv_2\), and \(\matnotn{concave}\cvvv = \cvvv_3\) are associative.
\end{lem}

\begin{proof}
Associativity of \(\cvvv_1\) follows by applying Lemma \ref{lem:pre-assoc} with \((\mc X, \preceq) = (\mc X_1, \le)\) and \((S, \sstar, \le) = (S_0, +, \le)\); associativity of \(\cvvv_2\) then follows by applying Lemma \ref{lem:pre-assoc} again with \((\mc X, \preceq) = (\mc X_1, \ge)\) and \((S, \sstar, \le) = (S_1, \cvvv_1, \le)\); and associativity of \(\cvvv_3\) is then obvious, since, if we write \(\rho = (A_v)_{v \in \mc X_2}\), \(\sigma = (B_v)_{v \in \mc X_2}\), and \(\tau = (C_v)_{v \in \mc X_2}\), then, for every \(v \in \mc X_2\), the doubly indexed components of \(\rho \cvvv_3 (\sigma \cvvv_3 \tau)\) and \((\rho \cvvv_3 \sigma) \cvvv_3 \tau\) at \(v\) both equal \(\inf \set{A_{v_1} \cvvv_2 B_{v_2} \cvvv_2 C_{v_3}}{v_1 + v_2 + v_3 = v}\).
\end{proof}

\begin{cor}
\label{cor:matrix-vs-spec-vee}
Suppose that \(\mc X_1\) is \(\set{i \in \R}{0 \le i \le r}\).
Let \(A = (\vec a_i)_{i \in \mc X_1}\) and \(B = (\vec b_i)_{i \in \mc X_1}\) be doubly indexed vectors, and define depth matrices \(f\) and \(g\), in the sense of and with the notation of \xcite{spice:asymptotic}*{Definition \xref{defn:vGvr}}, by \(f\textsup n_i = \vec a_i\) and \(g\textsup n_i = \vec b_i\) for all \(i \in \mc X_1\), and \(f\textsup d_i = f\textsup c_i = g\textsup d_i = g\textsup c_i = (\infty, \dotsc, \infty)\) for all \(i \in \tR \cup \sset{-\infty}\) with \(i < 0\).

With the notation of \xcite{spice:asymptotic}*{Definition \xref{defn:concave-vee}}, we have that \((f \bowtie g)\supn\) equals \(A \cvvv_2 B\).  If we have \(A \le A \cvvv_2 A\), then also \((f \rtimes g)\supn\) equals \(A \cvvv_2 B\).  If we also have \(B \le B \cvvv_2 B\), then also \((f \vee g)\supn\) equals \(A \cvvv_2 B\).
\end{cor}

\subsection{Concave functions}
\label{subapp:concave}

Let \Root be a root system in some real vector space \(V\), and put \(\Weight = \Root \cup \sset0\).  Let \(f\) and \(f'\) be two functions \anonmap\Weight S.  Definition \ref{defn:concave-ops} recalls some operations on such functions from \xcite{spice:asymptotic}*{Lemma \xref{lem:der-master-comm}}.

\begin{defn}
\label{defn:concave-ops}
The functions \(f \bowtie f'\), \(f \rtimes f'\), and \(f \vee f'\) from \Weight to \(S\) are defined by
\begin{align*}
(f \bowtie f')(\alpha) & {}= \inf_{a + a' = \alpha} f(a) + f'(a'), \\
(f \rtimes f')(\alpha) & {}= \inf_{\sum_i a_i + a' = \alpha} \sum f(a_i) + f'(a'), \\
\intertext{and}
(f \vee f')(\alpha)    & {}= \inf_{\sum_i a_i + \sum_j a'_j = \alpha} \sum f(a_i) + \sum f'(a'_j)
\end{align*}
for all \(\alpha \in \Weight\).
\end{defn}

Let \(\vec\Root = (\Root^0 \subseteq \dotsb \subseteq \Root^{\ell - 1} \subseteq \Root^\ell = \Root)\) be a tower of closed subsystems of \(\Root\), and \(\lambda\) an element of \(\contra V\).

\begin{defn}
\label{defn:depth-to-concave}
If \(\tau = (t_{v i j})_{v, i, j}\) is a triply indexed vector, in the sense of Definition \ref{defn:depth-ops}, then we define \map{f_{\vec\Root, \lambda, \tau}}\Weight S by
\[
f_{\vec\Root, \lambda, \tau}(\alpha) = \sup \set{t_{v_\alpha i j_\alpha}}{i \in \mc X_1},
\]
where \(v_\alpha\) equals \(\pair\alpha\lambda\) and \(j_\alpha\) is the least index \(j\) such that \(\alpha\) belongs to \(\Weight^j\), for all \(\alpha \in \Weight\).  If \(\vec\Root\) and \(\lambda\) are understood, then we may just write \matnotn{ftau}{f_\tau}.
\end{defn}

The idea of Lemma \ref{lem:depth-vs-concave-vee} was the inspiration for \xcite{spice:asymptotic}*{Definition \xref{defn:concave-vee}}---although in fact, since we are now using depth specifications instead of depth matrices, we actually no longer need \xcite{spice:asymptotic}*{Definition \xref{defn:concave-vee}}.

\begin{lem}
\initlabel{lem:depth-vs-concave-vee}
Let \(\tau = (t_{v i j})_{v, i, j}\) and \(\tau' = (t'_{v i j})_{v, i, j}\) be triply indexed vectors.
\begin{enumerate}
\item\sublabel{concave}
If \(\tau\) is concave, in the sense of Definition \ref{defn:depth-ops}, then \(f_\tau\) is concave, in the sense of \cite{bruhat-tits:reductive-groups-1}*{\S6.4.3}.
\setcounter{tempenumi}{\value{enumi}}
\end{enumerate}
With the notation of Definition \ref{defn:depth-to-concave}, the following inequalities hold.
\begin{enumerate}
\setcounter{enumi}{\value{tempenumi}}
\item\sublabel{vvv} \(f_{\tau \cvvv \tau'} \le f_\tau \bowtie f_{\tau'}\).
\item\sublabel{vev} If \(\tau\) is concave, then also \(f_{\tau \cvvv \tau'} \le f_\tau \rtimes f_{\tau'}\).
\item\sublabel{vee} If \(\tau\) and \(\tau'\) are concave, then also \(f_{\tau \cvvv \tau'} \le f_\tau \vee f_{\tau'}\).
\end{enumerate}
\end{lem}

\begin{proof}
Write \((c_{v i j})_{v, i, j}\) for \(\tau \cvvv \tau'\), and \(f\), \(f'\), and \(F\) for \(f_\tau\), \(f_{\tau'}\), and \(f_{\tau \cvvv \tau'}\).

Suppose that \(\alpha, a, a' \in \Weight\) satisfy \(\alpha = a + a'\).  Put \(v_\alpha = \pair\alpha\lambda\), and let \(j_\alpha\) be the least index \(j\) such that \(\alpha\) belongs to \(\Weight^j\), and similarly for \(a\) and \(a'\).

We have that \(v_\alpha\) equals \(v_a + v_{a'}\).

Since the various \(\Root^j\) are closed, one of the inequalities \(j_\alpha \le j_a = j_{a'}\) or \(j_{a'} \le j_\alpha = j_a\) or \(j_a \le j_{a'} = j_\alpha\) holds.  Therefore, for every index \(i \in \mc X_1\), Remark \ref{rem:depth-ops} gives the inequality \(c_{v_\alpha i j_\alpha} \le t_{v_a i j_a} + t'_{v_{a'}i j_{a'}}\), which, by the definition of \(f\) and \(f'\), implies the inequality \(c_{v_\alpha i j_\alpha} \le f(a) + f'(a')\).  Taking the supremum over all \(i\) gives the inequality \(F(\alpha) \le f(a) + f'(a')\), and then taking the infimum over all \(a\) and \(a'\) gives the inequality \(F(\alpha) \le (f \bowtie f')(\alpha)\).  This establishes \locpref{vvv}, which, by \cite{yu:supercuspidal}*{Lemma 1.1}, establishes \locpref{concave}.

Now suppose that \(\tau\) is concave, so that we have \(\tau \le \tau'\).  Since \(\tau\) is concave, Lemma \ref{lem:vee-assoc} gives the inequality \(\tau \cvvv \tau' \le (\tau \cvvv \tau) \cvvv \tau' = \tau \cvvv (\tau \cvvv \tau')\), so that \locpref{vvv} gives the inequality \(F \le f \bowtie F\).

Suppose that \(\alpha\) and \(a'\) lie in \(\Weight\), and that \((a_i)_i\) is a 
sequence in \(\Weight\) such that \(\sum_i a_i + a'\) equals \(\alpha\).  We prove by induction on the length of \((a_i)_i\) that the inequality \(F(\alpha) \le \sum_i f(a_i) + f'(a')\) holds.  The base case, where \((a_i)_i\) 
is empty, is clear.  We thus may, and do, suppose that \((a_i)_i\) is non-empty.

If \(\alpha - a' = \sum_i a_i\) belongs to \(\Weight\), then \locpref{vvv} gives the inequality \(F(\alpha) \le (f \bowtie f')(\alpha) \le f(\alpha - a') + f'(a')\), which, since \(f\) is concave by \locpref{concave}, implies the inequality \(F(\alpha) \le \sum_i f(a_i) + f'(a')\).  Otherwise, there is some index \(i_0\) such that \(\alpha - a_{i_0} = \sum_{i \ne i_0} a_i + a'\) belongs to \(\Weight\).  Then we have by induction that the inequality \(F(\alpha - a_{i_0}) \le \sum_{i \ne i_0} f(a_i) + f'(a')\) holds, and we deduce the inequality
\[
F(\alpha) \le (f \bowtie F)(\alpha) \le f(a_{i_0}) + F(\alpha - a_{i_0}) \le \sum_i f(a_i) + f'(a').
\]
This completes the induction; and then taking the infimum over all non-empty \((a_i)_i\) and \(a'\) gives \locpref{vev}.

Now suppose that \(\tau\) and \(\tau'\) are both concave.  We proceed similarly.  Once again, \locpref{vvv} gives the inequalities \(F \le f \bowtie F\) and \(F \le F \bowtie f'\).

Suppose that \(\alpha\) belongs to \(\Weight\), and that \((a_i)_i\) and \((a'_j)_j\) are non-empty sequences in \(\Weight\) such that \(\sum_i a_i + \sum_j a'_j\) equals \(\alpha\).  We prove by induction on the sum of the lengths of \((a_i)_i\) and \((a'_j)_j\) that the inequality \(F(\alpha) \le \sum_i f(a_i) + \sum_j f'(a'_j)\) holds.  If \((a_i)_i\) \emph{or} \((a'_j)_j\) has length \(1\), then the desired inequality follows from \locpref{vev}.  This establishes the base case of the induction.

Thus we may, and do, assume that both \((a_i)_i\) and \((a'_j)_j\) have more than \(1\) term.  Upon switching \(\tau\) and \(\tau'\) if necessary, we may, and do, assume that there is some index \(i_0\) such that \(\alpha - a_{i_0} = \sum_{i \ne i_0} a_i + \sum_j a'_j\) belongs to \(\Weight\).  Then induction gives the inequality
\[
F(\alpha) \le f(a_{i_0}) + F(\alpha - a_{i_0}) \le f(a_{i_0}) + \Bigl(\sum_{i \ne i_0} f(a_i) + \sum_j f'(a'_j)\Bigr) = \sum_i f(a_i) + \sum_j f'(a'_j).
\]
This completes the induction, and establishes \locpref{vee}.
\end{proof}
}

\Printindex{\jobname-terminology}{Index of terminology}
\Printindex{\jobname-notation}{Index of notation}

\end{document}